\documentclass{these}
\usepackage[english,francais]{babel}
\usepackage{times}
\usepackage{amsmath,amsfonts,amssymb,amsopn,amscd,amsthm}
\usepackage{pdfpages}
\usepackage[french]{nomencl}
\usepackage{mathrsfs}
\usepackage{euscript}
\usepackage{graphicx,xspace}
\usepackage{epsfig}
\usepackage{enumerate} 
\usepackage{enumitem}
\usepackage{xcolor}
\usepackage[all]{xypic}
\usepackage{url}

\usepackage[T1]{fontenc}
\usepackage[utf8]{inputenc}
\usepackage{hyperref}
\usepackage{psfrag}
\usepackage{bm}
\usepackage{youngtab}
\usepackage{tikz}

\usetikzlibrary{snakes}
\usepackage[all]{xy}
\usepackage{varioref}

\theoremstyle{plain}
\newtheorem{definition}{Définition}[chapter]
\newtheorem{theoreme}[definition]{Théorème}
\newtheorem{que}[definition]{Question}
\newtheorem{prop}[definition]{Proposition}
\newtheorem{ex}[definition]{Exemple}
\newtheorem{c-ex}[definition]{Contre-exemple}
\newtheorem{cor}[definition]{Corollaire}
\newtheorem{lem}[definition]{Lemme}

\theoremstyle{definition}

\theoremstyle{remark}
\newtheorem*{dem}{Démonstration}
\newtheorem*{rem}{Remarque}
\newtheorem{obs}[subsection]{Observation}
\newtheorem*{notation}{Notation}

\renewcommand\thesection{\thechapter.\arabic{section}}

\newcommand{\hecke}{$(\mathcal{S}_{2n},\mathcal{B}_n)$}
\newcommand{\Hecke}{$\mathbb{C}[\mathcal{B}_n\setminus \mathcal{S}_{2n}/\mathcal{B}_n]$}
\date{}
\def\pr{\ast}
\DeclareUnicodeCharacter{03C7}{\chi}
\DeclareUnicodeCharacter{B2}{^{2}}
\DeclareMathOperator{\supp}{supp}
\DeclareMathOperator{\type-cyclique}{type-cyclique}
\DeclareMathOperator{\coset-type}{coset-type}
\DeclareMathOperator{\ct}{ct}
\DeclareMathOperator{\SC}{SC}
\DeclareMathOperator{\NCSym}{NCSym}
\DeclareMathOperator{\I}{I}
\DeclareMathOperator{\tr}{tr}
\DeclareMathOperator{\sign}{sign}

\DeclareMathOperator{\Ind}{Ind}
\DeclareMathOperator{\Res}{Res}
\DeclareMathOperator{\End}{End}
\DeclareMathOperator{\Mat}{Mat}

\DeclareMathOperator{\diag}{diag}

\DeclareMathOperator{\id}{id}
\title{Polynomialité des coefficients de structure des algèbres de doubles-classes}
\author{Omar Tout}
\date{24 novembre 2014}
\labo{Bordelais de Recherche en Informatique}
\school{de Mathématiques et Informatique de Bordeaux}
\speciality{Informatique}
\university{l'Université de Bordeaux}{Université de Bordeaux}
\advisor[M]{Jean-Christophe Aval et Valentin Féray} 
\jury{  
\jurymember[M]{Florent Hivert}{Rapporteur}\\
\jurymember[M]{Alain Goupil}{Rapporteur}\\
\jurymember[M]{Valentin Féray}{Directeur de thèse}\\
\jurymember[M]{Jean-Christophe Aval}{Directeur de thèse}\\
\jurymember[Mme]{Mireille Bousquet-Mélou}{Examinatrice}\\
\jurymember[Mme]{Ekaterina Vassilieva}{Examinatrice}\\
\jurymember[M]{Robert Cori}{Examinateur}\\
\jurymember[M]{Gilles Schaeffer}{Examinateur}\\
\jurymember[M]{Pierre-Loïc Méliot}{Examinateur}\\
}
\makenomenclature
\begin{document}
\frontmatter
\maketitle
\chapter*{État de l'art}

Étant donnée une algèbre de dimension finie avec une base fixée, le produit de deux éléments de base peut s'écrire comme combinaison linéaire des éléments de la base. Les coefficients qui apparaissent dans un tel développement sont appelés \textit{coefficients de structure}. Le calcul de ces coefficients est essentiel, car, par linéarité, ils permettent de calculer tous les produits dans l'algèbre considérée. Mais il est difficile, même dans des cas particuliers d'algèbres, de donner des formules explicites pour ces coefficients. Dans cette thèse, on s'intéresse aux coefficients de structure et plus particulièrement à leur dépendance en $n$ dans le cadre de suites d'algèbres  particulières.

Les coefficients de structure des centres des algèbres de groupe et des algèbres de doubles-classes ont été les plus étudiés dans la littérature. En effet, l'étude des coefficients de structure des centres des algèbres de groupe est reliée à la théorie des représentations des groupes finis par le théorème de Frobenius, voir \cite[Lemme 3.3]{JaVi90} et l'appendice de Zagier dans \cite{lando2004graphs}. On présente ce théorème à la fin de la Section \ref{sec:Th_de_rep}. Ce théorème donne, dans le cas général d'un centre d'une algèbre de groupe, une formule exprimant les coefficients de structure en fonction des caractères irréductibles.

Une paire $(G,K),$ où $G$ est un groupe fini et $K$ un sous-groupe de $G,$ est appelée paire de Gelfand si l'algèbre de doubles-classes $\mathbb{C}[K\setminus G/K]$ est commutative, voir \cite[chapitre VII.1]{McDo}. Si $(G,K)$ est une paire de Gelfand, l'algèbre des fonctions constantes sur les doubles-classes de $K$ dans $G$ possède une base particulière dont les éléments sont appelés les \textit{fonctions sphériques zonales} de la paire $(G,K).$ On rappelle la théorie des paires de Gelfand et des fonctions sphériques zonales dans la Section \ref{sec:paire_des_Gelfand}. Les algèbres de doubles-classes associées aux paires de Gelfand généralisent les centres des algèbres de groupe et les fonctions sphériques zonales généralisent les caractères irréductibles.

Un théorème similaire à celui de Frobenius reliant les coefficients de structure des algèbres de doubles-classes à la théorie des fonctions sphériques zonales des paires de Gelfand a été établi et utilisé dans des cas particuliers, voir \cite{jackson2012character} et \cite{GouldenJacksonLocallyOrientedMaps}. L'auteur de cette thèse n'a pas réussi à trouver le théorème traitant le cas général dans la littérature. On propose donc à la fin de la Section \ref{sec:Gen_fct_sph_zon} un tel théorème pour une paire de Gelfand générale.

Le cas du centre de l'algèbre du groupe symétrique $Z(\mathbb{C}[\mathcal{S}_{n}])$ est particulièrement intéressant et plusieurs auteurs l'ont étudié. Le calcul des coefficients de structure de $Z(\mathbb{C}[\mathcal{S}_{n}])$ nécessite l'étude du type cyclique d'un produit de permutations, ce qui est difficile en général, voir par exemple les papiers \cite{bertram1980decomposing}, \cite{boccara1980nombre}, \cite{Stanley1981255}, \cite{walkup1979many}, \cite{GoupilSchaefferStructureCoef} et \cite{JaVi90} qui traitent de cas particuliers de ces coefficients. Plusieurs auteurs, voir \cite{Stanley1981255}, \cite{JaVi90}, \cite{jackson1987counting}, \cite{GoupilSchaefferStructureCoef}, ont utilisé les caractères irréductibles pour calculer ces coefficients. Malgré tout ces efforts, il n'y a pas de formule générale pour calculer les coefficients de structure du centre de l'algèbre du groupe symétrique et il semble que c'est un objectif hors de portée. On illustre la difficulté du calcul des coefficients de structure de $Z(\mathbb{C}[\mathcal{S}_{n}])$ par un exemple dans la Section \ref{sec:calc_coeff_sym}.

L'algèbre de doubles-classes de la paire $(\mathcal{S}_{2n}, \mathcal{B}_n),$ où $\mathcal{B}_n$ est le sous-groupe hyperoctaédral de $\mathcal{S}_{2n},$ introduit par James en 1961 dans \cite{James1961}, possède beaucoup de propriétés similaires à celles du centre de l'algèbre du groupe symétrique, voir \cite{Aker20122465}. La plus connue parmi elles est que ces deux algèbres possèdent une base indexée par les partitions de $n,$ voir les Sections \ref{sec:base_du_centre_de_S_n} et \ref{sec:base_de_Hecke}. De plus, les coefficients de structure de ces algèbres sont liées aux fonctions symétriques et aux graphes plongés dans des surfaces comme nous l'expliquons dans les paragraphes suivants.

D'après la formule de Frobenius, voir par exemple \cite{sagan2001symmetric}, les caractères irréductibles du groupe symétrique apparaissent dans le développement des fonctions de Schur dans la base des fonctions puissance. Cela relie l'étude des coefficients de structure du centre de l'algèbre du groupe symétrique à celle de la théorie des fonctions symétriques. On présente cette formule à la fin de la Section \ref{sec:tab_young}.

La paire $(\mathcal{S}_{2n}, \mathcal{B}_n)$ est une paire de Gelfand, voir \cite[Section VII.2]{McDo}. L'étude des coefficients de structure de l'algèbre de doubles-classes de la paire $(\mathcal{S}_{2n}, \mathcal{B}_n)$ est reliée elle aussi à celle des fonctions symétriques. En effet, les fonctions sphériques zonales de la paire $(\mathcal{S}_{2n}, \mathcal{B}_n)$ apparaissent dans le développement des polynômes zonaux dans la base des fonctions puissance. Les polynômes zonaux sont des spécialisations des polynômes de Jack, définis par Jack dans les articles \cite{jack1970class} et \cite{jack1972xxv}, qui forment une base de l'algèbre des fonctions symétriques.

En 1975, Cori a montré dans sa thèse, voir \cite{CoriHypermaps}, que les coefficients de structure du centre de l'algèbre du groupe symétrique comptent le nombre de graphes plongés dans des surfaces orientées (couramment appelés cartes) quand on impose certaines conditions de degré aux graphes. Ce résultat est aussi expliqué dans le livre \cite{lando2004graphs} de Lando et Zvonkin et l'article \cite{jackson1990character} de Jackson et Visentin. Dans la Section \ref{sec:int_com}, on présente ce résultat.

Il existe un résultat similaire pour l'algèbre de doubles-classes de la paire $(\mathcal{S}_{2n}, \mathcal{B}_n).$ D'après un résultat de Goulden et Jackson dans \cite{GouldenJacksonLocallyOrientedMaps} que l'on présente dans la Section \ref{sec:coef_surf_non-orie}, ces coefficients de structure comptent le nombre de cartes plongées dans des surfaces non nécessairement orientables (avec les mêmes conditions de degré). 

Un tel lien entre coefficients de structure et graphes existe aussi dans d'autres cas que le centre de l'algèbre du groupe symétrique et l'algèbre de doubles-classes de la paire $(\mathcal{S}_{2n}, \mathcal{B}_n).$ Par exemple, les coefficients de structure de l'algèbre de doubles-classes de la paire $(\mathcal{S}_n\times \mathcal{S}_{n-1}^{opp},\diag(\mathcal{S}_{n-1}))$ sont reliés à des graphes particuliers appelés dipoles, voir Jackson et Sloss \cite{Jackson20121856}. Pour plus de détails sur cette algèbre, le lecteur peut regarder les articles \cite{brender1976spherical}, \cite{strahov2007generalized} et \cite{jackson2012character}. Dans \cite{strahov2007generalized}, Strahov montre que les fonctions sphériques zonales associées à cette algèbre possèdent des propriétés analogues à celles des caractères irréductibles du groupe symétrique.


\bigskip

On a expliqué jusqu'ici pourquoi l'étude des coefficients de structure est intéressant. Dans la suite, on discute les résultats obtenus dans cette direction. 

Comme nous l'avons mentionné ci-dessus, le calcul, d'une façon directe ou bien en utilisant les caractères irréductibles, des coefficients de structure du centre de l'algèbre du groupe symétrique est un problème difficile. Il en est de même pour le calcul des coefficients de structure de l'algèbre de doubles-classes de la paire $(\mathcal{S}_{2n}, \mathcal{B}_n),$ voir \cite{bernardi2011counting}, \cite{morales2011bijective} et \cite{vassilieva2012explicit} pour des résultats partiels sur le sujet. Dans cette thèse, on s'intéresse à un problème plus facile que le calcul des coefficients de structure, celui de la dépendance en $n$ de ces coefficients. On donne, sous certaines conditions, la forme des coefficients de structure dans le cas général des centres d'algèbres de groupe et d'algèbres de doubles-classes. On cite les résultats les plus connus sur ce sujet dans les paragraphes suivants.

En 1954, Farahat et Higman ont montré dans \cite{FaharatHigman1959} une propriété de polynomialité en $n$ pour les coefficients de structure du centre de l'algèbre du groupe symétrique. En 1999, Ivanov et Kerov ont donné dans \cite{Ivanov1999} une preuve combinatoire de ce résultat en introduisant les permutations partielles. Ces résultats sont présentés dans les Sections \ref{sec:th_Far_Hig} et \ref{sec:approche_Ivanov_Kerov}.

Dans \cite{Aker20122465}, Aker et Can donnent une propriété de polynomialité pour les coefficients de structure de l'algèbre de doubles-classes de la paire $(\mathcal{S}_{2n}, \mathcal{B}_n)$ similaire à celle de Farahat et Higman. Malheureusement leur preuve de polynomialité contenait une erreur, \cite{Can}. Récemment, dans \cite{2014arXiv1407.3700B}, Can et {\"O}zden ont proposé une correction de cette preuve. Ce résultat de polynomialité a été retrouvé par Do{\l}{\c e}ga et Féray dans \cite{2014arXiv1402.4615D}, d'une manière indirecte en utilisant les polynômes de Jack. 

Récemment, Méliot a trouvé, voir \cite{meliot2013partial}, une propriété de polynomialité, similaire à celle de Farahat et Higman, dans le cas du centre de l'algèbre du groupe des matrices inversibles à coefficients dans un corps fini. 
Il faut mentionner que Méliot a déjà prouvé dans \cite{meliot2010products} une propriété de polynomialité pour les coefficients de structure du centre de l'algèbre de Iwahori–Hecke du groupe symétrique. Ce dernier résultat avait été conjecturé par Francis et Wang dans \cite{francis1992centers}.


\bigskip

La propriété de polynomialité des coefficients de structure est importante pour l'étude asymptotique de certains objets aléatoires. Dans le cas du centre de l'algèbre du groupe symétrique, elle intervient dans l'étude des moments des caractères irréductibles normalisés. Ces derniers sont vus comme des variables aléatoires définies sur l'ensemble des représentations irréductibles du groupe symétrique muni de la mesure de Plancherel, voir \cite{hora2007quantum} et \cite{ivanov2002olshanski}. Cette étude permet à Ivanov et Olshanski de retrouver dans \cite{ivanov2002olshanski} un résultat de convergence de la forme du bord d'un diagramme de Young obtenu en 1977 par Logan et Shepp dans \cite{logan1977variational} et indépendamment par Vershik et Kerov dans \cite{vervsik1977asymptotic}. On rappelle ces résultats et on explique le lien avec la polynomialité des coefficients de structure dans la Section \ref{sec:appl_diag_de_young}.

Dans \cite{2014arXiv1402.4615D}, Do{\l}{\c e}ga et F{\'e}ray ont illustré l'importance de la propriété de polynomialité dans le cas de l'algèbre de doubles-classes de la paire $(\mathcal{S}_{2n}, \mathcal{B}_n)$ en étudiant le comportement asymptotique des diagrammes de Young avec une déformation de la mesure de Plancherel. 

Des résultats probabilistes similaires dans le cas du centre de l'algèbre du groupe des matrices inversibles ont été établis, voir les articles \cite{fulman2008convergence} et \cite{dudko2008asymptotics}. 

\bigskip

Dans cet état de l'art, nous avons expliqué l'importance de l'étude des coefficients de structure et la difficulté du calcul de ces coefficients. Par ailleurs, on a vu pourquoi une propriété de polynomialité pour ces coefficients peut être particulièrement intéressante. On s'intéresse dans cette thèse à la dépendance en $n$ des coefficients de structure dans le cas général des algèbres de doubles-classes. Les résultats obtenus sont présentés dans la partie <<Résumé>>.
\chapter*{State of the art}

Given an algebra of finite dimension with a fixed basis, the product of two basis elements can be written as a linear combination of basis elements. The coefficients that appear in this expansion are called \textit{structure coefficients}. Computing these coefficients is essential because, by linearity, they allow us to compute all products in the considered algebra. However, it is difficult, even in particular cases of algebras, to give explicit formulas for these coefficients. In this thesis we studied the structure coefficients and especially their dependence on $n$ in the case of a sequence of special algebras.

The structure coefficients of centers of group algebras and of double-class algebras had been the most studied in the literature. In fact, the study of the structure coefficients of centers of group algebras is related to the representation theory of finite groups by Frobenius theorem, see \cite{JaVi90}[Lemma 3.3] and the appendix of Zagier in \cite{lando2004graphs}. We give this theorem at the end of Section \ref{sec:Th_de_rep}. This theorem expresses, in the general case of centers of group algebras, the structure coefficients in terms of irreducible characters. 

A pair $(G,K),$ where $G$ is a finite group and $K$ a sub-group of $G,$ is said to be a Gelfand pair if its associated double-class algebra $\mathbb{C}[K\setminus G/K]$ is commutative, see \cite[chapter VII.1]{McDo}. If $(G,K)$ is a Gelfand pair, the algebra of constant functions over the $K$-double-classes in $G$ has a particular basis, whose elements are called \textit{zonal spherical functions} of the pair $(G,K).$ We recall the theory of Gelfand pairs and zonal spherical functions in Section \ref{sec:paire_des_Gelfand}. The double-class algebras associated to Gelfand pairs generalize the centers of group algebras and in this case the zonal spherical functions generalize the irreducible characters.

A similar theorem which links the structure coefficients of double-class algebras with the theory of zonal spherical functions of Gelfand pairs has been established in particular cases, see \cite{jackson2012character} and \cite{GouldenJacksonLocallyOrientedMaps}. The author cannot find in the literature a similar to Frobenius theorem's in the general case of double-class algebras of Gelfand pairs. We present in the end of Section \ref{sec:Gen_fct_sph_zon} such a theorem for a general Gelfand pair.

The case of the center of the symmetric group algebra $Z(\mathbb{C}[\mathcal{S}_{n}])$ is particularly interesting and many authors have studied it in details. To compute the structure coefficients of $Z(\mathbb{C}[\mathcal{S}_{n}])$ one should study the cycle-type of the product of two permutations which can be quite difficult, see for example the papers \cite{bertram1980decomposing}, \cite{boccara1980nombre}, \cite{Stanley1981255}, \cite{walkup1979many}, \cite{GoupilSchaefferStructureCoef} and \cite{JaVi90} which deal with particular cases of these coefficients. Despite all these efforts, there is no general formula for the structure coefficients of the center of the symmetric group algebra and it seems that it is an elusive goal. We show the difficulty of computing the structure coefficients of $Z(\mathbb{C}[\mathcal{S}_{n}])$ with an example in Section \ref{sec:calc_coeff_sym}. 

The double-class algebra of the pair $(\mathcal{S}_{2n}, \mathcal{B}_n),$ where $\mathcal{B}_n$ is the Hyperoctahedral sub-group of $\mathcal{S}_{2n},$ was introduced by James in 1961 in \cite{James1961}. It has a long list of properties similar to that of the center of the symmetric group algebra, see \cite{Aker20122465}. The most known among them is that they both have a basis indexed by the set of partitions of $n,$ see Sections \ref{sec:base_du_centre_de_S_n} and \ref{sec:base_de_Hecke}. In addition, the structure coefficients of these algebras are related to symmetric functions and graphs embedded into surfaces. This is explained in the next paragraphs.

By Frobenius formula, see \cite{sagan2001symmetric}, the irreducible characters of the symmetric group appear in the expansion of Schur functions in the basis of power-sums. This relates the study of the structure coefficients of the center of the symmetric group algebra to the theory of symmetric functions. We present this formula at the end of Section \ref{sec:tab_young}.

The pair $(\mathcal{S}_{2n}, \mathcal{B}_n)$ is a Gelfand pair, see \cite[Section VII.2]{McDo}. The study of the structure coefficients of the double-class algebra of the pair $(\mathcal{S}_{2n}, \mathcal{B}_n),$ is also related to the theory of symmetric functions since the zonal spherical functions of the pair $(\mathcal{S}_{2n}, \mathcal{B}_n)$ appear in the development of zonal polynomials in terms of power-sums. The zonal polynomials are specialization of Jack polynomials, defined by Jack in \cite{jack1970class} and \cite{jack1972xxv}, which form a basis for the algebra of symmetric functions. 

In 1975, Cori proved in his thesis, see \cite{CoriHypermaps}, that the structure coefficients of the center of the symmetric group algebra count the number of graphs embedded into oriented surfaces (called maps) with some degree conditions. This result can also be found in the book \cite{lando2004graphs} of Lando and Zvonkin and the paper \cite{jackson1990character} of Jackson and Visentin. In Section \ref{sec:int_com}, we recall this result by a simple presentation.

The double-class algebra of the pair $(\mathcal{S}_{2n}, \mathcal{B}_n)$ has a similar combinatorial interpretation. Its structure coefficients count the number of graphs embedded into non-oriented surfaces, with similar degree conditions, according to a result of Goulden and Jackson given in \cite{GouldenJacksonLocallyOrientedMaps}. We present this result in Section \ref{sec:coef_surf_non-orie}.


The relation between structure coefficients and graphs is not limited to the cases of the center of the symmetric group algebra and the double-class algebra of the pair $(\mathcal{S}_{2n}, \mathcal{B}_n).$ For example, the structure coefficients of the double-class algebra of the pair $(\mathcal{S}_n\times \mathcal{S}_{n-1}^{opp},\diag(\mathcal{S}_{n-1}))$ are related to particular graphs called dipoles according to Jackson and Sloss in \cite{Jackson20121856}. For more details about this algebra, the reader is invited to see the papers \cite{brender1976spherical}, \cite{strahov2007generalized} and \cite{jackson2012character}. In \cite{strahov2007generalized}, Strahov shows that the zonal spherical functions of this pair have analogous properties to irreducible characters of the symmetric group. 

\bigskip

We have shown up to now the importance of the study of structure coefficents. In what follows, we discuss the most known results in this direction. 

As already stated, the computation, by a direct way or by using irreducible characters, of the structure coefficients of the center of the symmetric group algebra is difficult. This difficulty of computing structure coefficients also occurs in the case of the double-class algebra of the pair $(\mathcal{S}_{2n}, \mathcal{B}_n),$ see \cite{bernardi2011counting}, \cite{morales2011bijective} and \cite{vassilieva2012explicit}. In this thesis, we are interested in a simpler problem than the computation of structure coefficients, which is their dependence on $n.$  We give, under some conditions, the form of structure coefficients in the general case of centers of group algebras and double-class algebras. We mention the most known results in this subject in the next paragraphs.

In 1954, Farahat and Higman published the paper \cite{FaharatHigman1959} in which they prove a theorem  that shows a polynomiality property in $n$ for the structure coefficients of the center of the symmetric group algebra. In 1999, Ivanov and Kerov gave in \cite{Ivanov1999} a combinatorial proof to Farahat and Higman's theorem. They introduce and use in this paper combinatorial objects that they call partial permutations. We give details about these results  in Sections \ref{sec:th_Far_Hig} and \ref{sec:approche_Ivanov_Kerov}.

In \cite{Aker20122465}, Aker and Can give a polynomiality property for the structure coefficients of the double-class algebra of the pair $(\mathcal{S}_{2n}, \mathcal{B}_n)$ similar to that obtained by Farahat and Higman in the case of the center of the symmetric group algebra. Unfortunately, their proof of polynomiality contains an error \cite{Can}. Recently, in \cite{2014arXiv1407.3700B}, Can and {\"O}zden have suggested a correction to this proof. This result of polynomiality was found also by Do{\l}{\c e}ga and Féray, in \cite{2014arXiv1402.4615D}, by an indirect approach using Jack polynomials. 

Recently, Méliot has found, see \cite{meliot2013partial}, a polynomiality property, similar to that of Farahat and Higman, for the structure coefficients of the center of the group algebra of invertible matrices with coefficients in a finite field. 
It is worth mentioning that Méliot has already proved in \cite{meliot2010products} a polynomiality property for the structure coefficients of the Iwahori–Hecke algebra of the symmetric group. This latter result was conjectured by Francis and Wang in \cite{francis1992centers}.

\bigskip

The polynomiality property for the structure coefficients is important for the study of the asymptotic  behaviour of combinatorial objects. In the case of the center of the symmetric group algebra, this property is used in the study of moments of normalised irreducible characters viewed as random variables defined over the set of irreducible representations of the symmetric group with the Plancherel measure, see \cite{hora2007quantum} and \cite{ivanov2002olshanski}. This allows Ivanov and Olshanski to give a new proof in \cite{ivanov2002olshanski} of a result of convergence of the shape of the border of a Young diagram obtained in 1977 by Logan and Shepp, see \cite{logan1977variational}, and independently by Vershik and Kerov in \cite{vervsik1977asymptotic}. We recall these results and explain the link with the polynomiality property of the structure coefficients in Section \ref{sec:appl_diag_de_young}.

In \cite{2014arXiv1402.4615D}, Do{\l}{\c e}ga and F{\'e}ray illustrate the importance of the polynomiality property in the case of the double-class algebra of the pair $(\mathcal{S}_{2n}, \mathcal{B}_n)$ while studying the asymptotic behaviour of Young diagrams with a deformation of the Plancherel measure.

There exists similar probabilistic results in the case of the center of the group algebra of invertible matrices with coefficients in a finite field, see the papers \cite{fulman2008convergence} and \cite{dudko2008asymptotics}. 

\bigskip

In this state of art section, we have presented the importance to study the structure coefficients and the difficulty of computing these coefficients. In addition, we have seen why a polynomiality property for these coefficients is of a special interest. We are interested, in this thesis, in the dependence on $n$ of the structure coefficients in the general case of double-class algebras. The results obtained are presented in the <<Abstract>> section.
\resume{
Cette thèse est consacrée à l'étude des coefficients de structure d'algèbres de doubles-classes --rappelons que cela contient le cas des centres d'algèbres de groupe. On s'intéresse particulièrement à la dépendance en $n$ de ces coefficients dans le cas d'une suite d'algèbres de doubles-classes.

\bigskip

Le premier chapitre est un chapitre de généralités consacré à l'étude des coefficients de structure dans le cas général des centres d'algèbres de groupes finis et des algèbres de doubles-classes. On fixe dans ce chapitre les définitions et les notations utilisées tout au long de cette thèse. On présente aussi la théorie des représentations des groupes finis et son lien avec les coefficients de structure.

On montre à la fin de ce chapitre que l'étude des coefficients de structure des algèbres de doubles-classes est reliée à la théorie des paires de Gelfand et aux fonctions sphériques zonales en donnant un théorème similaire à celui de Frobenius. Ce théorème exprime les coefficients de structure de l'algèbre de doubles-classes $\mathbb{C}[K\setminus G/K],$ associée à une paire de Gelfand $(G,K),$ en fonction des fonctions sphériques zonales de la paire $(G,K).$ Il apparaît ici pour la première fois -- selon les connaissances de l'auteur -- dans un cadre général. Des cas particuliers de ce théorème se trouvent dans les articles \cite{jackson2012character} et \cite{GouldenJacksonLocallyOrientedMaps}.

\bigskip

Dans le deuxième chapitre on s'intéresse aux coefficients de structure du centre de l'algèbre du groupe symétrique. On donne le lien entre ces coefficients et les cartes dessinées sur des surfaces orientées et on présente le résultat de polynomialité de Farahat et Higman, voir \cite{FaharatHigman1959}. On détaille la méthode combinatoire d'Ivanov et Kerov donnée dans \cite{Ivanov1999} pour prouver ce résultat. On définit l'algèbre des fonctions symétriques dans ce chapitre et on rappelle que la méthode d'Ivanov et Kerov fait apparaître une algèbre isomorphe à celle des fonctions symétriques décalées d'ordre $1$ définie dans l'article \cite{okounkov1997shifted} d'Okounkov et Olshanski.

\bigskip

Le but du troisième chapitre de cette thèse est de montrer d'une manière combinatoire une propriété de polynomialité, similaire à celle obtenue par Farahat et Higman, pour les coefficients de structure de l'algèbre de doubles-classes de la paire $(\mathcal{S}_{2n}, \mathcal{B}_n).$ Ce résultat a été publié dans \cite{tout2013structure} et \cite{toutarxiv} en est une version plus détaillée. 

On a suivi l'approche d'Ivanov et Kerov dans \cite{Ivanov1999} pour arriver à notre résultat. On utilise des objets combinatoires qu'on appelle bijections partielles, qui jouent un rôle analogue à celui des permutations partielles de \cite{Ivanov1999}. Contrairement au cas des permutations partielles, la preuve de l'associativité du produit défini entre les bijections partielles est loin d'être évidente car ce produit est défini comme "moyenne" de bijections partielles. 

Pour arriver à notre résultat, on construit une algèbre universelle qui se projette sur l'algèbre de doubles-classes de la paire $(\mathcal{S}_{2n}, \mathcal{B}_n)$ pour tout $n$ et une relation avec l'algèbre des fonctions symétriques décalées d'ordre $2$ est établie. Comme dans \cite{Ivanov1999} pour le centre de l'algèbre du groupe symétrique, on obtient des majorations des degrés des coefficients de structure (comme polynôme en $n$) en définissant des filtrations sur cette algèbre.

\bigskip

Une généralisation concernant les algèbres de doubles-classes est donnée dans le quatrième chapitre. On présente dans ce chapitre un cadre général regroupant la propriété de polynomialité des coefficients de structure du centre de l'algèbre du groupe symétrique et l'algèbre de doubles-classes de la paire $(\mathcal{S}_{2n}, \mathcal{B}_n).$

Nous appliquons aussi notre généralisation à des algèbres de doubles-classes apparaissant dans la littérature. Par exemple, on montre que notre généralisation s'applique dans le cas de l'algèbre de doubles-classes de $\diag(\mathcal{S}_{n-1})$ dans $\mathcal{S}_n\times \mathcal{S}_{n-1}^{opp},$ et on donne une propriété de polynomialité pour les coefficients de structure dans ce cas. Ces derniers possèdent une interprétation combinatoire, voir \cite{Jackson20121856}. Le lecteur peut voir encore \cite[Section 1.3]{strahov2007generalized} dans lequel Strahov donne une liste de paires de Gelfand auxquelles notre généralisation pourrait s'appliquer.

Une autre application de notre généralisation est le centre de l'algèbre du groupe hyperoctaédral. Le lecteur peut se référer aux articles \cite{geissinger1978representations} et \cite{stembridge1992projective} pour plus de détail sur cette algèbre. Encore une fois, notre généralisation implique une propriété de polynomialité pour les coefficients de structure dans ce cas.

Malheureusement notre cadre général n'inclut pas le cas du centre de l'algèbre du groupe des matrices inversibles à coefficients dans un corps fini, ce qui nous permettrait de retrouver le résultat de Méliot \cite{meliot2013partial}. 

Un autre cas important est celui des super-classes des groupes uni-triangulaires. Ces objets sont le sujet de beaucoup de travail de recherche, voir \cite{andre2013supercharacters}, \cite{UnitriangulargroupAndre}, \cite{diaconis2008supercharacters} et \cite{yan2001representation}, et une propriété de polynomialité dans ce cas serait intéressante. En fait, les coefficients de structure des super-classes des matrices uni-triangulaires peuvent être vus comme les coefficients de structure d'une algèbre particulière de doubles-classes, voir la Section \ref{super-classe-matr-uni} où on explique ce fait. Malheureusement notre cadre général ne contient pas non plus ce cas. Un objectif important, dans notre futur travail de recherche, est d'essayer de modifier ce cadre afin qu'il puisse s'appliquer à ces deux algèbres.
}{

In this thesis we study the structure coefficients of double-class algebras --recall that this includes the case of centers of group algebras. Particularly, we are interested in the dependence on $n$ of these coefficients in the case of a sequence of double-class algebras.

\bigskip

The first chapter is a chapter of generalities dedicated to the study of the structure coefficients in the general cases of centers of group algebras and double-class algebras. We give in this chapter all the definitions and notations which we use throughout this thesis. We recall also the representation theory of finite groups and its link with structure coefficients.

We show at the end of this chapter that the study of the structure coefficients of double-class algebras is related to the theory of Gelfand pairs and zonal spherical functions by giving, in the case of Gelfand pairs, a theorem similar to that of Frobenius. This theorem writes the structure coefficients of the double-class algebra $\mathbb{C}[K\setminus G/K],$ associated to a Gelfand pair $(G,K),$ in terms of zonal spherical functions of the pair $(G,K).$ According to the author's knowledge, this theorem for the general case of Gelfand pairs appears for the first time in this manuscript. Particular cases had been treated in \cite{jackson2012character} by Jackson and Sloss and in \cite{GouldenJacksonLocallyOrientedMaps} by Goulden and Jackson.

\bigskip

In the second chapter, we consider the case of the structure coefficients of the center of the symmetric group algebra. We show how these coefficients are related to graphs embedded into orientable surfaces. We also recall Farahat and Higman's theorem in \cite{FaharatHigman1959} about the polynomiality property of the structure coefficients of the center of the symmetric group algebra. We study in details the combinatorial proof of this theorem given by Ivanov and Kerov in \cite{Ivanov1999}. This method uses a universal algebra which projects to the center of the symmetric group algebra for each $n.$ It turns out that this algebra is isomorphic to the algebra of $1$-shifted symmetric functions defined by Okounkov and Olshanski in \cite{okounkov1997shifted}.

\bigskip  

The goal of the third chapter of this thesis is to prove by a combinatorial way a polynomiality property, similar to that given by Farahat and Higman, for the structure coefficients of the double-class algebra of the pair $(\mathcal{S}_{2n}, \mathcal{B}_n).$ This result has already been published, see \cite{tout2013structure}, and a complete proof of it can be found in \cite{toutarxiv}.

We follow Ivanov and Kerov's approach in \cite{Ivanov1999} to achieve our result. We define new combinatorial objects that we call partial bijections. These objects are analogues in the case of the double-class algebra of the pair $(\mathcal{S}_{2n}, \mathcal{B}_n)$ of the partial permutations defined by Ivanov and Kerov in the case of the center of the symmetric group algebra. In contrast to partial permutations, the proof of associativity of the product defined between partial bijections is far from being obvious since the product is defined as an "average" of partial bijections. 

To prove our result, we build a universal algebra which projects on the double-class algebra of the pair $(\mathcal{S}_{2n}, \mathcal{B}_n)$ for each $n.$ We prove that a distinguished sub-algebra of this universal algebra is isomorphic to the algebra of $2$-shifted symmetric functions. In addition, using some filtrations on this universal algebra, we are able to give upper bounds to the degree of the structure coefficients (seen as polynomials in $n$) in this case.  

\bigskip


The fourth and final chapter is dedicated to build a general framework for the polynomiality property of the structure coefficients in the case of double-class algebras. This framework contains both polynomiality properties of the center of the symmetric group algebra and of the double-class algebra of the pair $(\mathcal{S}_{2n}, \mathcal{B}_n).$

Our generalization is particularly interesting for some double-class algebras which appear in the literature. For example, we show that our general framework contains the double-class algebra of $\diag(\mathcal{S}_{n-1})$ in $\mathcal{S}_n\times \mathcal{S}_{n-1}^{opp},$ and we give a polynomiality property for the structure coefficients in this case. These latter have a combinatorial interpretation using special graphs, see \cite{Jackson20121856}. The reader can also have a look to \cite[Section 1.3]{strahov2007generalized} in which Strahov gives a list of Gelfand pairs that our generalisation may be applied to.

Another application of our generalisation is the case of the center of the Hyperoctahedral group algebra. The reader is invited to see the papers \cite{geissinger1978representations} and \cite{stembridge1992projective} for details about this algebra. Once again, our generalisation implies a polynomiality property for the structure coefficients in this case.

Unfortunately, our general framework does not contain the case of the center of the group of invertible matrices with coefficients in a finite field algebra, which could allow us to reobtain Méliot's result \cite{meliot2013partial}.

Another important case is that of the super-classes of uni-triangular matrices groups. Recently, these objects had been the subject of a lot of research work, see \cite{andre2013supercharacters}, \cite{UnitriangulargroupAndre}, \cite{diaconis2008supercharacters} and \cite{yan2001representation}, and a polynomiality property for the structure coefficients in this case would be interesting. In fact, the structure coefficients of super-classes of uni-triangular matrices groups can be viewed as structure coefficients of a particular double-class algebra, see Section \ref{super-classe-matr-uni}. Unfortunately our general framework does not contain this case either. An important goal, in a future work about our generalisation, is to see whether or not our general framework can be modified in order to contain these two cases (Méliot's one and the super-classes of uni-triangular groups).

}

\chapter*{Remerciements}

Tout d'abord, je remercie l'EDMI de l'Université de Bordeaux et les membres du Laboratoire Bordelais de Recherche en Informatique (LaBRI) où tout mon travail de recherche a été conduit. 

Je remercie grandement mes directeurs de thèse, Jean-Christophe Aval et Valentin Féray, de m'avoir accepté pour faire une thèse en combinatoire malgré les faibles connaissances que j'avais dans ce domaine. Ils ont toujours été présents pendant les trois dernières années pour m'écouter, m'encourager et me guider dans mes travaux de recherche. Ils ont fait beaucoup plus que ce que j'attendais d'eux et je veux particulièrement les remercier pour leur patience sans faille. Ils ont toujours su me pousser à faire de mon mieux et ils m'ont donné beaucoup de conseils qui vont surement m'aider dans mon avenir scientifique.

Alain Goupil et Florent Hivert ont accepté d'être les rapporteurs de cette thèse, et je les en remercie. Ils ont également contribué par leurs remarques et suggestions à améliorer la qualité de ce mémoire, et je leur en suis très reconnaissant.

Je remercie encore Pierre-Loïc Méliot avec qui j'ai échangé de nombreux mail pendant cette dernière année et qui n'a pas hésité à répondre à mes questions. Il a aussi accepté de faire partie du jury de ma thèse avec Mireille Bousquet-Mélou, Robert Cori, Gilles Schaeffer et Ekaterina Vassilieva. C'était un honneur de vous voir juger mes travaux et je vous en suis reconnaissant. 

Mes remerciements vont également à tous ceux et celles qui m'ont transmis d'une manière ou d'une autre  leurs connaissances scientifiques durant toute ma vie. Je tiens à remercier spécifiquement les membres de l'équipe Combinatoire Algébrique et Énumérative du LaBRI pour les exposés et les présentations de haut niveau qu'ils ont présentés pendant les rencontres hebdomadaires du GT.


Enfin, je remercie mes amis et ma famille qui ont été toujours présents pour me soutenir. En particulier, je remercie ma mère et mon père qui m'ont énormément encouragé à continuer mes études et à faire une thèse de doctorat. Ils ont pris aussi la peine de venir assister à ma soutenance avec Moustpha, Mida, Ammoun, Racha, Lili et Reina. Merci à vous tous !

\tableofcontents
\printnomenclature
\mainmatter
\chapter{Les coefficients de structure: Généralités et liens avec la théorie des représentations}
\label{chapitre1}

On donne au début de ce chapitre la définition formelle des coefficients de structure d'une algèbre de dimension finie.

Dans le cas général, il n'existe pas de formule explicite pour les coefficients de structure. Même pour une algèbre fixée, dans les cas qu'on étudie -- le centre de l'algèbre du groupe symétrique au chapitre \ref{chapitre2} et l'algèbre de Hecke de la paire $(\mathcal{S}_{2n},\mathcal{B}_n)$ au chapitre \ref{chapitre3} --, il n'y a pas de formules explicites. Il apparaît que ce problème est difficile vu que même pour une algèbre donnée, disons le centre de l'algèbre du groupe symétrique, les formules connues pour des cas particuliers de coefficients de structure sont compliquées, voir \cite{GoupilSchaefferStructureCoef} pour plus de détails. 

Il existe pourtant, dans le cas des centres des algèbres des groupes finis une formule pour les coefficients de structure en fonction des caractères irréductibles, mais elle aussi non explicite. On va présenter ce lien dans ce chapitre, en rappelant ce qui est nécessaire de la théorie des représentations des groupes finis. 

On va s'intéresser aussi aux coefficients de structure des algèbres de doubles-classes. Celles-ci sont définies à partir des doubles-classes d'un sous-groupe dans un groupe fini. On va montrer que les centres des algèbres des groupes finis peuvent être vus comme des cas particuliers des algèbres de doubles-classes.

Une partie de ce chapitre est consacrée aux paires de Gelfand et aux fonctions sphériques zonales. Ces dernières peuvent être considérées comme l'analogue des caractères irréductibles dans le cas des algèbres des doubles-classes. Lorsqu'on a une paire de Gelfand, il est possible, comme on va le voir à la fin de ce chapitre, d'exprimer les coefficients de structure de l'algèbre de doubles-classes en fonction des fonctions sphériques zonales.

\section{Généralités}

 Soit $\mathcal{K}$ un corps commutatif. Une $\mathcal{K}$-algèbre est un ensemble $\mathcal{B}$ muni de deux lois internes notées $+, \times$ et d'une loi externe notée $\cdot$ tel que:
\begin{enumerate}
\item[1-] $(\mathcal{B},+,\cdot)$ est un $\mathcal{K}$-espace vectoriel ;
\item[2-] $x\times (y+z)=x\times y + x\times z$ pour tout $x,y$ et $z$ dans $\mathcal{B}$ ;
\item[3-] $(a\cdot x) \times (b\cdot y)=(ab)\cdot (x\times y)$ pour tout $a,b$ dans $\mathcal{K}$ et pour tout $x,y$ dans $\mathcal{B}$ où $ab$ est la multiplication de $a$ et $b$ dans $\mathcal{K}.$
\end{enumerate}
Dans ce qui suit on va alléger les notations pour les lois de multiplication: ainsi on va noter 
\begin{enumerate}
\item $ax$ au lieu de $a\cdot x$ pour tout $a\in \mathcal{K}$ et tout $x\in \mathcal{B}$
\item $xy$ au lieu de $x\times y$ pour tout $x$ et $y$ dans $\mathcal{B}$
\item $ab$ la multiplication de deux éléments $a$ et $b$ de $\mathcal{K}.$
\end{enumerate}
Il faut noter que toutes les algèbres que l'on va étudier dans cette thèse sont des $\mathbb{C}$-algèbres. On va donc écrire "algèbre" au lieu de "$\mathbb{C}$-algèbre" sauf si une confusion est possible.

\begin{definition}
Soit $I$ un ensemble fini. Supposons que $\mathcal{B}$ est une algèbre de dimension finie et que la famille $(b_k)_{k\in I}$ forme une base pour $\mathcal{B}.$ Prenons deux éléments de base $b_i$ et $b_j$ de $\mathcal{B}$ (ici $i, j\in I$), alors le produit $b_ib_j$ est un élément de $\mathcal{B},$ donc il peut s'écrire comme combinaison linéaire des éléments $(b_k)_{k\in I} :$ 
\begin{equation}\label{coef_de_str}
b_ib_j=\sum_{k\in I}c_{ij}^kb_k
\end{equation}
où les coefficients $c_{ij}^k$\label{nomen:cij} \nomenclature[01]{$c_{ij}^k$}{Coefficient de structure d'une algèbre de dimension finie \quad \pageref{nomen:cij}} sont des éléments de $\mathcal{K}$ pour tout $k\in I.$ Les coefficients $c_{ij}^k$ sont appelés {\em coefficients de structure} de $\mathcal{B}$ dans la base $(b_k)_{k\in I}.$ S'il n'y a pas de confusion, on parlera simplement des coefficients de structure sans spécifier la base.
\end{definition}

Dans certains cas (incluant les algèbres étudiées dans cette thèse), on peut donner une description combinatoire pour les coefficients de structure comme on va le voir dans le paragraphe suivant.

\begin{definition}
Une base d'une algèbre est dite multiplicative si le produit de deux éléments de base est bien un élément de base.
\end{definition}
Supposons qu'on a une algèbre $\mathcal{A}$ de dimension finie avec une base multiplicative $(b_k)_{k\in K},$ que l'on fixe dans cette section, et une fonction "type": $K \rightarrow I$ où $I$ est un ensemble fini. Pour un élément $i\in I,$ définissons $B_i$ comme l'ensemble de tous les éléments de base $b_k$ de $\mathcal{A}$ de type $i$ (les éléments $b_k$ tels que $type(k)=i$). Quitte à restreindre l'ensemble $I,$ on peut supposer que tous les $B_i$ sont non-vides.
\begin{notation}
Pour un ensemble fini $X,$ on va noter par $\textbf{X}$\label{nomen:X} \nomenclature[02]{$\textbf{X}$}{La somme des éléments de l'ensemble $X$ \quad \pageref{nomen:X}} la somme de ses éléments,
$$\textbf{X}=\sum_{x\in X}x.$$
\end{notation}

On observe facilement que les $({\bf B}_i)_{i\in I}$ sont linéairement indépendants et supposons que l'espace vectoriel $\mathcal{B}$ engendré par les $({\bf B}_i)_{i\in I}$ forme une sous-algèbre de $\mathcal{A}.$ Dans ce cas, il existe une description des coefficients de structure de l'algèbre $\mathcal{B}$ donnée par la proposition suivante. Cette description va nous aider à calculer les coefficients de structure dans les chapitres qui suivent.
\begin{prop}\label{desc_coef}
Soit $C_{ij}^k$ les coefficients de structure de l'algèbre $\mathcal{B},$ définis par l'équation suivante:
\begin{equation}\label{coef_de_str}
{\bf B}_i{\bf B}_j=\sum_{k\in I}C_{ij}^k{\bf B}_k.
\end{equation}
Alors $C_{ij}^k$ est la taille de l'ensemble suivant:
\begin{equation}
C_{ij}^k=|\lbrace (x,y)\in \mathcal{A}^2 \text{ tel que $x$ et $y$ sont deux éléments de la base de $\mathcal{A}$ de type $i$ et $j$ et xy=z}\rbrace|
\end{equation}
où $z\in \mathcal{A}$ est un élément de la base de $\mathcal{A}$ fixé de type $k.$ 
\end{prop}
\begin{proof}
Immédiat.
\end{proof}

Même si on travaille avec des algèbres sur le corps des nombres complexes, il est à remarquer que dans les conditions de cette proposition, on a le corollaire suivant. 

\begin{cor}
Les coefficients de structure de l'algèbre $\mathcal{B}$ sont des entiers naturels (positifs ou nuls).
\end{cor}

\section{Les centres des algèbres de groupe}

\subsection{Définition}
Dans cette partie, $G$ désigne un groupe fini. On note multiplicativement la loi sur $G,$ ce qui veut dire que la multiplication de deux éléments $g$ et $g'$ de $G$ est notée $gg'.$ On va s'intéresser aussi à la loi de conjugaison sur $G$, notée par $\bullet.$ Pour deux élément $g$ et $g'$ de $G,$ on a:
$$g'\bullet g:=g'gg'^{-1}.$$
On va écrire $(G,\bullet)$ chaque fois que la loi de conjugaison est utilisée pour éviter les confusions de notation. 

\textit{L'algèbre du groupe} $G$ notée $\mathbb{C}[G]$\label{nomen:alg_grp} \nomenclature[03]{$\mathbb{C}[G]$}{L'algèbre du groupe fini $G$ \quad \pageref{nomen:alg_grp}} est une algèbre sur $\mathbb{C}$ ayant pour base les éléments de $G.$ Autrement dit, un élément $x$ de $\mathbb{C}[G]$ peut s'écrire de manière unique comme une combinaison linéaire formelle des éléments de $G$ avec des coefficients dans $\mathbb{C}$ :
$$x=\sum_{g\in G}x_gg,$$
où $x_g\in \mathbb{C}$ pour tout $g\in G.$ 

Comme les éléments de $G$ forment une base pour $\mathbb{C}[G],$ on peut étendre la multiplication dans $G$ pour définir la multiplication dans $\mathbb{C}[G].$ Si $x=\sum_{g\in G}x_gg$ et $x'=\sum_{g'\in G}x'_{g'}g'$ sont deux éléments dans $\mathbb{C}[G],$ le produit $xx'$ est défini par :
$$xx'=\sum_{g\in G}x_gg\sum_{g'\in G}x'_{g'}g'=\sum_{g\in G, g'\in G}x_gx'_{g'}(gg').$$

L'action de $G$ sur lui-même par conjugaison peut être prolongée par linéarité en une action de $G$ sur l'algèbre $\mathbb{C}[G].$ Une sous-algèbre particulière de $\mathbb{C}[G]$ est celle des éléments invariants par conjugaison, notée $\mathbb{C}[G]^{(G,\bullet)}:$
$$\mathbb{C}[G]^{(G,\bullet)}:=\lbrace x\in \mathbb{C}[G] ~~|~~g'\bullet x=x~~\forall g'\in G\rbrace=\lbrace x\in \mathbb{C}[G] ~~|~~g'x=xg'~~\forall g'\in G\rbrace.$$
Le \textit{centre de l'algèbre du groupe} $G$, noté $Z(\mathbb{C}[G]),$ est l'algèbre des invariants $\mathbb{C}[G]^{(G,\bullet)}.$ C'est la sous-algèbre de $\mathbb{C}[G]$ formée des éléments qui commutent avec tous les éléments de $G.$ Dans le cas général, si $\mathcal{A}$ est une algèbre, le centre de $\mathcal{A}$ noté $Z(\mathcal{A})$\label{nomen:ctr_alg} \nomenclature[04]{$Z(\mathcal{A})$}{Le centre de l'algèbre $\mathcal{A}$ \quad \pageref{nomen:ctr_alg}} est l'ensemble suivant :
$$Z(\mathcal{A}):=\lbrace a\in \mathcal{A} \text{ tel que } ab=ba \text{ pour tout }b\in \mathcal{A}\rbrace.$$

\subsection{Action de groupe sur un ensemble}
Supposons que le groupe $G$ agit par $"\cdot"$ sur un ensemble $S.$ Pour un élément $s$ de $S,$ on distingue deux ensembles particuliers, le \textit{stabilisateur} de $s$ par $G$ noté $Stab_s^G$ et \textit{l'orbite} de $s$ dans $G$ noté $O_s^G.$ Le stabilisateur de $s$ par $G$ est le sous-ensemble suivant de $G$ 
$$Stab_s^G:=\lbrace g \in G~~\vert~~g\cdot s=s\rbrace\subseteq G$$
et l'orbite de $s$ dans $G$ est le sous-ensemble $G\cdot s$ de $S$ défini comme
$$O_s^G:=\lbrace g\cdot s~~\vert~~g\in G\rbrace\subseteq S.$$

Il est facile de vérifier que $Stab_s^G$ est un sous-groupe de $G$ pour tout $s\in S$ et que l'ensemble quotient $G/Stab_s^G$ est en bijection avec l'orbite de $s$ dans $G.$ Explicitement, la bijection est l'application $G/Stab_s^G \rightarrow O_s^G$ qui à $g\cdot Stab_s^G$ associe $g\cdot s.$ Ainsi, on obtient l'importante proposition suivante.
\begin{prop}\label{action_grp_ens}
Soit $S$ un ensemble et $G$ un groupe qui agit sur $S$. Pour tout élément $s$ de $S$ on a:\\
$$|O_s^G|=\frac{|G|}{|Stab_s^G|}.$$
\end{prop}

\subsection{Les classes de conjugaison}
Il est naturel de regarder les ensembles du paragraphe précédent dans le cas où $S$ est le groupe $G$ lui-même. Encore une fois, on va s'intéresser à l'action de conjugaison de $G$ sur $G.$ Si $g\in G,$ le stabilisateur de $g$ par $G$ 
$$Stab_g^{(G,\bullet)}=\lbrace g' \in G~~\vert~~g'\bullet g=g\rbrace=\lbrace g' \in G~~\vert~~g'g=gg'\rbrace$$
est l'ensemble de tous les éléments de $G$ qui commutent avec $g$, appelé \textit{centralisateur} de $g.$ La \textit{classe de conjugaison} d'un élément $g$ de $G,$ notée $C_g,$\label{nomen:classe_conj} \nomenclature[05]{$C_g$}{La classe de conjugaison de $g$ \quad \pageref{nomen:classe_conj}} est l'orbite de $g$ dans $G$ :
$$C_g:=O_g^{(G,\bullet)}=\lbrace g'\bullet g~~\vert~~g'\in G\rbrace=\lbrace g'gg'^{-1}~~\vert~~g'\in G\rbrace.$$
\begin{cor}\label{cor:taille_classe}
Soit $G$ un groupe. Pour tout $g\in G$ on a:
$$|C_g|=\frac{|G|}{|Stab_g^{(G,\bullet)}|}.$$
\end{cor}
On va noter $\mathcal{C}(G)$ l'ensemble des classes de conjugaison de $G.$ C'est un ensemble fini car $G$ est fini. On va indexer les classes de conjugaison par un ensemble fini $\mathcal{I}$ 
dont les éléments sont notés par des lettres grecques (habituellement $\lambda,$ $\delta,$ $\rho$). Ainsi, on a :
$$\mathcal{C}(G)=\lbrace C_\lambda \mid \lambda\in \mathcal{I}\rbrace.$$

Une des propriétés importantes du centre de l'algèbre du groupe $G,$ est qu'il possède une base indexée par les classes de conjugaison de $G$ comme le montre la proposition suivante.
\begin{prop}\label{prop:base_centre_alg-de_grp} La famille $({\bf C}_\lambda)_{\lambda\in \mathcal{I}}$, où
$${\bf C}_\lambda=\sum_{g\in C_\lambda}g,$$ 
forme une base de $Z(\mathbb{C}[G]).$
\end{prop}
\begin{proof}
Supposons que $X=\sum_{g\in G} c_gg$ est un élément du centre de l'algèbre $G,$ alors pour tout $g'\in G$ on a par définition 
$$g'\bullet (\sum_{g\in G} c_gg)=\sum_{g\in G} c_gg'gg'^{-1}=\sum_{g\in G} c_gg.$$
Ça veut dire que pour tout élément $g$ de $G$ et pour tout élément $g'$ de $G,$ les coefficients de $g$ et de $g'gg'^{-1}$ dans $X$ sont égaux. Donc pour que $X$ soit dans le centre de $G,$ il faut que pour tout élément $g$ de $G,$ le coefficient de $g$ dans $X$ soit égal au coefficient de tout élément conjugué à $g.$ Donc les sommes des classes de conjugaison engendrent $Z(\mathbb{C}[G])$ et il n'est pas difficile de voir qu'elles sont linéairement indépendantes.
\end{proof}

Soient $\lambda$ et $\delta$ deux éléments de $\mathcal{I}.$ Les \textit{coefficients de structures} $c_{\lambda\delta}^\rho$\label{nomen:coef_ctr_G} \nomenclature[06]{$c_{\lambda\delta}^\rho$}{Coefficients de structure du $Z(\mathbb{C}[G])$ \quad \pageref{nomen:coef_ctr_G}} du centre de l'algèbre du groupe $G$ sont définis par l'équation suivante:
\begin{equation}\label{coef_centre}
{\bf C}_\lambda{\bf C}_\delta=\sum_{\rho\in \mathcal{I}} c_{\lambda\delta}^\rho {\bf C}_g.
\end{equation}
\section{Les algèbres de doubles-classes}

Les algèbres de doubles-classes, ou algèbres de Hecke, ont pour base la somme des éléments des doubles-classes d'un sous-groupe $K$ dans un groupe $G.$ À la fin de cette section on va montrer que les centres des algèbres des groupes finis peuvent être considérés comme des cas particuliers d'algèbres de doubles-classes.

Il faut noter que le terme "algèbre de Hecke" est souvent utilisé dans la littérature pour désigner l'algèbre de Iwahori-Hecke, voir Section \ref{sec:alg_Iwahori_Hecke}. Dans cette thèse, nous l'utiliserons dans son sens d'origine, celui d'algèbre de double-classes.

\subsection{Définition}
Dans cette partie, on désigne par $G$ un groupe et par $H,$ $K$ et $L$ trois sous-groupes de $G.$
Une classe à droite (resp. \textit{à gauche}) de $H$ dans $G$ est un ensemble $Hg$ (resp. $gH$) pour un élément $g$ du groupe $G$ où 
$$Hg:=\lbrace hg \, ; \, h\in H\rbrace \text{ (resp. $gH=\lbrace gh \, ; \, h\in H\rbrace$}).$$
L'ensemble $H\backslash G:=\lbrace Hg \, ; \, g\in G\rbrace$ (resp. $G/H:=\lbrace gH \, ; \, g\in G\rbrace$) est l'ensemble des classes à droite (resp. gauche) de $H$ dans $G.$ Une double-classe de $H$ et $K$ dans $G$ est un ensemble $HgK$ pour un élément $g$ de $G$ où :
$$HgK:=\lbrace hgk \, ; \, h\in H \text{ et } k\in K\rbrace.$$
L'ensemble $H\backslash G/K:=\lbrace HgK \, ; \, g\in G\rbrace$ est l'ensemble des doubles-classes de $H$ et $K$ dans $G.$ D'après sa définition, une double-classe est une union disjointe de classes à droite (ou à gauche).

Dans la proposition suivante on donne une formule pour la taille d'une double-classe.
\begin{prop}
Soient $H$ et $K$ deux sous-groupes d'un groupe fini $G.$ Alors pour tout $g\in G$ on a :
$$|HgK|=\frac{|H||K|}{|H\cap gKg^{-1}|}.$$
\end{prop}

\begin{proof}
Comme $|hgK|=|K|$ pour tout $h\in H$ on aura :
$$|HgK|=|\lbrace hgK \text{ tel que } h\in H\rbrace||K|.$$
Donc d'après la Proposition \ref{action_grp_ens}, si on considère l'action du sous-groupe $H$ sur l'ensemble des classes à gauche de $K$ dans $G,$ on obtient :
$$|HgK|=\frac{|H|}{|\lbrace h\in H \text{ tel que }hgK=gK\rbrace|}|K|.$$
Mais $hgK=gK\Leftrightarrow h\in gKg^{-1},$ donc on a le résultat.
\end{proof}
 
 
\begin{prop}
Soit $G$ un groupe fini et $K$ un sous-groupe de $G.$ Les sommes des doubles-classes $KgK$ de $K$ dans $G$ engendrent linéairement une sous-algèbre de $\mathbb{C}[G].$
\end{prop}

\begin{proof}
On traite un cas plus général dans cette démonstration. Si $H$ et $L$ sont deux sous-groupes de $G$, en utilisant la décomposition des doubles-classes en union disjointe des classes à droites, on peut écrire :
$$\textbf{H}x\textbf{K}=\sum_i \textbf{H}x_i,$$
où les $x_i$ sont dans l'ensemble $xK.$ Pour une autre double-classe $KyL$ de $K$ et $L$ dans $G$, on peut encore écrire:
$$\textbf{K}y\textbf{L}=\sum_j y_j\textbf{L},$$
où les $y_j$ sont dans l'ensemble $Ky.$
En utilisant ces deux décompositions, la multiplication des deux doubles-classes $\textbf{H}x\textbf{K}$ et $\textbf{K}y\textbf{L}$ peut être écrite ainsi :
\begin{eqnarray*}
{\textbf{H}} x {\textbf{K}}\textbf{K}y\textbf{L}&=& \sum_i \textbf{H}x_i \sum_j y_j\textbf{L}\\
&=&\sum_i\sum_j\textbf{H}x_iy_j\textbf{L}.
\end{eqnarray*}
On obtient notre résultat dans le cas particulier où $H=L=K.$
\end{proof}
L'\textit{algèbre des doubles-classes} de $K$ dans $G,$ notée $\mathbb{C}[K\setminus G/K],$\label{nomen:coef_dbl_classe} \nomenclature[07]{$\mathbb{C}[K\setminus G/K]$}{Algèbre de doubles-classes de $K$ dans $G$ \quad \pageref{nomen:coef_dbl_classe}} est l'algèbre ayant comme base les sommes des double-classes de $K$ dans $G.$
Pour simplifier les notations, soit $\mathcal{J}$ 
un ensemble fini indexant l'ensemble $K\setminus G/K$ des doubles-classes de $K$ dans $G.$ Ainsi on obtient:
$$K\setminus G/K=\lbrace DC_\lambda \mid \lambda\in \mathcal{J}\rbrace,$$
où $DC_\lambda$\label{nomen:DC_lambda} \nomenclature[08]{$DC_\lambda$}{La double-classe associée à $\lambda$ \quad \pageref{nomen:DC_lambda}} est la double-classe associée à $\lambda.$

Soient $\lambda$ et $\delta$ deux éléments de $\mathcal{J}.$ Les \textit{coefficients de structure} $k^\rho_{\lambda\delta}$\label{nomen:coef_dbl} \nomenclature[08]{$k^\rho_{\lambda\delta}$}{Coefficients de structure de $\mathbb{C}[K\setminus G/K]$ \quad \pageref{nomen:coef_dbl}} de l'algèbre $\mathbb{C}[K\setminus G/K]$ sont définis par l'équation suivante:
\begin{equation}
\textbf{DC}_\lambda\textbf{DC}_\delta=\sum_{\rho \in \mathcal{J}}k^\rho_{\lambda\delta}\textbf{DC}_\rho.
\end{equation}

\subsection{Le centre de l'algèbre d'un groupe en tant qu'algèbre de doubles-classes}\label{sec:ctr_double_classe}

Dans cette partie on va montrer qu'une forte relation existe entre les coefficients de structure de $Z(\mathbb{C}[G])$ et les coefficients de structure de $\mathbb{C}[\diag(G)\setminus G\times G^{opp}/\diag(G)]$ pour un groupe $G.$

Expliquons les notations ci-dessus, $G^{opp}$ est le \textit{groupe opposé} de $G.$ La multiplication de $x$ par $y$ dans $G^{opp}$ est $yx$ au lieu de $xy.$ Le produit cartésien $G\times G^{opp}$ forme un groupe et sa diagonale $\diag(G):=\lbrace (x,x^{-1})|x\in G\rbrace$\label{nomen:diag} \nomenclature[08]{$\diag(G)$}{le sous-groupe diagonal de $G\times G^{opp}$\quad \pageref{nomen:diag}} est un de ses sous-groupes. Pour un élément $(a,b)\in G\times G^{opp}$ et deux éléments $(x,x^{-1})$ et $(y,y^{-1})$ de $\diag(G)$, on a :
\begin{equation}\label{pr_double_classe}
(x,x^{-1})(a,b)(y,y^{-1})=(xay,y^{-1}bx^{-1}).
\end{equation}

Cette égalité nous permet de relier les doubles-classes de $\diag(G)$ dans $G\times G^{opp}$ aux classes de conjugaison de $G.$ En effet, l'égalité \eqref{pr_double_classe} implique que deux éléments $(a,b)$ et $(c,d)$ sont dans la même double-classe de $\diag(G)$ dans $G\times G^{opp}$ si et seulement si $ab$ et $cd$ sont conjugués dans $G.$

On reprend l'ensemble $\mathcal{I}$ qui indexe les classes de conjugaison de $G.$ Alors $\mathcal{I}$ indexe aussi les doubles-classes de $\diag(G)$ dans $G\times G^{opp}.$ Soit $\lambda$ dans $\mathcal{I},$ la double-classe de $\diag(G)$ dans $G\times G^{opp}$ associée à $\lambda$ est notée par $C'_\lambda.$ Explicitement, on a : 
$$C'_\lambda=\lbrace (a,b)\in G\times G^{opp}~|~ ab\in C_\lambda\rbrace.$$
\begin{prop}\label{prop:relation_taille_dc_class}
Si $\lambda$ un élement de $\mathcal{I}$ alors on a
$$|C'_\lambda|=|G||C_\lambda|.$$
\end{prop}
\begin{proof}
L'application qui envoie un élément $(x,y)$ de $G\times C_\lambda$ sur $(x,x^{-1}y)\in G\times G^{opp}$ est une bijection entre $G\times C_\lambda$ et $C'_\lambda$ dont l'inverse est l'application qui envoie l'élément $(a,b)$ de $C'_\lambda$ à $(a,ab).$
\end{proof}

Soient $\lambda$ et $\delta$ deux éléments de $\mathcal{I},$ les coefficients de structure $c'^\rho_{\lambda\delta}$ de l'algèbre des doubles-classes de $\diag(G)$ dans $G\times G^{opp}, \mathbb{C}[\diag(G)\setminus G\times G^{opp}/\diag(G)],$ sont définis par l'équation suivante :
\begin{equation}
\mathbf{C}'_\lambda\mathbf{C}'_\delta=\sum_{\rho\in \mathcal{I}}c'^\rho_{\lambda\delta}\mathbf{C}'_\rho.
\end{equation} 


\begin{prop}\label{prop:lien_taille_classe_et_double_classe}
Soit $G$ un groupe fini et soit $\mathcal{I}$ un ensemble qui indexe les classes de conjugaison de $G.$ Si $\lambda,\delta$ et $\rho$ sont trois éléments de $\mathcal{I}$ alors les coefficients de structure $c_{\lambda\delta}^{\rho}$ et $c'^\rho_{\lambda\delta}$ du $Z(\mathbb{C}[G])$ et $\mathbb{C}[\diag(G)\setminus G\times G^{opp}/\diag(G)]$ sont reliés par la relation suivante :
\begin{equation*}
c'^\rho_{\lambda\delta}=|G|c^\rho_{\lambda\delta}.
\end{equation*}
\end{prop}
\begin{proof}
Tenons compte de la Proposition \ref{desc_coef}, le coefficient de structure $c_{\lambda\delta}^\rho$ dans $Z(\mathbb{C}[G])$ donné par l'équation (\ref{coef_centre}) n'est autre que le cardinal de l'ensemble $A_{\lambda\delta}^{\rho}$ défini par :
$$A_{\lambda\delta}^{\rho}=\lbrace (x,y)\in C_\lambda\times C_{\delta}\text{ tel que } xy=z \rbrace,$$
où $z$ est un élément fixé de $C_\rho.$
De même, on a $c'^\rho_{\lambda\delta}=|A'^{\rho}_{\lambda\delta}|$ où
$$A'^{\rho}_{\lambda\delta}=\lbrace \big((x_1,x_2),(y_1,y_2)\big)\in C'_\lambda\times C'_\delta \text{ tel que } (x_1y_1,y_2x_2)=(z,1)\rbrace.$$

L'application, 
$$\begin{array}{ccccc}
&  & A'^{\rho}_{\lambda\delta} & \to & A^{\rho}_{\lambda\delta}\times G \\
& & \big((x_1,x_2),(y_1,x_2^{-1})\big) & \mapsto & \big((x_1x_2,x_2^{-1}y_1),x_2\big) \\
\end{array}$$
est bien définie -- l'image d'un élément de $A'^{\rho}_{\lambda\delta}$ est bien dans $A^{\rho}_{\lambda\delta}\times G$ -- et elle forme une bijection 
dont l'inverse est donné par l'application 
\begin{equation*}
\big((x,y),t\big)\mapsto \big((xt^{-1},t),(ty,t^{-1})\big). \qedhere
\end{equation*}
\end{proof}

Le fait que le centre d'une algèbre de groupe $G$ peut être vu comme l'algèbre des doubles-classes de $\diag(G)$ dans $G\times G^{opp}$ est expliqué, d'une manière différente à celle donnée dans cette section, dans l'introduction de l'article \cite{strahov2007generalized} de Strahov.

\section{Théorie des représentations des groupes finis}\label{sec:Th_de_rep}

La théorie des représentations est un des outils pour décrire les coefficients de structure des centres des algèbres de groupe. Il est connu, comme on va le voir plus tard dans cette partie, que le nombre de $G$-modules irréductibles d'un groupe fini $G$ est égal au nombre de classes de conjugaison de $G.$ De plus, les coefficients de structure du centre d'une algèbre de groupe fini $G$ peuvent être écrits en fonction des caractères irréductibles de $G,$ ce qui relie directement l'étude des coefficients de structure à la théorie des représentations. Cette relation va être détaillée à la fin de cette partie.

On va commencer par reprendre les définitions nécessaires pour présenter cette relation. On va parfois omettre certaines démonstrations de résultats classiques. Pour plus de détails, le lecteur est invité à regarder le chapitre $1$ du livre \cite{sagan2001symmetric} de Bruce E. Sagan.

\subsection{Les $G$-modules} On va définir dans cette section les $G$-modules et les $G$-modules irréductibles. On présente aussi le théorème de Maschke qui permet de décomposer tout $G$-module, et en particulier l'algèbre du groupe $G,$ en somme directe de $G$-modules irréductibles.
\begin{definition}
On dit qu'un espace vectoriel $V$ est un $G$-module s'il existe un morphisme de groupes $$\rho :G \longrightarrow GL(V)$$ où $GL(V)$ est le groupe des applications linéaires bijectives de $V$ dans $V$.
\end{definition}
\begin{notation}Quand on considère plusieurs $G$-modules, pour éviter la confusion on va utiliser la notation $\rho_V$ au lieu de $\rho.$
\end{notation}
Pour simplifier les notations, on va écrire $gv$ au lieu de $\rho(g)(v)$ pour tout $g\in G$ et tout $v\in V.$ Lorsque $V$ est de dimension $n,$ on va parfois identifier $GL(V)$ avec le groupe des matrices inversibles de taille $n\times n$ à coefficients dans $\mathbb{C},$ on parle dans ce cas d'une représentation matricielle de $G.$
\begin{ex}\label{alg_de_grp}
L'algèbre de groupe $\mathbb{C}[G]$ est un $G$-module, muni de l'opération suivante :
$$g'x=g'\sum_{g\in G}c_gg=\sum_{g\in G}c_gg'g,$$
pour tout $g'\in G$ et tout $x=\sum_{g\in G}c_gg\in \mathbb{C}[G].$
\end{ex}
\begin{definition}
Soit $V$ un $G$-module et soit $W$ un sous espace vectoriel de $V$. On dit que $W$ est un sous $G$-module de $V$ si :
$$gw\in W \text{ pour tout $g\in G$ et pour tout $w \in W$}.$$
\end{definition}

\begin{definition}
Un $G$-module non nul $V$ est dit réductible s'il possède un sous module non trivial -- différent de $\lbrace 0\rbrace$ et de $V$ lui même --. Dans le cas contraire il est dit irréductible.
\end{definition}

Soient $V$ un espace vectoriel et $U$ et $W$ deux sous-espaces vectoriels de $V.$ On dit que $V$ est la somme directe de $U$ et $W$ et on écrit $V=U\oplus W$ si et seulement si tout élément de $V$ peut s'écrire d'une façon unique comme somme de deux éléments, l'un dans $U$ et l'autre dans $W.$ 
Dans une telle situation, si $V$ est de dimension finie, on a :$$\dim V=\dim U + \dim W.$$
De plus si $U$ et $W$ sont deux $G$-modules alors $V=U\oplus W$ est encore un $G$-module pour l'action définie par:
$$g(u+w):=gu+gw,$$
pour tout $g\in G, u\in U$ et $w\in W.$\\
Le Théorème de Maschke énoncé ci-dessous dit que pour chaque groupe fini $G$ et pour tout $G$-module $V$ non-nul, $V$ peut être décomposé en somme directe de sous-modules irréductibles de $V.$   

\begin{theoreme}[Maschke]\label{maschke}
Soit $G$ un groupe fini et $V$ un $G$-module non nul. Alors $$V=W^{(1)}\oplus W^{(2)}\oplus\ldots \oplus W^{(k)}$$
où les $W^{(i)}$ sont des sous modules irréductibles de $V.$
\end{theoreme}
\begin{proof}
Admise. Le lecteur est invité à regarder la preuve du Théorème 1.5.3 du livre \cite{sagan2001symmetric}.
\end{proof}

Soient $V$ et $W$ deux $G$-modules. Un $G$-morphisme $\theta:V\longrightarrow W$ est un morphisme d'espaces vectoriels qui vérifie:
$$\theta(gv)=g\theta(v)\text{ pour tout $g\in G$ et tout $v\in V$}.$$
Un $G$-isomorphisme est un $G$-morphisme et un isomorphisme d'espaces vectoriels en même temps.

Notons qu'il est possible de retrouver deux modules irréductibles et isomorphes dans la décomposition d'un $G$-module donnée par le théorème de Maschke. On note par $\hat{G}$\label{nomen:ens_mod_irr} \nomenclature[09]{$\hat{G}$}{Ensemble des $G$-modules irréductibles \quad \pageref{nomen:ens_mod_irr}} l'ensemble des $G$-modules irréductibles de $G$ à isomorphisme près,
$$\hat{G}:=\lbrace X \text{ tel que $X$ est un $G$-module irréductible}\rbrace.$$
Autrement dit $\hat{G}$ ne contient pas de $G$-modules isomorphes. Par ailleurs, si $Y$ est un $G$-module irréductible qui n'appartient pas à l'ensemble $\hat{G},$ alors il existe un $G$-module irréductible $X\in \hat{G}$ tel que $X$ est $G$-isomorphe à $Y.$ Dans ce cas là, on va regarder $Y$ comme étant $X$ à un $G$-isomorphisme près. En tenant compte de tout ça, l'égalité donnée par le théorème de Maschke devient une égalité à des $G$-isomorphismes près dans le corollaire suivant. On va continuer à utiliser le signe $"="$ bien que c'est une égalité à des $G$-isomorphismes près.
\begin{cor}
Soit $G$ un groupe fini et $V$ un $G$-module non nul. Alors il existe des entiers positifs $m_X$ tels que:
$$V=\bigoplus_{X\in \hat{G}}m_XX,$$
où $m_XX:=\underbrace{X\oplus X\oplus \cdots\oplus X}_{m_X \text{ fois }}$ pour tout $X\in \hat{G}.$
\end{cor}

\begin{cor}\label{dec_alg_de_grp}
Soit $G$ un groupe fini. L'algèbre de groupe $\mathbb{C}[G]$ peut s'écrire ainsi :
$$\mathbb{C}[G]=\bigoplus_{X\in \hat{G}}m_XX.$$
Par conséquent, on a :
$$\dim \mathbb{C}[G]=|G|=\sum_{X\in \hat{G}} m_X\dim X.$$
\end{cor}

\begin{lem}[Schur]\label{lem:schur}
Si $V$ et $W$ sont deux $G$-modules irréductibles et si $\theta:V\longrightarrow W$ est un $G$-morphisme, alors:
$$\text{$\theta$ est un $G$-isomorphisme ou $\theta$ est l'application nulle}.$$
Puisqu'on travaille avec des espaces vectoriels sur $\mathbb{C},$ qui est algébriquement clos, alors si $\theta$ est un $G$-isomorphisme entre deux $G$-modules irréductibles $V$ et $W,$ il existe un $c\in \mathbb{C}$ tel que : 
\begin{equation}\label{eq:G-iso}
\theta(v)=c\id(v),
\end{equation}
où $\id$ est l'application identique. Ce résultat se trouve détaillé en page $23$ de \cite{sagan2001symmetric}.
\end{lem}
\begin{proof}
Voir Théorème 1.6.5 du livre \cite{sagan2001symmetric}.
\end{proof}


Pour tout $G$-module $V,$ on associe l'algèbre $\End V$ des $G$-endomorphismes de $V,$
$$\End V:=\lbrace \theta: V\longrightarrow V \text{ tel que $\theta$ est un $G$-morphisme } \rbrace.$$ 
Si $V$ se décompose de la façon suivante en somme des $G$-modules irréductibles,
$$V=\bigoplus_{X\in \hat{G}}m_XX,$$
alors d'après le Théorème 1.7.9 du livre \cite{sagan2001symmetric} on a :
\begin{equation}\label{eq:eq_centre}
\End V\simeq \bigoplus_{X\in \hat{G}}\Mat_{m_X},
\end{equation}
où $\Mat_{m_X}$ est l'ensemble des matrices de taille $m_X\times m_X$ à coefficients dans $\mathbb{C}.$ 

\begin{lem} Pour tout $n>0,$ on a :
\begin{equation*}
\dim Z({\Mat_n})=1.
\end{equation*} 
\end{lem}
\begin{proof} Si $M\in Z({\Mat_n})$ alors par définition (voir page $4$) $MM'=M'M$ pour tout $M'\in \Mat_n$ et en particulier on a :
$$ME_{i,i}=E_{i,i}M,$$
pour tout $1\leq i\leq n,$ où $E_{i,i}$ est la matrice dont les entrées sont $0$ sauf l'entrée $(i,i)$ qui vaut $1.$ Ceci se traduit par le fait que :
$$m_{i,k}=0 \text{ et } m_{r,i}=0 \text{ pour tout $r,k\neq i$.}$$
Ce qui veut dire que $M$ est une matrice diagonale. De même, si $i\neq j,$ on a :
$$M(E_{i,j}+E_{j,i})=(E_{i,j}+E_{j,i})M.$$
Après développement, on arrive au fait que $M\in Z({\Mat_n})$ si et seulement si $M=c\I_n$ où\linebreak $c\in \mathbb{C}.$ 
\end{proof}
Donc, d'après l'équation \eqref{eq:eq_centre}, on a :
\begin{equation}\label{eq:dim-endV-plus-petit-nbderepsirr}
\dim Z({\End V})=|\lbrace X\in \hat{G} \text{ tel que } m_X>0\rbrace|\leq |\hat{G}|.
\end{equation}
\begin{prop}\label{egalite_classe_caract}
Soit $G$ un groupe fini. Alors on a:
$$|\mathcal{C}(G)|=|\hat{G}|.$$
\end{prop}
\begin{proof}
Pour une preuve complète de cette proposition, le lecteur est invité à voir la démonstration de la Proposition 1.10.1 de \cite{sagan2001symmetric}. Ici on donne l'idée principale de la preuve. On peut prouver que l'application 
$$\begin{array}{ccccc}
\theta & : & \mathbb{C}[G] & \to & \End \mathbb{C}[G] \\
&& x & \mapsto & \theta_x \\
\end{array},$$
où $\theta_x: \mathbb{C}[G]\rightarrow \mathbb{C}[G]$ est le $G$-endomorphisme défini par : $\theta_x(y)=yx,$
est un isomorphisme. En passant à la dimension on aura :
$$|\mathcal{C}(G)|=\dim Z(\mathbb{C}[G])=\dim Z({\End \mathbb{C}[G]})=|\hat{G}|.$$
La dernière égalité est une conséquence de l'équation \eqref{eq:dim-endV-plus-petit-nbderepsirr} et du Théorème \ref{th:dec_alg_de_grp} présenté plus loin.
\end{proof}
\subsection{Les caractères}

D'après la Proposition \ref{egalite_classe_caract}, la dimension du centre de l'algèbre d'un groupe est égale au nombre de représentations irréductibles. Il existe une relation encore plus forte qui exprime les coefficients de structures en fonction des caractères. 

\begin{definition}
Soit $X$ un $G$-module. Le caractère de $X$ noté $\mathcal{X}$\label{nomen:carac} \nomenclature[10]{$\mathcal{X}$}{Le caractère du module $X$ \quad \pageref{nomen:carac}} est la fonction $\mathcal{X}:G\longrightarrow \mathbb{C}$ définie par :
$$\mathcal{X}(g):=\tr (\rho_X(g)),$$ 
pour tout $g\in G,$ où $\tr$ est la fonction trace.
\end{definition} 

Deux représentations isomorphes ont le même caractère car la fonction trace est invariante sous conjugaison. La formule qui relie les coefficients de structure avec les caractères est donnée à la fin de cette partie. On va voir de plus que deux représentations sont isomorphes si et seulement si elles possèdent le même caractère.

\begin{definition}
 Soit $f: G\rightarrow \mathbb{C}$ une fonction. On dit que $f$ est centrale sur $G$ si $f$ est constante sur les classes de conjugaison de $G$.
 \end{definition}
 
Le caractère d'une représentation est une fonction centrale car la fonction trace est constante sur les classes de conjugaison.
 
On va noter par $\mathcal{R}(G)$\label{nomen:ens_fct_cent} \nomenclature[10]{$\mathcal{R}(G)$}{L'ensemble des fonctions centrales sur $G$ \quad \pageref{nomen:ens_fct_cent}} l'ensemble des fonctions centrales sur $G.$ Il n'est pas difficile de vérifier que $\mathcal{R}(G)$ est un sous-espace vectoriel de l'ensemble des fonctions de $G$ dans $\mathbb{C}.$
Pour une classe $\mathcal{K}\in\mathcal{C}(G),$ on associe la fonction (souvent appelée caractéristique) centrale $\delta_\mathcal{K}$\label{nomen:fct_carct} \nomenclature[10]{$\delta_\mathcal{K}$}{La fonction caractéristique \quad \pageref{nomen:fct_carct}} définie ainsi :
$$\delta_\mathcal{K}(g)=\left\{
\begin{array}{ll}
  1 & \qquad \mathrm{si}\quad g\in \mathcal{K} \\
  0 & \qquad \mathrm{si non.}\quad \\
 \end{array}
 \right.$$
 
Si $f\in \mathcal{R}(G)$ alors par définition, pour tout $\mathcal{K}\in \mathcal{C}(G),$ il existe un nombre $c_\mathcal{K}\in \mathbb{C}$ tel que $f(g)=c_\mathcal{K}$ pour tout $g\in \mathcal{K}.$ Donc $f$ peut s'écrire de la manière suivante:
$$f=\sum_{\mathcal{K}\in \mathcal{C}(G)}c_\mathcal{K}\delta_\mathcal{K}.$$
De plus il est clair que les fonctions $(\delta_\mathcal{K})_{\mathcal{K}\in \mathcal{C}(G)}$ sont linéairement indépendantes donc elles forment une base pour $\mathcal{R}(G)$ et on a:
\begin{equation}\label{eq:egalite_RG_CG}
\dim \mathcal{R}(G)=|\mathcal{C}(G)|.
\end{equation}

\begin{ex}\label{ex:calcul_carac_reg}
L'algèbre $\mathbb{C}[G]$ a pour base comme espace vectoriel les éléments du groupe $G.$ On note par $\mathcal{X}^{reg}$ le caractère de $\mathbb{C}[G].$ Soit $g$ un élément de $G,$ comme 
$$g=1_G \Leftrightarrow gg'=g' \text{ pour tout élément $g'$ de $G$},$$ alors on a :
$$\mathcal{X}^{reg}(g)=|G|\delta_{\lbrace 1_G\rbrace}(g)=\left\{
\begin{array}{ll}
  \vert G\vert & \qquad \mathrm{si}\quad g=1_G \\
  0 & \qquad \mathrm{si}\quad g\neq 1_G \\
 \end{array}
 \right.$$
\end{ex}

Pour une classe de conjugaison $\mathcal{K}\in \mathcal{C}(G)$ et pour un caractère $\mathcal{X},$ on note par $\mathcal{X}_{\mathcal{K}}$ la valeur de $\mathcal{X}$ sur n'importe quel élément de $\mathcal{K},$ 
$$\mathcal{X}_{\mathcal{K}}:=\mathcal{X}(g) \text{ pour tout }g\in \mathcal{K}.$$

On munit $\mathcal{R}(G)$ du produit scalaire défini par:
$$<f,h>=\frac{1}{\vert G\vert}\sum_{g\in G}f(g)\overline{h(g)},$$
pour tout $f,h\in \mathcal{R}(G)$ où $\overline{h(g)}$ désigne le conjugué de $h(g).$

\begin{prop}
Si $\mathcal{X}$ est un caractère, alors on a:
$$\mathcal{X}(g^{-1})=\overline{\mathcal{X}(g)},$$
pour tout $g\in G.$
\end{prop}
\begin{proof}
Voir \cite[page 34]{sagan2001symmetric} pour une démonstration complète. En résumé, il est possible de choisir une base orthonormale pour $X$ de sorte que :
$$X(g^{-1})=\overline{X(g)}^t,$$
pour tout $g\in G.$ Donc,
\begin{equation*}
\overline{\mathcal{X}(g)}=\tr \overline{X(g)}=\tr X(g^{-1})^t= \tr X(g^{-1}) =\mathcal{X}(g^{-1}). \qedhere
\end{equation*} 
\end{proof}
Soient $\mathcal{X}$ et $\mathcal{Y}$ deux caractères, en tenant compte du fait que $\mathcal{X}$ et $\mathcal{Y}$ sont constants sur les classes de conjugaison de $G,$ on obtient:
\begin{equation}
<\mathcal{X},\mathcal{Y}>=\frac{1}{\vert G\vert}\sum_{g\in G}\mathcal{X}(g)\mathcal{Y}(g^{-1})=\frac{1}{|G|}\sum_{\mathcal{K}\in \mathcal{C}(G)}|\mathcal{K}|\mathcal{X}_\mathcal{K}\overline{\mathcal{Y}_\mathcal{K}},
\end{equation}
où $\overline{\mathcal{Y}_\mathcal{K}}=\mathcal{Y}(g^{-1})$ pour tout $g\in \mathcal{K}.$

Ce produit scalaire possède une propriété particulière donnée par le théorème suivant.
\begin{theoreme}\label{prd_sc_carc}
Soient $\mathcal{X}$ et $\mathcal{Y}$ deux caractères de $G$-modules irréductibles, alors on a:
$$<\mathcal{X},\mathcal{Y}>=\delta_{\mathcal{X},\mathcal{Y}}=\left\{
\begin{array}{ll}
  1 & \qquad \mathrm{si}\quad \mathcal{X}=\mathcal{Y} \\
  0 & \qquad \mathrm{si}\quad \mathcal{X}\neq \mathcal{Y} \\
 \end{array}
 \right.$$
\end{theoreme}
\begin{proof}
Voir la preuve du Théorème 1.9.3 de \cite{sagan2001symmetric}.
\end{proof}
\begin{cor}\label{sum_sur_irr}
Soient $\mathcal{K}$ et $\mathcal{L}$ deux classes de conjugaison de $G$, alors on a :
$$\sum_{X\in \hat{G}}\mathcal{X}_\mathcal{K}\overline{\mathcal{X}_\mathcal{L}}=\frac{\vert G\vert}{\vert \mathcal{K}\vert}\delta_{\mathcal{K},\mathcal{L}}.$$
\end{cor}
\begin{proof}
Voir la démonstration du Théorème 1.10.3 page 42 dans \cite{sagan2001symmetric}.
\end{proof}
\begin{prop}
La famille $(\mathcal{X})_{X\in \hat{G}}$ forme une base orthonormale de $\mathcal{R}(G).$
\end{prop}
\begin{proof}
Cette proposition est une conséquence de la Proposition \ref{egalite_classe_caract}, de l'équation \eqref{eq:egalite_RG_CG} et du Théorème \ref{prd_sc_carc}.
\end{proof}
\subsection{Décomposition de l'algèbre du groupe}

Dans cette partie on va se concentrer sur le $G$-module particulier $\mathbb{C}[G].$ On a vu au Corollaire \ref{dec_alg_de_grp} que l'algèbre de groupe $\mathbb{C}[G]$ peut se décomposer de la manière suivante :
\begin{equation}
\mathbb{C}[G]=\bigoplus_{X\in \hat{G}}m_XX,
\end{equation}
où $\hat{G}$ est l'ensemble des $G$-modules irréductibles.

À l'aide du produit scalaire sur les caractères on peut avoir plus de détails sur les coefficients $m_X$ de cette équation.
\begin{lem} Si  $X^1,\cdots,X^r$ et $X$ sont des $G$-modules tel que $X=X^1\oplus \cdots \oplus X^r$ alors on a:
$$\mathcal{X}(g)=\mathcal{X}^1(g)+\cdots +\mathcal{X}^r(g),$$
pour tout $g\in G.$
\end{lem}
\begin{proof}
$\mathcal{X}(g)=\tr(\rho_X(g))=\tr((\rho_{X^1}\oplus \cdots \oplus \rho_{X^r}) (g))=\tr(\rho_{X^1}(g)\oplus \cdots \oplus \rho_{X^r}(g))=\tr(\rho_{X^1}(g))+\cdots +\tr(\rho_{X^r}(g))=\mathcal{X}^1(g)+\cdots +\mathcal{X}^r(g).$
\end{proof}
\begin{theoreme}\label{th:dec_alg_de_grp}
Soit $G$ un groupe fini et soit 
\begin{equation*}
\mathbb{C}[G]=\bigoplus_{X\in \hat{G}}m_XX
\end{equation*}
la décomposition en somme de représentations irréductibles de l'algèbre de groupe $\mathbb{C}[G],$ alors on a : $m_X=\dim X$ pour tout $X\in \hat{G}.$ Par conséquent, $\sum_{X\in \hat{G}}(\dim X)^2=|G|.$
\end{theoreme}
\begin{proof}
Soit $X$ un élément de $\hat{G}.$ D'après le Théorème \ref{prd_sc_carc}, on a:
$$<\mathcal{X}^{reg},\mathcal{X}>=\sum_{Y \in \hat{G}}m_Y<\mathcal{Y},\mathcal{X}>=m_X.$$
Par ailleurs,
$$<\mathcal{X}^{reg},\mathcal{X}>=\frac{1}{|G|}\sum_{\mathcal{K}\in \mathcal{C}(G)}|\mathcal{K}|\mathcal{X}^{reg}_\mathcal{K}\overline{\mathcal{X}_\mathcal{K}}=\frac{1}{|G|}|G|\mathcal{X}(1_G)=\dim X,$$
d'où le premier résultat. Pour obtenir la deuxième égalité il suffit de voir que $<\mathcal{X}^{reg},\mathcal{X}^{reg}>$ est d'une part $|G|$ et d'autre part $\sum_{X\in \hat{G}}m_X^2$ ou bien de passer à la dimension à partir de la décomposition de $\mathbb{C}[G].$
\end{proof}

\subsection{Les coefficients de structure du centre de l'algèbre de groupe en fonction des caractères} Dans cette partie, on va donner une formule qui exprime les coefficients de structure du centre de l'algèbre d'un groupe fini $G,$ donnés par l'équation \eqref{coef_centre}, en fonction des caractères irréductibles du groupe $G.$

Pour cela on va construire des idempotents indexés par les caractères irréductibles et qui forment une base du centre de l'algèbre du groupe $G.$ On va commencer par rappeler quelques résultats concernant les caractères irréductibles.
\begin{lem}\label{egalite_des_actions}
Soient $x,x^{'}\in \mathbb{C}[G],$ si pour tout $G$-module irréductible $V$ on a $xv=x'v$ pour tout $v\in V$ alors $x=x^{'}$. 
\end{lem}
\begin{proof}
Comme tout $G$-module de dimension finie se décompose en modules irréductibles, alors l'hypothèse est vérifiée pour tout $G$-module de dimension finie et en particulier l'algèbre de groupe $\mathbb{C}[G].$ Pour $v=1_G$, on a 
\begin{equation*}
x=x1_G=x'1_G=x'.\qedhere
\end{equation*}
\end{proof}
\begin{prop}\label{act_ducent_sur_irr}
Soit $V$ un $G$-module irréductible et $x$ un élément de $Z(\mathbb{C}[G]),$ le centre de l'algèbre du groupe $G$, alors on a:
$$xv=\frac{\mathcal{V}(x)}{\dim V}v \text{ pour tout $v\in V$},$$
où $\mathcal{V}$ est le caractère de $V.$
\end{prop}
\begin{proof}
Puisque $x\in Z(\mathbb{C}[G]),$ on a :
$$x(gv)=(xg)v=(gx)v=g(xv),$$
pour tout $g\in G$ et pour tout $v\in V.$ En d'autres termes, l'action de $x$ sur $V$ définit un $G$-morphisme. D'après le Lemme \ref{lem:schur}, ce $G$-morphisme est soit un $G$-isomorphisme soit identiquement nulle. Il existe donc, voir l'équation \eqref{eq:G-iso}, un $c\in \mathbb{C}$ tel que :
$$xv=cv,$$
pour tout $v\in V.$ Par conséquent, $\mathcal{V}(x)=c\dim V,$ d'où le résultat.
\end{proof}
\begin{cor}\label{cor_act_caract_norm}
Soient $x\in Z(\mathbb{C}[G])$, $g\in G$ et $V$ un $G$-module irréductible de caractère $\mathcal{V}$, alors: 
$$\mathcal{V}(xg)=\frac{\mathcal{V}(x)}{\dim V}\mathcal{V}(g).$$
\end{cor}
\begin{proof}
D'après la proposition précédente on a $xgv=\frac{\mathcal{V}(x)}{\dim V}gv$ pour tout $v\in V.$ Donc $\mathcal{V}(xg)=\tr(\frac{\mathcal{V}(x)}{\dim V}g)=\frac{\mathcal{V}(x)}{\dim V}\mathcal{V}(g).$
\end{proof}
\begin{prop}\label{caract_norma}
Soit $V$ un $G$-module irréductible, l'application $\frac{\mathcal{V}}{\dim V}$ définie par 
$$\begin{array}{ccccc}
\frac{\mathcal{V}}{\dim V} & : & Z(\mathbb{C}[G]) & \to & \mathbb{C} \\
& & x & \mapsto & \frac{\mathcal{V}(x)}{\dim V} \\
\end{array}$$
est un morphisme d'algèbres.
\end{prop}
\begin{proof}
Soient $x$ et $y$ deux éléments de $Z(\mathbb{C}[G]).$ En écrivant $y=\sum_{g\in G}y_gg,$ on obtient:
\begin{eqnarray*}
\frac{\mathcal{V}}{\dim V}(xy)&=&\frac{\mathcal{V}}{\dim V}\big(\sum_{g\in G}y_gxg\big)\\
&=&\sum_{g\in G}y_g\frac{\mathcal{V}}{\dim V}\big(xg\big)\\
&=&\sum_{g\in G}y_g\frac{\mathcal{V}(x)}{\dim V}\frac{\mathcal{V}(g)}{\dim V}~~~~~~~~~~~~(\text{d'après le Corollaire \ref{cor_act_caract_norm}})\\
&=&\frac{\mathcal{V}(x)}{\dim V}\frac{\mathcal{V}(y)}{\dim V}.
\end{eqnarray*}
On a donc $\frac{\mathcal{V}}{\dim V}(xy)=\frac{\mathcal{V}(x)}{\dim V}\frac{\mathcal{V}(y)}{\dim V}$ pour tout $x,y\in Z(\mathbb{C}[G])$ ce qui termine la preuve de la proposition.
\end{proof}

Cette proposition peut être écrite d'une autre manière.
\begin{prop}\label{isom_centre_fct}
Soit $G$ un groupe fini, l'application 
$$\begin{array}{ccccc}
Z(\mathbb{C}[G]) & \to & \mathcal{F}(\hat{G},\mathbb{C}) \\
x & \mapsto & (V\mapsto \frac{\mathcal{V}(x)}{\dim V}) \\
\end{array}$$
est un isomorphisme d'algèbres où $\mathcal{F}(\hat{G},\mathbb{C})$ est l'ensemble des applications de $\hat{G}$ à valeur dans $\mathbb{C}.$
\end{prop}
\begin{proof}
La Proposition \ref{caract_norma} implique que cette application est un morphisme d'algèbres. Il reste à remarquer que, d'après la Proposition \ref{egalite_classe_caract},
\begin{equation*}
\dim Z(\mathbb{C}[G])=|\mathcal{C}(G)|=|\hat{G}|=\dim \mathcal{F}(\hat{G},\mathbb{C}).\qedhere
\end{equation*}
\end{proof}

Soit $X\in \hat{G}$ une représentation irréductible de $G.$ On associe à $X$ l'élément $e_X$ suivant:
\begin{equation}
e_X=\frac{\mathcal{X}(1_G)}{|G|}\sum_{g\in G}\mathcal{X}(g^{-1})g,
\end{equation}
où $\mathcal{X}$ est le caractère associé à $X.$ Il est clair que $e_X\in Z(\mathbb{C}[G])$ vu que les caractères sont constants sur les classes de conjugaison. Ainsi $e_X$ peut s'écrire:
$$e_X=\frac{\mathcal{X}(1_G)}{|G|}\sum_{\mathcal{K}\in \mathcal{C}(G)}\overline{\mathcal{X}_\mathcal{K}}{\bf{\mathcal{\bf{K}}}},$$
où comme précédemment $\mathcal{\bf{K}}$ est la somme des éléments de $\mathcal{K}.$
\begin{lem} Soient $X$ et $Y$ deux éléments de $\hat{G}$ et soit $y\in Y,$ alors on a:
$$e_Xy=\left\{
\begin{array}{ll}
  y & \qquad \mathrm{si}\quad X=Y, \\
  0 & \qquad \mathrm{sinon.}\quad \\
 \end{array}
 \right.$$
 \end{lem}
 \begin{proof}
Comme $e_X\in Z(\mathbb{C}[G]),$ d'après la Proposition \ref{act_ducent_sur_irr}, on a:
$$e_Xy=\frac{\mathcal{Y}(e_X)}{\dim Y} y \text{ pour tout $y\in Y$}.$$
Mais,
\begin{eqnarray*}
\mathcal{Y}(e_X)&=&\mathcal{Y}\left(\frac{\mathcal{X}(1_G)}{|G|}\sum_{g\in G}\mathcal{X}(g^{-1})g\right)\\
&=&\frac{\mathcal{X}(1_G)}{|G|}\sum_{g\in G}\mathcal{X}(g^{-1})\mathcal{Y}(g)\\
&=&\frac{\mathcal{X}(1_G)}{|G|}|G|<\mathcal{X},\mathcal{Y}>.
\end{eqnarray*}
Donc $\frac{\mathcal{Y}(e_X)}{\dim Y}=\frac{\dim X}{\dim Y}<\mathcal{X},\mathcal{Y}>$ et on a directement le résultat grâce au Théorème \ref{prd_sc_carc}.
 \end{proof}
 \begin{cor}\label{idem}
 Pour $X$ et $Y$ deux éléments de $\hat{G}$ on a: $e_X^2=e_X$ et $e_Xe_Y=0$ si $X\neq Y.$
 \end{cor}
 \begin{proof}
 Ce corollaire s'obtient en utilisant le Lemme \ref{egalite_des_actions}.
 \end{proof}
Il est à remarquer que l'image de $e_X$ par l'isomorphisme d'algèbres donné dans la Proposition \ref{isom_centre_fct}, est la fonction $f_X$ suivante :
$$\begin{array}{ccccc}
f_X & : & \hat{G} & \to & \mathbb{C} \\
& & Y & \mapsto & \left\{
\begin{array}{ll}
  1 & \qquad \mathrm{si}\quad X=Y, \\
  0 & \qquad \mathrm{sinon.}\quad \\
 \end{array}
 \right. \\
\end{array}$$ Ceci est vrai puisque d'après le Théorème \ref{prd_sc_carc} on a : 
$$\frac{\mathcal{Y}(e_X)}{\dim Y}=\frac{\mathcal{X}(1_G)}{|G|\dim Y}\sum_{g\in G}\mathcal{X}(g^{-1})\mathcal{Y}(g)=<\mathcal{Y},\mathcal{X}>\frac{\mathcal{X}(1_G)}{\dim Y}=\left\{
\begin{array}{ll}
  1 & \qquad \mathrm{si}\quad X=Y, \\
  0 & \qquad \mathrm{sinon,}\quad \\
 \end{array}
 \right.$$
pour tout $Y\in \hat{G}.$ Le Corollaire \ref{idem} peut être retrouvé ainsi puisque la famille $(f_X)_{X\in \hat{G}}$ est une famille d'idempotents orthogonaux de $\mathcal{F}(\hat{G},\mathbb{C}).$

\begin{lem}\label{classe_en_fct_idemp}
Soit $\mathcal{K}\in \mathcal{C}(G)$ une classe de conjugaison de $G$ alors $\mathcal{\bf{K}}$ peut s'écrire en fonction des éléments $e_X$ de la façon suivante:
\begin{equation}
\mathcal{\bf{K}}=|\mathcal{K}|\sum_{X\in \hat{G}}\frac{\mathcal{X}_\mathcal{K}}{\mathcal{X}(1_G)}e_X.
\end{equation}
\end{lem}
\begin{proof}
On va développer explicitement la somme du membre de droite,
\begin{eqnarray*}
|\mathcal{K}|\sum_{X\in \hat{G}}\frac{\mathcal{X}_\mathcal{K}}{\mathcal{X}(1_G)}e_X&=&|\mathcal{K}|\sum_{X\in \hat{G}}\frac{\mathcal{X}_\mathcal{K}}{\mathcal{X}(1_G)}\frac{\mathcal{X}(1_G)}{|G|}\sum_{\mathcal{K'}\in \mathcal{C}(G)}\overline{\mathcal{X}_\mathcal{K'}}{\bf{\mathcal{\bf{K'}}}}\\
&=&|\mathcal{K}|\sum_{\mathcal{K'}\in \mathcal{C}(G)}\left( \sum_{X\in \hat{G}}\frac{\mathcal{X}_\mathcal{K}}{|G|}\overline{\mathcal{X}_\mathcal{K'}}\right){\bf{\mathcal{\bf{K'}}}}\\
&=&\sum_{\mathcal{K'}\in \mathcal{C}(G)}\delta_{\mathcal{K},\mathcal{K'}}{\bf{\mathcal{\bf{K'}}}} ~~~~\text{ (d'après le Corollaire \ref{sum_sur_irr}})\\
&=&\mathcal{\bf{K}}.
\end{eqnarray*}
\end{proof}
\begin{theoreme}\label{coef_fct_cara}
Soit $G$ un groupe fini et soit $\mathcal{I}$ l'ensemble des éléments qui indexent les classes de conjugaison de $G.$ Soient $\lambda, \delta$ et $\rho$ trois éléments de $\mathcal{I},$ le coefficient de structure $c_{\lambda\delta}^\rho$ dans la base des classes de conjugaison du centre de l'algèbre du groupe symétrique s'exprime en fonction des caractères irréductibles de $G$ de la façon suivante:
\begin{equation}
c_{\lambda\delta}^\rho=\frac{|C_\lambda||C_\delta|}{|G|}\sum_{X\in \hat{G}}\frac{\mathcal{X}_\lambda\mathcal{X}_\delta\overline{\mathcal{X}_\rho}}{\mathcal{X}(1_G)}.
\end{equation}
\end{theoreme}
\begin{proof}
D'après le Lemme \ref{classe_en_fct_idemp}, on peut développer le produit ${\bf{ C_\lambda C_\delta}}$ de la manière suivante:
\begin{eqnarray*}
{\bf{ C_\lambda C_\delta}}&=&\left(|C_\lambda|\sum_{X\in \hat{G}}\frac{\mathcal{X}_\lambda}{\mathcal{X}(1_G)}e_X\right)\left(|C_\delta|\sum_{Y\in \hat{G}}\frac{\mathcal{Y}_\delta}{\mathcal{Y}(1_G)}e_Y\right)\\
&=&|C_\lambda||C_\delta|\sum_{X\in \hat{G}}\frac{\mathcal{X}_\lambda}{\mathcal{X}(1_G)}\frac{\mathcal{X}_\delta}{\mathcal{X}(1_G)}e_X ~~~~~~ \text{ ( d'après le Corollaire \ref{idem} )}\\
&=&|C_\lambda||C_\delta|\sum_{X\in \hat{G}}\frac{\mathcal{X}_\lambda}{\mathcal{X}(1_G)}\frac{\mathcal{X}_\delta}{|G|}\sum_{\rho\in \mathcal{I}}\overline{\mathcal{X}_\rho} {\bf{ C_\rho}}~~~~~~ \text{ ( en développant $e_X$ ) }\\
\end{eqnarray*}
ce qui termine la preuve.
\end{proof}

La formule donnée dans le Théorème \ref{coef_fct_cara} est un cas particulier d'une formule générale décrivant le produit de plusieurs éléments et non pas seulement deux, appelée formule de Frobenius, voir l'appendice de Don Zagier dans \cite{lando2004graphs}. Cette formule est donnée aussi dans le Lemme 3.3 de l'article \cite{JaVi90}. Même si ce n'est pas une formule explicite, c'est très utilisé dans le cas du groupe symétrique, comme on va voir dans le chapitre \ref{chapitre2}, pour calculer les coefficients de structure.

\section{Les paires de Gelfand et les fonctions sphériques zonales}\label{sec:paire_des_Gelfand}  Pour plus de détails concernant certains résultats donnés dans cette section, le lecteur est invité à regarder la première section du chapitre VII du livre \cite{McDo} de Ian. G. Macdonald.

Il faut noter que notre présentation ici diffère un peu de celle de Macdonald dans \cite{McDo}. De plus, on donne de nouveaux résultats dans le cas général des paires de Gelfand qui selon les connaissances de l'auteur ne sont pas présentés ailleurs. Parmi eux et peut-être le plus important est une formule, similaire à celle de Frobenius, qui exprime les coefficients de structure de l'algèbre de doubles-classes $\mathbb{C}[K\setminus G/K],$ associée à une paire de Gelfand $(G,K),$ en fonction des fonctions sphériques zonales de la paire $(G,K).$

\subsection{Définitions}

On donne dans cette section plusieurs définitions équivalentes pour les paires de Gelfand. On montre de plus qu'à chaque paire de Gelfand on peut associer une famille des fonctions particulières, appelées fonctions sphériques zonales. On va voir dans la Section \ref{sec:Gen_fct_sph_zon} que ces fonctions généralisent de nombreuses propriétés des caractères normalisés.

\bigskip

Soit $(G,K)$ une paire où $G$ est un groupe fini et $K$ un sous-groupe de $G.$ On va construire une algèbre similaire à $\mathcal{R}(G),$ l'ensemble des fonctions constantes sur les classes de conjugaison. Pour cela, on considère l'ensemble $K\setminus G/ K$ des doubles-classes de $K$ dans $G.$ On note par $C(G,K)$ l'ensemble des fonctions $f:G\rightarrow \mathbb{C}$ constantes sur les doubles-classes de $K$ dans $G$
$$C(G,K):=\lbrace f:G\longrightarrow \mathbb{C} \text{ tel que }f(kxk')=f(x) \text{ pour tout $x\in G$ et tout $k,k'\in K$} \rbrace.$$
L'ensemble $C(G,K)$ est une algèbre pour la multiplication définie par le produit de convolution des fonctions :
$$(fg)(x)=\sum_{y\in G}f(y)g(y^{-1}x)\text{ pour tout $f,g\in C(G,K)$}.$$

L'algèbre $C(G,K)$ est à un isomorphisme près l'algèbre de doubles-classes $\mathbb{C}[K\setminus G/K].$ 
\begin{prop}
L'application $\psi$ définie par :
$$\begin{array}{ccccc}
\psi & : &\mathbb{C}[K\setminus G/ K] & \to & C(G,K) \\
& & x=\sum_{\lambda\in \mathcal{J}}x_\lambda \bf{DC}_\lambda & \mapsto &  \psi_x:=\sum_{\lambda\in \mathcal{J}}x_\lambda \delta_{DC_\lambda},\\
\end{array}$$
est un isomorphisme d'algèbres.
\end{prop}
\begin{proof}
Soient $x=\sum_{\lambda\in \mathcal{J}}x_\lambda \bf{DC}_\lambda$ et $y=\sum_{\delta\in \mathcal{J}}y_\delta \bf{DC}_\delta$ deux éléments de $\mathbb{C}[K\setminus G/ K].$ Pour tout $g\in G,$ on a :
\begin{equation*}
{xy}(g)=\psi_x\psi_y(g)=\left\{
\begin{array}{ll}
  \sum_{\lambda\in \mathcal{J}}\sum_{\delta\in \mathcal{J}}x_\lambda y_\delta k_{\lambda\delta}^{\rho} & \qquad \mathrm{si}\quad g\in DC_\rho, \\
  0 & \qquad \mathrm{sinon ,}\quad \\
 \end{array}
 \right.
 \end{equation*}
où  $k_{\lambda\delta}^{\rho}$ est le coefficient de $\bf{DC}_\rho$ dans le produit $\bf{DC}_\lambda\bf{DC}_\delta.$
\end{proof}
\begin{definition}
On dit qu'une paire $(G,K)$\label{nomen:paire_Gelfand} \nomenclature[11]{$(G,K)$}{Une paire de Gelfand \quad \pageref{nomen:paire_Gelfand}} où $G$ est un groupe fini et $K$ un sous groupe de $G$ est une paire de Gelfand si l'algèbre $C(G,K)$ est une algèbre commutative.
\end{definition}

Il existe plusieurs définitions pour les paires de Gelfand équivalentes à celle donnée ci-dessus. Pour aller plus loin dans les détails on va revenir un petit peu à la théorie des représentations.

Il n'est pas difficile de remarquer que la restriction à $K$ de l'action de $G$ sur un $G$-module $V$ définit un $K$-module. Ce $K$-module est noté $\Res_K^G V.$ La réciproque n'est pas évidente. Si $V$ est un $K$-module, il est plus difficile d'obtenir un $G$-module à partir de $V$ mais c'est possible. Ce $G$-module est noté $\Ind_K^GV$ et son morphisme est défini matriciellement par:
$$\rho_{\Ind_K^GV}(g)=(\rho_V(k_i^{-1}gk_j))_{i,j},$$
où les $k_i$ sont les représentants de l'ensemble $G/K=\lbrace k_1K,k_2K,\cdots ,k_lK\rbrace.$ 
\begin{theoreme}[Frobenius]\label{frob}
Soient $G$ un groupe et $K$ un sous-groupe de $G.$ Si $X$ est un $G$-module et $Y$ est un $K$-module alors on a :
$$<\Ind_K^G \mathcal{Y}, \mathcal{X}>=<\mathcal{Y},\Res_K^G \mathcal{X}>$$
où $\Res_K^G \mathcal{X}$ est le caractère du $K$-module $\Res_K^G X$ et $\Ind_K^G \mathcal{Y}$ est le caractère du $G$-module $\Ind_K^G Y.$
\end{theoreme}
\begin{proof}
Voir Théorème 1.12.6 du livre \cite{sagan2001symmetric}.
\end{proof}
Dans le cas où $V={1}$ est muni de la représentation triviale définie par $g\cdot 1=1,$ on note $\Ind_K^GV$ par $1_K^G.$ En fait, dans ce cas-là $1_K^G$ peut être identifié avec le $G$-module $\mathbb{C}[G/K]$ défini par:
$$g(k_iK)=(gk_i)K \text{ pour tout $g\in G$ et tout représentant $k_i$ de $G/K$}.$$  
\begin{prop}
Le $G$-module $\mathbb{C}[G/K]$ peut se décomposer en $G$-modules irréductibles de façon que chaque $G$-module irréductible apparaisse au plus une fois dans la décomposition si et seulement si la paire $(G,K)$ est une paire de Gelfand. 
\end{prop}
\begin{proof}
Voir le résultat (1.1) page 389 de \cite{McDo}.
\end{proof}
Supposons que $(G,K)$ est une paire de Gelfand et qu'on a :
$$\mathbb{C}[G/K]=\bigoplus_{i=1}^{s}X_i,$$
où les $X_i$ sont des $G$-modules irréductibles. On définit les fonctions $\omega_i:G\rightarrow \mathbb{C}$ en utilisant les caractères irréductibles $\mathcal{X}_i$ de la façon suivante:
\begin{equation}
\omega_i(x)=\frac{1}{|K|}\sum_{k\in K}\mathcal{X}_i(x^{-1}k), 
\end{equation}
pour tout $x\in G.$ Les fonctions $(\omega_i)_{1\leq i\leq s}$\label{nomen:fct_sph} \nomenclature[12]{$\omega_i$}{Une fonction sphérique zonale d'une paire de Gelfand \quad \pageref{nomen:fct_sph}} sont appelées \textit{fonctions sphériques zonales} de la paire $(G,K).$ Les fonctions sphériques zonales possèdent plusieurs propriétés remarquables, voir page 389 du livre \cite{McDo}. Elles forment une base orthogonale de $C(G,K)$ pour le produit scalaire défini par :
\begin{equation*}
<f,g>=\frac{1}{|G|}\sum_{x\in G}f(x)\overline{g(x)},
\end{equation*}
pour tout $f,g\in C(G,K)$ et elles vérifient l'égalité suivante:
\begin{equation}\label{prop_fon_zon}
\omega_i(x)\omega_i(y)=\frac{1}{|K|}\sum_{k\in K}\omega_i(xky),
\end{equation}
pour tout $1\leq i\leq s$ et pour tout $x,y\in G.$ 

\subsection{Un exemple: la paire de Gelfand $(G\times G^{opp},\diag(G))$}

On a démontré dans la Section \ref{sec:ctr_double_classe} que la paire $(G\times G^{opp},\diag(G))$ nous permet de voir le centre de l'algèbre du groupe $G$ comme une algèbre de doubles-classes. Le but de cette section est de montrer que, en plus, cette paire nous permet encore de voir les caractères normalisés de $G$ comme des fonctions sphériques zonales de la paire $(G\times G^{opp},\diag(G)).$ Ce fait est expliqué dans l'introduction de l'article \cite{strahov2007generalized} de Strahov ainsi que dans l'Exemple 9 de la section VII.1 de \cite{McDo}.

Supposons qu'on a deux groupes $G$ et $H$ et que $X$ est un $G$-module et $Y$ un $H$-module. On peut construire un $G\times H$-module noté $X\otimes Y.$ Le $G\times H$-module $X\otimes Y$ est muni de l'action définie matriciellement par le produit tensoriel des matrices. Si $A$ et $B$ sont deux matricces, $A\otimes B$ est la matrice obtenue en multipliant toutes les entrées $a_{ij}$ de la matrice $A$ par la matrice $B,$

$$A\otimes B=\begin{pmatrix}
a_{11}B & a_{12}B & \cdots \\
a_{21}B & a_{22}B & \cdots \\
\vdots & \vdots & \ddots
\end{pmatrix}.$$

On note par $\mathcal{X}\otimes \mathcal{Y}$ le caractère du $G\times H$-module $X\otimes Y.$ Par définition du produit tensoriel des matrices, on a :
$$\tr(A\otimes B)=\sum_ia_{ii}\tr(B)=\tr(A)\tr(B),$$
pour n'importe quelles matrices $A$ et $B.$ Donc on a pour tout $(g,h)\in G\times H,$
$$\mathcal{X}\otimes \mathcal{Y}(g,h)=\tr(\rho_{X\otimes Y}(g,h))=\tr(\rho_{X}(g)\otimes \rho_Y(h))=\tr(\rho_{X}(g))\tr(\rho_{Y}(h))=\mathcal{X}(g)\mathcal{Y}(h).$$

Il apparaît que si on a tous les $G$-modules irréductibles et tous les $H$-modules irréductibles de deux groupes $G$ et $H,$ on peut donner tous les $G\times H$-modules irréductibles à partir du produit tensoriel des modules. Le Théorème 1.11.3 du livre \cite{sagan2001symmetric}, nous donne l'égalité suivante :

$$\widehat{G\times H}=\hat{G}\otimes \hat{H},$$
où $$\hat{G}\otimes \hat{H}:= \lbrace X\otimes Y \text{ où $X\in \hat{G}$ et $Y\in \hat{H}$}\rbrace.$$ 

Si $X$ est un $G$-module, $X$ est encore un $G^{opp}$-module pour l'action $\cdot$ définie par:
$$g\cdot x=g^{-1}x,$$
pour tout $g\in G$ et tout $x\in X.$ Le caractère de $X$ vu comme $G^{opp}$-module est donc la fonction $\overline{\mathcal{X}}.$ Supposons que 
$$\hat{G}=\lbrace X_i \text{ tel que } 1\leq i \leq |\mathcal{C}(G)|\rbrace,$$
alors on a : 
$$\widehat{G\times G^{opp}}=\lbrace X_i\otimes \overline{X_j}\text{ tel que } 1\leq i,j \leq |\mathcal{C}(G)|\rbrace.$$
\begin{prop}
Si $G$ est un groupe fini alors la paire $(G\times G^{opp},\diag(G))$ est une paire de Gelfand.
\end{prop}
\begin{proof}
Les $G\times G^{opp}$-modules irréductibles sont de la forme $X_i\otimes \overline{X_j}$ où $1\leq i,j \leq |\mathcal{C}(G)|.$ Il suffit de  montrer que chaque $X_i\otimes \overline{X_j}$ apparaît au plus une fois dans la décomposition du $G\times G^{opp}$-module  $\Ind_{\diag(G)}^{G\times G^{opp}}1.$ Le nombre d'apparition du $G\times G^{opp}$-module irréductible $X_i\otimes \overline{X_j}$ dans la décomposition du $G\times G^{opp}$-module  $\Ind_{\diag(G)}^{G\times G^{opp}}1$ est égale à :
$$<\Ind_{\diag(G)}^{G\times G^{opp}}1, \mathcal{X}_i\otimes \mathcal{\overline{X}}_j>,$$
où $\mathcal{X}_i\otimes \mathcal{\overline{X}}_j(x,y)=\mathcal{X}_i(x)\mathcal{\overline{X}}_j(y)$ pour tout $x,y\in G.$
Par le Théorème \ref{frob} de Frobenius on a : 
\begin{eqnarray*}
<\Ind_{\diag(G)}^{G\times G^{opp}}1, \mathcal{X}_i\otimes \mathcal{\overline{X}}_j>&=&<1, \Res_{\diag(G)}^{G\times G^{opp}} \mathcal{X}_i\otimes \mathcal{\overline{X}}_j>\\
&=&\frac{1}{|\diag(G)|}\sum_{(g,g^{-1})\in \diag(G)}\mathcal{X}_i\otimes \mathcal{\overline{X}}_j(g,g^{-1})\,\,\,\,\,\,\text{(définition du $<,>$)}\\
&=&\frac{1}{|G|}\sum_{g\in G}\mathcal{X}_i(g)\mathcal{\overline{X}}_j(g^{-1})\\
&=&<\mathcal{X}_i,\mathcal{\overline{X}}_j>\\
&=&\delta_{\mathcal{X}_i,\mathcal{\overline{X}}_j}.
\end{eqnarray*}
Donc on a :
$$\Ind_{\diag(G)}^{G\times G^{opp}}1=\sum_{i=1}^{|\mathcal{C}(G)|}X_i\otimes X_i.$$
\end{proof}
On déduit de ce résultat le corollaire suivant dont une preuve directe a été donnée dans la Section \ref{sec:ctr_double_classe}.
\begin{cor} Soit $G$ un groupe fini alors on a :
$$|\mathcal{C}(G)|=|\diag(G)\setminus G\times G^{opp}/\diag(G)|.$$
\end{cor}
Les fonctions sphériques zonales $(\omega_i)_{1\leq i \leq |\mathcal{C}(G)|}$ de la paire $(G\times G^{opp},\diag(G))$ sont définies ainsi :
\begin{eqnarray*}
\omega_i(x,y)&=&\frac{1}{|G|}\sum_{g\in G}\mathcal{X}_i\otimes \mathcal{X}_i((x,y)^{-1}(g,g^{-1}))\\
&=&\frac{1}{|G|}\sum_{g\in G}\mathcal{X}_i(x^{-1}g) \mathcal{X}_i(g^{-1}y^{-1})\\
&=&\frac{1}{\dim X_i}\mathcal{X}_i(yx),
\end{eqnarray*}
la dernière égalité vient de la définition du produit et de la Proposition \ref{caract_norma}. En utilisant la propriété \eqref{prop_fon_zon} des fonctions sphériques zonales dans ce cas là, on obtient :
\begin{equation*}
\omega_i(x,1)\omega_i(y,1)=\frac{1}{|G|}\sum_{g\in G}\omega_i((x,1)(g,g^{-1})(y,1)),
\end{equation*}
pour tout $x,y\in G.$ En développant, on aura :
\begin{equation*}
\frac{1}{\dim X_i}\mathcal{X}_i(x)\frac{1}{\dim X_i}\mathcal{X}_i(y)=\frac{1}{|G|}\sum_{g\in G}\frac{1}{\dim X_i}\mathcal{X}_i(g^{-1}xgy),
\end{equation*}
pour tout $x,y\in G.$ Il faut remarquer qu'une fois cette relation étendue pour l'algèbre de groupe $\mathbb{C}[G],$ si $x$ et $y$ sont dans $Z(\mathbb{C}[G]),$ on obtient :
\begin{equation*}
\frac{1}{\dim X_i}\mathcal{X}_i(x)\frac{1}{\dim X_i}\mathcal{X}_i(y)=\frac{1}{\dim X_i}\mathcal{X}_i(xy).
\end{equation*}
Donc on retrouve le résultat de la Proposition \ref{caract_norma}. Ceci permet de voir les fonctions sphériques zonales comme des généralisations des caractères normalisés. En particulier la formule \eqref{prop_fon_zon} généralise la Proposition \ref{caract_norma}.

\subsection{Généralisation des propriétés des caractères normalisés aux fonctions sphériques zonales}\label{sec:Gen_fct_sph_zon}

Notre but ici est de donner plusieurs résultats concernant les fonctions sphériques zonales qui généralisent des résultats sur les caractères irréductibles de groupe donnés dans la Section \ref{sec:Th_de_rep}. Plus particulièrement, on montre dans la Proposition \ref{homomorphism} que les fonctions sphériques zonales sont des morphismes ce qui est similaire au résultat de la Proposition \ref{caract_norma}. On montre encore, dans le Théorème \ref{Th_tt}, que dans le cas d'une paire de Gelfand, les coefficients de structure s'expriment en fonction des fonctions sphériques zonales, ce qui est un résultat similaire au Théorème \ref{coef_fct_cara} dans le cas du centre d'une algèbre de groupe. 

Il est à noter -- selon les connaissances de l'auteur -- que les résultats donnés dans cette partie peuvent être retrouvés dans des papiers qui traitent des cas particuliers mais pas sous leur forme générale présentée ici. 

\begin{prop}\label{homomorphism}
Les fonctions sphériques zonales d'une paire de Gelfand $(G,K)$ définissent des morphimes entre $\mathbb{C}[K\setminus G/ K]$ et $\mathbb{C}^*.$\label{nomen:C_etoile} \nomenclature[13]{$\mathbb{C}^*$}{L'ensemble $\mathbb{C}\setminus \lbrace 0\rbrace$ \quad \pageref{nomen:C_etoile}}
\end{prop}
\begin{proof}
Etant définies sur le groupe $G,$ les fonctions sphériques zonales peuvent être étendues linéairement pour être définies sur l'algèbre du groupe $\mathbb{C}[G].$ On aura dans ce cas-là encore la propriété \eqref{prop_fon_zon}. Si $x$ et $y$ sont deux éléments de $\mathbb{C}[K\setminus G/ K],$ alors on a:
\begin{equation*}
\frac{1}{|K|}\sum_{k\in K}\omega_i(xky)=\frac{1}{|K|}\sum_{k\in K}\omega_i(xy)=\omega_i(xy),
\end{equation*}
D'où l'égalité $\omega_i(xy)=\omega_i(x)\omega_i(y)$ pour tout $x,y \in \mathbb{C}[K\setminus G/ K].$
\end{proof}

On reprend l'ensemble fini $\mathcal{J}$ dont les éléments $\lambda$ indexent les doubles-classes de $K$ dans $G,$
$$K\setminus G/K=\lbrace DC_\lambda \mid \lambda\in \mathcal{J}\rbrace.$$ 
Pour tout $\lambda\in \mathcal{J}$ on associe la fonction $f_\lambda: G\rightarrow \mathbb{C}$ définie ainsi :
$$f_\lambda(g)=\left\{
\begin{array}{ll}
  1 & \qquad \mathrm{si}\quad g\in DC_\lambda, \\
  0 & \qquad \mathrm{sinon .}\quad \\
 \end{array}
 \right.$$
Il n'est pas difficile de vérifier que la famille $(f_\lambda)_{\lambda\in \mathcal{J}}$ forme une base pour $C(G,K),$ donc :
\begin{equation}\label{dimC(G,K)}
\dim C(G,K) = |K\setminus G/K|.
\end{equation}
\begin{prop}\label{egalité_doubleclasse_fctzonale}
Soit $G$ un groupe et $K$ un sous-groupe de $G.$ Si $(G,K)$ est une paire de Gelfand alors le nombre des fonctions sphériques zonales de la paire $(G,K)$ est égal au nombre de doubles-classes de $K$ dans $G.$
\end{prop}
\begin{proof}
D'après \eqref{dimC(G,K)}, la dimension de $C(G,K)$ est égale au nombre des doubles-classes de $K$ dans $G.$ De plus, si $(G,K)$ est une paire de Gelfand alors les fonctions sphériques zonales forment une base pour $C(G,K),$ d'où le résultat.
\end{proof}

On utilise l'ensemble $\mathcal{J}$ pour indexer les doubles-classes de $K$ dans $G.$ D'après la Proposition \ref{egalité_doubleclasse_fctzonale}, si $(G,K)$ est une paire de Gelfand alors on a $\vert\mathcal{J}\vert$ fonctions sphériques zonales pour $(G,K).$ On va utiliser un ensemble qu'on note $\mathcal{J}^{'}$ ($\vert\mathcal{J}^{'}\vert=\vert\mathcal{J}\vert$) pour indexer les fonctions sphériques zonales de $(G,K)$ car celles-ci ne sont pas naturellement indexées par les doubles-classes de $K$ dans $G.$ Si $(G,K)$ est une paire de Gelfand on peut écrire :
$$\mathbb{C}[G/K]=\bigoplus_{\theta\in \mathcal{J}^{'}}X^\theta,$$
où les $X^\theta$ sont des $G$-modules irréductibles. 
Soit $\theta\in \mathcal{J}^{'},$ comme les fonctions sphériques zonales sont constantes sur les doubles-classes de $K$ dans $G,$ si $\lambda\in \mathcal{J},$ on va noter par $\omega^{\theta}_\lambda$ la valeur de la fonction sphérique $\omega^{\theta}$ sur n'importe quel élément de la double-classe $DC_\lambda,$
$$\omega^{\theta}_\lambda:=\omega^{\theta}(g),$$
pour un élément $g\in DC_\lambda.$ Parmi les propriétés remarquables des fonctions sphériques zonales données au chapitre VII de \cite{McDo}, rappelons les suivantes :
\begin{enumerate}\label{1}
\item[1-] $\omega^{\theta}\omega^{\psi}=\delta_{\theta\psi}\frac{|G|}{\mathcal{X}^{\theta}(1)}\omega^{\theta}$
\item[2-]$<\omega^{\theta},\omega^{\psi}>=\delta_{\theta\psi}\frac{1}{\mathcal{X}^{\theta}(1)}$
\end{enumerate}
où $\mathcal{X}^{\theta}(1)=\dim X^\theta.$

Dans le cas où $(G,K)$ est une paire de Gelfand, on va noter par $\widehat{(G,K)}$\label{nomen:ens_fct_sph} \nomenclature[13]{$\widehat{(G,K)}$}{Ensemble des fonctions sphériques zonales de la paire de Gelfand $(G,K)$ \quad \pageref{nomen:ens_fct_sph}} l'ensemble des fonctions sphériques zonales de $(G,K).$

\begin{prop}\label{isom_algebredoubleclasse_fct}
Soit $G$ un groupe fini et soit $K$ un sous-groupe de $G$ tel que $(G,K)$ est une paire de Gelfand, l'application 
$$\begin{array}{ccccc}
\mathbb{C}[K\setminus G / K] & \to & \mathcal{F}(\widehat{(G,K)},\mathbb{C}) \\
x & \mapsto & (\omega^{\phi}\mapsto \omega^{\phi}(x)) \\
\end{array}$$
est un isomorphisme d'algèbres où $\mathcal{F}(\widehat{(G,K)},\mathbb{C})$ est l'ensemble des applications de $\widehat{(G,K)}$, l'ensemble des fonctions sphériques zonales de la paire $(G,K),$ à valeur dans $\mathbb{C}.$
\end{prop}
\begin{proof}
La preuve s'obtient directement des Propositions \ref{homomorphism} et \ref{egalité_doubleclasse_fctzonale}.
\end{proof}

La Proposition \ref{isom_algebredoubleclasse_fct} est l'analogue de la Proposition \ref{isom_centre_fct} dans le cadre des fonctions sphériques zonales.

Pour un élément $\theta\in \mathcal{J}^{'},$ on définit l'élément $E_\theta$ ainsi :
$$E_\theta:=\frac{\mathcal{X}^\theta(1)}{|G|}\sum_{g\in G}\overline{\omega^{\theta}(g)}g=\frac{\mathcal{X}^\theta(1)}{|G|}\sum_{\rho\in \mathcal{J}}\overline{\omega^{\theta}_\rho}\bf{DC}_\rho.$$
\begin{prop}\label{eq:dc_foncsphe}
La famille $(E_\theta)_{\theta\in \mathcal{J}^{'}}$ est une base de $\mathbb{C}[K\setminus G/K]$ constituée d'idempotents. De plus, on a :
$${\bf{DC}_\lambda}=|DC_\lambda|\sum_{\psi\in \mathcal{J}^{'}}\omega^\psi_\lambda E_\psi.$$
\end{prop}
\begin{proof}
Le fait que $(E_\theta)_{\theta\in \mathcal{J}^{'}}$ est une famille d'idempotents découle de la propriété $2$ ci-dessus des fonctions sphériques zonales. En faite l'image de $E_\theta$ par l'isomorphisme donné dans la Proposition \ref{isom_algebredoubleclasse_fct} est la fonction 
$$\begin{array}{ccccc}
& & \widehat{(G,K)} & \to & \mathbb{C} \\
& & \omega^\phi & \mapsto & \left\{
\begin{array}{ll}
  1 & \qquad \mathrm{si}\quad \phi=\theta \\
  0 & \qquad \mathrm{sinon}\quad \\
 \end{array}
 \right. \\
\end{array}.$$
car $$\omega^\phi(E_\theta)=\mathcal{X}^\theta(1)<\omega^\phi,\omega^\theta>,$$
pour tout $\phi\in \mathcal{J}^{'}.$ Pour montrer que ${\bf{DC}_\lambda}=|DC_\lambda|\sum_{\psi\in \mathcal{J}^{'}}\omega^\psi_\lambda E_\psi,$ on commence par développer la somme à droite :
\begin{eqnarray*}
|DC_\lambda|\sum_{\psi\in \mathcal{J}^{'}}\omega^\psi_\lambda E_\psi &=&\sum_{\psi\in \mathcal{J}^{'}}\omega^\psi_\lambda  \frac{\mathcal{X}^\psi(1)}{|G|}\sum_{\delta\in \mathcal{J}}|DC_\lambda|\overline{\omega^{\psi}_\delta}\bf{DC}_\delta\\
&=&\sum_{\psi\in \mathcal{J}^{'}}\sum_{\delta\in \mathcal{J}}\mathcal{X}^\psi(1)\frac{|DC_\lambda|}{|G|}\omega^\psi_\lambda  \overline{\omega^{\psi}_\delta}\bf{DC}_\delta\\
&=&\sum_{\psi\in \mathcal{J}^{'}}\mathcal{X}^\psi(1)\frac{|DC_\lambda|}{|G|}\omega^\psi_\lambda \overline{\omega^{\psi}_\lambda}\bf{DC}_\lambda\\
&=&{\bf{DC}_\lambda} \,\,\,\, (\text{car $\sum_{\psi\in \mathcal{J}^{'}}\frac{|DC_\lambda|}{|G|}\omega^\psi_\lambda  \overline{\omega^{\psi}_\lambda}=<\omega^\psi,\omega^\psi>=\frac{1}{\mathcal{X}^\psi(1)}$ }).
\end{eqnarray*}
Il est clair que les $(E_\theta)_{\theta\in \mathcal{J}^{'}}$ forment une base de $\mathbb{C}[K\setminus G/K]$ car $|\mathcal{J}^{'}|$ est égal au nombre des doubles-classes de $K$ dans $G.$
\end{proof}
Cette proposition nous permet de donner une formule pour les coefficients de structure de l'algèbre de doubles-classes en fonction des fonctions sphériques zonales dans le cas où la paire $(G,K)$ est une paire de Gelfand.
\begin{theoreme}\label{Th_tt}
Soit $(G,K)$ une paire de Gelfand et soit $\mathcal{J}$ l'ensemble qui indexe les double-classes de $K$ dans $G.$ Soient $\lambda,\delta$ et $\rho$ trois éléments de $\mathcal{J},$ le coefficient de structure $k_{\lambda\delta}^\rho$ de l'algèbre de doubles-classes $\mathbb{C}[K\setminus G/K]$ s'exprime en fonction des fonctions sphériques zonales de la façon suivante:
$$k_{\lambda\delta}^\rho=\frac{|DC_\lambda||DC_\delta|}{|G|}\sum_{\theta\in \mathcal{J}^{'}}\mathcal{X}^\theta(1)\omega^\theta_\lambda \omega^\theta_\delta \overline{\omega^{\theta}_\rho}.$$
\end{theoreme}
\begin{proof}
D'après la Proposition \ref{eq:dc_foncsphe} on peut écrire :
\begin{eqnarray*}
\bf{DC}_\lambda\bf{DC}_\delta &=& |DC_\lambda|\sum_{\theta\in \mathcal{J}^{'}}\omega^\theta_\lambda E_\theta |DC_\delta|\sum_{\phi\in \mathcal{J}^{'}}\omega^\phi_\delta E_\phi \\
&=&|DC_\lambda||DC_\delta|\sum_{\theta\in \mathcal{J}^{'}}\omega^\theta_\lambda \omega^\theta_\delta E_\theta \\
&=&|DC_\lambda||DC_\delta|\sum_{\theta\in \mathcal{J}^{'}}\omega^\theta_\lambda \omega^\theta_\delta \frac{\mathcal{X}^\theta(1)}{|G|}\sum_{\rho\in \mathcal{J}}\overline{\omega^{\theta}_\rho}\bf{DC}_\rho,
\end{eqnarray*}
d'où on obtient le résultat voulu.
\end{proof}

Le Théorème \ref{Th_tt} est un analogue du Théorème \ref{coef_fct_cara}. L'auteur n'a pas pu trouver dans la littérature un énoncé similaire à celui de ce théorème qui traite le cas général des paires de Gelfand. Des cas particuliers ont été donnés dans les articles \cite{jackson2012character} et \cite{GouldenJacksonLocallyOrientedMaps}.

\chapter{Le centre de l'algèbre du groupe symétrique}
\label{chapitre2}

Ce chapitre est consacré à l'étude d'un cas particulier de centre d'algèbre de groupe fini, celui du centre de l'algèbre du groupe symétrique. Le \textit{groupe symétrique} $\mathcal{S}_n,$ où $n$ est un entier strictement positif, est le groupe des fonctions bijectives de l'ensemble $[n]:=\lbrace 1,\cdots ,n\rbrace,$ à valeurs dans $[n].$ Les éléments de $\mathcal{S}_n$ sont appelés permutations de $n.$ L'algèbre de ce groupe présente de nombreuses propriétés combinatoires.

Après avoir rappelé les définitions principales concernant les partitions, on va montrer dans la première section que ces objets indexent les classes de conjugaison du groupe symétrique $\mathcal{S}_n$ et donc, aussi une base de $Z(\mathbb{C}[\mathcal{S}_n]).$ Les coefficients de structure de $Z(\mathbb{C}[\mathcal{S}_n])$ associés à cette base ont un sens combinatoire, car ils comptent le nombre de graphes dessinés sur des surfaces orientées, voir \cite{CoriHypermaps}, \cite{lando2004graphs}. Nous allons présenter ce résultat en Section \ref{sec:int_com}.

L'ensemble $\hat{\mathcal{S}_n}$ des $\mathcal{S}_n$-modules irréductibles peut être indexé par les partitions de $n.$ Dans la Section \ref{sec:rep_grp_sym}, on va rappeler comment construire le $\mathcal{S}_n$-module irréductible $S^\lambda$ associé à une partition $\lambda$ de $n.$ Dans cette même section, on donne une relation entre la théorie des représentations du groupe symétrique et les fonctions symétriques: les caractères irréductibles du groupe symétrique apparaissent dans le développement des fonctions de Schur dans la base des fonctions puissance. 

Puisqu'il n'existe pas une formule générale pour calculer les coefficients de structure du centre de l'algèbre du groupe symétrique, plusieurs auteurs, voir \cite{Stanley1981255}, \cite{JaVi90}, \cite{jackson1987counting}, \cite{GoupilSchaefferStructureCoef}, ont utilisé la formule exprimant les coefficients de structure en fonction des caractères irréductibles pour calculer ces coefficients dans des cas particuliers. Ainsi, Goupil et Schaeffer ont pu donner une formule pour les coefficients de structure lorqu'une des partitions est de la forme $(n),$ voir \cite{GoupilSchaefferStructureCoef}. 

On donnera aussi une formule à la fin de la Section \ref{sec:def} quand une des partitions est de la forme $(2,1^{n-2}).$ Cette formule, obtenue d'une manière directe, est loin d'être évidente bien que le cas considéré soit probablement le plus simple. Pour l'obtenir on observe comment la multiplication par une transposition transforme le type cyclique d'une permutation.

Dans une autre direction, Farahat et Higman ont montré en $1959$ une propriété de polynomialité pour les coefficients de structure du $Z(\mathbb{C}[\mathcal{S}_n])$ dans l'article \cite{FaharatHigman1959}. En $1999,$ Ivanov et Kerov ont proposé, dans \cite{Ivanov1999}, une preuve combinatoire de ce résultat en introduisant les permutations partielles. Leur preuve utilise une algèbre universelle qui se projette sur $Z(\mathbb{C}[\mathcal{S}_n])$ pour tout $n.$ En donnant des filtrations sur cette algèbre ils ont pu majorer le degré des polynômes. 

Cette algèbre universelle est en bijection avec l'algèbre des fonctions symétriques décalées d'ordre $1.$ Les fonctions symétriques décalées sont des déformations des fonctions symétriques. La notion d'algèbre des fonctions symétriques décalées apparaît pour la première fois dans l'article \cite{okounkov1997shifted} d'Okounkov et Olshanski. Les détails de ces résultats se trouvent dans les Sections \ref{sec:filt} et \ref{sec:iso}.

\section{Préliminaires}\label{sec:def}

On rappelle dans cette section les principales définitions concernant les partitions. On montre ensuite que les classes de conjugaison du groupe symétrique sont indexées par les partitions de $n.$ Cela nous permet, d'après les résultats du premier chapitre, de donner une base indexée aussi par les partitions de $n$ pour le centre de l'algèbre du groupe symétrique\label{nomen:grp_symm} \nomenclature[14]{$\mathcal{S}_n$}{Le groupe symétrique de $[n]$ \quad \pageref{nomen:grp_symm}} $\mathcal{S}_n.$
 
\subsection{Les partitions} 
Une \textit{partition} $\lambda=(\lambda_1,\lambda_2,\cdots,\lambda_r)$ est une suite décroissante d'entiers strictement positifs. Les $\lambda_i$ sont les \textit{parts} de la partition $\lambda$ et la \textit{taille} de $\lambda$ est la somme des $\lambda_i,$ notée $|\lambda|$\label{nomen:taille_part} \nomenclature[15]{$|\lambda|$}{La taille de la partition $\lambda$ \quad \pageref{nomen:taille_part}}. La \textit{longueur} de $\lambda$ notée $l(\lambda)$\label{nomen:long_part} \nomenclature[16]{$l(\lambda)$}{La longueur de la partition $\lambda$ \quad \pageref{nomen:long_part}} est le nombre $r$ de ses parts. On dit que $\lambda$ est une partition de $n$ et on écrit $\lambda\vdash n$ si $|\lambda|=n.$ Si on note par $m_i(\lambda)$ le nombre des parts de $\lambda$ égales à $i,$ alors $\lambda$ peut être écrite ainsi (notation exponentielle) :
$$\lambda=(1^{m_1(\lambda)},2^{m_2(\lambda)},\cdots).$$
On va noter par $\mathcal{P}$ l'ensemble de toutes les partitions et par $\mathcal{P}_n$ l'ensemble des partitions de $n.$

Pour une partition $\lambda,$ on définit le nombre $z_\lambda$ 
ainsi :
$$z_\lambda=\prod_{i\geq 1}i^{m_i(\lambda)}m_i(\lambda)!.$$

L'union des deux partitions $\lambda=(1^{m_1(\lambda)},2^{m_2(\lambda)},\cdots)$ et $\delta=(1^{m_1(\delta)},2^{m_2(\delta)},\cdots)$ est la partition obtenue par réunion des parts de $\lambda$ et $\delta,$ autrement dit:
$$\lambda\cup \delta:=(1^{m_1(\lambda)+m_1(\delta)},2^{m_2(\lambda)+m_2(\delta)},\cdots).$$

Une \textit{partition propre} est une partition qui ne possède aucune part égale à $1.$ L'ensemble de toutes les partitions propres est noté par $\mathcal{PP},$\label{nomen:ens_pro} \nomenclature[17]{$\mathcal{PP}$}{L'ensemble des partitions propres \quad \pageref{nomen:ens_pro}} et $\mathcal{PP}_n$\label{nomen:ens_pro_n} \nomenclature[18]{$\mathcal{PP}_n$}{L'ensemble des partitions propres de taille égale à $n$ \quad \pageref{nomen:ens_pro_n}} est l'ensemble de toutes les partitions propres de taille $n.$ L'ensemble $\mathcal{P}_n$\label{nomen:ens_part_n} \nomenclature[19]{$\mathcal{P}_n$}{L'ensemble des partitions de taille $n$ \quad \pageref{nomen:ens_part_n}} des partitions de $n$ est en bijection avec l'ensemble $\mathcal{PP}_{\leq n},$\label{nomen:ens_part_pro_petit_n} \nomenclature[20]{$\mathcal{PP}_{\leq n}$}{L'ensemble des partitions propres de taille plus petite ou égale à $n$ \quad \pageref{nomen:ens_part_pro_petit_n}} défini par :
$$\mathcal{PP}_{\leq n}:=\bigsqcup_{0\leq r \leq n}\mathcal{PP}_r.$$
La bijection est donnée explicitement par :
$$\begin{array}{ccccc}
&  & \mathcal{P}_n & \longrightarrow & \mathcal{PP}_{\leq n} \\
& & \lambda & \longmapsto & \overline{\lambda}:=(1^0,2^{m_2(\lambda)},\cdots)\\
\end{array}$$
et son inverse est\label{nomen:compl_lambda} \nomenclature[21]{$\underline{\lambda}_n$}{La partition $(1^{n-|\lambda|},2^{m_2(\lambda)},\cdots)$ \quad \pageref{nomen:compl_lambda}} :
$$\begin{array}{ccccc}
&  & \mathcal{PP}_{\leq n} & \longrightarrow & \mathcal{P}_{n} \\
& & \lambda & \longmapsto & \underline{\lambda}_n:=(1^{n-|\lambda|},2^{m_2(\lambda)},\cdots)\\
\end{array}.$$
Informellement, la bijection consiste à enlever les parts égales à $1$ d'une partition de $n$ pour obtenir une partition propre de taille inférieure ou égale à $n$ et à rajouter des parts égales à $1$ à une partition propre de taille inférieure ou égale à $n$ jusqu'à obtenir une partition de $n.$

\subsection{Classes de conjugaisons du groupe symétrique}\label{sec:base_du_centre_de_S_n}

Parmi les multiples façons d'écrire une permutation de $n,$ on va considérer ici l'écriture en produit de cycles disjoints. L'image d'un entier $i$ avec cette écriture est l'élément qui suit $i$ dans son cycle. Par exemple, la permutation $\omega= (1\,\,5\,\,6\,\,7)(2\,\,10)(3\,\,9\,\,4)(8)$ est bijectivement présentée en deux lignes ainsi :
$$\omega=\begin{matrix}
1&2&3&4&5&6&7&8&9&10\\
5&10&9&3&6&7&1&8&4&2
\end{matrix}.$$
Cette écriture est unique à une permutation des cycles près. Un cycle de taille $r,$
$$(c_1\,\,c_2\,\,\cdots \,\,c_r)$$
représente la bijection qui envoie $c_1$ à $c_2,$ $c_2$ à $c_3,$ ainsi de suite et $c_r$ à $c_1.$ Il y a $r$ façon différente de l'écrire selon le choix du premier élément. Par exemple : $(1\,\,2\,\,3)=(2\,\,3\,\,1)=(3\,\,1\,\,2).$

Le \textit{type-cyclique} d'une permutation $\sigma$ de $n$ est la partition de $n$ formée à partir des tailles des cycles de $\sigma$ en ordre décroissant. Par exemple, la permutation $\omega$ ci-dessus a la partition $(4,3,2,1)$ comme type-cyclique.

Dans tout ce qui suit, si $f$ et $g$ sont deux permutations, $fg$ désigne la composition $g\circ f.$

\begin{prop}
Deux permutations de $\mathcal{S}_n$ sont conjuguées si et seulement si elles possèdent le même type-cyclique.
\end{prop}
\begin{proof}
Soit $\sigma$ une permutation de $n$ écrite comme produit des cycles disjoints,$$\sigma=(i_1~~i_2~~\ldots~~i_l)~~\cdots~~(i_m~~i_{m+1}~~\ldots~~i_n).$$ 
Alors pour tout $\varepsilon\in \mathcal{S}_n$ on a:
$$\varepsilon\sigma\varepsilon^{-1}=(\varepsilon^{-1}(i_1)~~\varepsilon^{-1}(i_2)~~\ldots~~\varepsilon^{-1}(i_l))~~\cdots~~(\varepsilon^{-1}(i_m)~~\varepsilon^{-1}(i_{m+1})~~\ldots~~\varepsilon^{-1}(i_n)),$$
ce qui signifie que le type cyclique de $\sigma$ est égale au type cyclique de $\varepsilon\sigma\varepsilon^{-1}$. Réciproquement si deux permutations $\sigma$ et $\tau$ possèdent le même type cyclique, supposons que  $$\sigma=(i_1~~i_2~~\ldots~~i_l)~~\cdots~~(i_m~~i_{m+1}~~\ldots~~i_n)$$
et 
$$\tau=(j_1~~j_2~~\ldots~~j_l)~~\cdots~~(j_m~~j_{m+1}~~\ldots~~j_n),$$
alors on a $\sigma=\varepsilon\tau\varepsilon^{-1}$ où $\varepsilon$ est la permutation de $n$ définie par : 
$$\varepsilon (i_{r})=j_{r}, ~~ 1\leq r\leq n.$$
Cela termine la preuve.
\end{proof}

Comme conséquence de cette proposition, la classe de conjugaison d'une permutation $\omega$ de $n$ est l'ensemble des permutations qui possèdent le même type-cyclique que $\omega.$ Cela nous permet d'indexer les classes de conjugaisons du groupe symétrique $\mathcal{S}_n$ par l'ensemble des partitions de $n.$

\begin{cor}
Soit $\omega$ une permutation de $n$ et $\lambda$ une partition de $n$ tel que $\type-cyclique(\omega)=\lambda,$ alors la classe de conjugaison de $\omega$ est :
$$C_\omega=\lbrace \sigma\in \mathcal{S}_n ~~|~~ \type-cyclique(\sigma)=\lambda\rbrace.$$
\end{cor}

Pour une permutation $\omega$ de $n$, il existe une formule pour le cardinal de $Stab_\omega^{\mathcal{S}_n}$ donnée explicitement par la proposition suivante.

\begin{prop}Soit $\omega$ une permutation de $n$ de type cyclique $\lambda=(1^{m_{1}},2^{m_{2}},\ldots,n^{m_{n}}).$ On a:
\begin{equation}\label{eq:taille_stab_permu}
\vert Stab_\omega^{\mathcal{S}_n}\vert=z_{\lambda}=1^{m_{1}}m_{1}!2^{m_{2}}m_{2}!\cdots n^{m_{n}}m_{n}!.
\end{equation}
\end{prop}
\begin{proof} On va calculer la taille du centralisateur (voir page $4$ pour la définition générale ) d'une permutation. D'abord faisons la remarque suivante. Pour une permutation $\tau$ de $n$ on a : 
$$\tau (i_{1}\ldots i_{k})=(i_{1}\ldots i_{k})\tau \Leftrightarrow (i_{1}\ldots i_{k})=(\tau(i_{1})\ldots \tau(i_{k})).$$
Donc si $(i_{1}\ldots i_{k})$ est un cycle de longueur $k,$ on a:
$$Stab_{(i_{1}\ldots i_{k})}^{\mathcal{S}_n}=\lbrace\tau\in \mathcal{S}_n\text{ tel que } (i_{1}\ldots i_{k})=(\tau(i_{1})\ldots \tau(i_{k}))\rbrace.$$
Pour qu'une permutation $\tau\in \mathcal{S}_n$ soit dans $Stab_{(i_{1}\ldots i_{k})}^{\mathcal{S}_n}$ il faut que $\tau(i_1)$ soit dans $\lbrace i_{1},\ldots, i_{k}\rbrace$ ce qui fait $k$ choix pour $\tau(i_1).$ La condition $(i_{1}\ldots i_{k})=(\tau(i_{1})\ldots \tau(i_{k}))$ implique les valeurs de $\tau(i_{2})\ldots \tau(i_{k}).$ Comme il n'y a aucune condition pour le choix des images des éléments de $[n]\setminus \lbrace i_{1},\ldots, i_{k}\rbrace$ par $\tau,$ on obtient :
$$\vert Stab_{(i_{1}\ldots i_{k})}^{\mathcal{S}_n}\vert=k\cdot (n-k)!.$$
De même, on peut montrer que :
$$\tau (i_{1}\ldots i_{k})\ldots (j_{1}\ldots j_{k})=(i_{1}\ldots i_{k})\ldots (j_{1}\ldots j_{k})\tau $$ $$\Leftrightarrow  $$
$$(\tau(i_{1})\ldots \tau(i_{k}))\ldots (\tau(j_{1})\ldots \tau(j_{k}))=(i_{1}\ldots i_{k})\ldots (j_{1}\ldots j_{k}),$$
et donc (même démarche que dans le cas d'un cycle):
$$\vert Stab_{\underbrace{(i_{1}\ldots i_{k})\ldots (j_{1}\ldots j_{k})}_{m_{k}}}^{\mathcal{S}_n}\vert=m_{k}!k^{m_{k}}\cdot (n-km_{k})!.$$
Dans le cas général, si $\omega$ est une permutation de $n$ et si $\mathcal{C}_1\cdots \mathcal{C}_r$ est la décomposition de $\omega$ en produit de cycles disjoints, pour une permutation $\tau$ de $n,$ on a :
$$\tau \omega =\omega \tau \Leftrightarrow \mathcal{C}_1\cdots \mathcal{C}_r = \tau(\mathcal{C}_1)\cdots \tau(\mathcal{C}_r),$$
où $\tau(\mathcal{C}_i)=(\tau(a_1)\cdots \tau (a_t))$ si $\mathcal{C}_i=(a_1\cdots a_t).$ On observe que la taille du stabilisateur de $\omega$ ne dépend que de son type-cyclique. Si $\omega$ est de type-cyclique $\lambda=(1^{m_{1}},2^{m_{2}},\ldots,n^{m_{n}}),$ on obtient la formule \eqref{eq:taille_stab_permu}.
\end{proof}
\begin{cor}\label{cor:taille_classe_de_conj_S_n}
Soit $\omega$ une permutation de $n,$ on a:
$$|C_\omega|=\frac{n!}{z_{\type-cyclique(\omega)}}.$$
\end{cor}
\begin{proof}
Cette formule est le cas particulier, dans le cas du groupe symétrique, de la formule donnée dans le Corollaire \ref{cor:taille_classe}.
\end{proof}

Comme l'ensemble des classes de conjugaison du groupe symétrique est indexé par les partitions de $n,$ d'après la Proposition \ref{prop:base_centre_alg-de_grp}, le centre de l'algèbre du groupe symétrique, $Z(\mathbb{C}[\mathcal{S}_n]),$ possède une base indexée par les partitions de $n$ :
\begin{cor} La famille $({\bf C}_\lambda)_{\lambda \in \mathcal{P}_n},$ où,
$${\bf C}_\lambda=\sum_{\omega \in \mathcal{S}_n \atop {\type-cyclique(\omega)=\lambda}}\omega,$$
forme une base pour $Z(\mathbb{C}[\mathcal{S}_n]).$
\end{cor}
\section{Coefficients de structure et interprétation combinatoire}

Soient $\lambda$ et $\delta$ deux partitions de $n.$ Les coefficients de structure du centre de l'algèbre du groupe symétrique $c_{\lambda\delta}^{\rho}$ sont définis par l'équation suivante:
\begin{equation}\label{Coef_S_n}
{\bf C}_\lambda{\bf C}_\delta=\sum_{\rho \in \mathcal{P}_n}c_{\lambda\delta}^{\rho}{\bf C}_\rho.
\end{equation}

Comme l'ensemble $\mathcal{P}_n$ des partitions de $n$ est en bijection avec l'ensemble $\mathcal{PP}_{\leq n},$ on peut indexer les éléments de base du centre de l'algèbre du groupe symétrique, $Z(\mathbb{C}[\mathcal{S}_n]),$ par les partitions propres de taille plus petite ou égale à $n.$ Dans ce cas, les coefficients de structure apparaissent naturellement comme des fonctions de $n.$

Soient $\lambda$ et $\delta$ deux partitions propres. Pour tout entier naturel $n$ tel que $n\geq |\lambda|, |\delta|$ on a l'équation suivante dans le centre de l'algèbre du groupe symétrique $\mathcal{S}_n:$

\begin{equation}\label{eq:def_str_cof_S_n}
{\bf C}_{\underline{\lambda}_n}{\bf C}_{\underline{\delta}_n}=\sum_{\rho \in \mathcal{PP}_{\leq n}} c_{\lambda\delta}^{\rho}(n){\bf C}_{\underline{\rho}_n}.
\end{equation} 

Il n'existe pas une formule explicite pour les coefficients de structure du centre de l'algèbre du groupe symétrique donnés par l'équation \eqref{Coef_S_n}. Néanmoins, il existe une description combinatoire de ces coefficients: la Proposition \ref{desc_coef} implique le résultat suivant.

\begin{cor}\label{appr_comb_coef_de_str}
Soient $\lambda,\delta$ et $\rho$ trois partitions de $n,$ alors on a :
$$c_{\lambda\delta}^{\rho}=|\lbrace (\sigma,\alpha)\in \mathcal{S}_n^2 \text{ tel que $\type-cyclique(\sigma)=\lambda,$ $\type-cyclique(\alpha)=\delta$ et $\sigma\alpha=\varphi$}\rbrace|,$$
où $\varphi$ est une permutation fixée de $n$ de type-cyclique $\rho.$
\end{cor}

\subsection{Autour du calcul des coefficients de structure}\label{sec:calc_coeff_sym}

Le but de cette section est de calculer les coefficients de structure lorsqu'une des partitions est $(2,1^{n-2})$ pour illustrer la difficulté à manier cette reformulation combinatoire.

Dans le cas trivial où $\lambda=(1^n),$ on a ${\bf C}_{(1^n)}=id_n,$ avec $id_n$ l'application identité de $[n]$ dans $[n].$ Il est facile donc de remarquer que si $\mu$ et $\rho$ sont deux partitions de $n$ on a :
$$c_{(1^n)\mu}^\rho=
\left\{
\begin{array}{ll}
  1 & \qquad \mathrm{si}\quad \mu=\rho, \\
  0 & \qquad \mathrm{sinon.} \\
 \end{array}
 \right.$$

Comme on va le voir dans la Proposition \ref{prop:comp}, le calcul devient beaucoup plus difficile dans le cas où $\lambda=(2,1^{n-2}),$ qui est apparemment le plus facile des cas non triviaux. Pour comprendre ce qui se passe quand on fait le produit ${\bf C}_{(2,1^{n-2})}{\bf C}_{\mu},$ on va expliquer dans le lemme suivant comment la multiplication par une transposition transforme une permutation.

\begin{lem}\label{lem:act_tr_per}
Soient $\sigma$ une permutation de $n$ et $\tau_{ij}$ la transposition dans $\mathcal{S}_n$ qui permute $i$ et $j,$ on a:
$$l(\sigma\circ \tau_{ij})=l(\sigma)\pm 1,$$
où $l(\sigma)$ est le nombre des cycles qui apparaissent dans l'écriture de $\sigma$ en produit de cycles disjoints.
\end{lem}
\begin{proof}
On note comme d'habitude $O_i$ l'orbite de $i$ par $\sigma$. On a deux cas à distinguer :
\begin{enumerate}
\item \underline{$j\in O_i,$ c'est-à-dire $j=\sigma^k(i)$ pour un certain $k$:}
dans ce cas, $\sigma\circ \tau_{ij}$ possède les mêmes cycles que $\sigma,$ sauf le cycle $(i~\sigma(i)~\ldots~j~\ldots~\sigma^{r}(i))$ qui sera coupé en deux cycles :
$$(i~\sigma^{k+1}(i)~\sigma^{k+2}(i)~\ldots~\sigma^{r-1}(i)~\sigma^{r}(i))$$ et $$(j~\sigma(i)~\sigma^2(i)~\ldots~\sigma^{k-1}(i)).$$
Donc, dans ce cas on a $l(\sigma\circ \tau_{ij})=l(\sigma)+1.$
\item \underline{$j\notin O_i$ :}
dans ce cas, $\sigma\circ \tau_{ij}$ possède les mêmes cycles que $\sigma,$ sauf les cycles $(i~\sigma(i)~\ldots~\sigma^{r}(i))$ et $(j~\sigma(j)~\ldots~\sigma^{p}(j)),$ qui sont réunis pour former le cycle :
$$(i~\sigma(j)~\sigma^{2}(j)~\ldots~\sigma^{p}(j)~j~~\sigma(i)~\sigma^2(i)~\ldots~\sigma^r(i)).$$
Donc, dans ce cas on a $l(\sigma\circ \tau_{ij})=l(\sigma)-1.$ \qedhere
\end{enumerate}
\end{proof}


Cette description de l'action d'une transposition sur une permutation nous permet d'obtenir la proposition suivante.

\begin{prop}\label{prop:comp}
Soit $\mu=(1^{m_{1}},2^{m_{2}},\ldots,n^{m_{n}})$ une partition de $n.$ On a :\\
$${\bf C}_{(2^{1},1^{n-2})}{\bf C}_{\mu}=\displaystyle {\sum_{{i=2},{m_{i}}\neq 0}^{n} \sum_{k\leq \lfloor i/2\rfloor } \phi_{i}^k{\bf C}_{\mu_{i}^{k}}+\sum_{i=1,m_{i}\neq 0}^{n} \sum_{k> i,m_{k}\neq 0}^{n-i} \phi_{i}^{'k}{\bf C}_{\mu^{'k}_{i}}}$$
où $$\mu_{i}^{k}=
\left\{
\begin{array}{ll}
  (1^{m_{1}},2^{m_{2}},\ldots,k^{m_{k}+1},\ldots,(i-k)^{m_{i-k}+1},\ldots,i^{m_{i}-1},\ldots,n^{m_{n}}) & \qquad \mathrm{si}\quad i-k\neq k \\
  (1^{m_{1}},2^{m_{2}},\ldots,k^{m_{k}+2},\ldots,i^{m_{i}-1},\ldots,n^{m_{n}}) & \qquad \mathrm{si}\quad i-k=k \\
 \end{array}
 \right.$$$$\mu_{i}^{'k}=
\left\{
\begin{array}{ll}
  (1^{m_{1}},2^{m_{2}},\ldots,i^{m_{i}-1},\ldots,k^{m_{k}-1},\ldots,(i+k)^{m_{i+k}+1},\ldots,n^{m_{n}}) & \qquad \mathrm{si}\quad k\neq i \\
  (1^{m_{1}},2^{m_{2}},\ldots,i^{m_{i}-2},\ldots,(2i)^{m_{2i}+1},\ldots,n^{m_{n}}) & \qquad \mathrm{si}\quad k=i \\
 \end{array}
 \right.$$
 $$\phi_{i}^{k}=
\left\{
\begin{array}{ll}
  k(m_k+1)(i-k)(m_{i-k}+1) & \qquad \mathrm{si}\quad i-k\neq k \\
  \frac{k^2(m_k+1)(m_k+2)}{2} & \qquad \mathrm{si}\quad i-k=k \\
 \end{array}
 \right.$$$$\phi_{i}^{'k}=
\left\{
\begin{array}{ll}
  (i+k)(m_{i+k}+1) & \qquad \mathrm{si}\quad k\neq i \\
  \frac{2i(m_{2i}+1)}{2} & \qquad \mathrm{si}\quad k=i. \\
 \end{array}
 \right.$$
\end{prop}
\begin{proof} On ne donne pas les détails du calcul pour obtenir les formules. Les formes $\mu_{i}^{k}$ et $\mu_{i}^{'k}$ sont une conséquence du Lemme \ref{lem:act_tr_per}. Pour calculer $\phi_{i}^{k}$ et $\phi_{i}^{'k}$, on fixe une permutation $\sigma$ de type cyclique $\mu_{i}^{k}$ et une permutation $\delta$ de type cyclique $\mu_{i}^{'k}$, d'après le Corollaire \ref{appr_comb_coef_de_str} on a :
\begin{equation*}
\phi_{i}^{k}=\vert\lbrace\tau \text{ tel que : $\tau\epsilon=\sigma$ où $\epsilon$ est une permutation de type cyclique $\mu$}\rbrace\vert
\end{equation*}
et
\begin{equation*}
\phi_{i}^{'k}=\vert\lbrace\tau^{'} \text{ tel que : $\tau^{'}\epsilon^{'}=\delta$ où $\epsilon^{'}$ est une permutation de type cyclique $\mu$}\rbrace\vert. \qedhere
\end{equation*}
\end{proof}

\begin{ex}\label{ex:calcul_coef_str}
$${\bf C}_{(1^{n-2},2)}^2=\frac{n(n-1)}{2}{\bf C}_{(1^n)}+3{\bf C}_{(1^{n-3},3)}+2{\bf C}_{(1^{n-4},2^2)}.$$
\end{ex}

\begin{ex}\label{ex_pr_cl_de_con_2}
$${\bf C}_{(1^{n-2},2)}{\bf C}_{(1^{n-3},3)}=2(n-2){\bf C}_{(1^{n-2},2)}+4{\bf C}_{(1^{n-4},4)}+{\bf C}_{(1^{n-5},2,3)}.$$
\end{ex}

Le résultat de la Proposition \ref{prop:comp} est classique. Il apparaît pour la première fois en 1986 dans le papier \cite{KatrielPaldus} de Katriel et Paldus.

La Proposition \ref{prop:comp} montre combien il est difficile de donner des formules pour les coefficients de structure de $Z(\mathbb{C}[\mathcal{S}_n])$ en utilisant l'approche combinatoire. L'action d'une transposition sur une permutation est la plus facile à décrire dans le groupe symétrique et pourtant les formules données dans la proposition sont complexes. Avec cette approche, les coefficients vont être de plus en plus durs à calculer et le cas général paraît hors de portée.

Dans \cite{GoupilSchaefferStructureCoef}, Goupil et Schaeffer donnent une expression explicite, mais relativement compliquée, pour le coefficient $c_{\lambda\delta}^{(n)},$ pour n'importe quelles partitions $\lambda,\delta\in \mathcal{P}_n.$ Ils utilisent la formule du Théorème \ref{coef_fct_cara} qui exprime les coefficients de structure en fonction des caractères irréductibles. Il semble difficile de généraliser leur méthode car dans le cas où une des partitions est $(n),$ presque tous les termes s'annulent dans la formule du Théorème \ref{coef_fct_cara}, ce qui n'est pas le cas en général.

\subsection{Interprétation combinatoire}\label{sec:int_com}

Dans \cite{CoriHypermaps}, Cori a montré que les coefficients de structure du centre de l'algèbre du groupe symétrique comptent des graphes dessinés sur des surfaces orientées. On rappelle dans cette section cette relation entre graphes et coefficients de structure.

Un graphe est dit \textit{bicolorié} si ses sommets sont coloriés avec deux couleurs -- disons noir et blanc -- de sorte que chaque arête relie deux sommets de couleurs différentes. Si un graphe bicolorié possède $a$ arêtes, on peut étiqueter ces arêtes avec les entiers de l'ensemble $[a].$ 

Précisons qu'on s'intéresse juste aux graphes connexes dans cette section. On prend comme exemple tout au long de cette partie le graphe bicolorié avec un étiquetage donné par la Figure \ref{fig:etiq_bic}. Bien que c'est un graphe plongé dans le plan (ou sur la sphère), il faut noter que l'on considèrera aussi des graphes plongés dans des surfaces de genre supérieur.

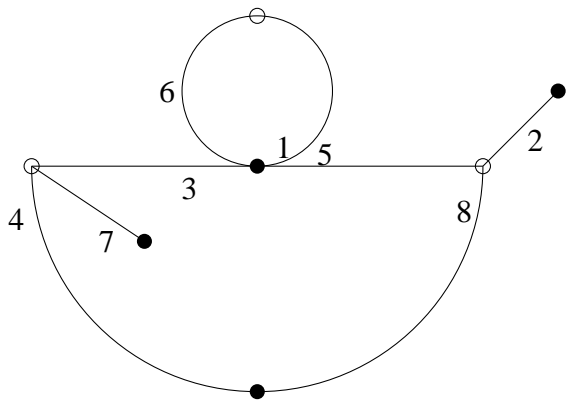
\begin{figure}[htbp]
\begin{center}
\begin{tikzpicture}
\draw (0,0) circle (0.1cm);
\fill[black] (3,0) circle (0.1cm);
\draw (6,0) circle (0.1cm);
\fill[black] (7,1) circle (0.1cm);
\fill[black] (1.5,-1) circle (0.1cm);
\draw (3,2) circle (0.1cm);
\fill[black] (3,-3) circle (0.1cm);
\draw (3,1) circle (1cm);
\draw (0,0) -- (3,0) ;
\draw (3,0) -- (6,0) ;
\draw (6,0) -- (7,1) ;
\draw (0,0) -- (1.5,-1) ;
\draw (0,0) arc (180:360:3cm) ;
\node at (2.1,-.3) {3};
\node at (-0.2,-.7) {4};
\node at (1,-1) {7};
\node at (3.35,.25) {1};
\node at (1.8,1) {6};
\node at (3.9,.15) {5};
\node at (6.7,.35) {2};
\node at (5.75,-.6) {8};
\end{tikzpicture}
\caption{Un étiquetage pour un graphe bicolorié.}
\label{fig:etiq_bic}
\end{center}
\end{figure}

Un graphe plongé dans une surface est appelé \textit{une carte} si ses faces sont homéomorphes à des disques ouverts. Les face d'une carte sont les zones -- régions disjointes de la surface séparées par les arêtes du graphe -- obtenues sur la surface une fois le graphe dessiné, la zone extérieure y compris. Par exemple la carte de la Figure \ref{fig:etiq_bic} possède trois faces.

Le \textit{degré d'un sommet} est le nombre des arêtes incidentes à ce sommet et la \textit{partition des sommets} est la partition dont les parts sont les degrés des sommets du graphe. Comme on travaille avec des cartes bicoloriés, on va distinguer la partition des sommets blancs et la partition des sommets noirs. Pour les faces, on s'intéresse au nombre d'arêtes qui entourent -- appelées arêtes incidentes -- la face. Certaines arêtes sont doublement incidentes à une face, par exemple l'arête étiquetée par $2$ de la Figure \ref{fig:etiq_bic} est doublement incidentes à la face extérieure. Ces arêtes seront comptées deux fois. Dans le cas des cartes bicoloriées, on a toujours un nombre pair d'arêtes incidentes à une face et donc on définit le degré d'une face comme la moitié de ce nombre d'arêtes. Dans la carte bicoloriée de la Figure \ref{fig:etiq_bic}, la face extérieure est de degré $4,$ la face à l'intérieur du cercle est de degré $1$ et la face à l'intérieur du demi-cercle est de degré $3.$ La \textit{ partition des faces } est la partition dont les parts sont les degrés des faces de la carte. La partition des sommets blancs de la carte de la Figure \ref{fig:etiq_bic} est $(3,3,2),$ celle des sommets noirs est $(4,2,1^2)$ tandis que la partition des faces est $(4,3,1).$ Les partitions des sommets blancs, des sommets noirs et des faces sont des partitions de $a,$ où $a$ est le nombre d'arêtes de la carte bicolorié.

Étant donnée un étiquetage d'une carte bicoloriée $G$ de $a$ arêtes, on va associer trois permutations à $G.$ La \textit{permutation des sommets noirs}, appelée $\sigma,$ est une permutation de $[a]$ avec un cycle associé à chaque sommet noir. Le cycle de $\sigma$ associé à un sommet noir est formé des étiquettes de ses arêtes incidentes placées dans le sens opposé aux aiguilles d'une montre -- on prend comme orientation le sens trigonométrique (contraire aux aiguilles d'une montre) sur le plan dans notre exemple. Dans le cas général, le parcours est fait dans le sens de l'orientation de la surface sous-jacente --. De la même façon on définit la permutation $\alpha$ des sommets blancs. 

Pour la permutation $\varphi$ des faces, on ne lit qu'une arête sur deux. Chaque arête ayant deux côtés, on procède de la manière suivante pour déterminer, étant donné une étiquette $i,$ à partir de quel côté il faut suivre la face pour trouver l'image de $i.$ On met l'étiquette à gauche de l'arête lorsqu'on la parcourt de son extrémité noire vers son extrémité blanche. Pour l'exemple de la Figure \ref{fig:etiq_bic}, l'image de $2$ par $\varphi$ est $4$ et non $8.$

On donne ci-dessous les permutations $\sigma,$ $\alpha$ et $\varphi$ du graphe de la Figure \ref{fig:etiq_bic}.

$$\sigma=(1\,\,\,\, 6\,\,\,\, 3\,\,\,\, 5)(4\,\,\,\, 8)(2)(7),$$ 
$$\alpha=(7\,\,\,\, 3\,\,\,\, 4)(2\,\,\,\, 5\,\,\,\, 8)(1\,\,\,\, 6),$$ 
$$\varphi=(2\,\,\,\, 4\,\,\,\, 6\,\,\,\, 5)(8\,\,\,\, 7\,\,\,\, 3)(1).$$

Remarquons que le type-cyclique de $\sigma$ est la partition des sommets noirs, le type-cyclique de $\alpha$ est la partition des sommets blancs et le type-cyclique de $\varphi$ est la partition des faces. 

\begin{prop}
La permutation $\varphi$ associée aux faces vérifie l'égalité suivante:
$$\varphi=\alpha^{-1}\sigma^{-1}.$$
\end{prop}
\begin{proof}
Voir Proposition 1.3.16. de \cite{lando2004graphs} pour le cas particulier où $\alpha$ est une involution et la généralisation pour $\alpha\in \mathcal{S}_n$ dans la Section 1.5 du même livre. 
\end{proof}

Une carte bicoloriée non-nécessairement connexe est une union disjointe de cartes bicoloriées connexes. En particulier, cela veut dire que chaque composante connexe du graphe est plongée dans une composante connexe différente de la surface et que les faces ainsi obtenues sont homéomorphes à des disques ouverts.  

On a vu pour l'instant comment associer à une carte bicoloriée étiquetée "connexe" un triplet de permutation $(\sigma,\alpha,\varphi).$ En réalité, cette association peut être établie aussi pour les cartes bicoloriées étiquetées non-nécessairement connexes et forme dans ce cas une bijection avec l'ensemble des triplet de permutations dont la réciproque -- comment construire une carte bicolorié étiquetée non-nécessairement connexe à partir d'un triplet $(\sigma,\alpha,\varphi)$ de permutation -- est détaillé dans le livre \cite[Construction 1.3.20]{lando2004graphs} de Lando et Zvonkin.

Soient $\lambda,$ $\delta$ et $\rho$ trois partitions de $n,$ le nombre $g_{\lambda\delta}^{\rho}$ de cartes bicoloriées étiquetées non-nécessairement connexes avec $n$ arêtes, $l(\lambda)$ sommets blancs, $l(\delta)$ sommets noirs, $l(\rho)$ faces et dont la partition des sommets noirs est $\lambda$, celle des sommets blancs est $\delta$ et celle des faces est $\rho$ 
est donné par :
$$g_{\lambda\delta}^\rho=|\lbrace (\sigma,\alpha)\in \mathcal{S}_n^2 \text{ tel que $\type-cyclique(\sigma)=\lambda,$ $\type-cyclique(\alpha)=\delta,$} $$ 
$$\text{ et } \type-cyclique(\alpha^{-1}\sigma^{-1})=\rho \rbrace|.$$

\begin{cor}
Soient $\lambda,$ $\delta$ et $\rho$ trois partitions de $n,$ alors on a :
$$g_{\lambda\delta}^\rho=\frac{n!}{z_\rho}c_{\lambda\delta}^{\rho}.$$
\end{cor}
\begin{proof}
Le résultat s'obtient directement en tenant compte de la formule des coefficients de structure donnée dans le Corollaire \ref{appr_comb_coef_de_str}.
\end{proof}

\section[Représentations de $\mathcal{S}_n$ et fonctions symétriques]{Représentations du groupe symétrique et fonctions symétriques}\label{sec:rep_grp_sym}

Dans cette partie on rappelle brièvement les définitions de quelques fonctions symétriques et on donne une description des $\mathcal{S}_n$-modules irréductibles. Les caractères irréductibles du groupe symétrique $\mathcal{S}_n$ apparaissent dans le changement de base entre fonctions de Schur et fonctions puissance dans l'algèbre des fonctions symétriques (voir l'équation \eqref{rel:sch_pui}), ce qui relie les deux théories.
    
\subsection{Algèbre des fonctions symétriques}\label{sec:alg_fct_sym}

On présente ici les principales fonctions symétriques qui vont nous servir dans cette thèse ainsi que l'algèbre de ces fonctions. Le lecteur peut se référer aux deux livres \cite{McDo} et \cite{sagan2001symmetric} pour plus des détails sur ces fonctions.

Soit $\mathbb{Z}[x_1,\cdots,x_n]$ l'anneau des polynômes en $n$ variables $x_1,\cdots,x_n.$ Le groupe $\mathcal{S}_n$ agit sur $\mathbb{Z}[x_1,\cdots,x_n]$ par :
$$\sigma f(x_1,\cdots,x_n)=f(x_{\sigma(1)},\cdots,x_{\sigma(n)}),$$
pour tout $\sigma\in \mathcal{S}_n$ et pour tout $f\in \mathbb{Z}[x_1,\cdots,x_n].$ Un polynôme $f\in \mathbb{Z}[x_1,\cdots,x_n]$ est dit \textit{symétrique} s'il est invariant par l'action de n'importe quelle permutation $\sigma\in \mathcal{S}_n.$ L'ensemble des polynômes symétriques forme un sous-anneau de $\mathbb{Z}[x_1,\cdots,x_n]$ noté $\Lambda_n.$ Par exemple, $f(x_1,x_2,x_3)=x_1x_2x_3\in \Lambda_3.$

Le \textit{degré} d'un monôme $x_1^{d_1}\cdots x_l^{d_l}$ est la somme des $d_i,$
$$\deg x_1^{d_1}\cdots x_l^{d_l}:=\sum_{i=1}^l d_i.$$ 
Un polynôme symétrique est dit \textit{homogène} de degré $d$ si tous ses monômes possèdent le même degré $d.$ Informellement, lorsqu'on travaille avec un nombre infini de variables, on parle de fonctions symétriques. 
\begin{definition}\label{def:fct_sym}
Une fonction $f$ avec un nombre infini de variables est une fonction symétrique homogène de degré $d$ si $f(x_1,x_2,\cdots,x_i, 0,0,\cdots )$ est un polynôme symétrique homogène de degré $d$ en $i$ variables pour tout $i.$ Une fonction symétrique est une combinaison linéaire (finie) de fonctions symétriques homogènes.
\end{definition}
Si $\lambda=(\lambda_1,\cdots ,\lambda_l)\in \mathcal{P},$ on définit la \textit{fonction symétrique monomiale}, notée $m_\lambda,$ comme étant la somme sans multiplicité de tous les monômes de degré $|\lambda|$ de la forme $x_{i_1}^{\lambda_1}x_{i_2}^{\lambda_2}\cdots x_{i_l}^{\lambda_l}$ ($i_1,i_2,\cdots,i_l$ distincts). Par exemple, on a :
$$m_{(2,1)}(x_1,x_2,\cdots)=x_1^2x_2+x_2^2x_1+x_1^2x_3+x_3^2x_1+x_2x_3^2+\cdots.$$
Par définition la fonction $m_\lambda$ est homogène de degré $|\lambda|.$ On définit $\Lambda^n$ comme étant l'espace vectoriel engendré par les $m_\lambda$ où $\lambda\in \mathcal{P}_n.$ L'anneau des fonctions symétriques $\Lambda$\label{nomen:alg_fct_sym} \nomenclature[22]{$\Lambda$}{Algèbre des fonctions symétriques \quad \pageref{nomen:alg_fct_sym}} est l'anneau gradué 
$$\Lambda:=\bigoplus_{n\geq 0}\Lambda^n.$$
Les deux définitions d'une fonction symétrique, celle de la définition \ref{def:fct_sym} et celle comme étant un élément de $\Lambda,$ sont équivalentes. L'anneau $\Lambda$ est engendré linéairement par tous les $m_\lambda.$ Il existe plusieurs familles de fonctions symétriques homogènes indexées par les partitions qui forment des bases de $\Lambda.$ Parmi elles, la famille $(p_\lambda)_{\lambda\in \mathcal{P}}$ où $p_\lambda$ est la \textit{fonction puissance} définie par :
$$p_\lambda(x_1,x_2\cdots):=p_{\lambda_1}(x_1,x_2\cdots)p_{\lambda_2}(x_1,x_2\cdots)\cdots p_{\lambda_l}(x_1,x_2\cdots),$$
où $$p_{i}(x_1,\cdots,x_n):=\sum_{j=1}^nx_j^i.$$
Par exemple, on a :
$$p_{(2,1)}(x_1,x_2\cdots)=(x_1^2+x_2^2+\cdots)(x_1+x_2+\cdots)=m_{(2,1)}(x_1,x_2\cdots)+m_{(3)}(x_1,x_2\cdots).$$


Soit $\lambda=(\lambda_1,\cdots \lambda_{l(\lambda)})$ une partition. Pour tout $n\geq \lambda,$ on définit la fonction $s_\lambda (x_1,x_2,\cdots ,x_{n})$ en $n$ variables ainsi : 
$$s_\lambda (x_1,x_2,\cdots ,x_{n})=\frac{\left|
\begin{matrix}
x_1^{\lambda_1+n-1} & x_1^{\lambda_2+n-2} & \cdots & x_1^{\lambda_n} \\
x_2^{\lambda_1+n-1} & x_2^{\lambda_2+n-2} & \cdots & x_2^{\lambda_n} \\
\vdots & \vdots & \ddots & \vdots \\
x_{n}^{\lambda_1+n-1} & x_{n}^{\lambda_2+n-2} & \cdots & x_{n}^{\lambda_n}
\end{matrix}
\right|}{\left|
\begin{matrix}
x_1^{n-1} & x_1^{n-2} & \cdots & 1 \\
x_2^{n-1} & x_2^{n-2} & \cdots & 1 \\
\vdots & \vdots & \ddots & \vdots \\
x_{n}^{n-1} & x_{n}^{n-2} & \cdots & 1
\end{matrix}
\right|}.$$

Le dénominateur est le déterminant de Vandermonde qui est égale à :
$$\prod_{1\leq i,j\leq l(\lambda)}(x_i-x_j).$$
Comme $x_i-x_j$ divise le numérateur pour tout $1\leq i,j\leq l(\lambda),$ alors le déterminant de Vandermonde divise le numérateur, ce qui fait de $s_\lambda$ un polynôme. C'est un polynôme symétrique de degré $\vert \lambda\vert$ parce qu'il est le quotient de deux polynômes anti-symétriques -- un polynôme $f\in \mathbb{Z}[x_1,\cdots,x_n]$ est anti-symétrique si pour tout $\omega\in \mathcal{S}_n,$ $f(x_{\omega(1)},x_{\omega(2)},\cdots , x_{\omega(n)})=\sign(\omega)f (x_1,x_2,\cdots ,x_{n})$ --.

La \textit{fonction de Schur} $s_\lambda$ associée à une partition $\lambda$ est définie comme étant l'unique élément de $\Lambda$ obtenu à partir de la suite  $(s_\lambda (x_1,x_2,\cdots ,x_{n}))_n.$
Les fonctions de Schur peuvent être définies de plusieurs manières différentes. On a choisi de présenter la définition utilisant les déterminants puisqu'elle est la plus connue et la plus facile à définir. On va donner une autre définition utilisant les tableaux dans la section suivante.

Les fonctions $(s_\lambda)_{\lambda\in \mathcal{P}}$ forment aussi une base pour $\Lambda,$ voir \cite{sagan2001symmetric} et \cite{McDo} pour une démonstration complète de ce fait et pour plus de détails sur les fonctions de Schur.

L'algèbre des fonctions symétriques à coefficients dans l'anneau de polynômes en 1 variable $\mathbb{C}[\alpha]$ est notée par $\Lambda(\alpha).$ On rappelle quelques résultats d'ici la fin de ce chapitre concernant cette algèbre mais on ne l'étudie pas en détails. Elle est utilisée à la fin de ce chapitre pour donner un isomorphisme intéressant. La famille des polynômes de Jack $J_{\rho}(\alpha)$, indexée par les partitions, est une base de $\Lambda(\alpha).$ Ces polynômes apparaissaient en 1971 pour la première fois dans les articles \cite{jack1970class} et \cite{jack1972xxv} de Jack. La Section 10 du Chapitre VI du livre de Macdonald \cite{McDo} est consacrée à l'étude de ces polynômes. Pour $\alpha=1$ ces polynômes sont les fonctions de Schur à un multiple près. Pour d'autres paramètres, ces polynômes coïncident à un multiple près avec d'autres fonctions symétriques intéressantes, voir \cite{2014arXiv1402.4615D}. Pour finir, nous introduisons les nombres $\theta^{\rho}_{\lambda}(\alpha),$ qui sont définis comme les coefficients dans le développement de $J_{\rho}(\alpha)$ dans la base des fonctions puissance $p_\lambda$ :
\begin{equation}\label{jackpolynomial1}J_{\rho}(\alpha)=\sum_{|\lambda|=|\rho|}\theta^{\rho}_{\lambda}(\alpha)p_\lambda.
\end{equation}
Les coefficients $\theta^{\rho}_{\lambda}(\alpha)$ nous seront utiles à la sous-section \ref{sec:fct_sym_dec}.

\subsection{Tableaux de Young et $\mathcal{S}_n$-modules irréductibles}\label{sec:tab_young}

Cette section est consacrée à l'étude des modules irréductibles du groupe symétrique $\mathcal{S}_n.$ D'après les résultats du Chapitre \ref{chapitre1}, le nombre de modules irréductibles du groupe $\mathcal{S}_n$ est égal au nombre de partitions de taille $n$ -- car il s'agit du nombre de classes de conjugaison de $\mathcal{S}_n$ --. Pour une partition $\lambda$ de $n,$ on va expliquer comment construire un $\mathcal{S}_n$-module irréductible $\mathcal{S}^\lambda$ associé à $\lambda.$ 

Signalons que la construction des $\mathcal{S}_n$-modules irréductibles ne va pas nous être utile pour la suite et que nous la rappelons seulement pour montrer que dans le cas du groupe symétrique les modules irréductibles sont explicitement connus. 

Soit $\lambda$ une partition de $n.$ Un $\textit{tableau de Young}$ de forme $\lambda=(\lambda_1,\cdots,\lambda_l)$ ou $\textit{$\lambda$-tableau}$ est un diagramme de $l$ lignes avec comme entrées les entiers de l'ensemble $[n]$ tel que la ligne $i$ possède $\lambda_i$ entrées pour tout $1\leq i \leq l.$ Dans un $\lambda$-tableau, il n'y a pas de répétition d'entiers ce qui fait que chaque entier de l'ensemble $[n]$ apparaît une et une seule fois. La Figure \ref{fig:tableau} montre un $(3,2,1)$-tableau.

\begin{figure}[htbp]
\begin{center}
$\begin{matrix}
5&3&6\\
1&2\\
4\\
\end{matrix}$
\caption{Un $(3,2,1)$-tableau.}
\label{fig:tableau}
\end{center}
\end{figure}

Deux $\lambda$-tableaux $t_1$ et $t_2$ sont équivalents si et seulement si $t_2$ peut être obtenu à partir de $t_1$ par une suite des permutations des éléments dans la même ligne. La classe d'équivalence d'un tableau $t$ est notée $\lbrace t\rbrace$ et parfois représentée par des lignes horizontales entre les lignes du tableau. La classe d'équivalence du tableau de la Figure \ref{fig:tableau} contient $12$ tableaux, parmi eux les trois suivants :
$$
\begin{matrix}
3&5&6\\
2&1\\
4\\
\end{matrix}\,\,\,\,\,\,\,\,\,\,\,\,\,\,\,\,\,\,
\begin{matrix}
6&3&5\\
1&2\\
4\\
\end{matrix}\,\,\,\,\,\,\,\,\,\,\,\,\,\,\,\,\,\,
\begin{matrix}
5&3&6\\
2&1\\
4\\
\end{matrix}\,\,
$$

Dans le cas général la classe d'un $\lambda$-tableau contient $\lambda !:=\lambda_1!\cdots \lambda_l!$ tableaux. Pour une colonne $C$ d'un tableau $t,$ on définit $k_C$ de la manière suivante :
$$k_C=\sum_{p}\sign(p)p,$$
où la somme est sur l'ensemble des permutations des éléments de la colonne $C$ et $\sign(p)$ est la signature de la permutation $p.$ On définit $k_t$ comme le produit des $k_C$ :
$$k_t:=\prod _{\text{$C$ colonne de $t$}}k_C.$$ Pour le tableau $t$ de la Figure \ref{fig:tableau} on a :
$$k_t=(\id-(1\,\,5)-(1\,\,4)-(4\,\,5)+(1\,\,4)(1\,\,5)+(1\,\,4)(4\,\,5))(\id-(2\,\,3)).$$

L'action d'une permutation $p$ sur un $\lambda$-tableau $t$ est définie naturellement, $p\cdot t$ est le $\lambda$-tableau dont les entrées sont les images des entrées de $t$ par $p.$ Par exemple, 
$$(1\,\,3\,\,5)(2\,\,6)(4)\,\cdot \,\begin{matrix}
5&3&6\\
1&2\\
4\\
\end{matrix}=\begin{matrix}
1&5&2\\
3&6\\
4\\
\end{matrix}$$

On associe à chaque tableau $t$ l'élément $e_t$ défini par :
$$e_t=k_t\lbrace t\rbrace.$$

On dit qu'un tableau est standard si les entrées de chaque ligne et de chaque colonne sont strictement croissantes. Par exemple, le tableau de la Figure \ref{fig:tableau}, ainsi que tous les tableaux dans sa classe d'équivalence, ne sont pas standard, mais le tableau suivant l'est :
$$\begin{matrix}
1&3&5\\
2&6\\
4\\
\end{matrix}$$

D'après la Proposition \ref{egalite_classe_caract}, le nombre de $\mathcal{S}_n$-modules irréductibles est égal au nombre de classes de conjugaisons de $\mathcal{S}_n.$ Donc, le nombre de $\mathcal{S}_n$-modules irréductibles est égal à $|\mathcal{P}_n|.$ On peut associer à $\lambda$ un $\mathcal{S}_n$-module $S^\lambda,$\label{nomen:Slambda} \nomenclature[23]{$S^\lambda$}{Le $\mathcal{S}_n$-module irréductible associé à $\lambda$ \quad \pageref{nomen:Slambda}} défini comme étant le module engendré par la famille $(e_t)_{\text{$t$ $\lambda$-tableau }}.$

\begin{prop}
Soit $\lambda$ une partition de $n.$ La famille $(e_t)_{\text{$t$ $\lambda$-tableau standard}}$ forme une base pour $S^\lambda.$ De plus, on a :
$$\widehat{\mathcal{S}_n}=\lbrace S^\lambda \,;\, \lambda\in \mathcal{P}_n \rbrace.$$
\end{prop}
\begin{proof}
Voir Théorème 2.4.6 et Théorème 2.5.2 de \cite{sagan2001symmetric}.
\end{proof}
On note par $f^\lambda$ le nombre de $\lambda$-tableaux standards. D'après la proposition précédente on a :
$$\dim S^\lambda=f^\lambda.$$
En appliquant le Corollaire \ref{dec_alg_de_grp} dans le cas du groupe symétrique, on obtient le corollaire suivant.

\begin{cor}
L'algèbre du groupe symétrique se décompose ainsi :
$$\mathbb{C}[\mathcal{S}_n]=\bigoplus_{\lambda\in \mathcal{P}_n}f^\lambda S^{\lambda}$$
ce qui implique l'égalité suivante :
\begin{equation}\label{eq:per_pairedetableauxstand}
n!=\sum_{\lambda\in \mathcal{P}_n}(f^\lambda)^2.
\end{equation}
\end{cor}

L'équation \eqref{eq:per_pairedetableauxstand} possède une preuve bijective à partir de la correspondance de Robinson-Schensted qui associe à chaque permutation de $n$ une paire de tableaux standard de même forme.

\begin{notation}Le caractère d'un $G$-module $X$ est noté par $\mathcal{X}$ dans le premier chapitre. On va utiliser la notation standard pour les caractères des $\mathcal{S}_n$-modules irréductibles. Ainsi, pour une partition $\lambda$ de $n,$ on note $\mathcal{X}^\lambda$ le caractère du module irréductible $S^\lambda.$
\end{notation}
D'après le Théorème \ref{coef_fct_cara}, les coefficients de structure du centre de l'algèbre du groupe symétrique, donnés par l'équation \eqref{Coef_S_n}, s'expriment en fonction des caractères $\mathcal{X}^\lambda$ ainsi :

\begin{theoreme}\label{Th_coef_en_fct_carac}
Soient $\lambda,$ $\delta$ et $\rho$ trois partitions de $n.$ On a :
\begin{equation*}
c_{\lambda\delta}^\rho=\frac{n!}{z_\lambda z_\delta}\sum_{\gamma\in \mathcal{P}_n}\frac{\mathcal{X}^\gamma_\lambda\mathcal{X}^\gamma_\delta\overline{\mathcal{X}^\gamma_\rho}}{\dim S^\gamma}.
\end{equation*}
\end{theoreme} 

Une formule explicite pour les caractères $\mathcal{X}^\lambda$ n'existe pas. Il existe cependant une description combinatoire de ces coefficients en utilisant la règle de Murnaghan-Nakayama, voir \cite[Section 4.10]{sagan2001symmetric}. Le centre de l'algèbre du groupe symétrique est lié à l'algèbre des fonctions symétriques par la correspondance de Frobenius. Les caractères du groupe symétrique apparaissent dans le développement des \textit{fonctions de Schur} dans la base des fonctions puissance. On rappelle que la famille $(s_\lambda)_{\lambda\in \mathcal{P}}$ est une base pour $\Lambda.$ L'écriture de $s_\lambda$ comme combinaison linéaire des $p_\mu$ fait apparaître les caractères $\mathcal{X}^\lambda.$ Explicitement, on a le théorème suivant, dû à Frobenius.
\begin{theoreme}[Frobenius]
Soit $\lambda$ une partition, $s_\lambda$ s'écrit dans la base des fonctions puissance ainsi :  
\begin{equation}\label{rel:sch_pui}
s_\lambda=\sum_{\mu \atop{|\mu|=|\lambda|}}\frac{1}{z_\mu}\mathcal{X}^\lambda_\mu p_\mu.
\end{equation}
\end{theoreme}
\begin{proof}
Voir Théorème 4.6.6 de \cite{sagan2001symmetric}.
\end{proof}

On présente maintenant la définition d'une fonction de Schur en utilisant les tableaux. Un tableau est \textit{semistandard} si les entrées de chaque ligne sont croissantes (au sens large) et de chaque colonne sont strictement croissantes. On autorise la répétition des entiers dans un tableau semistandard. Un tableau semistandard est à valeurs dans $[m]$ si ses entrées sont dans l'ensemble $[m].$ Si $\lambda=(\lambda_1,\cdots,\lambda_{l(\lambda)})$ est une partition et si 
$$T=\big(T_{ij}\big)_{1\leq i\leq l(\lambda)\atop{1\leq j\leq \lambda_i}}$$ 
est un $\lambda$-tableau, on va noter par $(x_1,x_2,\cdots )^T$ le monôme suivant:
$$(x_1,x_2,\cdots )^T:=\prod_{i=1}^{l(\lambda)}\prod_{j=1}^{\lambda_i}x_{T_{ij}}.$$
Par exemple, si $T$ est le $(3,3,2)$-tableau semistandard suivant:
$$
\begin{matrix}
1&1&2\\
2&2&3\\
3&4
\end{matrix}$$
alors $(x_1,x_2,\cdots)^T=x_1^2x_2^3x_3^2x_4.$ Il faut remarquer que si $T$ est à valeurs dans $[m],$ le monôme $(x_1,x_2,\cdots)^T$ ne contient aucun $x_i$ avec $i> m.$

Pour $\lambda\in \mathcal{P},$ une définition équivalente de la fonction de Schur $s_\lambda$ est donnée ainsi :

$$s_\lambda(x_1,x_2,\cdots):=\sum_{T \text{ $\lambda$-tableau semistandard}}(x_1,x_2,\cdots)^T.$$
Par exemple, comme les $(2,1)$-tableau semistandard à valeurs dans $[3]$ sont :
$$
\begin{matrix}
1&1\\
2
\end{matrix}\,\,\,\,\,\,\,\,\,\,\,\,\,\,\,\,\,\,
\begin{matrix}
1&1\\
3
\end{matrix}\,\,\,\,\,\,\,\,\,\,\,\,\,\,\,\,\,\,
\begin{matrix}
1&2\\
3
\end{matrix}\,\,\,\,\,\,\,\,\,\,\,\,\,\,\,\,\,\,
\begin{matrix}
1&3\\
2
\end{matrix}\,\,\,\,\,\,\,\,\,\,\,\,\,\,\,\,\,\,
\begin{matrix}
1&2\\
2
\end{matrix}\,\,\,\,\,\,\,\,\,\,\,\,\,\,\,\,\,\,
\begin{matrix}
1&3\\
3
\end{matrix}\,\,\,\,\,\,\,\,\,\,\,\,\,\,\,\,\,\,
\begin{matrix}
2&2\\
3
\end{matrix}\,\,\,\,\,\,\,\,\,\,\,\,\,\,\,\,\,\,
\begin{matrix}
2&3\\
3
\end{matrix}
$$
On a donc :
$$s_{(2,1)}(x_1,x_2,x_3)=x_1^2x_2+x_1^2x_3+2x_1x_2x_3+\cdots =m_{(2,1)}(x_1,x_2,x_3)+2m_{(1^3)}(x_1,x_2,x_3).$$


\subsection{Fonctions symétriques décalées}\label{sec:fct_sym_dec} Les fonctions symétriques décalées sont des déformations des fonctions symétriques. Par exemple, la définition déterminantale des \textit{fonction de Schur décalées}, voir \cite{okounkov1997shifted}, est similaire à celle des fonctions de Schur donnée dans la Section \ref{sec:alg_fct_sym}. La différence est que la symétrie n'est pas en $x_i$ mais en $x_i-i.$ De même il est possible de définir des fonctions décalées similaires aux fonctions spécifiques utilisées dans la Section \ref{sec:alg_fct_sym}, comme les fonctions puissance décalées. Ces fonctions décalées particulières engendrent une algèbre appelée l'algèbre des fonctions symétriques décalées. 

Dans cette section on donne leur définition ainsi que plusieurs propriétés et résultats concernant cette algèbre, nécessaires pour la suite de cette thèse. En particulier, des algèbres universelles isomorphes à l'algèbre des fonctions décalées vont apparaître dans la preuve de polynomialité des coefficients de structure des deux algèbres étudiées dans ce chapitre et dans le chapitre \ref{chapitre3}. 

\begin{definition} Une $\alpha$-\textit{fonction symétrique décalée} $f$ de degré au plus $d$ avec un nombre infini des  variables $(x_1,x_2,\cdots)$ est une famille $(f_i)_{i\geq 1}$ de polynômes de degré au plus $d$ avec les deux propriétés suivantes :
\begin{enumerate}
\item $f_i$ est un polynôme symétrique en $(x_1-\frac{1}{\alpha},x_2-\frac{2}{\alpha},\cdots,x_i-\frac{i}{\alpha}).$
\item $f_{i+1}(x_1,x_2,\cdots,x_i,0)=f_i(x_1,x_2,\cdots,x_i).$
\end{enumerate}
\end{definition}

L'\textit{algèbre des fonctions symétriques décalées} à coefficients dans $\mathbb{C}[\alpha]$ est notée par $\Lambda^{*}(\alpha).$\label{nomen:alg_fct_sym_dec} \nomenclature[23]{$\Lambda^{*}$}{Algèbre des fonctions symétriques décalées \quad \pageref{nomen:alg_fct_sym_dec}} Il existe un isomorphisme d'algèbres entre $\Lambda^{*}(\alpha)$ et l'algèbre des fonctions symétriques à coefficients dans $\mathbb{C}[\alpha]$, $\Lambda(\alpha).$ Cet isomorphisme a été donné par Lassalle (voir \cite{Lassalle}) qui le note $(\#).$ On le note ici $sh_{\alpha}$\label{nomen:isom_LambdaetoileetLambda} \nomenclature[23]{$sh_{\alpha}$}{Isomorphisme entre $\Lambda^{*}(\alpha)$ et $\Lambda(\alpha)$ \quad \pageref{nomen:isom_LambdaetoileetLambda}} pour que la dépendance en $\alpha$ soit explicite.

Soit $f$ un élément de $\Lambda^{*}(\alpha).$ Pour toute partition $\lambda=(\lambda_1,\lambda_2,\cdots,\lambda_l)$, on note par $f(\lambda)$ la valeur $f_{l}(\lambda_1,\lambda_2,\cdots,\lambda_{l}).$ La fonction symétrique décalée $f$ est déterminée par ses valeurs sur les partitions, voir \cite[Section 2]{1996q.alg.....8020O}.

Soient $\lambda$ et $\rho$ deux partitions telles que $|\rho|\geq |\lambda|.$ La fonction symétrique décalée $sh_\alpha(p_\lambda)$ est reliée à $\theta^{\rho}_{\lambda\cup (1^{n-|\lambda|})}(\alpha)$ par l'équation suivante donnée par Lassalle, voir \cite[Proposition 2]{Lassalle} :
\begin{equation}\label{lassalleequation}\alpha^{|\lambda|-l(\lambda)}sh_\alpha(p_\lambda)(\rho)=\begin{pmatrix}
|\rho|-|\lambda|+m_1(\lambda)\\
m_1(\lambda)
\end{pmatrix}z_\lambda \theta^{\rho}_{\lambda\cup (1^{n-|\lambda|})}(\alpha),
\end{equation}
où on rappelle que $\begin{displaystyle}z_\lambda=\prod_{i\geq 1}i^{m_i(\lambda)}m_i(\lambda)!\end{displaystyle}.$ D'après la définition de $sh_\alpha(f)$ pour toute fonction symétrique $f$, donnée dans \cite[Eq. (3.1)]{Lassalle}, on a $sh_\alpha(p_\lambda)(\rho)=0$ si $|\rho|<|\lambda|.$

Pour $\alpha=1,$ si $\rho$ et $\lambda$ sont deux partitions de la même taille, on a d'après \cite[page 662]{Lassalle} :
$$\theta^{\rho}_{\lambda}(1)=\frac{|\rho|!}{z_\lambda}\frac{\mathcal{X}^\rho_\lambda}{\dim S^\rho}.$$

La définition de $z_\lambda$ implique que si $\rho$ et $\lambda$ sont deux partitions tel que $|\rho|\geq |\lambda|,$ on a :
$$z_{\lambda\cup (1^{|\rho|-|\lambda|})}=z_\lambda \frac{(|\rho|-|\lambda|+m_1(\lambda))!}{m_1(\lambda)!}.$$
Par conséquent, l'équation \eqref{lassalleequation} nous donne l'équation suivante pour $\alpha=1$ :
\begin{equation}\label{sh1}
sh_1(p_\lambda)(\rho)=|\rho| \cdots (|\rho|-|\lambda|+1)\frac{\mathcal{X}^\rho_{\lambda\cup (1^{|\rho|-|\lambda|})}}{\dim S^\rho}.
\end{equation}
Cette équation nous sera utile dans la Section \ref{sec:iso}.

Le fait qu'on présente la définition des fonctions symétriques décalées pour un $\alpha$ arbitraire tandis qu'on ne s'intéresse dans ce chapitre qu'au cas $\alpha=1$ peut paraître surprenant. Cependant, le cas $\alpha=2$ est particulièrement intéressant puisqu'il fait apparaître les fonctions sphériques zonales d'une paire de Gelfand particulière, la paire $(\mathcal{S}_{2n},\mathcal{B}_n).$ Le chapitre \ref{chapitre3} est consacré à l'étude de cette paire et le cas $\alpha=2$ sera traité plus en détail dans ce chapitre. 

\section{Polynomialité des coefficients de structures}\label{sec:poly}

La polynomialité des coefficients de structure du centre de l'algèbre du groupe symétrique donnés par l'équation (\ref{Coef_S_n}) a été établie par Farahat et Higman en $1959$ dans l'article \cite{FaharatHigman1959}. En $1999,$ Ivanov et Kerov ont donné une preuve combinatoire de ce fait en introduisant les permutations partielles, voir \cite{Ivanov1999}. On va présenter ici ces deux approches.

\subsection{Théorème de Farahat et Higman}\label{sec:th_Far_Hig}
On rappelle dans cette partie le résultat principal de \cite{FaharatHigman1959} avec sa démonstration. Si $d$ est un ensemble, on désigne par $\mathcal{S}_d$ l'ensemble des permutations de $d$ --ensemble des bijections de $d$ dans $d$--. Avec cette notation $\mathcal{S}_{[n]}$ n'est autre que le groupe symétrique $\mathcal{S}_n.$ On utilise les notations standards des ensembles d'entiers, $\mathbb{N}$ est l'ensemble des entiers naturels et $\mathbb{N}^*$ est l'ensemble des entiers strictement positifs. On note par $\mathcal{S},$ l'ensemble des permutations de $\mathbb{N}^*$ à support fini.

Pour un ensemble $P\subseteq \mathcal{S}$, on définit $\mathbb{N}^*(P)$ ainsi : $$\mathbb{N}^*(P)=\lbrace j\in \mathbb{N}^*,~~\exists \sigma \in P~~/~~\sigma (j)\neq j\rbrace$$ et on note par $n(P)$ le cardinal de $\mathbb{N}^*(P)$. Le groupe symétrique $\mathcal{S}_n$ s'identifie avec l'ensemble suivant : $$\mathcal{S}_n=\lbrace \sigma\in \mathcal{S}~~/\mathbb{N}^*(\sigma)\subseteq [n]\rbrace.$$

Soient $(\alpha_1,\ldots,\alpha_r)$ et $(\beta_1,\ldots,\beta_r)$ deux éléments de $\mathcal{S}^r.$ On dit qu'ils sont conjugués dans $\mathcal{S}$ s'il existe un $\sigma \in \mathcal{S}$ tel que $\alpha_i=\sigma\beta_i\sigma^{-1} ,~~i=1,2,\ldots,r$. Cette relation de conjugaison définit bien une relation d'équivalence.

Si  $(\alpha_1,\ldots,\alpha_r)$ et $(\beta_1,\ldots,\beta_r)$ sont deux éléments de $\mathcal{S}^r$ conjugués dans $\mathcal{S},$ alors l'application $\sigma_{\vert \mathbb{N}^*(\alpha_1,\ldots,\alpha_r)}$ est une bijection entre $\mathbb{N}^*(\alpha_1,\ldots,\alpha_r)$ et $\mathbb{N}^*(\beta_1,\ldots,\beta_r),$ ce qui signifie que $n(\alpha_1,\ldots,\alpha_r)=n(\beta_1,\ldots,\beta_r).$ Donc si $L$ est la classe de conjugaison de $(\alpha_1,\ldots,\alpha_r) \in \mathcal{S}^r$, on peut noter par $n(L)$ l'entier $n(\alpha_1,\ldots,\alpha_r)$.

\begin{lem}\label{lem:Farahat_Higman}
Si $L$ est la classe de conjugaison de $(\alpha_1,\ldots,\alpha_r) \in \mathcal{S}^r$, alors :
\begin{enumerate}
\item[1-] si $n< n(L)$, $L\cap \mathcal{S}_n^r=\emptyset$.
\item[2-] si $n\geq n(L)$, $L\cap \mathcal{S}_n^r$ est la classe de conjugaison d'un élément de $\mathcal{S}_n^r$. De plus, $$\vert L\cap \mathcal{S}_n^r\vert=\frac{n(n-1)\ldots(n-n(L)+1)}{k(L)},$$ où $k(L)$ ne dépend pas de $n.$
\end{enumerate}
\end{lem} 
\begin{proof}
\
\begin{enumerate}
\item[1-] Supposons  $n< n(L).$ 
 S'il existe $(\beta_1,\ldots,\beta_r) \in L\cap \mathcal{S}_n^r$, on a d'une part $$\vert  \mathbb{N}^*(\beta_1,\ldots,\beta_r)\vert\leq n$$ puisque les $\beta_i\in \mathcal{S}_n$ et d'autre part  $\vert  \mathbb{N}^*(\beta_1,\ldots,\beta_r)\vert=n(L)>n$, ce qui est impossible.
\item[2-] Supposons maintenant $n\geq n(L).$ L'ensemble $\mathbb{N}^*(\alpha_1,\ldots,\alpha_r)$ possède au plus $n$ éléments, ce qui veut dire qu'il existe $\tau \in \mathcal{S}$ tel que $\tau ( \mathbb{N}^*(\alpha_1,\ldots,\alpha_r))\subseteq [n].$ On a $\tau\circ\alpha_i\circ\tau^{-1} \in  \mathcal{S}_n,$ pour tout $i=1,\ldots,r$. En effet, soit $j \in [n]$:\\
- si $j \notin \tau ( \mathbb{N}^*(\alpha_1,\ldots,\alpha_r))$ alors $(\tau\circ\alpha_i\circ\tau^{-1})(j)=j \in [n],$\\
- si $j \in \tau ( \mathbb{N}^*(\alpha_1,\ldots,\alpha_r))$ alors $j=\tau(h), ~ h\in \mathbb{N}^*(\alpha_1,\ldots,\alpha_r)$ et donc on a  $$(\tau\circ\alpha_i\circ\tau^{-1})(j)=\tau(\alpha_i(h))=
\left\{
\begin{array}{ll}
  \tau(h)=j \in [n] & \qquad \mathrm{si}\quad \alpha_i(h)=h \\
  \tau(\alpha_i(h))\in \tau( \mathbb{N}^*(\alpha_1,\ldots,\alpha_r))\subseteq [n]& \qquad \mathrm{si}\quad \alpha_i(h)\neq h \\
 \end{array}
 \right.$$
car si $\alpha_i(h)\neq h$ alors $h$ et $\alpha_i(h)$ appartiennent à $\mathbb{N}^*(\alpha_1,\ldots,\alpha_r).$ De plus, si $j >n$, alors $(\tau \circ \alpha_i \circ \tau^{-1} (j))=j.$
Donc, $$\tau\circ(\alpha_1,\ldots,\alpha_r)\circ\tau^{-1} \in L\cap \mathcal{S}_n^r\Rightarrow L\cap \mathcal{S}_n^r\neq\emptyset.$$

Il nous reste à montrer que si $\alpha_1,\ldots,\alpha_r,\beta_1,\ldots,\beta_r \in \mathcal{S}_n,$ avec $\beta_i=\sigma\circ\alpha_i\circ\sigma^{-1}$ pour tout $1\leq i\leq r$ et $\sigma \in \mathcal{S}$ alors il existe $\tau \in \mathcal{S}_n$ tel que $\beta_i=\tau\circ\alpha_i\circ\tau^{-1}$ pour tout $1\leq i\leq r.$ En effet il suffit de prendre $$\tau_{\vert \mathbb{N}^*(\alpha_1,\ldots,\alpha_r)}=\sigma_{\vert \mathbb{N}^*(\alpha_1,\ldots,\alpha_r)}:\mathbb{N}^*(\alpha_1,\ldots,\alpha_r)\longrightarrow \mathbb{N}^*(\beta_1,\ldots,\beta_r),$$ et de compléter $\tau$ par n'importe quelle bijection entre $[n]\setminus \mathbb{N}^*(\alpha_1,\ldots,\alpha_r)$ et \linebreak $[n]\setminus \mathbb{N}^*(\beta_1,\ldots,\beta_r).$ L'application $\tau$ est bien une permutation de  $\mathcal{S}_n$ car \linebreak $\mathbb{N}^*(\alpha_1,\ldots,\alpha_r),\mathbb{N}^*(\beta_1,\ldots,\beta_r)\subseteq [n]$ puisque $\alpha_1,\ldots,\alpha_r,\beta_1,\ldots,\beta_r \in \mathcal{S}_n.$
\end{enumerate}
Supposons maintenant que $L\cap \mathcal{S}_n^r$ est la classe de conjugaison de $(\gamma_1,\ldots,\gamma_r) \in \mathcal{S}_n^r$, alors  $\vert L\cap \mathcal{S}_n^r\vert=\frac{n!}{\vert S_{n_{(\gamma_1,\ldots,\gamma_r)}}\vert}$, où
\begin{eqnarray*}
S_{n_{(\gamma_1,\ldots,\gamma_r)}}&=&\lbrace \sigma \in \mathcal{S}_n \text{ tel que } \sigma\circ(\gamma_1,\ldots,\gamma_r)=(\gamma_1,\ldots,\gamma_r)\circ\sigma\rbrace \\
&\equiv&\lbrace\sigma\in \mathcal{S}_{\mathbb{N}^*(\gamma_1,\ldots,\gamma_r)} \text{ tel que } \sigma\circ(\gamma_1,\ldots,\gamma_r)=(\gamma_1,\ldots,\gamma_r)\circ\sigma\rbrace\times  \mathcal{S}_{E_{n}\setminus \mathbb{N}^*(\gamma_1,\ldots,\gamma_r)}
\end{eqnarray*}
Donc on a :
$$\vert S_{n_{(\gamma_1,\ldots,\gamma_r)}}\vert=k(L) \times  \vert \mathcal{S}_{E_{n}\setminus \mathbb{N}^*(\gamma_1,\ldots,\gamma_r)}\vert=k(L)(n-n(L))!,$$ 
où $$k(L)=\vert\lbrace\sigma\in \mathcal{S}_{\mathbb{N}^*(\gamma_1,\ldots,\gamma_r)}/\sigma\circ(\gamma_1,\ldots,\gamma_r)=(\gamma_1,\ldots,\gamma_r)\circ\sigma\rbrace\vert$$ est une constante indépendante de $n$. D'où le résultat.
\end{proof}

Les classes de conjugaisons des éléments de $\mathcal{S}$ correspondent aux partitions propres. En effet, le type-cyclique d'une permutation $p$ de $\mathcal{S}$ est de la forme $\lambda \cup (1^\infty)$ où $\lambda$ est la partition propre associée à $p_{|_{\mathbb{N}^*(p)}}.$ Deux permutations $p$ et $p'$ de $\mathcal{S}$ de type cyclique $\lambda \cup (1^\infty)$ et $\lambda' \cup (1^\infty)$ respectivement sont conjuguées si et seulement si $\lambda =\lambda'.$ On note par $C_\lambda$ la classe de conjugaison correspondant à la partition propre $\lambda$. On a:
\begin{enumerate}
\item[1-]  si $n<|\lambda|$, $C_\lambda\cap \mathcal{S}_n=\emptyset,$
\item[2-]  si $n\geq |\lambda|$, $C_\lambda\cap \mathcal{S}_n=C_{\underline{\lambda}_n}$ qui est la classe de conjugaison dans $\mathcal{S}_n$ qui correspond à la partition $\underline{\lambda}_n=\lambda\cup (1^{n-|\lambda|})$.
\end{enumerate} 
La somme des éléments de $C_{\underline{\lambda}_n}$ n'est autre que l'élément ${\bf C}_{\underline{\lambda}_n}\in Z(\mathbb{C}[\mathcal{S}_{n}])$. Si $n< |\lambda|$,  on a ${\bf C}_{\underline{\lambda}_n}=0$. 

\begin{theoreme}\label{pol_coef_cen_grp_sym}
Soient $\lambda,\delta$ et $\rho$ trois partitions propres. Il existe un unique polynôme $f_{\lambda\delta}^{\rho}(x)$ tel que $c_{\lambda\delta}^{\rho}(n)=f_{\lambda\delta}^{\rho}(n)$ pour tout entier $n\geq |\lambda|,|\delta|$ et $|\rho|.$ Le degré de $f_{\lambda\delta}^{\rho}(x)$ est la plus grande valeur de $n(L)-|\rho|$ pour toutes les classes $L$ de paires $(\alpha,\beta)$ tel que $\alpha\in C_\lambda, \beta\in C_\delta$ et $\alpha \beta \in C_\rho.$
\end{theoreme}
\begin{proof}
Si $(\alpha,\beta)\in \mathcal{S}^2$ tel que 
\begin{equation}\label{fareq}
\alpha\in C_\lambda,\beta\in C_\delta \text{ et }\alpha\beta\in C_\rho,
\end{equation}
alors pour tout $(\alpha_1,\beta_1)$ conjugué à $(\alpha,\beta),$ le couple $(\alpha_1,\beta_1)$ satisfait (\ref{fareq}). Donc l'ensemble des solutions de (\ref{fareq}) est une union disjointe de classes des conjugaisons. Mais si $(\alpha,\beta)$ satisfait (\ref{fareq}), on a $n(\alpha,\beta)\leq |\lambda|+|\delta|$, donc d'après le Lemme \ref{lem:Farahat_Higman}, la classe de conjugaison de $(\alpha,\beta)$ est représentée par une paire dans $\mathcal{S}_ {|\lambda|+|\delta|}^2.$ En particulier, le nombre de classes de conjugaison parmi les solutions de \eqref{fareq} est fini. Notons $L_1,L_2,\ldots,L_r$ ces classes, alors à nouveau d'après le Lemme \ref{lem:Farahat_Higman}, le nombre des solutions de (\ref{fareq}) avec $\alpha,\beta \in \mathcal{S}_n$ est: $$\sum_{i=1}^{r}\frac{n(n-1)\ldots(n-n(L_i)+1)}{k(L_i)}. $$
Mais ce nombre est le cardinal de: $$\lbrace (\alpha,\beta)\in \mathcal{S}^2~~/~~\alpha\in C_{\underline{\lambda}_n}, \beta\in C_{\underline{\delta}_n}, \alpha\beta \in C_{\underline{\rho}_n}\rbrace.$$ Or $c_{\lambda\delta}^{\rho}(n)$ est le cardinal de l'ensemble suivant:$$c_{\lambda\delta}^{\rho}(n)=\lbrace (\alpha,\beta)\in \mathcal{S}^2~~/~~\alpha\in C_{\underline{\lambda}_n}, \beta\in C_{\underline{\delta}_n}, \alpha\beta =\sigma \rbrace,$$ où $\sigma$ est une permutation fixée de type cyclique $\underline{\rho}_n$. 
D'où $$c_{\lambda\delta}^{\rho}(n)=\frac{\sum_{i=1}^{u}\frac{n(n-1)\ldots(n-n(L_i)+1)}{k(L_i)}}{\vert C_{\underline{\rho}_n}\vert}.$$
Or $\vert C_{\underline{\rho}_n}\vert=\frac{n(n-1)\ldots(n-|\rho|+1)}{k(\rho)}$ d'après la Lemme \ref{lem:Farahat_Higman}. D'autre part,
$$|\rho|=n(\alpha\beta)\leq n(\alpha,\beta)=n(L_i)$$ pour $(\alpha,\beta)\in L_i$, ce qui nous donne:
$$c_{\lambda\delta}^{\rho}(n)=k(\rho)\sum_{i=1}^{r}\frac{(n-|\rho|)(n-|\rho|-1)\ldots(n-n(L_i)+1)}{k(L_i)}$$qui est un polynôme de degré $\max(n(L_i)-|\rho|).$ D'où le résultat.
\end{proof}

\begin{cor}
Soient $\lambda$ et $\delta$ deux partitions propres. Alors, pour tout entier $n\geq |\lambda|,|\delta|$ on a:
\begin{equation*}
{\bf C}_{\underline{\lambda}_n}{\bf C}_{\underline{\delta}_n}=\sum_{\rho \in \mathcal{PP}_{\leq n}\atop{|\rho|\leq |\lambda|+|\delta|}} c_{\lambda\delta}^{\rho}(n){\bf C}_{\underline{\rho}_n},
\end{equation*}
où les coefficients $c_{\lambda\delta}^{\rho}(n)$\label{nomen:coeff_str_centre_S_n} \nomenclature[24]{$c_{\lambda\delta}^{\rho}(n)$}{Coefficient de structure dans le centre de l'algèbre du groupe symétrique quand $\lambda,\delta,\rho$ sont propres et $n$ très grand\quad \pageref{nomen:coeff_str_centre_S_n}} sont des polynômes en $n.$ 
\end{cor}

Vu que les coefficients de structure s'expriment à l'aide des caractères, une question naturelle se pose.
\begin{que}
Peut-on retrouver la propriété de polynomialité des coefficients de structure du centre de l'algèbre du groupe symétrique d'une façon directe en utilisant la formule du Théorème \ref{Th_coef_en_fct_carac} ?
\end{que}

\subsection{Approche d'Ivanov et Kerov}\label{sec:approche_Ivanov_Kerov}
En $1999,$ Ivanov et Kerov ont proposé une nouvelle preuve pour le Théorème \ref{pol_coef_cen_grp_sym} de Farahat et Higman dans leur article \cite{Ivanov1999}. Ils ont introduit les permutations partielles pour arriver à ce résultat. Leur approche a plusieurs avantages. Elle leur a permis d'obtenir une borne explicite sur le degré des polynômes. De plus, leur preuve fait apparaître une algèbre universelle qui se projette sur $Z(\mathbb{C}[\mathcal{S}_n])$ pour tout $n.$ Ils montrent aussi que cette algèbre est isomorphe à l'algèbre des fonctions symétriques décalées d'ordre $1.$ Cet isomorphisme va être détaillé en Section \ref{sec:iso}. Ici, on va rappeler les idées importantes de leur démonstration de la propriété de polynomialité.

\subsubsection{Le semi-groupe des permutations partielles}

\begin{definition}
Une permutation partielle de $[n]$ est un couple $(d,\omega)$\label{nomen:perm_part} \nomenclature[25]{$(d,\omega)$}{Notation d'une permutation partielle \quad \pageref{nomen:perm_part}} où $d\subseteq [n]$ et $\omega \in \mathcal{S}_d.$
\end{definition}
On note par $\mathfrak{P}_n$\label{nomen:ens_perm_part} \nomenclature[26]{$\mathfrak{P}_n$}{Ensemble des permutations partielles de $[n]$ \quad \pageref{nomen:ens_perm_part}} l'ensemble des permutations partielles de $[n]$. Comme pour tout $1\leq k\leq n$ on a $\vert \mathcal{S}_k\vert=k!$ et $\vert\lbrace d\subseteq [n]~~/~~\vert d\vert=k\rbrace\vert={n \choose k}$, on obtient: 
\begin{equation*}
\vert \mathfrak{P}_n\vert=\sum_{k=0}^n{n \choose k}k!=\sum_{k=0}^n(n-k+1)(n-k+2)\cdots(n-1)n.
\end{equation*}

Si $(d,\omega)\in \mathfrak{P}_n$, on note par $\widetilde{\omega}$ la permutation de $[n]$ définie par :
$$\widetilde{\omega}(a)=
\left\{
\begin{array}{ll}
  \omega(a) & \qquad \mathrm{si}\quad a\in d, \\
  a & \qquad \mathrm{si}\quad a\in [n]\setminus d. \\
 \end{array}
 \right.$$
Étant données deux permutations partielles $(d_1,\omega_1)$ et $(d_2,\omega_2)$ de $\mathfrak{P}_n$, on définit leur produit ainsi :
\begin{equation*}(d_1,\omega_1)\cdot (d_2,\omega_2)=(d_1\cup d_2,\widetilde{\omega}_{1_{|d_1\cup d_2}}\circ \widetilde{\omega}_{2_{|d_1\cup d_2}}).\end{equation*}
Ce produit donne à $\mathfrak{P}_n$ une structure de semi-groupe d'élément neutre $(\emptyset,e_0)$ où $e_0$ est la permutation (triviale) de l'ensemble vide.

On note par $\mathcal{B}_n=\mathbb{C}[\mathfrak{P}_n]$ l'algèbre du semi-groupe $\mathfrak{P}_n$. Soit $x\subset [n]$ tel que $\vert x\vert=k$, on définit l'application $\varphi_x:\mathcal{B}_n\longrightarrow \mathbb{C}[\mathcal{S}_x]$ par :
$$\varphi_x(d,\omega)=
\left\{
\begin{array}{ll}
  \widetilde{\omega}_{|x} & \qquad \mathrm{si}\quad d\subset x, \\
  0 & \qquad \mathrm{sinon.}\quad \\
 \end{array}
 \right.$$
On obtient aisément la proposition suivante.
\begin{prop}
$\varphi_x$ est un homomorphisme d'algèbre.
\end{prop}
\subsubsection{Classes de conjugaison dans $\mathfrak{P}_n$}
\begin{prop}
Le groupe $\mathcal{S}_n$ agit sur le semi-groupe $\mathfrak{P}_n$ par l'action définie ci-dessous:
$$\sigma\cdot (d,\omega)=(\sigma(d),\sigma\circ \omega\circ \sigma^{-1})$$ 
pour tout $\sigma\in \mathcal{S}_n$ et tout $(d,\omega)\in \mathfrak{P}_n.$
\end{prop}
Les orbites de cette action sont appelées classes de conjugaisons de $\mathfrak{P}_n$. Deux permutations partielles $(d_1,\omega_1)$ et $(d_2,\omega_2)$ sont dans la même orbite si et seulement s'il existe $\sigma\in \mathcal{S}_n$ tel que $(d_2,\omega_2)=(\sigma(d_1),\sigma\circ \omega_1\circ \sigma^{-1}),$ c'est-à-dire que $d_2=\sigma(d_1)$ et $\omega_2=\sigma\circ \omega_1\circ \sigma^{-1}.$ Donc, $(d_1,\omega_1)$ et $(d_2,\omega_2)$ sont dans la même orbite si et seulement si $\vert d_1\vert=\vert d_2\vert$ et $\type-cyclique(\omega_1)=\type-cyclique(\omega_2).$

Par conséquent, les classes de conjugaisons de $\mathfrak{P}_n$ sont indexées par les éléments de l'ensemble $\mathcal{P}_{\leq n},$ les partitions de taille plus petite ou égale à $n.$ Pour une partition $\rho\in \mathcal{P}_{\leq n},$ la classe de conjugaison de $\mathfrak{P}_n$ associée à $\rho$ est notée par $A_\rho(n)$ et est définie par:
$$A_\rho(n):=\lbrace (d,\omega)\in \mathfrak{P}_n\text{ tel que } |d|=|\rho| \text{ et }\type-cyclique(\omega)=\rho\rbrace.$$
En particulier $A_{\emptyset}(n)=\lbrace(\emptyset, e_0)\rbrace$.
\begin{prop}\label{prop:taille_A_rho_n}
Soit $\rho\in \mathcal{P}_{\leq n},$ alors on a :
$$\vert A_{\rho}(n)\vert=\begin{pmatrix}
n-|\rho|+m_1(\rho) \\
m_1(\rho)
\end{pmatrix}\cdot \vert C_{\underline{\rho}_n}\vert.$$
\end{prop}
\begin{proof} 
Considérons l'application $$\begin{array}{ccccc}
\Theta & : & A_{\rho}(n) & \to & C_{\underline{\rho}_n} \\
& & (d,\omega) & \mapsto & \widetilde{\omega}. \\
\end{array}$$
Pour tout $v,v^{'}\in C_{\underline{\rho}_n}$ avec $v\neq v^{'}$, on a $\Theta^{-1}(v)\cap \Theta^{-1}(v^{'})=\emptyset$, donc $$\vert A_{\rho}(n)\vert=\sum_{v\in C_{\underline{\rho}_n}}\vert \Theta^{-1}(v)\vert.$$
Soit $v\in C_{\underline{\rho}_n}$, on note par $\supp(v)$ le support de $v$ :
$$\supp(v)=\lbrace x\in [n]~\vert~ v(x)\neq x\rbrace.$$
Comme $v$ est de type cyclique $\rho \cup (1^{n-\vert\rho \vert}),$ on a $\vert \supp(v)\vert=\vert\rho\vert-m_1(\rho)$. Rappelons que 
$$\Theta^{-1}(v)=\lbrace (d,\omega)\in A_{\rho}(n) \text{ tel que } \widetilde{\omega}=v\rbrace.$$ 
Pour construire un $(d,\omega)\in \Theta^{-1}(v)$, il faut que $\omega$ coïncide avec $v$ sur $\supp(v)$ (nécessairement $\supp(v)\subset d$) et on choisit $m_1(\rho)=\vert d\vert-\vert \supp(v)\vert$ points fixes pour $\omega$ parmi les $n-|\rho|+m_1(\rho)$ points fixes de $v$, donc $\vert \Theta^{-1}(v)\vert=\begin{pmatrix}
n-|\rho|+m_1(\rho) \\
m_1(\rho)
\end{pmatrix}$.
\end{proof}
Prolongeons l'action de $\mathcal{S}_n$ sur $\mathfrak{P}_n$ par linéarité en une action sur $\mathcal{B}_n$ et notons 
$$\mathcal{A}_n=\mathcal{B}_n^{\mathcal{S}_n}:=\lbrace b\in \mathcal{B}_n \text{ tel que pour tout $\sigma\in \mathcal{S}_n,$ } \sigma\cdot b=b \rbrace$$ la sous-algèbre des éléments invariants de $\mathcal{B}_n$.
L'homomorphisme surjectif $$\begin{array}{ccccc}
\psi & : & \mathfrak{P}_n & \to & \mathcal{S}_n \\
& & (d,\omega) & \mapsto & \widetilde{\omega} \\
\end{array}$$
peut être prolongé linéairement en un homomorphisme surjectif d'algèbres $\psi:\mathcal{B}_n\rightarrow \mathbb{C}[\mathcal{S}_n]$.\\
Pour tout $\sigma\in \mathcal{S}_n$ et pour tout $b\in \mathcal{B}_n,$ on a :
$$\psi (\sigma\cdot b)=\sigma\cdot \psi(b)$$
(ceci découle du fait que $\widetilde{\sigma^{-1} \omega\sigma}=\sigma^{-1} \widetilde{\omega}\sigma$). Donc
$$\psi(\mathcal{A}_n)=(\psi(\mathcal{B}_n))^{\mathcal{S}_n}=Z(\mathbb{C}[\mathcal{S}_n]).$$
Soit $\rho\in \mathcal{P}_{\leq n},$ on note
$${\bf A}_{\rho}(n)=\sum_{(d,\omega)\in A_{\rho}(n)}(d,\omega).$$
Les éléments de la famille $({\bf A}_\rho(n))_{\rho \in \mathcal{P}_{\leq n}}$ forment une base de l'algèbre $\mathcal{A}_n.$ Une des propriétés importantes du morphisme $\psi$ est que pour toute partition $\rho \in \mathcal{P}_{\leq n},$ on a:
\begin{equation}\label{imagepsi}
\psi({\bf A}_\rho(n))=\begin{pmatrix}
n-|\rho|+m_1(\rho) \\
m_1(\rho)
\end{pmatrix} {\bf C}_{\underline{\rho}_n}.
\end{equation}
Cette propriété s'obtient directement de la Proposition \ref{prop:taille_A_rho_n} quand on développe $\psi({\bf A}_\rho(n)).$
\subsubsection{Les algèbres $\mathcal{A}_\infty$ et $\mathcal{B}_{\infty}$}
Considérons la famille $(\mathcal{B}_n, \varphi_n)$ où $\varphi_n:\mathcal{B}_{n+1}\rightarrow \mathcal{B}_n$ est le morphisme défini par:
$$\varphi_n(d,\omega)=
\left\{
\begin{array}{ll}
  (d,\omega) & \qquad \mathrm{si}\quad d\subset [n]; \\
  0 & \qquad \mathrm{sinon.}\quad \\
 \end{array}
 \right.$$
On note par $\mathcal{B}_\infty$ la limite projective des algèbres $\mathcal{B}_n$ munie des morphismes $\varphi_n$.
\begin{definition}
Une permutation partielle de $\mathbb{N}$ est un couple $(d,\omega)$ où $d\subsetneq \mathbb{N}$ est un sous ensemble fini de $\mathbb{N}$ et $\omega\in \mathcal{S}_d.$
\end{definition} 
On note par $\mathfrak{P}$ l'ensemble de toute les permutations partielles de $\mathbb{N}.$ L'écriture générique d'un élément $b\in \mathcal{B}_\infty$ est :
$$b=\sum_{k=0}^{\infty}\sum_{\vert d\vert=k}\sum_{\omega\in \mathcal{S}_d}b_{d,\omega}(d,\omega),$$
où les $b_{d,\omega}$ sont des nombres complexes.
On note par $\mathcal{S}_\infty$ l'ensemble des permutations de $\mathbb{N}.$ Comme précédemment $\mathcal{S}_\infty$ agit sur $\mathcal{B}_\infty$ et on note $\mathcal{A}_\infty=(\mathcal{B}_\infty)^{\mathcal{S}_\infty}$\label{nomen:alg_A_infini} \nomenclature[27]{$\mathcal{A}_\infty$}{Algèbre qui se projette sur $Z(\mathbb{C}[\mathcal{S}_n])$ pour tout $n$ \quad \pageref{nomen:alg_A_infini}} la sous-algèbre des éléments invariants de $\mathcal{B}_\infty$ par l'action de $\mathcal{S}_\infty.$ Un élément $b\in \mathcal{B}_\infty$ est dans $\mathcal{A}_\infty$ si et seulement si :
$$b_{d,\omega}=b_{\sigma (d),\sigma\circ\omega\circ \sigma^{-1}},~\forall \sigma\in \mathcal{S}_\infty.$$
Pour une partition $\rho\in \mathcal{P},$ on définit $A_\rho$ ainsi :
$$A_\rho:=\lbrace (d,\omega)\in \mathfrak{P}\text{ tel que }|d|=|\rho|\text{ et }\type-cyclique(\omega)=\rho\rbrace.$$
Il est facile de voir que tout élément de $\mathcal{A}_\infty$ s'écrit de manière unique comme combinaison linéaire infinie des éléments  $({\bf A}_\rho)_{\rho \in \mathcal{P}}$, où
$${\bf A}_\rho=\sum_{(d,\omega)\in A_\rho}(d,\omega). $$
Soient $\lambda$ et $\delta$ deux partitions de $\mathcal{P},$ les coefficients de structure $k_{\lambda\delta}^\rho$ de l'algèbre $\mathcal{A}_\infty$ sont définis par l'équation suivante:
\begin{equation}\label{eq:str_coef_A_infini}
{\bf A}_\lambda{\bf A}_\delta=\sum_{\rho \in \mathcal{P}}k_{\lambda\delta}^\rho{\bf A}_\rho.
\end{equation}
On note par $\theta_n$ l'homomorphisme naturel entre $\mathcal{B}_\infty$ et $\mathcal{B}_n$,
$$\begin{array}{ccccc}
\theta_n & : & \mathcal{B}_\infty & \to & \mathcal{B}_n \\
& & \sum_{{\overset{d\subseteq \mathbb{N}}{\omega\in \mathcal{S}_{d}}}}b_{d,\omega}(d,\omega) & \mapsto & \sum_{{\overset{d\subseteq [n]}{\omega\in \mathcal{S}_{d}}}}b_{d,\omega}(d,\omega) \\
\end{array}.$$
Remarquons que: $$\theta_n({\bf A}_\rho)={\bf A}_{\rho}(n).$$

Comme conséquence de l'équation \eqref{eq:str_coef_A_infini}, les $k_{\lambda\delta}^{\rho}$ sont aussi les coefficients de structure de l'algèbre $\mathcal{A}_n$ :
\begin{equation*}
{\bf A}_\lambda(n){\bf A}_\delta(n)=\sum_{\rho\in \mathcal{P}_{\leq n}}k_{\lambda\delta}^{\rho}{\bf A}_\rho(n),
\end{equation*}
où $\lambda$ et $\delta$ sont deux partitions de taille inférieure ou égale à $n.$ Autrement dit, bien que les coefficients de structure $k_{\lambda\delta}^{\rho}$ de l'algèbre $\mathcal{A}_\infty$ ne dépendent pas de $n,$ se sont les coefficients de structure de $\mathcal{A}_n$ pour tout $n.$

Soit $\rho$ une partition de taille inférieure ou égale à $n.$ Considérons la permutation partielle $(d,\omega)$ de $n$ où $d=[|\rho|]$ et 
$$\omega=(1~2\ldots\rho_1)(\rho_1+1\ldots\rho_1+\rho_2)(\rho_1+\rho_2+1\ldots\rho_1+\rho_2+\rho_3)\ldots(r-\rho_l+1\ldots r).$$
D'après la Proposition \ref{coef_de_str}, le coefficient $k_{\sigma,\tau}^\rho$ n'est autre que le cardinal de l'ensemble $E_{\lambda\delta}^\rho(n)$ défini ci-dessous:
\begin{eqnarray*}
 E_{\lambda\delta}^\rho(n)&=&\lbrace((d_1,\omega_1),(d_2,\omega_2))\in A_{\lambda}(n)\times A_{\delta}(n)~\mid~(d_1,\omega_1)\cdot (d_2,\omega_2)=(d,\omega)\rbrace \\
 &=&\lbrace((d_1,\omega_1),(d_2,\omega_2))\in A_{\lambda}(n)\times A_{\delta}(n)~\mid~d_1\cup d_2=[|\rho|],~\omega_1\circ \omega_2=\omega\rbrace.
\end{eqnarray*}
\begin{cor}\label{deg1}
Si $k_{\lambda\delta}^\rho\neq 0$, alors $\vert \rho\vert\leq\vert \lambda\vert+\vert\delta \vert.$
\end{cor}
\begin{proof}
Si $k_{\lambda\delta}^\rho\neq 0$ on a $E_{\lambda\delta}^\rho(n)\neq \emptyset\Rightarrow \exists d_1,d_2\subset [n]$ tel que $\vert d_1\vert=\vert\lambda\vert,\vert d_2\vert=\vert\delta\vert$ et $\vert\rho\vert=\vert d_1\cup d_2\vert\leq \vert d_1\vert+\vert d_2\vert=\vert\lambda\vert+\vert\delta\vert.$
\end{proof}

Dans le corollaire suivant, on présente une sous-algèbre particulière de $\mathcal{A}_\infty.$ Comme on va le montrer dans la Section \ref{sec:iso}, cette sous-algèbre est isomorphe à l'algèbre $\Lambda^*(1)$ des fonctions symétriques décalées d'ordre $1.$

\begin{cor}\label{cor:sous_alg_de_A_infty}
L'ensemble des combinaisons linéaire finies des (${\bf{A}}_\lambda$), noté par $\widetilde{\mathcal{A}_\infty}$, est une sous-algèbre de $\mathcal{A}_\infty$. La famille $({\bf{A}}_\lambda)_{\lambda\in \mathcal{P}}$ forme une base pour $\widetilde{\mathcal{A}_\infty}$.
\end{cor}

\subsubsection{Autre preuve du théorème de Farahat et Higman}
Ici on présente la preuve d'Ivanov et Kerov du Théorème \ref{pol_coef_cen_grp_sym} de Farahat et Higman.
\begin{theoreme}
Soient $\lambda, \delta$ et $\rho$ trois partitions propres et soit $n$ un entier naturel tel que $n\geq |\lambda|, |\delta|,|\rho|,$ les coefficients de structure $c_{\lambda\delta}^{\rho}(n)$ donnée par l'équation \eqref{eq:def_str_cof_S_n} peuvent s'écrire ainsi :
\begin{equation*}
c_{\lambda\delta}^{\rho}(n)=\sum_{k=0}^{n-\vert\rho\vert}k_{\lambda\delta}^{\rho\cup (1^k)}\begin{pmatrix}
n-\vert\rho\vert \\
k
\end{pmatrix}, 
\end{equation*}
où les coefficients $k_{\lambda\delta}^{\rho\cup (1^k)}$ sont des nombres indépendants de $n.$
\end{theoreme}
\begin{proof}
Soient $\lambda$ et $\delta$ deux partitions propres, d'après l'équation \eqref{imagepsi}, on a :\linebreak $\psi({\bf A}_\lambda(n))={\bf C}_{\underline{\lambda}_n}$ et $\psi({\bf A}_\delta(n))={\bf C}_{\underline{\delta}_n}.$
On reprend l'équation suivante dans $\mathcal{A}_n:$
\begin{equation*}
{\bf A}_\lambda(n){\bf A}_\delta(n)=\sum_{\tau\in \mathcal{P}_{\leq n}\atop{ |\tau|\leq |\lambda|+|\delta|}}k_{\lambda\delta}^{\tau}{\bf A}_\tau(n).
\end{equation*}
Si on applique $\psi,$ on obtient :
\begin{equation*}
{\bf C}_{\underline{\lambda}_n}{\bf C}_{\underline{\delta}_n}=\sum_{\tau\in \mathcal{P}_{\leq n}\atop{ |\tau|\leq |\lambda|+|\delta|}}k_{\lambda\delta}^{\tau}\begin{pmatrix}
n-|\tau|+m_1(\tau) \\
m_1(\tau)
\end{pmatrix}{\bf C}_{\underline{\tau}_n}.
\end{equation*}
Il faut remarquer que si $\rho$ est une partition propre de taille inférieure ou égale à $n,$ alors on a:
\begin{equation*}
{\bf C}_{\underline{\rho}_n}={\bf C}_{\underline{\rho\cup (1^k)}_n} \text{ pour tous les $0\leq k\leq n-|\rho|.$}
\end{equation*}
Alors la somme à droite s'écrit:
$$\sum_{\rho \in \mathcal{PP}_{\leq n}\atop{|\rho|\leq |\lambda|+|\delta|}} \left[\sum_{k=0}^{n-\vert\rho\vert}k_{\lambda\delta}^{\rho\cup (1^k)}\begin{pmatrix}
n-\vert\rho\cup (1^k)\vert+m_1(\rho\cup (1^k)) \\
m_1(\rho\cup (1^k))
\end{pmatrix}\right] {\bf C}_{\underline{\rho}_n}.$$
Après simplification, on obtient :
$$\sum_{\rho \in \mathcal{PP}_{\leq n}\atop{|\rho|\leq |\lambda|+|\delta|}} \left[\sum_{k=0}^{n-\vert\rho\vert}k_{\lambda\delta}^{\rho\cup (1^k)}\begin{pmatrix}
n-\vert\rho\vert \\
k
\end{pmatrix} \right] {\bf C}_{\underline{\rho}_n},$$
d'où le résultat.
\end{proof}
\begin{cor}\label{cor:deg_cof_str_S_n}
Soient $\lambda$ et $\delta$ deux partitions propres, pour toute partition propre $\rho$ de taille inférieure ou égale à $|\lambda|+|\delta|$ et pour tout $n\geq |\lambda|,|\delta|,|\rho|$ on a:
\begin{equation*}
\deg(c_{\lambda\delta}^{\rho}(n))=\max_{\overset{0\leq k\leq n-\vert\rho\vert}{k_{\lambda\delta}^{\rho\cup (1^k)}\neq 0}}\underbrace{\deg\left( \begin{pmatrix}
n-\vert\rho\vert \\
k
\end{pmatrix} \right) }_{k}=\displaystyle \max_{\overset{0\leq k\leq n-\vert\rho\vert}{k_{\lambda\delta}^{\rho\cup (1^k)}\neq 0}}k.
\end{equation*}
\end{cor}
\subsection{Filtrations sur $\mathcal{A}_\infty$ et le degré des coefficients de structure}\label{sec:filt}

Cette section a pour but de donner plusieurs filtrations sur $\mathcal{A}_\infty$ et de montrer comment les utiliser pour obtenir des majorations pour le degré des polynômes $c_{\lambda\delta}^{\rho}(n).$ On commence par les filtrations données par Ivanov et Kerov dans \cite{Ivanov1999}.

D'après le Corollaire \ref{deg1}, il est clair que la fonction
$$\deg_1 ({\bf A}_\rho) = |\rho|,$$
définit une filtration sur $\mathcal{A}_\infty.$

\begin{prop}\label{filt_deg_2}
La fonction $\deg_{2}$ définie sur $\mathcal{B}_\infty$ par:
$$\deg_2(d,\omega)=\vert d \vert+\vert m_1(\omega)\vert,$$
où $m_1(\omega)$ est l'ensemble des points fixes de $\omega$ est une filtration.
\end{prop}
\begin{proof} Montrons que pour tout $(d_1,\omega_1)$, $(d_2,\omega_2)$
$$\deg_2((d_1,\omega_1)\cdot (d_2,\omega_2))\leq \deg_2(d_1,\omega_1)+\deg_2(d_2,\omega_2).$$
Décomposons d'abord l'ensemble $d_1\cup d_2$ en union d'ensembles disjoints de la façon suivante :
\begin{eqnarray*}
d_1\setminus d_2 &=& d_{mf}^{12}\sqcup d_{ff}^{12} \\
d_2\setminus d_1 &=& d_{fm}^{21}\sqcup d_{ff}^{21} \\
d_1\cap d_2 &=& d_{mm}\sqcup d_{fm}\sqcup d_{mf}\sqcup d_{ff}
\end{eqnarray*}
où le premier indice est $f$ si les points de l'ensemble correspondant sont fixés par $\omega_1$ et $m$ s'ils ne sont pas fixés par $\omega_1$. Le deuxième indice est $f$ si les points de l'ensemble correspondant sont fixés par $\omega_2$ et $m$ s'ils ne sont pas fixés par $\omega_2$. Or par définition, on a :
$$\deg_2(d_1,\omega_1)=\vert d_1 \vert+\vert m_1(\omega_1)\vert,$$
où $m_1(\omega_1)$ est l'ensemble des points fixes de $\omega_1.$ Donc, on obtient :
\begin{eqnarray*}
\deg_2(d_1,\omega_1) &=& \vert d_1\setminus d_2\vert +\vert d_1\cap d_2\vert+\vert d_{ff}^{12}\cup d_{fm} \cup d_{ff}\vert \\
&=& \vert d_{mf}^{12}\vert +\vert d_{mm}\vert +\vert d_{mf}\vert +2\vert d_{ff}\vert +2\vert d_{ff}^{12}\vert +2\vert d_{fm}\vert 
\end{eqnarray*}
De même on a :
\begin{eqnarray*}
\deg_2(d_2,\omega_2) &=& \vert d_{fm}^{21}\vert +\vert d_{mm}\vert +\vert d_{fm}\vert +2\vert d_{ff}\vert +2\vert d_{ff}^{21}\vert +2\vert d_{mf}\vert 
\end{eqnarray*}
et 
\begin{eqnarray*}
\deg_2((d_1,\omega_1).(d_2,\omega_2))&=&\deg_2(d_1\cup d_2,\omega_1\circ \omega_2) \\
&=& \vert d_1\cup d_2 \vert+\vert m_1(\omega_1\circ \omega_2)\vert.
\end{eqnarray*}
Or les points fixes par $\omega_1\circ \omega_2$ sont les points fixés par $\omega_1$ et par $\omega_2$ (c'est-à-dire les éléments des ensembles $d_{ff}^{12}$, $d_{ff}^{21}$, $d_{ff}$) et certains points déplacés par $\omega_2$ puis par $\omega_1$ (c'est-à-dire un sous-ensemble de $d_{mm}$). Par suite on a :
$$\vert m_1(\omega_1\circ \omega_2)\vert\leq \vert d_{ff}^{12}\vert +\vert d_{ff}^{21}\vert+ \vert d_{ff}\vert +\vert d_{mm}\vert,$$
d'où \begin{eqnarray*}
\deg_2(d_1\cup d_2,\omega_1\circ \omega_2) &\leq & \vert d_{mf}^{12}\vert +\vert d_{mf}\vert +\vert d_{fm}\vert +\vert d_{fm}^{21}\vert+ \\
&& 2\vert d_{ff}\vert +2\vert d_{ff}^{12}\vert +2\vert d_{mm}\vert+ 2\vert d_{ff}^{21}\vert \\
&\leq & \vert d_{mf}^{12}\vert +3\vert d_{mf}\vert +3\vert d_{fm}\vert +\vert d_{fm}^{21}\vert+ \\
&& 4\vert d_{ff}\vert +2\vert d_{ff}^{12}\vert +2\vert d_{mm}\vert+ 2\vert d_{ff}^{21}\vert \\
&=& \deg_2(d_1,\omega_1)+ \deg_2(d_2,\omega_2).
\end{eqnarray*}
Cela termine la preuve.
\end{proof}

On peut restreindre la filtration de la Proposition \ref{filt_deg_2} à $\mathcal{A}_\infty.$ Cette filtration a été introduite par Kerov dans \cite{kerov1993gaussian}. Dans \cite{Ivanov1999}, Ivanov et Kerov montre qu'en fait, la fonction 
$$\deg_J({\bf A}_{\rho}):=|\rho|+\sum_{k\in J}m_k(\rho)$$
définit une filtration sur $\mathcal{A}_\infty$ pour tout $J\subset \mathbb{N}.$ On a décidé de présenter le cas particulier $\deg_2$ seulement car cela nous suffit pour présenter une bonne majoration pour le degré des coefficients de structure.

On présente maintenant des résultats sur la décomposition des permutations en produit de transpositions afin de donner une autre filtration sur $\mathcal{A}_\infty.$
\begin{lem}
Le nombre minimal de transpositions qui peut apparaître dans la décomposition d'une permutation $\sigma$ de $n$ en produit de transpositions est :
$$m(\sigma):=n-l(\sigma),$$
où $l(\sigma)$ est le nombre de cycles disjoints de $\sigma.$
\end{lem}
\begin{proof}
Remarquons d'abord que tout cycle $(i_1~i_2~\ldots~i_{k-1}~i_k)$ peut s'écrire comme $(i_1~i_2)\circ (i_2~i_3)\ldots(i_{k-2}~i_{k-1})\circ (i_{k-1}~i_k)$, autrement dit tout cycle de longueur $k$ peut s'écrire comme produit de $k-1$ transpositions. On peut donc écrire toute permutation $\sigma$ de $\mathcal{S}_n$ comme produit des $m(\sigma)=n-l(\sigma)$ transpositions. Maintenant démontrons par récurrence que si $\sigma=\tau_1\ldots\tau_k$ où les $\tau_r$ sont des transpositions, alors on a forcément $k\geq m(\sigma)$. Considérons $\sigma^{'}=\tau_1\ldots\tau_{k-1}$, par hypothèse de récurrence on a $k-1\geq m(\sigma^{'})=n-l(\sigma^{'})$, or $\sigma=\sigma^{'}\circ \tau_k$ ce qui implique, d'après le Lemme \ref{lem:act_tr_per}, que $l(\sigma)=l(\sigma^{'})\pm 1.$ Donc, $$k-1\geq n-(l(\sigma)\pm 1)\geq n-l(\sigma)-1=m(\sigma)-1,$$ d'où $k\geq m(\sigma).$
\end{proof}
\begin{prop}Soient $\alpha$ et $\beta$ deux éléments de $\mathcal{S}_n$, alors on a :
$$m(\alpha\circ \beta)\leq m(\alpha)+m(\beta).$$
\end{prop}
\begin{proof}
Écrivons $\alpha$ et $\beta$ comme des produits de transpositions de longueur minimale :
\begin{eqnarray*}
\alpha &=& \tau_1\tau_2\ldots \tau_{m(\alpha)} \\
\beta &=& \tau_{m(\sigma)+1}\tau_{m(\sigma)+2}\ldots\tau_{m(\alpha)+m(\beta)}.
\end{eqnarray*}
Alors $\sigma\circ \tau=\tau_1\tau_2\ldots \tau_{m(\alpha)}\tau_{m(\sigma)+1}\tau_{m(\sigma)+2}\ldots\tau_{m(\alpha)+m(\beta)},$ et on déduit le résultat.
\end{proof}
\begin{cor}
La fonction $\deg_{3}$ définie sur $\mathcal{A}_\infty$ par:
$$\deg_{3}({\bf A}_{\rho})=\vert\rho\vert-l(\rho)$$ 
est une filtration.
\end{cor}

On montre maintenant comment les filtrations sur $\mathcal{A}_\infty$ nous aident à obtenir des bornes supérieures pour les coefficients de structure du centre de l'algèbre du groupe symétrique. Si $\deg$ est une filtration sur $\mathcal{A}_\infty$, alors l'égalité
$${\bf A}_\lambda {\bf A}_\delta=\sum_\rho k_{\lambda\delta}^\rho {\bf A}_\rho,$$
implique (en appliquant $\deg$ sur cette équation) : 
\begin{equation}\label{eq:filtr}
\deg({\bf A}_\lambda {\bf A}_\delta):=\max_{\rho, k_{\lambda\delta}^\rho\neq 0} \deg({\bf A}_\rho)\leq \deg({\bf A}_\lambda)+ \deg ({\bf A}_\delta).
\end{equation}
\begin{prop}\label{prop:borne}
Soient $\lambda,$ $\delta$ et $\rho$ trois partitions propres, alors on a :
$$\deg(c_{\lambda\delta}^\rho(n))\leq \frac{|\lambda|+|\delta|-|\rho|}{2}.$$
\end{prop}
\begin{proof}
D'après le Corollaire \ref{cor:deg_cof_str_S_n}, on a :
$$\deg(c_{\lambda\delta}^\rho(n))=\max_{\overset{0\leq k\leq n-\vert\rho\vert}{k_{\lambda\delta}^{\rho\cup (1^k)}\neq 0}}k.$$
Dans le cas où la filtration est $\deg_2,$ l'équation \eqref{eq:filtr} ci-dessus nous donne:
$$\deg_2({\bf A}_{\rho\cup (1^k)})\leq \max_{\rho, k_{\lambda\delta}^\rho\neq 0} \deg_2({\bf A}_\rho)\leq \deg_2({\bf A}_\lambda)+ \deg_2 ({\bf A}_\delta),$$
pour tout $k\geq 0$ tel que $k_{\lambda\delta}^{\rho\cup (1^k)}\neq 0.$ Ceci veut dire que pour tout $k\geq 0$ tel que $k_{\lambda\delta}^{\rho\cup (1^k)}\neq 0,$ on a:
$$|\rho|+2k\leq |\lambda|+|\delta|,$$
d'où le résultat.
\end{proof}
Il est à remarquer que le même raisonnement avec la filtration $\deg_1$ implique que le degré des $c_{\lambda\delta}^\rho(n)$ est majoré par $|\lambda|+|\delta|-|\rho|$ tandis que le même raisonnement avec la filtration $\deg_3$ n'implique aucune majoration pour le degré des coefficients de structure du centre de l'algèbre du groupe symétrique.

Dans \cite{feray2012complete}, Féray donne une nouvelle filtration sur $\mathcal{A}_\infty.$ Cette filtration n'apparait pas dans l'article \cite{Ivanov1999} d'Ivanov et Kerov donc on a décidé de la présenter à part. On la présente dans la proposition suivante et on donne une nouvelle démonstration différente de celle donnée par Féray, voir \cite[Lemme 2.10]{feray2012complete}. L'objectif de cette nouvelle filtration est de modifier $\deg_3$ pour en déduire une borne sur le degré des coefficients de structure. 

\begin{prop}La fonction $\deg_4$ définie sur $\mathcal{B}_\infty$ par
$$\deg_4(d,\omega)=|d|-l(\omega)+m_1(\omega),$$
est une filtration.
\end{prop}
\begin{proof}
Dans un premier temps, nous allons montrer que pour tout \linebreak $(d_1,\omega_1),(\lbrace a,b\rbrace,\tau_{a,b})\in \mathfrak{P}$, on a :
$$\deg_4(d_1\cup \lbrace a,b\rbrace,\omega_1\circ \tau_{a,b})\leq \deg_4(d_1,\omega_1)+\underbrace{\deg_4(\lbrace a,b\rbrace,\tau_{a,b})}_1.$$
En effet on peut distinguer les 3 cas suivants :
\begin{enumerate}
\item[1-] $\underline{a,b\in d_1:}$ dans ce cas on distingue 2 cas:
\begin{enumerate}
\item[a-] $\underline{l(\omega_1\circ \tau_{a,b})=l(\omega_1)+1}:$ dans ce cas, $a$ et $b$ sont dans le même cycle de la permutation $\omega_1$ d'après le Lemme \ref{lem:act_tr_per}. Si ce cycle est $(a\,\,b),$ alors $m_1(\omega_1\circ \tau_{a,b})= m_1(\omega_1)+2,$ sinon, $m_1(\omega_1\circ \tau_{a,b})< m_1(\omega_1)+2.$ Dans tous les cas, on a :
$$m_1(\omega_1\circ \tau_{a,b})\leq m_1(\omega_1)+2.$$
D'où
\begin{eqnarray*}
\deg_4(d_1\cup \lbrace a,b\rbrace,\omega_1\circ \tau_{a,b})&\leq & |d_1|-l(\omega_1)-1+m_1(\omega_1)+2 \\
&\leq & \deg_4(d_1,\omega_1)+\deg_4(\lbrace a,b\rbrace,\tau_{a,b}).
\end{eqnarray*}
\item[b-] $\underline{l(\omega_1\circ \tau_{a,b})=l(\omega_1)-1}:$ dans ce cas, $a$ et $b$ ne sont pas dans le même cycle de la permutation $\omega_1$ d'après le Lemme \ref{lem:act_tr_per}. Si $a$ et $b$ sont des points fixes, alors on a $m_1(\omega_1\circ \tau_{a,b})= m_1(\omega_1)-2.$ Si $a$ est un point fixe, mais pas $b$, on a $m_1(\omega_1 \circ \tau_{a,b})=m_1(\omega_1)-1.$ Si $a$ et $b$ ne sont pas fixes, on a $m_1(\omega_1\circ \tau_{a,b}) =m_1(\omega_1).$ Donc, dans tous les cas, on a :
$$m_1(\omega_1\circ \tau_{a,b})\leq m_1(\omega_1).$$
D'où
\begin{eqnarray*}
\deg_4(d_1\cup \lbrace a,b\rbrace,\omega_1\circ \tau_{a,b})&\leq & |d_1|-l(\omega_1)+1+m_1(\omega_1) \\
&\leq & \deg_4(d_1,\omega_1)+\deg_4(\lbrace a,b\rbrace,\tau_{a,b}).
\end{eqnarray*}
\end{enumerate}
\item[2-]$\underline{a\in d_1,~ b\notin d_1 (\text{ou, symétriquement, } a\notin d_1, b\in d_1):}$ dans ce cas, d'après le Lemme \ref{lem:act_tr_per}, on a $$l(\omega_1\circ \tau_{a,b})=l(\omega_1)$$ et selon le cas si $a$ est un point fixe de $\omega_1$ ou non, on a : $$m_1(\omega_1\circ \tau_{a,b})=m_1(\omega_1)-1 \text{ ou bien }m_1(\omega_1\circ \tau_{a,b})=m_1(\omega_1).$$
Dans tous les cas, on a :
$$m_1(\omega_1\circ \tau_{a,b})\leq m_1(\omega_1).$$
D'où
\begin{eqnarray*}
\deg_4(d_1\cup \lbrace a,b\rbrace,\omega_1\circ \tau_{a,b})&\leq & |d_1|+1-l(\omega_1)+m_1(\omega_1) \\
&\leq & \deg_4(d_1,\omega_1)+\deg_4(\lbrace a,b\rbrace,\tau_{a,b}).
\end{eqnarray*}
\item[3-]$\underline{a,b\notin d_1:}$ dans ce cas et toujours d'après le Lemme \ref{lem:act_tr_per}, on a $$l(\omega_1\circ \tau_{a,b})=l(\omega_1)+1~et~ m_1(\omega_1\circ \tau_{a,b})=m_1(\omega_1).$$
\begin{eqnarray*}
\deg_4(d_1\cup \lbrace a,b\rbrace,\omega_1\circ \tau_{a,b})&=& |d_1|+2-l(\omega_1)-1+m_1(\omega_1) \\
&\leq & \deg_4(d_1,\omega_1)+\deg_4(\lbrace a,b\rbrace,\tau_{a,b}).
\end{eqnarray*}
\end{enumerate}

Remarquons maintenant que pour toute permutation partielle $(d_1,\omega_1)$ et pour tout $n\in \mathbb{N}^*$, on a $\deg_4(\lbrace n\rbrace,id_{\lbrace n\rbrace})=1$ et :
\begin{eqnarray*}
\deg_4(d_1\cup \lbrace n\rbrace,\omega_1\circ id_{\lbrace n\rbrace})&=&\left\{
\begin{array}{ll}
  |d_1|-l(\omega_1)+m_1(\omega_1) & \qquad \mathrm{si}\quad n\in d_1 \\
  |d_1|+1-l(\omega_1)-1+m_1(\omega_1)+1 & \qquad \mathrm{si}\quad n\notin d_1 \\
 \end{array}
 \right.\\
 &\leq & \deg_4(d_1,\omega_1)+ \deg_4(\lbrace n\rbrace,id_{\lbrace n\rbrace})
 \end{eqnarray*}
 
Pour $k\geq 2,$ une permutation partielle en forme de cycle de longueur $k$ peut être décomposée ainsi :
$$
(\lbrace a_1,\ldots a_k\rbrace, (a_1~\ldots~a_k))=\underbrace{(\lbrace a_1, a_2\rbrace, \tau_{a_1,a_2})}_{\in S}\circ\ldots\circ \underbrace{(\lbrace a_{k-1}, a_k\rbrace, \tau_{a_{k-1},a_k})}_{\in S}$$
et avec cette décomposition on a :
$$\underbrace{\deg_4(\lbrace a_1,\ldots a_k\rbrace, (a_1~\ldots~a_k))}_{k-1}=\underbrace{\deg_4(\lbrace a_1, a_2\rbrace, \tau_{a_1,a_2})}_{1}+\ldots+ \underbrace{\deg_4(\lbrace a_{k-1}, a_k\rbrace, \tau_{a_{k-1},a_k})}_{1}.
$$

Donc, toute permutation partielle $(d_2,\omega_2)$ peut se décomposer en produit des "transpositions partielles" :
$$(d_2,\omega_2)=(d_2^1,\tau_{d_2^1})\cdot (d_2^2,\tau_{d_2^2})\cdot \cdots\cdot (d_2^t,\tau_{d_2^t}),$$
de sorte que
$$\deg_4(d_2,\omega_2)=\deg_4(d_2^1,\tau_{d_2^1})+\deg_4(d_2^2,\tau_{d_2^2})+\ldots+\deg_4(d_2^t,\tau_{d_2^t}).$$
Avec cette décomposition, on obtient :
\begin{eqnarray*}
\deg_4((d_1,\omega_1)\cdot (d_2,\omega_2)) &=& \deg_4((d_1,\omega_1)\cdot (d_2^1,\tau_{d_2^1})\cdot (d_2^2,\tau_{d_2^2})\cdot \cdots\cdot (d_2^t,\tau_{d_2^t})) \\
&\leq &\deg_4((d_1,\omega_1)\cdot (d_2^1,\tau_{d_2^1})\cdot \cdots\cdot (d_2^{t-1},\tau_{d_2^{t-1}}))+\deg_4(d_2^t,\tau_{d_2^t}) \\
&\vdots & \\
&\leq & \deg_4(d_1,\omega_1)+\underbrace{\deg_4(d_2^1,\tau_{d_2^1})+\cdots+\deg_4(d_2^t,\tau_{d_2^t})}_{\deg_4(d_2,\omega_2)}.
\end{eqnarray*}

D'où le résultat.
\end{proof}
Comme dans le cas de la filtration $\deg_2,$ on peut restreindre la filtration $\deg_4$ à $\mathcal{A}_\infty$ et obtenir le corollaire suivant.
\begin{cor}\label{cor:borne_2}
Soient $\lambda,$ $\delta$ et $\rho$ trois partitions propres alors on a:
$$\deg(c_{\lambda\delta}^\rho(n))\leq |\lambda|+|\delta|-|\rho|+l(\rho)-l(\lambda)-l(\delta).$$
\end{cor}

D'après ce corollaire, le coefficient $c_{(2^2)(2^2)}^{(3^2)}(n)$ est un polynôme constant car son degré est nul ce qui est plus précis que ce qu'on peut obtenir en appliquant la Proposition \ref{prop:borne}. Cette dernière implique que $\deg(c_{(2^2)(2^2)}^{(3^2)}(n))\leq 1.$ Il faut noter que la borne sur le degré des polynômes obtenue dans le Corollaire \ref{cor:borne_2} n'est pas toujours meilleure que celle donnée par la Proposition \ref{prop:borne}. Par exemple, on sait, d'après l'Exemple \ref{ex_pr_cl_de_con_2}, que le degré de $c_{(2)(3)}^{(2)}(n)$ est exactement $1.$ La Proposition \ref{prop:borne} montre qu'il est inférieur ou égale à $1$ ce qui est plus précis que la majoration $2$ donnée par le Corollaire \ref{cor:borne_2}. Donc, les résultats du Corollaire \ref{cor:borne_2} et de la Proposition \ref{prop:borne} sont incomparables.

\subsection{Isomorphisme entre $\widetilde{\mathcal{A}_\infty}$ et l'algèbre des fonctions symétriques décalées d'ordre $1$}\label{sec:iso}

On va donner dans cette section un isomorphisme entre l'algèbre $\Lambda^*(1)$ des $1$-fonctions symétriques décalées et l'algèbre $\widetilde{\mathcal{A}_\infty}$, la sous-algèbre de l'algèbre universelle $\mathcal{A}_\infty$ définie dans le Corollaire \ref{cor:sous_alg_de_A_infty}. Cet isomorphisme a été établi par Ivanov et Kerov, voir \cite[Section 9]{Ivanov1999}.

On considère le diagramme suivant :
$$\xymatrix{
    \widetilde{\mathcal{A}_\infty} \ar[dd] \ar[rd]^{\theta_n}& \\
     & \mathcal{A}_n \ar[dl]^\psi \\
     Z(\mathbb{C}[\mathcal{S}_n])}$$
Pour tout entier $n,$ la fonction $\psi\circ \theta_n:\widetilde{\mathcal{A}_\infty} \rightarrow Z(\mathbb{C}[\mathcal{S}_n])$ est un morphisme d'algèbre car c'est une composition de morphismes d'algèbre. En appliquant $\psi\circ \theta_n$ sur les éléments de base de $\widetilde{\mathcal{A}_\infty},$ on a d'après l'équation \eqref{imagepsi} :
$$\psi\circ \theta_n(\bf{A}_\lambda)=\psi({\bf{A}}_\lambda(n))=\begin{pmatrix}
n-|\lambda|+m_1(\lambda) \\
m_1(\lambda)
\end{pmatrix} {\bf C}_{\underline{\lambda}_n},$$
pour toute partition $\lambda$ de taille inférieure ou égale à $n.$ Si $\rho\in \mathcal{P}_n,$ comme\linebreak $\psi\circ \theta_n(\bf{A}_\lambda)\in Z(\mathbb{C}[\mathcal{S}_n]),$ on peut appliquer le caractère $\mathcal{X}^\rho$ et on obtient :
\begin{equation}
\mathcal{X}^\rho(\psi\circ \theta_n({\bf{A}}_\lambda))=\begin{pmatrix}
n-|\lambda|+m_1(\lambda) \\
m_1(\lambda)
\end{pmatrix} |C_{\underline{\lambda}_n}|\mathcal{X}^\rho_{\underline{\lambda}_n}=\frac{n\cdots (n-|\lambda|+1)}{z_\lambda}\mathcal{X}^\rho_{\underline{\lambda}_n}.
\end{equation}  

D'après la Proposition \ref{caract_norma}, la fonction $\frac{\mathcal{X}^\rho}{\dim S^\rho}\circ \psi\circ \theta_n$ est un morphisme. De plus, tenant compte de la relation \eqref{sh1}, on obtient :
\begin{equation}\label{iso_A_fctdecale}
\frac{\mathcal{X}^\rho}{\dim S^\rho}\circ \psi\circ \theta_n({\bf{A}}_\lambda)=\frac{1}{z_\lambda}sh_1(p_\lambda)(\rho).
\end{equation}
Il est à remarquer que si $|\lambda|>|\rho|$ dans l'équation \eqref{iso_A_fctdecale}, le membre à gauche de l'égalité est nul car $\theta_{|\rho|}({\bf{A}_\lambda})={\bf{A}_\lambda}(|\rho|)=0$ et le membre à droite est également nul par définition de $sh_\alpha,$ voir \eqref{lassalleequation}. Donc l'égalité \eqref{iso_A_fctdecale} est toujours vraie.\\

\begin{prop}
L'application linéaire 
$$\begin{array}{ccccc}
F & : & \widetilde{\mathcal{A}_\infty} & \to & \Lambda^*(1) \\
& & {\bf{A}_\lambda} & \mapsto & \frac{sh_1(p_\lambda)}{z_\lambda} \\
\end{array}$$
est un isomorphisme d'algèbres.
\end{prop}
\begin{proof}
D'après \eqref{iso_A_fctdecale}, $F$ coïncide avec $\frac{\mathcal{X}^\rho}{\dim S^\rho}\circ \psi\circ \theta_{|\rho|}$ et est donc un morphisme car $\frac{\mathcal{X}^\rho}{\dim S^\rho},$  $\psi$ et $\theta_{|\rho|}$ sont des morphismes.

Le fait que $F$ est un isomorphisme vient du fait que les ${\bf{A}_\lambda}$ sont une base de $\widetilde{\mathcal{A}_\infty}$ et les $sh_1(p_\lambda)$ une base pour $\Lambda^*(1).$
\end{proof}

\section{Applications pour les diagrammes de Young aléatoires}\label{sec:appl_diag_de_young}

Le but de cette section est de montrer l'importance et l'utilité de la propriété de polynomialité des coefficients de structure du centre de l'algèbre du groupe symétrique. Je me suis servi pour présenter cette section des notes d'un cours de Féray donné à Edinburgh, voir \cite{feray2014}.

D'après le Théorème \ref{th:dec_alg_de_grp} du premier chapitre, on peut définir une mesure de probabilité sur l'ensemble $\hat{G}$ des $G$-modules irréductibles d'un groupe fini $G$ de la manière suivante. 

\begin{prop}
Soit $G$ un groupe fini. La fonction $\mathrm{P}_G:\hat{G}\rightarrow \mathbb{R}$ définie par :
$$\mathrm{P}_G(X)=\frac{(\dim X)^2}{|G|}$$
est une mesure de probabilité sur $\hat{G}.$
\end{prop}
\begin{proof}
Le fait que 
$$\sum_{X\in \hat{G}}\mathrm{P}_G(X)=\sum_{X\in \hat{G}}\frac{(\dim X)^2}{|G|}=1$$
est une conséquence du Théorème \ref{th:dec_alg_de_grp} du chapitre \ref{chapitre1}.
\end{proof}

La mesure de probabilité $\mathrm{P}_{G}$ est appelée \textit{mesure de Plancherel}. On fixe un élément $g$ de $G,$ et on considère la variable aléatoire $\mathfrak{F}_g$ définie ainsi :
 
$$\begin{array}{ccccc}
\mathfrak{F}_g & : & \hat{G} & \to & \mathbb{R} \\
& & X & \mapsto & \frac{\mathcal{X}(g)}{\dim X} \\
\end{array}$$

Notre objectif dans ce qui suit est d'étudier la variable aléatoire $\mathfrak{F}_g.$

\subsection{Espérance et moment d'ordre $2$}

L'espérance $\mathbb{E}_{\mathrm{P}_{G}}(\mathfrak{F}_g)$ de la variable $\mathfrak{F}_g$ est définie ainsi :
$$\mathbb{E}_{\mathrm{P}_{G}}(\mathfrak{F}_g):=\sum_{X\in \hat{G}}\mathrm{P}_G(X)\mathfrak{F}_g(X).$$

\begin{prop}\label{esprerance}
Soit $G$ un groupe fini et soit $g$ un élément fixé de $G,$ alors on a :
$$\mathbb{E}_{\mathrm{P}_{G}}(\mathfrak{F}_g)=\delta_{\lbrace 1_G\rbrace}(g).$$
\end{prop}
\begin{proof}
Après simplification, on a :
$$\mathbb{E}_{\mathrm{P}_{G}}(\mathfrak{F}_g)=\sum_{X\in \hat{G}}\frac{\dim X}{|G|}\mathcal{X}(g).$$
Le membre à droite de cette égalité est égal à $\frac{\mathcal{X}^{reg}(g)}{|G|}$ d'après le Théorème \ref{th:dec_alg_de_grp} du chapitre \ref{chapitre1}. Le résultat est donc une conséquence de l'Exemple \ref{ex:calcul_carac_reg} du premier chapitre.
\end{proof}

Soient maintenant $g$ et $g'$ deux éléments de $G$ et on considère l'ensemble $\mathcal{I}$ qui indexe les classes de conjugaison de $G.$ On suppose que $g\in C_\lambda$ et $g'\in C_\delta$ où $\lambda$ et $\delta$ sont deux éléments de $\mathcal{I}.$ Puisque les fonctions caractères sont constantes sur les classes de conjugaison, la définition de $\mathfrak{F}$ peut être étendue sur les classes de conjugaison par :
\begin{equation}\label{eq:var_ale}
\mathfrak{F}_{ {\bf{C}_\lambda} }(X):=|C_\lambda|\mathfrak{F}_g(X),
\end{equation}
pour tout $X\in \hat{G}.$ De plus pour tout $X\in \hat{G}$ on a :
\begin{eqnarray*}
\mathfrak{F}_g(X)\mathfrak{F}_{g'}(X)&=&\frac{1}{|C_\lambda|}\mathfrak{F}_{ {\bf{C}_\lambda} }(X)\frac{1}{|C_\delta|}\mathfrak{F}_{ {\bf{C}_\delta} }(X)\\
&=&\frac{1}{|C_\lambda||C_\delta|}\frac{\mathcal{X}({\bf{C}_\lambda})}{\dim X}\frac{\mathcal{X}({\bf{C}_\delta})}{\dim X}\\
&=&\frac{1}{|C_\lambda||C_\delta|}\frac{\mathcal{X}({\bf{C}_\lambda}{\bf{C}_\delta})}{\dim X} ~~~~~~~~(\text{d'après la Proposition \ref{caract_norma} du chapitre \ref{chapitre1}})\\
&=&\sum_{\rho\in \mathcal{I}}c_{\lambda\delta}^\rho\frac{1}{|C_\lambda||C_\delta|}\frac{\mathcal{X}({\bf{C}_\rho})}{\dim X}\\
&=&\sum_{\rho\in \mathcal{I}}c_{\lambda\delta}^\rho\frac{|C_\rho|}{|C_\lambda||C_\delta|}\mathfrak{F}_{g_\rho}(X)~~~~~~~~(\text{d'après l'équation \eqref{eq:var_ale}})
\end{eqnarray*} 
où $g_\rho$ est un élément arbitraire de $C_\rho$ pour tout $\rho\in \mathcal{I}.$ 

\begin{prop}\label{prop:esp_ordre_2}
Soit $G$ un groupe fini et soient $g\in C_\lambda$ et $g'\in C_\delta$ où $\lambda$ et $\delta$ sont deux éléments de l'ensemble $\mathcal{I}$ indexant les classes de conjugaison de $G.$ On a :
$$\mathbb{E}_{\mathrm{P}_{G}}(\mathfrak{F}_{g}\mathfrak{F}_{g'})=\frac{c_{\lambda\delta}^{\emptyset}}{|C_\lambda||C_\delta|},$$
où $c_{\lambda\delta}^{\emptyset}$ est le coefficient de structure (trivial) de ${\bf{C}}_{1_G}$ dans le produit ${\bf{C}}_\lambda {\bf{C}}_\delta.$
\end{prop}
\begin{proof}
On a :
$$\mathfrak{F}_{g}\mathfrak{F}_{g'}=\sum_{\rho\in \mathcal{I}}c_{\lambda\delta}^\rho\frac{|C_\rho|}{|C_\lambda||C_\delta|}\mathfrak{F}_{g_\rho},$$
où $g_\rho$ est un élément arbitraire de $C_\rho$ pour tout $\rho\in \mathcal{I}.$ En passant à l'espérance, on obtient :
$$\mathbb{E}_{\mathrm{P}_{G}}(\mathfrak{F}_{g}\mathfrak{F}_{g'})=\sum_{\rho\in \mathcal{I}}c_{\lambda\delta}^\rho\frac{|C_\rho|}{|C_\lambda||C_\delta|}\mathbb{E}_{\mathrm{P}_{G}}(\mathfrak{F}_{g_\rho}).$$ D'après la Proposition \ref{esprerance}, on a $\mathbb{E}_{\mathrm{P}_{G}}(\mathfrak{F}_{g_\rho})=\delta_{\lbrace 1_G\rbrace}(g_\rho),$ d'où le résultat.
\end{proof}

On retrouve dans l'exemple suivant un résultat particulier dans le cas du groupe symétrique, voir \cite{hora2007quantum}.

\begin{ex}\label{ex:esp_tran}
D'après l'Exemple \ref{ex:calcul_coef_str}, on obtient :
$$\mathbb{E}_{\mathrm{P}_{\mathcal{S}_n}}(\mathfrak{F}^2_{(1\,\,2)})=\frac{{n \choose 2} }{{n \choose 2}^2 }=\frac{2}{n(n-1)}.$$
\end{ex}

D'après l'Exemple \ref{ex:esp_tran}, l'espérance $\mathbb{E}_{\mathrm{P}_{\mathcal{S}_n}}((\sqrt{n}\mathfrak{F}_{(1\,\,2)})^2)$ tend vers $0$ quand $n$ tend vers l'infini. Cela montre que la variable aléatoire $\sqrt{n}\mathfrak{F}_{(1\,\,2)}$ converge en probabilité vers la variable aléatoire constante $0.$ Autrement dit, on a :
\begin{equation*}
\text{ pour tout $\varepsilon > 0$ on a : }\lim\limits_{\substack{n \to \infty}} \mathrm{P}_{\mathcal{S}_n} (\sqrt{n}|\mathfrak{F}_{(1\,\,2)}|>\varepsilon)=0.
\end{equation*}
Il est possible de généraliser ce résultat pour toutes les permutations de $n.$ 
De plus, si $\lambda$ est une partition propre et si $n$ est un entier tel que $|\lambda|<n,$ alors, d'après le Corollaire \ref{cor:taille_classe_de_conj_S_n} :
\begin{equation*}
|C_{\underline{\lambda}_n}|=\frac{n!}{z_\lambda (n-|\lambda|)!}
\end{equation*}
est un polynôme de degré $|\lambda|.$
\begin{cor}\label{esp_fct_rat_S_n}
Soient $\lambda$ et $\delta$ deux partitions propres et soit $n$ un entier tel que $n>|\lambda|,|\delta|.$ Soient $g$ et $g'$ deux permutations de $n$ tel que $\type-cyclique(g)=\underline{\lambda}_n$ et $\type-cyclique(g')=\underline{\delta}_n,$ alors :
\begin{equation*}
\mathbb{E}_{\mathrm{P}_{\mathcal{S}_n}}(\mathfrak{F}_{g}\mathfrak{F}_{g'})=\frac{z_\lambda z_\delta c_{\lambda\delta}^{\emptyset}(n)}{(n-|\lambda|+1)\cdots n\cdot (n-|\delta|+1)\cdots n}=\left\{
\begin{array}{ll}
  \frac{z_\lambda}{(n-|\lambda|+1)\cdots n} & \qquad \mathrm{si}\quad \lambda=\delta, \\
  0 & \qquad \mathrm{sinon,}\quad \\
 \end{array}
 \right.
\end{equation*}
est une fraction rationnelle.
\end{cor}
\begin{cor}\label{conv_F_g}
Soit $\lambda$ une partition propre et soit $n$ un entier plus grand que la taille de $\lambda.$ Pour toute permutation $\omega$ de $n$ de type-cyclique $\underline{\lambda}_n,$ la variable aléatoire $(\sqrt{n})^{|\lambda|-1}\mathfrak{F}_{\omega}$ converge en probabilité vers la variable aléatoire constante $0.$
\end{cor}
\begin{proof}
Ce résultat s'obtient de la valeur de $\mathbb{E}_{\mathrm{P}_{\mathcal{S}_n}}(\mathfrak{F}^2_{\omega})$ obtenue par le Corollaire \ref{esp_fct_rat_S_n}.
\end{proof}

D'après le Corollaire \ref{esp_fct_rat_S_n}, on obtient une formule pour $\mathbb{E}_{\mathrm{P}_{\mathcal{S}_n}}(\mathfrak{F}^2_{g})$ pour tout $g\in \mathcal{S}_n.$ Dans la sous-section suivante on passe aux calculs des moments d'ordre supérieurs et on illustre l'intérêt de la polynomialité des coefficients de structure dans le calcul.

\subsection{Moments d'ordre supérieurs}

Dans \cite{ivanov2002olshanski}, Ivanov et Olshanski s'intéressent au calcul de $\mathbb{E}_{\mathrm{P}_{\mathcal{S}_n}}(\mathfrak{F}^m_{g})$ pour tout $m\in \mathbb{N}^*.$ Le but de cette sous-section est de montrer comment les coefficients de structure interviennent dans ce calcul et le rôle important que joue la propriété de polynomialité ici.

On a déjà montré, avant d'énoncer la Proposition \ref{prop:esp_ordre_2}, que pour tout $X\in \hat{G}$ on a :
\begin{equation*}
\mathfrak{F}_g(X)\mathfrak{F}_{g'}(X)=\sum_{\rho\in \mathcal{I}}c_{\lambda\delta}^\rho\frac{|C_\rho|}{|C_\lambda||C_\delta|}\mathfrak{F}_{g_\rho}(X),
\end{equation*}
où $g_\rho$ est un élément arbitraire de $C_\rho$ pour tout $\rho\in \mathcal{I}.$ Cela implique que pour tout $g\in C_\lambda$ et pour tout $X\in \hat{G},$ on a :
\begin{equation}\label{eq:moment_ordre_2}
\mathfrak{F}^2_g(X)=\sum_{\rho\in \mathcal{I}}c_{\lambda\lambda}^\rho\frac{|C_\rho|}{|C_\lambda||C_\lambda|}\mathfrak{F}_{g_\rho}(X),
\end{equation}
où $g_\rho$ est un élément arbitraire de $C_\rho$ pour tout $\rho\in \mathcal{I}.$ Si on multiplie l'équation \eqref{eq:moment_ordre_2} par $\mathfrak{F}_g(X),$ on obtient :
\begin{equation}
\mathfrak{F}^3_g(X)=\sum_{\rho\in \mathcal{I}}\sum_{\rho'\in \mathcal{I}}c_{\lambda\lambda}^\rho\frac{|C_\rho|}{|C_\lambda||C_\lambda|}c_{\rho\lambda}^{\rho'}\frac{|C_{\rho'}|}{|C_\rho||C_\lambda|}\mathfrak{F}_{g_{\rho'}}(X),
\end{equation}
où $g_{\rho'}$ est un élément arbitraire de $C_{\rho'}$ pour tout $\rho'\in \mathcal{I}.$ D'après la Proposition \ref{esprerance}, on obtient :
$$\mathbb{E}_{\mathrm{P}_G}(\mathfrak{F}^3_{g})=\sum_{\rho\in \mathcal{I}}c_{\lambda\lambda}^\rho\frac{|C_\rho|}{|C_\lambda||C_\lambda|}c_{\rho\lambda}^{\emptyset}\frac{1}{|C_\rho||C_\lambda|}=c_{\lambda\lambda}^\lambda c_{\lambda\lambda}^{\emptyset}\frac{1}{|C_\lambda|^3}=\frac{c_{\lambda\lambda}^\lambda}{|C_\lambda|^2},$$
pour tout $g\in C_\lambda.$ De même, on peut montrer que :
\begin{equation*}
\mathbb{E}_{\mathrm{P}_G}(\mathfrak{F}^4_{g})=\sum_{\rho\in \mathcal{I}}\frac{c_{\lambda\lambda}^\rho c_{\rho\lambda}^\lambda}{|C_\lambda|^3},
\end{equation*}
pour tout $g\in C_\lambda.$ Cela met en évidence que les moments d'ordre $m\geq 3$ de la variable aléatoire $\mathfrak{F}_{g}$ s'expriment en fonction des coefficients de structure du centre de l'algèbre du groupe $G.$ Comme ces coefficients de structure sont des polynômes dans le cas du groupe symétrique d'après le Théorème \ref{pol_coef_cen_grp_sym} de Farahat et Higman, on obtient le résultat suivant.
\begin{prop}
Soit $\lambda$ une partition propre et soit $n$ un entier plus grand que la taille de $\lambda.$ Pour toute permutation $g$ de $n$ de type-cyclique $\lambda\cup (1^{n-|\lambda|})$ et pour tout $m\geq 3,$ l'espérance $\mathbb{E}_{\mathrm{P}_{\mathcal{S}_n}}(\mathfrak{F}^m_{g})$ est une fonction rationnelle.
\end{prop}

Par ailleurs, les filtrations présentées dans la section précédentes aident à décrire le terme dominant de cette fraction rationnelle (voir \cite[Section 4]{ivanov2002olshanski}) et donc un équivalent asymptotique des moments. Cela permet d'obtenir des lois limites pour les $\mathfrak{F}_{g}$ avec une renormalisation appropriée \cite[Théorèmes 6.1 et 6.5]{ivanov2002olshanski}.

Le lecteur aurait dû remarquer les coefficients de structure particuliers qui apparaissent dans le calcul des moments dans cette section. Ces coefficients sont: 
$$c_{\lambda\lambda}^\emptyset,\,\, c_{\lambda\delta}^\emptyset,\,\, c_{\lambda\lambda}^\lambda,\,\, c_{\lambda\lambda}^\delta\text{ et } c_{\lambda\delta}^\delta.$$  
Il est donc intéressant de savoir s'il existe des formules explicites pour ces coefficients dans le cas où $G$ est le groupe symétrique $\mathcal{S}_n$ (dans ce cas $\mathcal{I}$ est l'ensemble $\mathcal{P}_n$). On a commencé à parler de ce problème dans la Section \ref{sec:calc_coeff_sym} et on a donné des formules explicites pour les coefficients de structure $c_{\lambda\emptyset}^{\delta}$ et $c_{(2,1^{n-2})\lambda}^\delta.$ Comme toute permutation de $\mathcal{S}_n$ est conjuguée d'elle-même, il n'est donc pas difficile de voir que:
$$c_{\lambda\delta}^\emptyset=\left\{\begin{array}{ll}
  |C_\lambda| & \qquad \mathrm{si}\quad \lambda=\delta, \\
  0 & \qquad \mathrm{sinon.}\quad \\
 \end{array}
 \right.$$
On ne connait pas de formule explicite pour les autres coefficients particuliers. Les coefficients de structures $c_{\lambda\delta}^\rho$ du centre de l'algèbre du groupe symétrique sont (à multiplication par la taille de la classe de conjugaison appropriée) symétrique en les trois partitions $\lambda$, $\delta$ et $\rho.$ Ce fait est une conséquence immédiate de l'interprétation combinatoire de ces coefficients et de la formule de Frobenius (voir Théorème \ref{Th_coef_en_fct_carac}). Donc, une formule explicite pour $c_{\lambda\lambda}^\delta$ implique une formule explicite pour $c_{\lambda\delta}^\delta$ et vice-versa. Il serait donc intéressant de répondre à la question suivante:
\begin{que}
Soient $\lambda$ et $\delta$ deux partitions de $n,$ est-il possible de donner des formules explicites pour les coefficients de structure suivants:
$$c_{\lambda\lambda}^\lambda\text{ et } c_{\lambda\lambda}^\delta.$$  
\end{que}
Il semble que des formules existent pour $c_{\lambda\lambda}^\delta$
dans le cas où $\lambda$ est de la forme $\lambda=(2^k,1)$ \cite{Goupil-Personal}.
\subsection{Convergence des diagrammes de Young}

Soit $\lambda=(\lambda_1,\cdots,\lambda_r)$ une partition. Le diagramme de Young associé à $\lambda$ est formé de $r$ lignes de cellule tel que la première ligne possède $\lambda_1$ cellules, deuxième ligne possède $\lambda_2$ cellules et ainsi de suite. Par exemple le diagramme de Young associé à $(4,2,1)$ avec la convention française est donné par la Figure \ref{fig:diag_de_young}.
\begin{figure}[htbp]
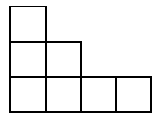

\begin{center}
\yng(1,2,4)

\caption{Le diagramme de Young associé à $(4,2,1).$}
\label{fig:diag_de_young}
\end{center}
\end{figure}

On associe à un diagramme de Young d'une partition $\lambda$ une fonction lipschitzienne qu'on note $\lambda$ de la façon suivante. On tourne le diagramme de Young d'une angle de $\pi/2$ à gauche, puis on applique une homothétie de rapport $\sqrt{2},$ la fonction $\lambda$ est alors la frontière du diagramme dans les nouvelles coordonnées étendue par la fonction $|x|.$ Par exemple, la fonction associée au diagramme de la partition $(4,2,1)$ est présentée dans la figure \ref{fig:fct_diag_young} suivante.
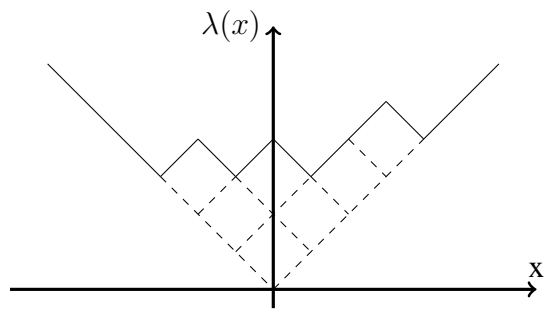
\begin{figure}[htbp]
\begin{center}
\begin{tikzpicture}[scale=0.5]
\draw[very thick, ->] (-7,0)--(7,0)node[above]{x};
\draw[very thick, ->] (0,-.5)--(0,7)node[left]{$\lambda(x)$};
\draw [domain=-6:-3] plot (\x,{-\x});
\draw [domain=-3:-2] plot (\x,{\x+6});
\draw [domain=-2:-1] plot (\x,{-\x+2});
\draw [domain=-1:0] plot (\x,{\x+4});
\draw [domain=0:1] plot (\x,{-\x+4});
\draw [domain=1:3] plot (\x,{\x+2});
\draw [domain=3:4] plot (\x,{-\x+8});
\draw [domain=4:6] plot (\x,{\x});
\draw[dashed] [domain=-3:0] plot (\x,{-\x});
\draw[dashed] [domain=-1:1] plot (\x,{-\x+2});
\draw[dashed] [domain=1:2] plot (\x,{-\x+4});
\draw[dashed] [domain=-1:1] plot (\x,{\x+2});
\draw[dashed] [domain=-2:-1] plot (\x,{\x+4});
\draw[dashed] [domain=2:3] plot (\x,{-\x+6});
\draw[dashed] [domain=0:4] plot (\x,{\x});
\end{tikzpicture}
\caption{La fonction associée au diagramme de Young de la partition $(4,2,1).$}
\label{fig:fct_diag_young}
\end{center} 
\end{figure}

Dans \cite{ivanov2002olshanski}, Ivanov et Olshanski montrent que la fonction normalisée,
\begin{equation*}
\overline{\lambda}^n(x):=\frac{1}{\sqrt{n}}\lambda(\sqrt{n}x)
\end{equation*}
converge en norme $\Vert\cdot \Vert_\infty$ en probabilité vers la fonction $\Omega$ définie ci-dessous,
$$\Omega(x)=
\left\{
\begin{array}{ll}
  \frac{2}{\pi}(x \arcsin(\frac{x}{2})+\sqrt{4-x^2}) & \qquad \mathrm{si}\quad |x|\leq 2, \\
  |x| & \qquad \mathrm{si}\quad |x|>2. \\
 \end{array}
 \right.$$
Cela veut dire que $$\Vert \overline{\lambda}^n- \Omega\Vert_\infty:=\sup_{x\in \mathbb{R}}|\overline{\lambda}^n(x)- \Omega(x)|$$ tend vers $0$ en probabilité quand $n$ tend vers l'infini. La convergence en probabilité de la variable aléatoire $\mathfrak{F}_{\omega}$ donnée dans le Corollaire \ref{conv_F_g} est une étape intermédiaire qu'Ivanov et Olshanski ont utilisée pour arriver à leur résultat. En fait, ce résultat a été trouvé pour la première fois, en $1977,$ par Logan et Shepp, voir \cite{logan1977variational}. Dans la même année, Vershik et Kerov ont donné une preuve indépendante de celle de Logan et Shepp, voir \cite{vervsik1977asymptotic}. La preuve d'Ivanov et Olshanski passant par les $\mathfrak{F}_{g}$ a l'avantage de permettre de décrire aussi les fluctuations de $\overline{\lambda}^n$ autour de la forme limite $\Omega$, voir \cite[Section 7]{ivanov2002olshanski}.

\chapter{L'algèbre de Hecke de la paire $(\mathcal{S}_{2n},\mathcal{B}_n)$}
\label{chapitre3}

Dans le deuxième chapitre on a étudié le centre de l'algèbre d'un groupe particulier, celui du groupe symétrique. Dans ce chapitre on va étudier une algèbre particulière de doubles-classes, celle des doubles-classes de $\mathcal{B}_n$ dans $\mathcal{S}_{2n},$ où $\mathcal{B}_n$ est le groupe hyperoctaédral. Cette algèbre est appelée algèbre de Hecke de la paire $(\mathcal{S}_{2n},\mathcal{B}_n)$ et a été introduite pour la première fois par James en 1961, voir \cite{James1961}.

Le choix d’étudier cette algèbre est motivé par une longue liste de propriétés similaires à celles du centre de l’algèbre du groupe symétrique. D’abord on va montrer que les doubles-classes de $\mathcal{B}_n$ dans $\mathcal{S}_{2n},$ peuvent être indexées, comme les classes de conjugaison de $\mathcal{S}_n,$ par les partitions de $n$ ce qui veut dire que l’algèbre de Hecke de la paire $(\mathcal{S}_{2n},\mathcal{B}_n)$ possède une base indexée par les éléments de $\mathcal{PP}_{\leq n}.$ 

D'autre part, les coefficients de structure associés à cette base comptent les graphes dessinés sur des surfaces non orientées avec certaines contraintes, comme on va le montrer dans la Section \ref{sec:coef_surf_non-orie}, alors que les coefficients de structure de $Z(\mathbb{C}[\mathcal{S}_n])$ comptent les graphes dessinés sur des surfaces orientées avec des contraintes similaires. Ce fait a été prouvé par Goulden et Jackson, voir \cite{GouldenJacksonLocallyOrientedMaps} pour plus de détails.

Un des résultats principaux de cette thèse est une propriété de polynomialité pour les coefficients de structure de l'algèbre de Hecke de la paire $(\mathcal{S}_{2n},\mathcal{B}_n)$ similaire à celle donnée par Farahat et Higman (Théorème \ref{pol_coef_cen_grp_sym}) dans le cas du centre de l'algèbre du groupe symétrique. On va démontrer cette propriété d'une façon combinatoire dans la Section \ref{sec:polynomialite}. Notre démonstration est inspirée de celle d'Ivanov et Kerov \cite{Ivanov1999} dans le cas du centre de l'algèbre du groupe symétrique.

Comme dans le cas de $Z(\mathbb{C}[\mathcal{S}_n]),$ on indexe les coefficients de structure de l'algèbre de Hecke de la paire $(\mathcal{S}_{2n},\mathcal{B}_n)$ par les partitions propres pour faire apparaître naturellement la dépendance en $n.$ Cela nous aide à présenter d'une manière claire notre résultat.

Cette propriété de polynomialité a été présentée par Aker et Can dans \cite{Aker20122465}. Les auteurs de cet article ont suivi l'approche de Farahat et Higman dans leur preuve. Leur théorème de polynomialité tel qu'il apparaît dans \cite{Aker20122465} contient des erreurs. Récemment, après la publication du contenu de ce chapitre Can et {\"O}zden ont proposé dans \cite{2014arXiv1407.3700B} une correction de la preuve. Le résultat de polynomialité a aussi été trouvé, d'une manière indirecte en utilisant les polynômes de Jack, par Do{\l}{\c e}ga et Féray, voir \cite[Proposition 5.3]{2014arXiv1402.4615D}. Le résultat présenté ici est plus fort que celui de Do{\l}{\c e}ga et Féray car nous montrons aussi que les polynômes obtenus ont des coefficients positifs dans une certaine base.

On définit des nouveaux objets combinatoires pour prouver notre résultat. On les appelle les bijections partielles. Elles sont des analogues des permutations partielles, introduites dans \cite{Ivanov1999}, dans le cas de l'algèbre de Hecke de la paire \hecke. Dans la Section \ref{sec:polynomialite}, on montre l'existence d'une algèbre universelle qui se projette sur l'algèbre de Hecke de la paire $(\mathcal{S}_{2n},\mathcal{B}_n)$ pour tout $n.$ Cette algèbre est isomorphe à l'algèbre des fonctions symétriques décalées d'ordre $2.$ Un isomorphisme est explicitement construit à la Section \ref{sec:isom_fct_decal_2}.

On donne dans la Section \ref{sec:filtr} plusieurs filtrations sur cette algèbre universelle. Cela va nous permettre de majorer le degré des polynômes décrivant les coefficients de structure de l'algèbre de Hecke de la paire \hecke.

Le résultat principal de ce chapitre a été publié par l'auteur dans \cite{tout2013structure}. Une version plus détaillée, voir \cite{toutarxiv}, a été soumise à un journal.

\section{Une algèbre de doubles-classes}
Soit $n$ un entier strictement positif. Pour tout entier  $k\geq 1$, on note par $p(k)$ la paire \linebreak $\lbrace 2k-1,2k\rbrace.$ Le \textit{groupe hyperoctaédral}\label{nomen:grp_hyper} \nomenclature[28]{$\mathcal{B}_n$}{Groupe hyperoctaédral \quad \pageref{nomen:grp_hyper}} $\mathcal{B}_n$ est le sous-groupe de $\mathcal{S}_{2n}$ formé des permutations qui envoient chaque paire de la forme $p(k)$ sur une autre paire de la même forme :
$$\mathcal{B}_n=\lbrace \omega\in \mathcal{S}_{2n}\text{ tel que pour tout $1\leq k\leq n$ on a } \omega(p(k))=p(k') \text{ où $1\leq k'\leq n$} \rbrace.$$ Par exemple, la permutation $4\,3\,1\,2\,6\,5\,8\,7$ de $8$ appartient à $\mathcal{B}_4.$ 

L'\textit{algèbre de Hecke de la paire} $(\mathcal{S}_{2n},\mathcal{B}_n)$ est l'algèbre des doubles-classes de $\mathcal{B}_n$ dans $\mathcal{S}_{2n}.$ D'après les notations du premier chapitre, cette algèbre s'écrit \Hecke. Le Chapitre VII du livre de \cite{McDo} contient une section dont le but est de montrer que la paire $(\mathcal{S}_{2n},\mathcal{B}_n)$ est une paire de Gelfand et d'étudier cette paire. 

\subsection{Coset-type d'une permutation}\label{sec:base_de_Hecke}
À chaque permutation $\omega$ de $[2n]$ on peut associer un graphe que l'on note $\Gamma(\omega).$ Le graphe $\Gamma(\omega)$ possède $2n$ sommets positionnés sur un cercle. Chaque sommet possède deux étiquettes (\textit{extérieure} et \textit{intérieure}). Les étiquettes extérieures sont les nombres entre $1$ et $2n.$ Pour chaque sommet avec $i$ comme étiquette extérieure, l'image de $i$ par $\omega$ est l'étiquette intérieure. On joint les sommets qui possèdent $2i-1$ et $2i$ comme étiquettes extérieures (resp. intérieures) par une arête. 
Ainsi, le graphe $\Gamma(\omega)$ est formé d'une union disjointe de cycles de longueurs paires -- dans un cycle, le nombre d'arêtes extérieures est égale au nombre d'arêtes intérieures --. 

\begin{definition}
Soit $\omega$ une permutation de $2n$ et soit $\Gamma(\omega)$ le graphe associé à $\omega.$ Supposons que $\Gamma(\omega)$ possède $r$ cycles de tailles $2\lambda_1\geq 2\lambda_2\geq 2\lambda_3\geq \cdots\geq 2\lambda_r$ alors le coset-type de $\omega$ est la partition $\ct(\omega):=(\lambda_1,\cdots, \lambda_r)$\label{nomen:cos_typ} \nomenclature[29]{$\ct()$}{Le coset-type d'une permutation de $2n$ \quad \pageref{nomen:cos_typ}} de $n.$
\end{definition}
\begin{ex}\label{exemple du coset type} Le graphe $\Gamma(\omega)$ associé à la permutation $\omega=2\,4\,9\,3\,1\,10\,5\,8\,6\,7\in \mathcal{S}_{10}$ est dessiné sur la Figure~\ref{fig:coset-type}.
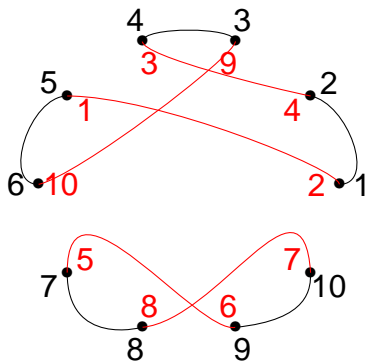
\begin{figure}[htbp]
\begin{center}
\begin{tikzpicture}
\foreach \angle / \label in
{ 0/2, 36/4, 72/9, 108/3, 144/1, 180/10, 216/5,
252/8, 288/6, 324/7}
{
\node at (\angle:2cm) {\footnotesize $\bullet$};
\draw[red] (\angle:1.7cm) node{\textsf{\label}};
}
\foreach \angle / \label in
{ 0/1, 36/2, 72/3, 108/4, 144/5, 180/6, 216/7,
252/8, 288/9, 324/10}
{\draw (\angle:2.3cm) node{\textsf{\label}};}
\draw (0:2cm) .. controls + (5mm,0) and +(.5,-.1) .. (36:2cm);
\draw[red] (36:2cm) .. controls + (-.5,.1) and + (0,-.2) .. (108:2);
\draw (108:2cm) .. controls + (0,.2) and + (0,.2) .. (72:2);
\draw[red]	(72:2)		  .. controls + (0,-.2) and + (+.5,.1) .. (180:2);
\draw	(180:2)		  .. controls + (-5mm,0) and + (-0.5,0) .. (144:2);
\draw[red]	(144:2)		  .. controls + (1.2,0) and + (0,.2) .. (0:2);
\draw[red] (216:2cm) .. controls + (0,15mm) and +(-.5,-.1) .. (288:2cm);
\draw	(288:2cm)		    .. controls + (1,0.1) and + (0,-.2) .. (324:2);
\draw[red]	(324:2)		    .. controls + (-0.1,1.5) and + (1,0.1) .. (252:2);
\draw	(252:2)		    .. controls + (0,0) and + (0,-1) .. (216:2);
\end{tikzpicture}
\caption{Le graphe $\Gamma(\omega)$ de l'Exemple \ref{exemple du coset type}.}
\label{fig:coset-type}
\end{center}
\end{figure}
Ce graphe possède deux cycles de longueurs $6$ et $4$, donc $\ct(\omega)=(3,2)$.
\end{ex}
Le coset-type d'une permutation nous permet de décrire les doubles-classes de $\mathcal{B}_n$ dans $\mathcal{S}_{2n}.$ Comme le montre la proposition suivante, la double-classe de $\mathcal{B}_n$ dans $\mathcal{S}_{2n}$ d'une permutation $x$ est l'ensemble de toutes les permutations de $2n$ qui possèdent le même coset-type que $x.$
\begin{prop}
Soit $x$ une permutation de $2n$ de coset-type $\lambda,$ on a:
$$\mathcal{B}_nx\mathcal{B}_n=\lbrace y\in \mathcal{S}_{2n}\text{ tel que } \ct(y)=\lambda\rbrace.$$
\end{prop} 
\begin{dem}
Voir le point (2.1) page 401 de \cite{McDo}.
\end{dem}
Cette proposition nous permet d'indexer les doubles-classes de $\mathcal{B}_n$ dans $\mathcal{S}_{2n}$ et ainsi une base de \Hecke, par les partitions de $n.$ Si $\lambda$ est un élément de $\mathcal{P}_n,$ on note par $K_\lambda$\label{nomen:K_lambda} \nomenclature[30]{$K_\lambda$}{L'ensemble de toutes les permutations de $2n$ de coset-type $\lambda$ \quad \pageref{nomen:K_lambda}} l'ensemble de toutes les permutations de $2n$ de coset-type $\lambda.$ 
\begin{cor}
Les éléments de la famille $({\bf K}_\lambda)_{\lambda\in \mathcal{P}_n}$ où 
$${\bf K}_\lambda:=\sum_{\omega\in \mathcal{S}_{2n}\atop{ \ct(\omega)=\lambda}}\omega,$$
forment une base de l'algèbre de Hecke \Hecke.
\end{cor}

\subsection{Coefficients de structures et cartes dessinées sur des surfaces non-orientées}\label{sec:coef_surf_non-orie}

Les coefficients de structures $\alpha_{\lambda\delta}^{\rho}$\label{nomen:coef_str_hecke} \nomenclature[31]{$\alpha_{\lambda\delta}^{\rho}$}{Les coefficients de structure de l'algèbre de Hecke de la paire $(\mathcal{S}_{2n},\mathcal{B}_n)$ \quad \pageref{nomen:coef_str_hecke}} de l'algèbre de Hecke de la paire $(\mathcal{S}_{2n},\mathcal{B}_n)$ sont donnés par l'équation suivante:
\begin{equation*}
{\bf K}_{\lambda}{\bf K}_{\delta}=\sum_{\rho\in \mathcal{P}_n}\alpha_{\lambda\delta}^{\rho}{\bf K}_{\rho},
\end{equation*}
où $\lambda$ et $\delta$ sont deux partitions de $n.$ En utilisant la proposition \ref{coef_de_str} du Chapitre \ref{chapitre1} on a le corollaire suivant.
\begin{cor}
Soient $\lambda,\delta$ et $\rho$ trois partitions de $n.$ Le coefficient $\alpha_{\lambda\delta}^{\rho}$ est égal à :
$$\alpha_{\lambda\delta}^{\rho}=|\lbrace (\alpha,\sigma)\in \mathcal{S}_{2n}^2 \text{ tel que $\ct(\alpha)=\lambda,$ $\ct(\sigma)=\delta$ et $\alpha\sigma=\varphi$}\rbrace|,$$
où $\varphi\in \mathcal{S}_{2n}$ est une permutation fixée telle que $\ct(\varphi)=\rho.$
\end{cor}

Pour interpréter combinatoirement les coefficients de structure de l'algèbre de Hecke de la paire \hecke, on s'intéresse aux graphes bicoloriés dessinés sur des surfaces non-orientées. Pour la définition des graphes bicoloriés, nous renvoyons à la Section \ref{sec:int_com}.

Dans le cas des graphes sur des surfaces non-orientées, pour définir un étiquetage, on étiquette les deux côtés d'une arête. Un graphe bicolorié étiqueté de $a$ arêtes dessiné sur une surface non-orientée -- appelé carte comme dans la Sous-section \ref{sec:int_com} du deuxième chapitre -- peut être identifié avec un triplet de permutations de $[2a].$

On reprend le graphe de la Figure \ref{fig:etiq_bic} du Chapitre \ref{chapitre2}. On regarde ce graphe comme dessiné sur une surface sans orientation donnée. 

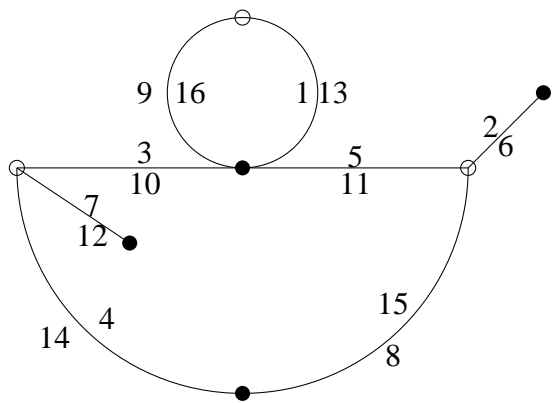
\begin{figure}[htbp]
\begin{center}
\begin{tikzpicture}
\draw (0,0) circle (0.1cm);
\fill[black] (3,0) circle (0.1cm);
\draw (6,0) circle (0.1cm);
\fill[black] (7,1) circle (0.1cm);
\fill[black] (1.5,-1) circle (0.1cm);
\draw (3,2) circle (0.1cm);
\fill[black] (3,-3) circle (0.1cm);
\draw (3,1) circle (1cm);
\draw (0,0) -- (3,0) ;
\draw (3,0) -- (6,0) ;
\draw (6,0) -- (7,1) ;
\draw (0,0) -- (1.5,-1) ;
\draw (0,0) arc (180:360:3cm) ;
\node at (1.7,0.2) {3};
\node at (1.2,-2) {4};
\node at (1,-.5) {7};
\node at (1,-.9) {12};
\node at (3.8,1) {1};
\node at (1.7,1) {9};
\node at (1.7,-.2) {10};
\node at (4.5,.15) {5};
\node at (6.3,.55) {2};
\node at (4.5,-.2) {11};
\node at (5,-2.5) {8};
\node at (6.5,.3) {6};
\node at (4.2,1) {13};
\node at (2.3,1) {16};
\node at (5,-1.8) {15};
\node at (.5,-2.25) {14};
\end{tikzpicture}
\caption{Un étiquetage pour un graphe bicolorié dessiné sur une surface non-orientée.}
\label{fig:surf_non_orient}
\end{center}
\end{figure}

On associe à une carte bicoloriée dessinée sur une surface non-orientée un triplet d'appariements de $[2a].$ Un appariement de $[2a]$ est un ensemble de paires d'éléments de $[2a]$ disjointes et dont l'union est $[2a].$ L'appariement $\mathfrak{B}$ (resp. $\mathfrak{N}$) est associé aux sommets blancs (resp. noirs) et l'appariement $\mathfrak{A}$ est associé aux arêtes. La paire associée à une arête est formée des étiquettes des deux côtés de l'arête. Pour un sommet ayant $r$ arêtes incidentes, on associe $r$ paires où chaque paire est formée des étiquettes intérieures au coin formé par deux arêtes. Par exemple, les appariements $\mathfrak{B},\mathfrak{N}$ et $\mathfrak{A}$ de la carte de la Figure \ref{fig:surf_non_orient} sont :
$$\mathfrak{B}=\lbrace\lbrace 2,5\rbrace,\lbrace 11,15\rbrace,\lbrace 6,8\rbrace,\lbrace 1,16\rbrace,\lbrace 13,9\rbrace,\lbrace 3,14\rbrace,\lbrace 7,10\rbrace,\lbrace 12,4\rbrace\rbrace,$$
$$\mathfrak{N}=\lbrace\lbrace 2,6\rbrace,\lbrace 13,5\rbrace,\lbrace 1,16\rbrace,\lbrace 9,3\rbrace,\lbrace 10,11\rbrace,\lbrace 12,7\rbrace,\lbrace 8,14\rbrace,\lbrace 15,4\rbrace\rbrace,$$
et
$$\mathfrak{A}=\lbrace\lbrace 2,6\rbrace,\lbrace 11,5\rbrace,\lbrace 15,8\rbrace,\lbrace 1,13\rbrace,\lbrace 16,9\rbrace,\lbrace 10,3\rbrace,\lbrace 7,12\rbrace,\lbrace 14,4\rbrace\rbrace.$$

À une paire $(\mathfrak{X},\mathfrak{Y})$ d'appariements de $[2a],$ on associe un graphe $\Gamma(\mathfrak{X},\mathfrak{Y})$ de $2a$ sommets de la manière suivante. On met $2a$ sommets étiquetés avec deux étiquettes (extérieure et intérieure) avec les entiers de l'ensemble $[2a]$ sur un cercle. Si la paire $\lbrace x,y\rbrace,$ est un élément de $\mathfrak{X},$ et si $y$ est l'étiquette extérieure d'un sommet $s,$ alors l'étiquette intérieure de $s$ est l'entier $z$ tel que $\lbrace x,z\rbrace,$ est un élément de $\mathfrak{Y}.$ Pour chaque paire d'entier $\lbrace x,y\rbrace$ de $\mathfrak{X},$ (resp. $\mathfrak{Y}$) on relie les deux sommets ayant les étiquettes extérieures (resp. intérieures) $x$ et $y$ par une arête extérieure (resp. intérieure -- dessinée en rouge --). Par exemple le graphe $\Gamma(\mathfrak{B},\mathfrak{N})$ est donnée par la Figure \ref{fig:graphe d'appariements}.

\begin{figure}[htbp]
\begin{center}
\begin{tikzpicture}
\foreach \angle / \label in
{ 0/13, 22.5/6, 45/4, 67.5/10, 90/14, 112.5/2, 135/16,
157.5/1, 180/3, 202.5/5, 225/8, 247.5/9, 270/11, 292.5/12, 315/15, 337.5/7}
{
\node at (\angle:2cm) {\footnotesize $\bullet$};
\draw[red] (\angle:1.7cm) node{\textsf{\label}};
}
\foreach \angle / \label in
{ 0/2, 22.5/5, 45/11, 67.5/15, 90/6, 112.5/8, 135/16,
157.5/1, 180/13, 202.5/9, 225/3, 247.5/14, 270/7, 292.5/10, 315/12, 337.5/4}
{\draw (\angle:2.3cm) node{\textsf{\label}};}
\draw (0:2cm) .. controls + (0,0) and +(0,0) .. (22.5:2cm);
\draw (45:2cm) .. controls + (0,0) and +(0,0) .. (67.5:2cm);
\draw (90:2cm) .. controls + (0,0) and +(0,0) .. (112.5:2cm);
\draw (135:2cm) .. controls + (0,0) and +(0,0) .. (157.5:2cm);
\draw (180:2cm) .. controls + (0,0) and +(0,0) .. (202.5:2cm);
\draw (225:2cm) .. controls + (0,0) and +(0,0) .. (247.5:2cm);
\draw (270:2cm) .. controls + (0,0) and +(0,0) .. (292.5:2cm);
\draw (315:2cm) .. controls + (0,0) and +(0,0) .. (337.5:2cm);
\draw[red] (0:2cm) .. controls + (0,0) and +(0,0) .. (202.5:2cm);
\draw[red] (22.5:2cm) .. controls + (0,0) and +(0,0) .. (112.5:2cm);
\draw[red] (45:2cm) .. controls + (0,0) and +(0,0) .. (315:2cm);
\draw[red] (67.5:2cm) .. controls + (0,0) and +(0,0) .. (270:2cm);
\draw[red] (90:2cm) .. controls + (0,0) and +(0,0) .. (225:2cm);
\draw[red] (135:2cm) .. controls + (0.5,0) and +(0,0) .. (157.5:2cm);
\draw[red] (180:2cm) .. controls + (0,0) and +(0,0) .. (247.5:2cm);
\draw[red] (292.5:2cm) .. controls + (0,0) and +(0,0) .. (337.5:2cm);
\end{tikzpicture}
\caption{Le graphe $\Gamma(\mathfrak{B},\mathfrak{N}).$}
\label{fig:graphe d'appariements}
\end{center}
\end{figure}
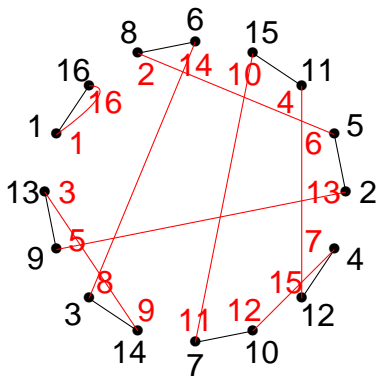

\begin{rem}
Le graphe associé à une paire d'appariements qu'on note $\Gamma$ est légèrement différent, mais équivalent à ce que Goulden et Jackson notent $\Lambda$ dans \cite{GouldenJacksonLocallyOrientedMaps}. On a fait ce choix pour qu'on puisse définir, un peut plus loin dans cette section, le "coset-type" d'un graphe $\Gamma$ de la même façon que pour une permutation.
\end{rem}

Comme dans le cas orienté, on définit une carte bicoloriée non-nécessairement connexe comme une union disjointe de cartes bicoloriées connexes. Cela veut dire que chaque composante connexe du graphe est plongé dans une composante connexe différente de la surface et que les faces ainsi obtenues sont homéomorphes à des disques ouverts.

\begin{prop}
L'ensemble des cartes bicoloriées étiquetées, avec $a$ arêtes, non-nécessairement connexes dessinées sur des surfaces non-orientées est en bijection avec l'ensemble des triplets d'appariement de $[2a].$
\end{prop}
\begin{proof}
On a vu, avant l'énoncé de cette proposition, comment on obtient un triplet d'appariement $[2a]$ à partir d'une carte bicoloriée non-nécessairement connexe de $a$ arêtes dessinée sur une surface non-orientée. Réciproquement, pour un triplet $(\mathfrak{B},\mathfrak{N},\mathfrak{A})$ d'appariement $[2a],$ la carte y associée s'obtient en réunissant d'une manière précise, voir la section 3.2 de l'article \cite{2013arXiv1301.6531D}, les trois graphes $\Gamma(\mathfrak{B},\mathfrak{N}),$ $\Gamma(\mathfrak{B},\mathfrak{A})$ et $\Gamma(\mathfrak{N},\mathfrak{A}).$ 
\end{proof}

Pour plus de détails sur la correspondance entre cartes bicoloriées sur des surfaces non-orientées et triplets d'appariement, le lecteur est invité à regarder l'article \cite{GouldenJacksonLocallyOrientedMaps} de Goulden et Jackson, l'article \cite{2013arXiv1301.6531D} de {Do{\l}ega}, {F{\'e}ray} et {{\'S}niady} et l'article \cite{Feray2011338} de {F{\'e}ray} et {{\'S}niady}.



Comme dans le cas des permutations de $2n,$ on définit le \text{coset-type} du graphe $\Gamma(\mathfrak{X},\mathfrak{Y})$ comme étant la partition formée des moitiés des longueurs des cycles de  $\Gamma(\mathfrak{X},\mathfrak{Y}).$ Par exemple le coset-type de $\Gamma(\mathfrak{B},\mathfrak{N})$ est la partition $(4,3,1)$ de $8.$ En fait, si on note par $\mathfrak{S}$ l'appariement suivant de $[2n]$ :
$$\mathfrak{S}:=\lbrace\lbrace 1,2\rbrace,\lbrace 3,4\rbrace,\cdots,\lbrace 2n-1,2n\rbrace\rbrace,$$
alors pour toute permutation $\omega$ de $[2n]$ on a :
$$\coset-type(\omega)=\coset-type\big( \Gamma(\mathfrak{S},\omega(\mathfrak{S}))\big),$$
où $\omega(\mathfrak{S})$ est l'appariement suivant de $[2n]$ :
$$\omega(\mathfrak{S}):=\lbrace\lbrace \omega(1),\omega(2)\rbrace,\lbrace \omega(3),\omega(4)\rbrace,\cdots,\lbrace \omega(2n-1),\omega(2n)\rbrace\rbrace.$$
Ceci explique la dénomination "coset-type".

On associe à une carte bicoloriée à $2a$ arêtes dessinée sur une surface non-orientée trois partitions de $a$ correspondant aux degrés des sommets blancs, des sommets noirs et des faces d'une manière similaire à celles obtenues dans le cas des cartes bicoloriées dessinées sur des surfaces orientées traité dans la Section \ref{sec:int_com}. Par exemple les partitions des sommets noirs, des sommets blancs et des faces de la Figure \ref{fig:surf_non_orient} sont $(4,2,1,1),$ $(3,3,2)$ et $(4,3,1)$ respectivement.

Soient $\lambda,$ $\delta$ et $\rho$ trois partitions de $n,$ le nombre ${g'}_{\lambda\delta}^{\rho}$ des cartes bicoloriés non-nécessairement connexes, dessinés sur une surface non-orientée, avec $n$ arêtes, $l(\lambda)$ sommets blancs, $l(\delta)$ sommets noirs, $l(\rho)$ faces et dont la partition des sommets noirs est $\lambda$, celle des sommets blancs est $\delta$ et celle des faces est $\rho,$ 
est donné par :
$${g'}_{\lambda\delta}^\rho=|\lbrace (\mathfrak{B},\mathfrak{N},\mathfrak{A}) \text{ triplet d'appariements de $[2n]$} \text{ tel que $coset-type(\Gamma(\mathfrak{B},\mathfrak{A}))=\lambda,$}$$ $$ coset-type(\Gamma(\mathfrak{N},\mathfrak{A}))=\delta,  \text{ et } coset-type(\Gamma(\mathfrak{B},\mathfrak{N}))=\rho \rbrace|.$$
Avec cette égalité, le nombre ${g'}_{\lambda\delta}^\rho$ n'apparaît pas immédiatement comme un coefficient de structure mais avec un peu de travail on peut le relier à ceux de l'algèbre de Hecke de la paire \hecke.

\begin{prop}[Goulden et Jackson]
Soient $\lambda,$ $\delta$ et $\rho$ trois partitions de $n.$ Alors on a :
$${g'}_{\lambda\delta}^\rho=\frac{(2n)!\frac{n!}{z_\lambda}}{(n!)^2 2^{n+l(\lambda)}}\alpha_{\lambda\delta}^{\rho}.$$
\end{prop}
\begin{proof}
Voir la preuve du Corollaire 2.3 de l'article \cite{GouldenJacksonLocallyOrientedMaps} de Goulden et Jackson.
\end{proof}

\section{Résultat de polynomialité et idée de la preuve}\label{sec:res_idee}

Comme dans le cas du centre de l'algèbre du groupe symétrique, on utilise les partitions propres pour présenter la propriété de polynomialité des coefficients de structure de l'algèbre de Hecke de la paire \hecke. Soient $\lambda$ et $\delta$ deux partitions propres et soit $n$ un entier positif tel que $|\lambda|,|\delta|\leq n,$ alors on a :
\begin{equation}
{\bf K}_{\underline{\lambda}_n}\cdot {\bf K}_{\underline{\delta}_n}=\sum_{\tau\text{ partition propre}\atop {|\tau|\leq n}}\alpha_{\lambda\delta}^{\tau}(n){\bf K}_{\underline{\tau}_n}.
\end{equation}
La propriété de polynomialité des coefficients de structure $\alpha_{\lambda\delta}^{\rho}(n)$ de l'algèbre de Hecke de la paire $(\mathcal{S}_{2n},\mathcal{B}_n)$ est donnée par le théorème suivant. Soit $n$ un entier positif, pour $0\leq k\leq n,$ on note $(n)_{k}:=\frac{n!}{(n-k)!}=n(n-1)\cdots (n-k+1).$
\begin{theoreme}\label{Theorem 2.1}
Soient $\lambda$, $\delta$ et $\rho$ trois partitions propres. Le coefficient de structure $\alpha_{\lambda\delta}^{\rho}(n)$ s'écrit
$$\alpha_{\lambda\delta}^{\rho}(n)=\left\{
\begin{array}{ll}
  2^nn!f_{\lambda\delta}^{\rho}(n) & \qquad \mathrm{si}\quad n\geq |\rho|,\\\\
  0 & \qquad \mathrm{si}\quad n< |\rho|, \\
 \end{array}
 \right.$$
 où $\displaystyle{  f_{\lambda\delta}^{\rho}(n)=\sum_{j=0}^{|\lambda|+|\delta|-|\rho|}a_j(n-|\rho|)_{j}}$ est un polynôme en $n$ et les $a_j$ sont des nombres rationnels positifs.
\end{theoreme}

L'idée de notre preuve est de construire une algèbre universelle $\mathcal{A'}_\infty$ sur $\mathbb{C}$ avec les propriétés suivantes :
\begin{enumerate}[topsep=1pt, partopsep=1pt, itemsep=1pt, parsep=1pt]
\item Pour tout $n\in \mathbb{N}^*$, il existe un morphisme d'algèbres $\theta_n:\mathcal{A'}_\infty \longrightarrow \mathbb{C}[\mathcal{B}_n\setminus \mathcal{S}_{2n}/\mathcal{B}_n].$
\item Chaque élément $x$ de $\mathcal{A'}_\infty$ s'écrit d'une façon unique comme une combinaison linéaire infinie d'éléments $T_\lambda$, indexés par les partitions. Cela implique que, pour deux partitions $\lambda$ et $\delta$, il existe des nombres rationnels positifs $b_{\lambda\delta}^\rho$ tel que :
\begin{equation}\label{equation 2}
T_\lambda \pr T_\delta=\sum_{\rho\text{ partition}}b_{\lambda\delta}^\rho T_\rho.
\end{equation} 
\item Le morphisme $\theta_n$ envoie $T_\lambda$ sur un multiple de ${K}_{\underline{\lambda}_n}.$
\end{enumerate}
Pour construire $\mathcal{A'}_\infty$, on utilise des objets combinatoires appelées "bijections partielles". Pour tout $n\in \mathbb{N}^*$, on construit une algèbre $\mathcal{A'}_n$ en utilisant l'ensemble des bijections partielles de $[2n].$ L'algèbre $\mathcal{A'}_\infty$ est définie comme étant la limite projective de la suite $(\mathcal{A'}_n).$

La projection $p_n:\mathcal{A'}_\infty\rightarrow \mathcal{A'}_n$ implique des coefficients qui sont des polynômes en $n.$ En définissant l'extension d'une bijection partielle de $[2n]$ en une permutation de $[2n]$, on construit un morphisme de $\mathcal{A'}_n$ à valeur dans $\mathbb{C}[\mathcal{B}_n\setminus \mathcal{S}_{2n}/\mathcal{B}_n].$ Ce morphisme fait intervenir des coefficients où apparaît le nombre $2^nn!.$ Il s'avère que le  morphisme $\theta_n$ est une composition de morphismes, voir Figure \ref{fig:diagram}. L'étape finale consiste à appliquer la chaine des morphismes de la Figure \ref{fig:diagram} à l'équation (\ref{equation 2}).
\begin{figure}[htbp]
\begin{center}
$$\xymatrix{
    \mathcal{A'}_\infty \ar[dd]_{\theta_{n}} \ar[rd]^*[@]{\hbox to 0pt{\hss\txt{$p_n$}\hss}}& \\
     & \mathcal{A'}_n \ar[dl] \\
     \mathbb{C}[\mathcal{B}_n\setminus \mathcal{S}_{2n}/\mathcal{B}_n]}$$
\end{center}
\caption{Diagramme des algèbres et morphismes impliqués dans notre preuve.}
\label{fig:diagram}
\end{figure}
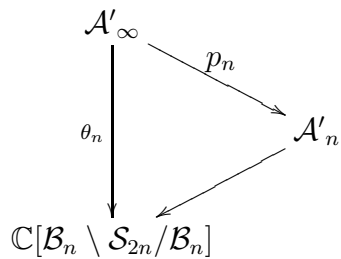

\section{L'algèbre des bijections partielles}

\subsection{Qu'est ce qu'une bijection partielle ?}

On va donner l'analogue d'une permutation partielle dans le cas de l'algèbre de Hecke de la paire \hecke. C'est une bijection dont l'ensemble de départ et l'ensemble d'arrivée ne sont pas nécessairement égaux. Pour garder la possibilité de définir le coset-type d'une telle bijection, il faut que l'ensemble de départ et l'ensemble d'arrivée soient des unions de paires de la forme $p(k).$ 

Comme d'habitude on note par $\mathbb{N}^{*}$ l'ensemble des entiers naturels strictement positifs. Pour $n\in \mathbb{N}^{*}$, on définit $\mathbf{P}_{n}$ ainsi : 
$$\mathbf{P}_{n}:=\lbrace p(k_1)\cup\cdots\cup p(k_i)~|~ 1 \leq i \leq n,~1\leq k_1<\cdots < k_i \leq n \rbrace.$$
\begin{definition}
Une \textit{bijection partielle}\label{nomen:bij_part} \nomenclature[32]{$(\sigma, d ,d^{'})$}{Notation d'une bijection partielle \quad \pageref{nomen:bij_part}} de $[2n]$ est un triplet $(\sigma, d ,d^{'})$ où $d,d^{'}\in \mathbf{P}_n$ et \linebreak $\sigma:d\longrightarrow d^{'}$ est une bijection. L'ensemble $d$ est appelé l'ensemble de départ de $(\sigma, d ,d^{'})$ et $d^{'}$ est son ensemble d'arrivée. On note par $Q_n$\label{nomen:ens_bij_part} \nomenclature[33]{$Q_n$}{L'ensemble des bijections partielles de $[2n]$ \quad \pageref{nomen:ens_bij_part}} l'ensemble des bijections partielles de $[2n].$
\end{definition}

\begin{rem}
Dans la théorie des semi-groupes, le terme "bijection partielle" fait référence à une injection arbitraire d'un sous-ensemble $A$ de $[n]$ dans $[n].$ De ce point de vue, une bijection partielle de $[2n]$ est une bijection partielle particulière de la théorie des semi-groupes. Cependant, on a choisi cette appellation pour être similaire, dans le cas de l'algèbre de Hecke de la paire $(\mathcal{S}_{2n},\mathcal{B}_n),$ à la notion de "permutation partielle" utilisée par Ivanov et Kerov dans \cite{Ivanov1999}. Cette appellation ne va pas causer d'ambiguïté car la notion de bijection partielle de la théorie des semi-groupes n'est pas du tout utilisée dans cette thèse.
\end{rem}

\begin{rem}Pour tout entier positif $n$, soit $R_n$ l'ensemble des bijections $f:d(f)\longrightarrow c(f)$ où $c(f),d(f)\subseteq [n].$ L'ensemble $R_n$ muni de la composition (définie sur le plus grand ensemble sur lequel la composition peut être faite) est un monoïde -- cela veut dire que la composition est associative et $R_n$ possède un élément neutre -- appelé le \textit{semi-groupe symétrique inverse}. Avec cette composition, l'ensemble des bijections partielles $Q_n$ forme un sous-monoïde de $R_{2n}.$ Il est connu, voir \cite{Solomon1990}, que $R_{2n}$ est en bijection avec le \textit{rook monoid} $\mathcal{R}_{2n}.$ Il est important de noter qu'on n'utilise pas cette structure sur $Q_n$ dans ce chapitre. Le produit utilisé ici est défini un peu plus loin dans ce chapitre.
\end{rem}

La taille de $Q_n$ est donnée par:
$$\displaystyle{| Q_n|=\sum_{k=0}^{n}\begin{pmatrix}
n\\
k
\end{pmatrix}^2 (2k)!}.$$
Il faut remarquer que $\mathcal{S}_{2n}$ est inclus dans $Q_n:$ en effet une permutation de $[2n]$ peut être vue comme la bijection partielle $(\sigma,[2n],[2n])$ de $[2n].$
\begin{notation}On va utiliser la lettre $\alpha$ pour représenter une bijection partielle. Pour $\alpha$ une bijection partielle donnée, on va noter $\sigma$ (resp. $d$, $d^{'}$) le premier (resp. deuxième, troisième) élément du triplet définissant $\alpha.$ Les mêmes conventions vont être utilisées pour $\widetilde{\alpha}$, $\alpha_i$, $\hat{\alpha}$ \ldots
\end{notation}
\begin{obs}
La définition d'une bijection partielle laisse la possibilité de définir le coset-type. À toute bijection partielle $\alpha$ on associe un graphe $\Gamma (\alpha)$ avec $|d|$ sommets placés sur un cercle. Les étiquettes extérieures (resp. intérieures) sont les éléments de l'ensemble $d$ (resp. $d^{'}$). L'étiquette intérieure d'un sommet est l'image par $\sigma$ de son étiquette extérieure. Puisque $d$ et $d^{'}$ sont dans $\mathbf{P}_n$, on peut joindre les étiquettes extérieures (resp. intérieures) $2i$ et $2i-1.$ Donc la définition du coset-type pour les permutations de $2n$ peut être étendue naturellement pour les bijections partielles de $n.$  On note par $ct(\alpha)$ ou bien $ct(\sigma)$ le coset-type de la bijection partielle $\alpha.$
\end{obs}
\begin{ex}\label{coset type partial bijection} Soit $\alpha=(\sigma,d,d')$ une bijection partielle de $[16]$ où $$d=\lbrace 3,4,5,6,9,10,11,12,13,14\rbrace, \,\, d'=\lbrace 1,2,3,4,7,8,9,10,15,16\rbrace$$ et $\sigma$ est donnée par l'écriture en deux lignes suivante :
$$\sigma=\begin{matrix}
3&4&5&6&9&10&11&12&13&14\\
9&16&1&15&10&2&4&8&3&7
\end{matrix},$$
ce qui signifie que $\sigma(3)=9$, $\sigma(4)=16$ et ainsi de suite. Le graphe $\Gamma(\alpha)$ est donné par la Figure~\ref{fig:coset-type partial bijection}.
\begin{figure}[htbp]
\begin{center}
\begin{tikzpicture}
\foreach \angle / \label in
{ 0/9, 36/16, 72/1, 108/15, 144/10, 180/2, 216/4,
252/8, 288/3, 324/7}
{
\node at (\angle:2cm) {\footnotesize $\bullet$};
\draw[red] (\angle:1.7cm) node{\textsf{\label}};
}
\foreach \angle / \label in
{ 0/3, 36/4, 72/5, 108/6, 144/9, 180/10, 216/11,
252/12, 288/13, 324/14}
{\draw (\angle:2.3cm) node{\textsf{\label}};}
\draw (0:2cm) .. controls + (5mm,0) and +(.5,-.1) .. (36:2cm);
\draw[red] (36:2cm) .. controls + (-.5,.1) and + (0,-.2) .. (108:2);
\draw (108:2cm) .. controls + (0,.2) and + (0,.2) .. (72:2);
\draw[red]	(72:2)		  .. controls + (0,-.2) and + (+.5,.1) .. (180:2);
\draw	(180:2)		  .. controls + (-5mm,0) and + (-0.5,0) .. (144:2);
\draw[red]	(144:2)		  .. controls + (1.2,0) and + (0,.2) .. (0:2);
\draw[red] (216:2cm) .. controls + (0,15mm) and +(-.5,-.1) .. (288:2cm);
\draw	(288:2cm)		    .. controls + (1,0.1) and + (0,-.2) .. (324:2);
\draw[red]	(324:2)		    .. controls + (-0.1,1.5) and + (1,0.1) .. (252:2);
\draw	(252:2)		    .. controls + (0,0) and + (0,-1) .. (216:2);
\end{tikzpicture}
\caption{Le graphe $\Gamma(\alpha)$.}
\label{fig:coset-type partial bijection}
\end{center}
\end{figure}
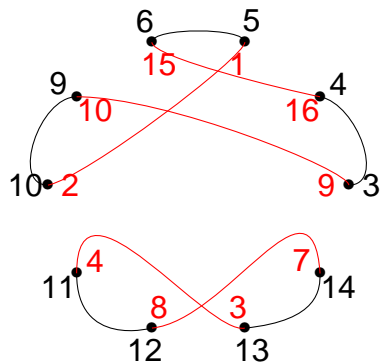
Il possède deux cycles de longueurs $6$ et $4$, donc $ct(\alpha)=(3,2).$
\end{ex}

\setcounter{subsection}{1}
\subsection{Problème d'extension} 

Nous commençons par rappeler le cas des permutations partielles du chapitre \ref{chapitre2}. Dans la Section \ref{sec:approche_Ivanov_Kerov}, on a besoin de considérer des prolongements de permutations partielles. Pour cela, si $(d,\omega)$ est une permutation partielle de $n,$ et si $d\subset \widetilde{d},$ l'extension de $(d,\omega)$ à $\widetilde{d}$ est la permutation partielle $(\widetilde{d},\widetilde{\omega})$ de $n$ où $$\widetilde{\omega}(a)=
\left\{
\begin{array}{ll}
  \omega(a) & \qquad \mathrm{si}\quad a\in d \\
  a & \qquad \mathrm{si}\quad a\in \widetilde{d}\setminus d. \\
 \end{array}
 \right.$$ 
Le type-cyclique de $(\widetilde{d},\widetilde{\omega})$ est l'union du type-cyclique de $(d,\omega)$ avec la partition $(1^{|\widetilde{d}\setminus d|}).$ Donc l'extension naturelle d'une permutation partielle se traduit en ajoutant des parts égales à $1$ au type-cyclique.

Dans le cas d'une bijection partielle, ajouter une part égale à $1$ au coset-type peut-être fait de plusieurs façons. Supposons que $\alpha$ est une bijection partielle de $[2n]$ et considérons les ensembles $\widetilde{d}=d\sqcup p(k)$ et $\widetilde{d'}=d'\sqcup p(k')$ pour deux entiers naturels $k$ et $k'.$ L'extension de $\alpha$ à $\widetilde{d}$ qui ajoute une part égale à $1$ au coset-type peut-être faite de deux manières. On peut envoyer $2k$ sur $2k'$ et $2k-1$ sur $2k'-1$ ou bien $2k$ sur $2k'-1$ et $2k-1$ sur $2k'$ et dans les deux cas on obtient une bijection partielle de coset-type égale à $ct(\alpha)\cup (1).$ On va prendre en considération toutes les extensions possibles d'une permutation partielle.
\begin{definition}\label{definition 3.2}
Soient $(\sigma,d,d^{'})$ et $(\widetilde{\sigma},\widetilde{d},\widetilde{d^{'}})$ deux bijections partielles de $n.$ On dit que $(\widetilde{\sigma},\widetilde{d},\widetilde{d^{'}})$ est une \textit{extension triviale} de $(\sigma,d,d^{'})$ si :
$$d\subseteq \widetilde{d},~ \widetilde{\sigma}_{|_{d}}=\sigma ~ \text{et} ~ ct(\widetilde{\sigma})=ct(\sigma)\cup \Big(1^\frac{{|\widetilde{d}\setminus d|}}{2}\Big).$$
On note par $P_\alpha(n)$ l'ensemble de toutes les extensions triviales de $\alpha$ dans $Q_n.$
\end{definition}
\begin{ex}
Soit $\alpha$ la bijection partielle de $[16]$ donnée dans l'Exemple \ref{coset type partial bijection}. Soit \linebreak $\widetilde{d}=d\cup \lbrace 1,2,15,16\rbrace$, $\widetilde{d^{'}}=d'\cup \lbrace 5,6,11,12\rbrace$ et considérons la bijection suivante : $$\widetilde{\sigma}=
\setcounter{MaxMatrixCols}{16}
\begin{matrix}
\bf{1}&\bf{2}&3&4&5&6&9&10&11&12&13&14&\bf{15}&\bf{16}\\
\bf{12}&\bf{11}&9&16&1&15&10&2&4&8&3&7&\bf{5}&\bf{6}
\end{matrix}.$$
Alors, $\widetilde{\alpha}=(\widetilde{\sigma},\widetilde{d},\widetilde{d^{'}})$ est une extension triviale de $\alpha.$ De la même manière $\hat{\alpha}=(\hat{\sigma},\hat{d},\hat{d^{'}})$, où $\hat{d}=\hat{d^{'}}=[16]$ et 
$$\hat{\sigma}=
\setcounter{MaxMatrixCols}{16}
\begin{matrix}
\bf{1}&\bf{2}&3&4&5&6&\bf{7}&\bf{8}&9&10&11&12&13&14&\bf{15}&\bf{16}\\
\bf{13}&\bf{14}&9&16&1&15&\bf{12}&\bf{11}&10&2&4&8&3&7&\bf{6}&\bf{5}
\end{matrix},$$
est encore une extension triviale de $\alpha.$
\end{ex}
\begin{lem}\label{Lemma 3.1}
Soit $\alpha$ une bijection partielle de $[2n]$ et $X$ un élément de $\mathbf{P}_n$ tel que $d\subseteq X$. Le nombre des extensions triviales $\widetilde{\alpha}$ de $\alpha$ tel que $\widetilde{d}=X$ est égale à :
$$(2n-|d|)\cdot (2n-|d|-2)\cdots (2n-|d|-|X\setminus d|+2)=2^{\frac{|X\setminus d|}{2}}\Big(n-\frac{|d|}{2}\Big)_{\frac{|X\setminus d|}{2}}.$$ 
On a la même formule pour le nombre des extensions triviales $\widetilde{\alpha}$ tel que $\widetilde{d'}=X.$
\end{lem}
\subsection{Produit des bijections partielles}
On note par $\mathcal{D}_n:=\mathbb{C}[Q_n]$\label{nomen:D_n} \nomenclature[33]{$\mathcal{D}_n$}{L'espace vectoriel de base $Q_n$ \quad \pageref{nomen:D_n}} l'espace vectoriel de base les éléments de $Q_n$. On va munir $\mathcal{D}_n$ d'une structure d'algèbre. Soient $\alpha_1$ et $\alpha_2$ deux bijections partielles. Si $d_1=d'_2$, on peut composer $\alpha_1$ et $\alpha_2$ et définir $\alpha_1\pr \alpha_2=\alpha_1\circ \alpha_2=(\sigma_1\circ \sigma_2,d_2,d'_1).$ Mais en général $d_1$ est différent de $d'_2.$ Dans une telle situation on va étendre $\alpha_1$ et $\alpha_2$ pour obtenir des bijections partielles $\widetilde{\alpha_1}$ et $\widetilde{\alpha_2}$ tel que $\widetilde{d_1}=\widetilde{d'_2}.$ Puisqu'il y  a plusieurs façons de le faire,  un choix naturel est de prendre la moyenne de la composition de toutes les extensions triviales possibles. Soit $E_{\alpha_1}^{\alpha_2}(n)$ l'ensemble suivant :
\begin{eqnarray*}
E_{\alpha_1}^{\alpha_2}(n)&:=&\lbrace(\widetilde{\alpha_1},\widetilde{\alpha_2})\in P_{\alpha_1}(n)\times P_{\alpha_2}(n) \text{ tel que }\widetilde{d_1}=\widetilde{d_2'}=d_1\cup d_2'\rbrace.
\end{eqnarray*}
Les éléments de $E_{\alpha_1}^{\alpha_2}(n)$ sont représentés schématiquement sur la Figure \ref{fig:composition}.
\begin{figure}[htbp]
\begin{center}
\begin{tikzpicture}[thick,fill opacity=0.5]
\draw (0,0) ellipse (.5cm and .8cm)++(2,0) ellipse (.5cm and .8cm);
\draw[->] (0.5,0) to node [sloped,above] {$\sigma_2$} (1.5,0);
\draw[->,dashed] (0.5,-1)[] to node [sloped,above] {$\widetilde{\sigma_2}$} (1.5,-1);
\draw (2,-1) ellipse (.5cm and .8cm);
\draw (4,-1) ellipse (.5cm and .8cm);
\draw[->] (2.5,-1) to node [sloped,above] {$\sigma_1$} (3.5,-1);
\draw[->, dashed] (2.5,0)[] to node [sloped,above] {$\widetilde{\sigma_1}$} (3.5,0);
\draw[dashed] (-0.4,-0.5) arc (140:400: 0.5cm and 0.8cm);
\draw[dashed] (4.4,-0.4) arc (-40:220: 0.5cm and .8cm);
\draw (0,0) node {$d_2$};
\draw (2,0) node {$d_2'$};
\draw (2,-1) node {$d_1$};
\draw (4,-1) node {$d_1'$};
\end{tikzpicture}
\caption{Représentation schématique des éléments de $E_{\alpha_1}^{\alpha_2}(n)$.}
\label{fig:composition}
\end{center}
\end{figure}
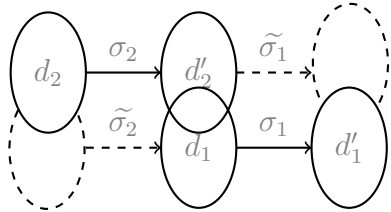 

\begin{notation} On va utiliser la convention suivante pour les figures représentant schématiquement un même ensemble.
\begin{enumerate}
\item[-] Les données qui définissent l'ensemble (fixées en passant d'un élément de l'ensemble à un autre) sont dessinées avec une forme en traits pleins.
\item[-] Les éléments de l'ensemble sont dessinés avec une forme en traits pointillés.
\end{enumerate}
\end{notation}
\begin{ex}\label{Ex E}
Considérons les deux bijections partielles $\alpha_1$ et $\alpha_2$ de $3$ :
$$\alpha_1=\begin{matrix}
1&2&5&6\\
3&2&1&4
\end{matrix}~~\text{ et }~~ \alpha_2=\begin{matrix}
3&4&5&6\\
5&6&3&4
\end{matrix}.$$
Alors, $E_{\alpha_1}^{\alpha_2}(3)$ est l'ensemble des quatre éléments suivants. 
$$\left(\begin{matrix}
1&2&\bf{3}&\bf{4}&5&6\\
3&2&\bf{5}&\bf{6}&1&4
\end{matrix}~~,~~\begin{matrix}
\bf{1}&\bf{2}&3&4&5&6\\
\bf{1}&\bf{2}&5&6&3&4
\end{matrix}\right),\left(\begin{matrix}
1&2&\bf{3}&\bf{4}&5&6\\
3&2&\bf{6}&\bf{5}&1&4
\end{matrix}~~,~~\begin{matrix}
\bf{1}&\bf{2}&3&4&5&6\\
\bf{1}&\bf{2}&5&6&3&4
\end{matrix}\right),$$
$$\left(\begin{matrix}
1&2&\bf{3}&\bf{4}&5&6\\
3&2&\bf{5}&\bf{6}&1&4
\end{matrix}~~,~~\begin{matrix}
\bf{1}&\bf{2}&3&4&5&6\\
\bf{2}&\bf{1}&5&6&3&4
\end{matrix}\right)\text{ et }\left(\begin{matrix}
1&2&\bf{3}&\bf{4}&5&6\\
3&2&\bf{6}&\bf{5}&1&4
\end{matrix}~~,~~\begin{matrix}
\bf{1}&\bf{2}&3&4&5&6\\
\bf{2}&\bf{1}&5&6&3&4
\end{matrix}\right).$$
\end{ex}
\begin{definition}
Soient $\alpha_1$ et $\alpha_2$ deux bijections partielles de $n.$ On définit le produit $\alpha_1\pr \alpha_2$ ainsi :
\begin{equation}\label{equation 3}
\alpha_1\pr \alpha_2:=\frac{1}{| E_{\alpha_1}^{\alpha_2}(n)|}\sum_{(\widetilde{\alpha_1},\widetilde{\alpha_2}) \in E_{\alpha_1}^{\alpha_2}(n)} \widetilde{\alpha_1}\circ \widetilde{\alpha_2}.
\end{equation}
\end{definition}
D'après le Lemme \ref{Lemma 3.1}, on a: 
\begin{equation}\label{cardinal de E}
| E_{\alpha_1}^{\alpha_2}(n)|=2^{\frac{| d'_2\setminus d_1|}{2}+\frac{| d_1\setminus d'_2|}{2}}\cdot \left(n-\frac{| d'_1|}{2}\right)_{\left(\frac{| d'_2\setminus d_1|}{2}\right)}\cdot \left(n-\frac{| d_2|}{2}\right)_{\left(\frac{| d_1\setminus d'_2|}{2}\right)}.
\end{equation}
\begin{prop}\label{prop:assoc}
Le produit $\pr$ est associatif ce qui implique que $\mathcal{D}_n$ est une algèbre (sans élément neutre).
\end{prop}
\begin{proof}
La section \ref{associativite} est dédiée à la démonstration de cette proposition.
\end{proof}
On va illuster l'associativité dans l'exemple suivant. 
\begin{ex}
Soient $\alpha_1$ et $\alpha_2$ les deux bijections partielles de $6$ données dans l'Exemple \ref{Ex E}. En utilisant l'ensemble $E_{\alpha_1}^{\alpha_2}(3)$, on a :
$$\alpha_1\pr \alpha_2=\frac{1}{4}\left[ \begin{matrix}
1&2&3&4&5&6\\
3&2&1&4&5&6
\end{matrix}~+~\begin{matrix}
1&2&3&4&5&6\\
3&2&1&4&6&5
\end{matrix}~+~\begin{matrix}
1&2&3&4&5&6\\
2&3&1&4&5&6
\end{matrix}~+~\begin{matrix}
1&2&3&4&5&6\\
2&3&1&4&6&5
\end{matrix}\right].$$
Considérons la bijection partielle $\alpha_3$ de $[6]$ définie par : 
$$\alpha_3=\begin{matrix}
1&2&3&4\\
5&1&6&2
\end{matrix}.$$
De la même manière on peut vérifier que :
$$\alpha_2\pr \alpha_3=\frac{1}{4}\left[ \begin{matrix}
1&2&3&4&5&6\\
3&1&4&2&5&6
\end{matrix}~+~\begin{matrix}
1&2&3&4&5&6\\
3&1&4&2&6&5
\end{matrix}~+~\begin{matrix}
1&2&3&4&5&6\\
3&2&4&1&5&6
\end{matrix}~+~\begin{matrix}
1&2&3&4&5&6\\
3&2&4&1&6&5
\end{matrix}\right],$$
et que :
$$\alpha_1\pr (\alpha_2\pr \alpha_3)=(\alpha_1\pr \alpha_2)\pr \alpha_3=\frac{1}{8}\left[ \begin{matrix}
1&2&3&4&5&6\\
5&3&6&2&1&4
\end{matrix}~+~\begin{matrix}
1&2&3&4&5&6\\
5&3&6&2&4&1
\end{matrix}~+~\begin{matrix}
1&2&3&4&5&6\\
6&3&5&2&1&4
\end{matrix}\right]$$
$$+\frac{1}{8}\left[\begin{matrix}
1&2&3&4&5&6\\
6&3&5&2&4&1
\end{matrix}~+~\begin{matrix}
1&2&3&4&5&6\\
5&2&6&3&1&4
\end{matrix}~+~\begin{matrix}
1&2&3&4&5&6\\
5&2&6&3&4&1
\end{matrix}\right]$$
$$+\frac{1}{8}\left[\begin{matrix}
1&2&3&4&5&6\\
6&2&5&3&1&4
\end{matrix}~+~\begin{matrix}
1&2&3&4&5&6\\
6&2&5&3&4&1
\end{matrix}\right].$$
\end{ex}

\subsection{Associativité du produit des bijections partielles}\label{associativite}
Le but de cette section est de montrer la Proposition \ref{prop:assoc}. Soient $\alpha_1,\alpha_2$ et $\alpha_3$ trois bijections partielles de $n.$ Par définition du produit on a :
\begin{equation*}
(\alpha_1\pr \alpha_2)\pr \alpha_3=
\frac{1}{| E_{\alpha_1}^{\alpha_2}(n)|}\sum_{(\widetilde{\alpha_1},\widetilde{\alpha_2})\in E_{\alpha_1}^{\alpha_2}(n)}\frac{1}{| E_{\widetilde{\alpha_1}\circ \widetilde{\alpha_2}}^{\alpha_3}(n)|}\sum_{(\widetilde{\widetilde{\alpha_1}\circ \widetilde{\alpha_2}},\widetilde{\alpha_3})\in E_{\widetilde{\alpha_1}\circ \widetilde{\alpha_2}}^{\alpha_3}(n)}(\widetilde{\widetilde{\sigma_1}\circ \widetilde{\sigma_2}}\circ \widetilde{\sigma_3}, \widetilde{d_3},\widetilde{\widetilde{d^{'}_1}}),
\end{equation*}
\begin{equation*} \alpha_1\pr (\alpha_2\pr \alpha_3)=
\frac{1}{| E_{\alpha_2}^{\alpha_3}(n)|}\sum_{(\widetilde{\alpha_2},\widetilde{\alpha_3})\in E_{\alpha_2}^{\alpha_3}(n)}\frac{1}{| E_{\alpha_1}^{\widetilde{\alpha_2}\circ \widetilde{\alpha_3}}(n)|}\sum_{(\widetilde{\alpha_1},\widetilde{\widetilde{\alpha_2}\circ \widetilde{\alpha_3}})\in E_{\alpha_1}^{\widetilde{\alpha_2}\circ \widetilde{\alpha_3}}(n)}(\widetilde{\sigma_1}\circ \widetilde{\widetilde{\sigma_2}\circ \widetilde{\sigma_3}}, \widetilde{\widetilde{d_3}},\widetilde{d^{'}_1}).
\end{equation*}
On considère les ensembles suivants qui indexent les sommes des équations au dessus :
$$X_1=\lbrace \big((\widetilde{\alpha_1},\widetilde{\alpha_2}),(\widetilde{\widetilde{\alpha_1}\circ \widetilde{\alpha_2}},\widetilde{\alpha_3})\big) \text{ tel que }(\widetilde{\alpha_1},\widetilde{\alpha_2})\in E_{\alpha_1}^{\alpha_2}(n)\text{ et } (\widetilde{\widetilde{\alpha_1}\circ \widetilde{\alpha_2}},\widetilde{\alpha_3})\in E_{\widetilde{\alpha_1}\circ \widetilde{\alpha_2}}^{\alpha_3}(n)\rbrace,$$
et
$$X_2=\lbrace \big((\widetilde{\alpha_2},\widetilde{\alpha_3}),(\widetilde{\alpha_1},\widetilde{\widetilde{\alpha_2}\circ \widetilde{\alpha_3}})\big)\text{ tel que } (\widetilde{\alpha_2},\widetilde{\alpha_3})\in E_{\alpha_2}^{\alpha_3}(n)\text{ et } (\widetilde{\alpha_1},\widetilde{\widetilde{\alpha_2}\circ \widetilde{\alpha_3}})\in E_{\alpha_1}^{\widetilde{\alpha_2}\circ \widetilde{\alpha_3}}(n)\rbrace.$$
Avec la notation $X_1$, le produit $(\alpha_1\pr \alpha_2)\pr \alpha_3$ peut s'écrire ainsi :
\begin{equation}\label{produit 1}
(\alpha_1\pr \alpha_2)\pr \alpha_3=
\sum_{\big((\widetilde{\alpha_1},\widetilde{\alpha_2}),(\widetilde{\widetilde{\alpha_1}\circ \widetilde{\alpha_2}},\widetilde{\alpha_3})\big)\in X_1}\frac{1}{| E_{\alpha_1}^{\alpha_2}(n)|\cdot| E_{\widetilde{\alpha_1}\circ \widetilde{\alpha_2}}^{\alpha_3}(n)|}(\widetilde{\widetilde{\sigma_1}\circ \widetilde{\sigma_2}}\circ \widetilde{\sigma_3}, \widetilde{d_3},\widetilde{\widetilde{d^{'}_1}}).
\end{equation}
Les éléments de l'ensemble $X_1$ sont représentés sur la Figure \ref{fig:associativity 1}.
\begin{figure}[htbp]
\begin{center}
\begin{tikzpicture}[thick,fill opacity=0.5]
\draw (1,0) ellipse (.7cm and 1cm) ++ (3,0) ellipse (.7cm and 1cm);
\draw (4,-1.5) ellipse (.7cm and 1cm)  ++(3,0) ellipse (.7cm and 1cm);
\draw (7,-3) ellipse (.7cm and 1cm)  ++(3,0) ellipse (.7cm and 1cm);
\draw[dashed] (3.6,-2.3) arc (115:415: 0.72cm and 0.9cm);
\draw[dashed] (10.4,-2.1) arc (310:600: 0.72cm and 0.9cm);
\draw[->] (1.7,0) to node [sloped,above] {$\sigma_3$} (3.3,0);
\draw[->] (4.7,-1.5) to node [sloped,above] {$\sigma_2$} (6.3,-1.5);
\draw[->] (7.7,-3) to node [sloped,above] {$\sigma_1$} (9.2,-3);
\draw[->,dashed] (1.7,-2.2) to node [sloped,above] {$\widetilde{\sigma_3}$} (3.3,-2.2);
\draw[->,dashed] (4.7,-3) to node [sloped,above] {$\widetilde{\sigma_2}$} (6.3,-3);
\draw[->,dashed] (7.7,-1.5) to node [sloped,above] {$\widetilde{\sigma_1}$} (9.2,-1.5);
\draw[snake=brace] (11,-0.5) -- (11,-4) node[midway,right] {$\widetilde{d_1^{'}}$};
\draw[->,dashed] (4.7,0) to node [sloped,above] {$\widetilde{\widetilde{\sigma_1}\circ \widetilde{\sigma_2}}$} (9.2,0);
\draw[snake=brace, mirror snake] (3.2,-4.2) -- (4.5,-4.2) node[midway,below] {$\widetilde{d'_3}$};
\draw[dashed] (0.6,-0.75) arc (120:415: 0.72cm and 1.75cm);
\draw[snake=brace, mirror snake] (-0.2,1) -- (-0.2,-4) node[midway,left] {$\widetilde{d_3}$};
\draw[dashed] (10.4,-0.7) arc (310:600: 0.72cm and 0.9cm);
\draw[snake=brace] (12,1) -- (12,-4) node[midway,right] {$\widetilde{\widetilde{d_1^{'}}}$};
\draw (1,0) node  {$d_3$};
\draw (4,0) node  {$d_3^{'}$};
\draw (4,-1.5) node  {$d_2$};
\draw (7,-1.5) node  {$d_2^{'}$};
\draw (7,-3) node  {$d_1$};
\draw (10,-3) node  {$d_1^{'}$};
\end{tikzpicture}
\caption{Représentations des éléments de l'ensemble $X_1$.}
\label{fig:associativity 1}
\end{center}
\end{figure}\\
De la même manière, avec la notation $X_2$, le produit $\alpha_1\pr (\alpha_2\pr \alpha_3)$ peut être écrit ainsi :
\begin{equation}\label{produit 2}
\alpha_1\pr (\alpha_2\pr \alpha_3)=
\sum_{\big((\widetilde{\alpha_2},\widetilde{\alpha_3}),(\widetilde{\alpha_1},\widetilde{\widetilde{\alpha_2}\circ \widetilde{\alpha_3}})\big) \in X_2}\frac{1}{| E_{\alpha_2}^{\alpha_3}(n)|\cdot| E_{\alpha_1}^{\widetilde{\alpha_2}\circ \widetilde{\alpha_3}}(n)|}(\widetilde{\sigma_1}\circ \widetilde{\widetilde{\sigma_2}\circ \widetilde{\sigma_3}}, \widetilde{\widetilde{d_3}},\widetilde{d^{'}_1}).
\end{equation}
Les éléments dans $X_2$ sont représentés sur la Figure \ref{fig:associativity 2}.
\begin{figure}[htbp]
\begin{center}
\begin{tikzpicture}[thick,fill opacity=0.5]
\draw (1,0) ellipse (.7cm and 1cm) ++ (3,0) ellipse (.7cm and 1cm);
\draw (4,-1.5) ellipse (.7cm and 1cm)  ++(3,0) ellipse (.7cm and 1cm);
\draw (7,-3) ellipse (.7cm and 1cm)  ++(3,0) ellipse (.7cm and 1cm);
\draw[dashed] (0.6,-0.8) arc (120:415: 0.70cm and 0.9cm);
\draw[dashed] (7.4,-0.7) arc (310:600: 0.72cm and 0.9cm);
\draw[->] (1.7,0) to node [sloped,above] {$\sigma_3$} (3.3,0);
\draw[->] (4.7,-1.5) to node [sloped,above] {$\sigma_2$} (6.3,-1.5);
\draw[->] (7.7,-3) to node [sloped,above] {$\sigma_1$} (9.3,-3);
\draw[->,dashed] (1.7,-1.5) to node [sloped,above] {$\widetilde{\sigma_3}$} (3.3,-1.5);
\draw[->,dashed] (4.7,0) to node [sloped,above] {$\widetilde{\sigma_2}$} (6.2,0);
\draw[->,dashed] (7.7,-0.8) to node [sloped,above] {$\widetilde{\sigma_1}$} (9.2,-0.8);
\draw[->,dashed] (1.7,-3) to node [sloped,below] {$\widetilde{\widetilde{\sigma_2}\circ\widetilde{\sigma_3}}$} (6.2,-3);
\draw[snake=brace] (11,1) -- (11,-4) node[midway,right] {$\widetilde{d_1^{'}}$};
\draw[snake=brace, snake] (6.2,1) -- (7.6,1) node[midway,above] {$\widetilde{d_1}$};
\draw[dashed] (0.6,-2.3) arc (115:415: 0.70cm and 0.9cm);
\draw[snake=brace, mirror snake] (0.1,1) -- (0.1,-2.5) node[midway,left] {$\widetilde{d_3}$};
\draw[snake=brace, mirror snake] (-0.5,1) -- (-0.5,-4) node[midway,left] {$\widetilde{\widetilde{d_3}}$};
\draw[dashed] (10.35,-2.1) arc (310:600: 0.72cm and 1.8cm);
\draw (1,0) node  {$d_3$};
\draw (4,0) node  {$d_3^{'}$};
\draw (4,-1.5) node  {$d_2$};
\draw (7,-1.5) node  {$d_2^{'}$};
\draw (7,-3) node  {$d_1$};
\draw (10,-3) node  {$d_1^{'}$};
\end{tikzpicture}
\caption{Représentation des éléments de l'ensemble $X_2$.}
\label{fig:associativity 2}
\end{center}
\end{figure}

Pour démontrer l'associativité du produit, on va construire un ensemble $X$ et deux fonctions surjectives $\phi_1:X\longrightarrow X_1$ et $\phi_2:X\longrightarrow X_2$ pour écrire les deux sommes des equations \eqref{produit 1} et \eqref{produit 2} comme sommes sur les éléments du même ensemble $X$. Soit $X$ l'ensemble des éléments
$$\big(\epsilon_1=(\tau_1,\delta_0,\delta_1),\epsilon_2=(\tau_2,\delta_1,\delta_2),\epsilon_3=(\tau_3,\delta_2,\delta_3)\big)\in P_{\alpha_1}(n)\times P_{\alpha_2}(n)\times P_{\alpha_3}(n),$$ qui vérifient les propriétés suivantes :
\begin{enumerate}[label=\roman*)]
\item $d_3^{'}\cup d_2\subseteq \delta_2$.
\item $d_2^{'}\cup d_1\subseteq \delta_1$.
\item \label{delta2} $\delta_{2}=d'_{3}\cup \tau_{2}^{-1}(d_1\cup d'_2)$.
\end{enumerate}
Les éléments de cet ensemble sont représentés par la Figure \ref{fig:associativity 3}. Il faut noter que \ref{delta2} est une condition de minimalité. On va voir dans la preuve du Lemme \ref{surjection de phi} ci-dessous pourquoi elle est utile.
\begin{figure}[htbp]
\begin{center}
\begin{tikzpicture}[thick,fill opacity=0.5]
\draw (1,0) ellipse (.7cm and 1cm) ++ (3,0) ellipse (.7cm and 1cm);
\draw (4,-1.5) ellipse (.7cm and 1cm)  ++(3,0) ellipse (.7cm and 1cm);
\draw (7,-3) ellipse (.7cm and 1cm)  ++(3,0) ellipse (.7cm and 1cm);
\draw[dashed] (7.4,-0.7) arc (310:600: 0.72cm and 0.9cm);
\draw[dashed] (3.6,-2.3) arc (115:415: 0.72cm and 0.9cm);
\draw[dashed] (0.6,-0.75) arc (120:415: 0.72cm and 1.75cm);
\draw[dashed] (10.35,-2.1) arc (310:600: 0.72cm and 1.8cm);
\draw[->] (1.7,0) to node [sloped,above] {$\tau_{3{_|{_{d_3}}}}$$=$$\sigma_3$} (3.3,0);
\draw[->] (4.7,-1.5) to node [sloped,above] {$\tau_{2{_|{_{d_2}}}}$$=$$\sigma_2$} (6.3,-1.5);
\draw[->] (7.7,-3) to node [sloped,above] {$\tau_{1{_|{_{d_1}}}}$$=$$\sigma_1$} (9.3,-3);
\draw[snake=brace, mirror snake] (9.3,-4.2) -- (10.7,-4.2) node[midway,below] {$\delta_0$};
\draw[snake=brace, mirror snake] (0.3,-4.2) -- (1.7,-4.2) node[midway,below] {$\delta_3$};
\draw[snake=brace, mirror snake] (3.3,-4.2) -- (4.7,-4.2) node[midway,below] {$\delta_2$};
\draw[snake=brace, mirror snake] (6.3,-4.2) -- (7.7,-4.2) node[midway,below] {$\delta_1$};
\draw[->, dashed] (1.5,-4.4) to node [sloped,below] {$\tau_{3}$} (3.5,-4.4);
\draw[->, dashed] (4.5,-4.4) to node [sloped,below] {$\tau_{2}$} (6.5,-4.4);
\draw[->, dashed] (7.5,-4.4) to node [sloped,below] {$\tau_{1}$} (9.5,-4.4);
\draw (1,0) node  {$d_3$};
\draw (4,0) node  {$d'_3$};
\draw (4,-1.5) node  {$d_2$};
\draw (7,-1.5) node  {$d'_2$};
\draw (7,-3) node  {$d_1$};
\draw (10,-3) node  {$d'_1$};
\end{tikzpicture}
\caption{Représentations des éléments de l'ensemble $X$.}
\label{fig:associativity 3}
\end{center}
\end{figure}

On définit les applications $\phi_1:X\longrightarrow X_1$ et $\phi_2:X\longrightarrow X_2$ ainsi : 
\begin{eqnarray*}
\phi_1(\epsilon_1,\epsilon_2,\epsilon_3)=
\big((\epsilon_{1_{|_{d_1\cup d'_2}}},\epsilon_{2_{|_{\tau_2^{-1}(d_1\cup d'_2)}}}),(\epsilon_1\circ \epsilon_2,\epsilon_3)\big),
\end{eqnarray*}
\begin{eqnarray*}
\phi_2(\epsilon_1,\epsilon_2,\epsilon_3)=
\big((\epsilon_{2_{|_{d_2\cup d'_3}}},\epsilon_{3_{|_{\tau_3^{-1}(d_2\cup d'_3)}}}),(\epsilon_2\circ \epsilon_3,\epsilon_1)\big).
\end{eqnarray*}
On peut dire que l'application $\phi_1$ (resp. $\phi_2$) "oublie" la forme pointillée en haut (resp. bas) de la troisième (resp. deuxième) colonne de la Figure \ref{fig:associativity 3}. On note par 
$$\text{$2a=|d_3|=|d'_3|, 2b=|d_2|=|d'_2|, 2c=|d_1|=|d'_1|, 2d=|d^{'}_3\cap d_2|$ and $2e=|d^{'}_2\cap d_1|$}.$$ 
On va démontrer le lemme suivant :
\begin{lem}\label{surjection de phi}
L'application $\phi_1$ est bien définie et surjective. Pour un élément \linebreak $\big((\widetilde{\alpha_1},\widetilde{\alpha_2}),(\widetilde{\widetilde{\alpha_1}\circ \widetilde{\alpha_2}},\widetilde{\alpha_3})\big)\in X_1$, on a :
$$|\phi_1^{-1}\Big(\big((\widetilde{\alpha_1},\widetilde{\alpha_2}),(\widetilde{\widetilde{\alpha_1}\circ \widetilde{\alpha_2}},\widetilde{\alpha_3})\big)\Big)|=
2^{a-d-f}\cdot(n-b-c+e)_{(a-d-f)},$$
où $2f=|(\widetilde{d_2}\setminus d_2)\cap d^{'}_3|$.\\
Il faut noter que $f$, contrairement à $a,b,c,d,e$, dépend de l'élément de $X_1$ qu'on considère.
\end{lem}
\begin{proof}
Pour tout élément $(\epsilon_1,\epsilon_2,\epsilon_3)\in X$, on va vérifier que $\phi_1\big((\epsilon_1,\epsilon_2,\epsilon_3)\big)$ est dans $X_1$. 
Soit $$(\widetilde{\alpha_1},\widetilde{\alpha_2})=(\epsilon_{1_{|_{d_1\cup d'_2}}},\epsilon_{2_{|_{\tau_2^{-1}(d_1\cup d'_2)}}}).$$

Premièrement, $(\widetilde{\alpha_1},\widetilde{\alpha_2})$ est dans $P_{\alpha_1}(n)\times P_{\alpha_2}(n)$ car $(\epsilon_1,\epsilon_2)$ est dans $P_{\alpha_1}(n)\times P_{\alpha_2}(n)$. Par ailleurs, l'ensemble d'arrivée de $\widetilde{\alpha_2}$ et l'ensemble de départ de $\widetilde{\alpha_1}$ sont $d_1\cup d'_2$. Donc, $(\widetilde{\alpha_1},\widetilde{\alpha_2})$ est dans $E_{\alpha_1}^{\alpha_2}(n)$. 

Deuxièmement, il n'est pas difficile de vérifier que $(\epsilon_1\circ \epsilon_2,\epsilon_3)$ est dans $P_{\widetilde{\alpha_1}\circ \widetilde{\alpha_2}}(n)\times P_{\alpha_3}(n)$. Pour être dans $E_{\widetilde{\alpha_1}\circ \widetilde{\alpha_2}}^{\alpha_3}(n)$, il faut que $(\epsilon_1\circ \epsilon_2,\epsilon_3)$ vérifie la condition que $\delta_{2}=d'_{3}\cup \tau_{2}^{-1}(d_1\cup d'_2)$, ce qui est donnée par \ref{delta2}. 

Donc, $\phi_1$ est bien définie.\medskip

Maintenant, fixons $\big((\widetilde{\alpha_1},\widetilde{\alpha_2}),(\widetilde{\widetilde{\alpha_1}\circ \widetilde{\alpha_2}},\widetilde{\alpha_3})\big)\in X_1$.
On va compter le nombre de ses pré-images par $\phi_1.$
Pour construire un élément de $\phi_1^{-1}\Big(\big((\widetilde{\alpha_1},\widetilde{\alpha_2}),(\widetilde{\widetilde{\alpha_1}\circ \widetilde{\alpha_2}},\widetilde{\alpha_3})\big)\Big)$, il faut juste définir $\epsilon_1, \epsilon_2$ et $\delta_1$ puisque les autres éléments $\epsilon_3,\delta_0,\delta_2$ et $\delta_3$ sont déterminés par $\widetilde{\alpha_1}, \widetilde{\alpha_2}, \widetilde{\alpha_3}$ et $\widetilde{\widetilde{\alpha_1}\circ \widetilde{\alpha_2}}.$ D'abord, pour construire $\delta_1$, il faut étendre $d_1\cup d^{'}_2$ tout en ajoutant des paires de la forme $p(k)$ pour obtenir un ensemble qui a la même taille que $\delta_2=\widetilde{d_2}\cup d_3^{'}$ qui est $2a+2b+2c-2e-2d-2f.$ On a $|d_1\cup d^{'}_2|=2b+2c-2e$, donc le nombre des façons possibles pour étendre $d_1\cup d^{'}_2$ est le nombre de choix de $2a+2b+2c-2e-2d-2f-(2b+2c-2e)=2a-2d-2f$ éléments parmi $2n-(2b+2c-2e).$ Comme notre choix doit respecter la condition que $\delta_1$ est dans $\mathbf{P}_n$, ce nombre est :
$$\begin{pmatrix}
n-(b+c-e)\\
a-d-f
\end{pmatrix}.$$ 
Une fois l'ensemble $\delta_1$ déterminé, on doit étendre $\widetilde{\sigma_1}$ à $\tau_1$ (on a l'ensemble de départ $\delta_0$ et l'ensemble d'arrivée $\delta_1$) tout en envoyant les paires de la forme $p(k)$ sur des paires de la même forme. Le nombre de façons de faire est :
$$2^{a-d-f}\cdot(a-d-f)!.$$ 
Après l'extension de $\widetilde{\sigma_1}$ à $\tau_1$, on aura directement $\tau_2$ parce que $\tau_1\circ\tau_2=\widetilde{\widetilde{\sigma_1}\circ \widetilde{\sigma_2}}$ est donné. Donc, le cardinal de l'ensemble $\phi_1^{-1}\Big(\big((\widetilde{\alpha_1},\widetilde{\alpha_2}),(\widetilde{\widetilde{\alpha_1}\circ \widetilde{\alpha_2}},\widetilde{\alpha_3})\big)\Big)$ est égal à :\\
$$\begin{pmatrix}
n-(b+c-e)\\
a-d-f
\end{pmatrix}\cdot2^{a-d-f}\cdot(a-d-f)!=
2^{a-d-f}\cdot(n-b-c+e)_{(a-d-f)}.
$$
\end{proof}
De la même manière on peut démontrer le lemme suivant.
\begin{lem}
L'application $\phi_2$ est bien définie et surjective. Pour un élément \linebreak $\big((\widetilde{\alpha_2},\widetilde{\alpha_3}),(\widetilde{\alpha_1},\widetilde{\widetilde{\alpha_2}\circ \widetilde{\alpha_3}})\big)\in X_2$, on a :
$$|\phi_2^{-1}\Big(\big((\widetilde{\alpha_2},\widetilde{\alpha_3}),(\widetilde{\alpha_1},\widetilde{\widetilde{\alpha_2}\circ \widetilde{\alpha_3}})\big)\Big)|=
2^{c-e-g}\cdot(n-a-b+d)_{(c-e-g)},$$
où $2g=|(\widetilde{d^{'}_{2}}\setminus d^{'}_{2})\cap d_{1}|$.
\end{lem}
On peut encore vérifier que :
 $$| E_{\alpha_1}^{\alpha_2}(n)|=2^{b+c-2e}\cdot(n-c)_{(b-e)}\cdot(n-b)_{(c-e)},$$
 $$| E_{\alpha_2}^{\alpha_3}(n)|=2^{a+b-2d}\cdot(n-b)_{(a-d)}\cdot(n-a)_{(b-d)},$$
et
$$| E_{\widetilde{\alpha_1}\circ \widetilde{\alpha_2}}^{\alpha_3}(n)|=2^{a+b+c-2d-e-2f}\cdot(n-b-c+e)_{(a-d-f)}\cdot(n-a)_{(b+c-d-e-f)},$$
$$| E_{\alpha_1}^{\widetilde{\alpha_2}\circ \widetilde{\alpha_3}}(n)|=2^{a+b+c-d-2e-2g}\cdot(n-a-b+d)_{(c-e-g)}\cdot(n-c)_{(a+b-d-e-g)}.$$
Les produits $(\alpha_1\pr \alpha_2)\pr \alpha_3$ et $\alpha_1\pr (\alpha_2\pr \alpha_3)$ donnés par les équations \eqref{produit 1} et \eqref{produit 2} comme des sommes sur les éléments de $X_1$ et $X_2$ peuvent être écrits comme des sommes sur les éléments de l'ensemble $X$ ainsi :
\begin{eqnarray*}
(\alpha_1\pr \alpha_2)\pr \alpha_3 =
\sum_{(\epsilon_1,\epsilon_2,\epsilon_3)\in X}&&\frac{1}{2^{2a+2b+2c-3d-3e-3f}}\cdot ((n-b-c+e)_{(a-d-f)})^2\\
&&\cdot(n-a)_{(b+c-d-e-f)}\cdot(n-c)_{(b-e)}\cdot(n-b)_{(c-e)}\\
&&\epsilon_1\circ\epsilon_2\circ\epsilon_3,
\end{eqnarray*} 
et
\begin{eqnarray}\label{second product}
\alpha_1\pr (\alpha_2\pr\alpha_3)=
\sum_{(\epsilon_1,\epsilon_2,\epsilon_3)\in X}&&\frac{1}{2^{2a+2b+2c-3d-3e-3g}}\cdot ((n-a-b+d)_{(c-e-g)})^2\\
&&\cdot (n-c)_{(a+b-d-e-g)}\cdot(n-a)_{(b-d)}\cdot (n-b)_{(a-d)} \nonumber \\
&&\epsilon_1\circ\epsilon_2\circ\epsilon_3. \nonumber
\end{eqnarray} 
Pour tout entier naturel strictement positif $n$, on a les identités suivantes:
\begin{eqnarray*}
(n-b-c+e)_{(a-d-f)}.(n-c)_{(b-e)}&=&(n-c)_{(a+b-d-e-f)}\\
(n-a)_{(b-d)}.(n-a-b+d)_{(c-e-f)}&=&(n-a)_{(b+c-d-e-f)}\\
(n-b)_{(c-e)}.(n-b-c+e)_{(a-d-f)}&=&(n-b)_{(a-d)}.(n-a-b+d)_{(c-e-f)}.
\end{eqnarray*}
Donc, le produit $(\alpha_1\pr \alpha_2)\pr \alpha_3$ peut être écrit ainsi :
\begin{eqnarray}\label{first product}
(\alpha_1\pr \alpha_2)\pr \alpha_3 =
\sum_{(\epsilon_1,\epsilon_2,\epsilon_3)\in X}&&\frac{1}{2^{2a+2b+2c-3d-3e-3f}}\cdot ((n-a-b+d)_{(c-e-f)})^2\\
&&\cdot(n-c)_{(a+b-d-e-f)}\cdot(n-a)_{(b-d)}\cdot (n-b)_{(a-d)} \nonumber \\
&&\epsilon_1\circ\epsilon_2\circ\epsilon_3. \nonumber
\end{eqnarray}
Pour tout élément de $X$, l'égalité $|\delta_2|=|\delta_1|$ implique $2c+2b-2e+2a-2d-2f=2a+2b-2d+2c-2e-2g$, donc on a $f=g$. En comparant \eqref{second product} et \eqref{first product}, on voit que les produits $(\alpha_1\pr \alpha_2)\pr \alpha_3$ et $\alpha_1\pr (\alpha_2\pr \alpha_3)$ sont égaux, donc on a l'associativité.

\section{L'algèbre des invariants $\mathcal{A'}_n$}

\subsection{Action de $\mathcal{B}_n\times \mathcal{B}_n$ sur l'algèbre des bijections partielles}

Dans cette partie, on montre que le groupe $\mathcal{B}_n\times \mathcal{B}_n$ agit sur l'ensemble des bijections partielles $Q_n$ et on considère $\mathcal{A}'_n$ l'algèbre des invariants par l'action de $\mathcal{B}_n\times \mathcal{B}_n$ étendue linéairement à $\mathcal{D}_n.$
\begin{prop}
Le groupe $\mathcal{B}_n\times \mathcal{B}_n$ agit sur $Q_n$ par :
$$(a,b)\bullet(\sigma,d,d')=(a\sigma b^{-1},b(d),a(d')),$$
pour tout $(a,b)\in \mathcal{B}_n\times \mathcal{B}_n$ et $(\sigma,d,d') \in Q_n.$
\end{prop}
\begin{proof}
Soit $(\sigma,d,d') \in Q_n.$ On a : $1_{\mathcal{B}_n\times \mathcal{B}_n}\bullet(\sigma,d,d') = (\sigma,d,d').$ Il reste à démontrer que pour tout $(a,b),$ $(c,d)\in \mathcal{B}_n\times \mathcal{B}_n$ et pour tout $(\sigma,d,d') \in Q_n,$ on a : $$\left((a,b)\cdot (c,d)\right) \bullet(\sigma,d,d')=(a,b)\bullet \left( (c,d)\bullet(\sigma,d,d') \right),$$ ce qui est facile à vérifier.
\end{proof}
\begin{obs}
Deux bijections partielles de $n$ sont dans la même orbite par cette action si et seulement elles possèdent le même coset-type.
\end{obs}
On peut étendre cette action linéairement pour obtenir une action de $\mathcal{B}_n\times \mathcal{B}_n$ sur $\mathcal{D}_n$.
\begin{lem}\label{compatibility of E}
Soient $a,b$ et $c$ trois permutations de $\mathcal{B}_n$ et soient $\alpha_1$ et $\alpha_2$ deux bijections partielles de $n.$ Alors l'ensemble $E_{\alpha_1}^{\alpha_2}(n)$ est en bijection avec $E_{(a,b)\bullet \alpha_1}^{(b,c)\bullet \alpha_2}(n).$
\end{lem}
\begin{proof}
Il n'est pas difficile de vérifier que les deux applications :
$$\begin{array}{ccccc}
\Theta & : & E_{\alpha_1}^{\alpha_2}(n) & \to & E_{(a,b)\bullet \alpha_1}^{(b,c)\bullet \alpha_2}(n) \\
& & (\widetilde{\alpha_1},\widetilde{\alpha_2}) & \mapsto & ((a,b)\bullet \widetilde{\alpha_1},(b,c)\bullet \widetilde{\alpha_2}) \\
\end{array},$$
et
$$\begin{array}{ccccc}
\Psi & : & E_{(a,b)\bullet \alpha_1}^{(b,c)\bullet \alpha_2}(n) & \to & E_{\alpha_1}^{\alpha_2}(n) \\
& & (\beta_1,\beta_2) & \mapsto & ((a^{-1},b^{-1})\bullet \beta_1,(b^{-1},c^{-1})\bullet \beta_2) \\
\end{array},$$
sont bien définies. On va démontrer qu'elles sont inverses l'une de l'autre.
\begin{eqnarray*}
\Psi\Big(\Theta\big((\widetilde{\alpha_1},\widetilde{\alpha_2})\big)\Big)&=&\Psi\Big(\big((a,b)\bullet \widetilde{\alpha_1},(b,c)\bullet \widetilde{\alpha_2}\big)\Big)\\
&=&\Big((a^{-1},b^{-1})\bullet (a,b)\bullet \widetilde{\alpha_1},(b^{-1},c^{-1})\bullet (b,c)\bullet \widetilde{\alpha_2}\Big)\\
&=&(\widetilde{\alpha_1},\widetilde{\alpha_2}),
\end{eqnarray*}
et de même,
\begin{eqnarray*}
\Theta\Big(\Psi\big(\beta_1,\beta_2\big)\Big)&=&(\beta_1,\beta_2).
\end{eqnarray*}
Donc $\Theta$ définit bien une bijection entre $E_{\alpha_1}^{\alpha_2}(n)$ et $E_{(a,b)\bullet \alpha_1}^{(b,c)\bullet \alpha_2}(n)$ dont l'inverse est $\Psi.$
\end{proof}
Ce lemme implique que l'action $\bullet$ est compatible avec le produit de $Q_n.$ Précisément, on peut démontrer le corollaire suivant.
\begin{cor}\label{compatibility}
Pour tout $(a,b,c)\in \mathcal{B}_n^3$ et pour n'importe quelles bijections partielles $\alpha_1$ et $\alpha_2$ de $n,$ on a :
\begin{equation}\label{equation 4}
(a,c)\bullet(\alpha_1\pr \alpha_2)=((a,b)\bullet \alpha_1)\pr ((b,c)\bullet \alpha_2).
\end{equation}
\end{cor}
\begin{proof} Si $(\widetilde{\alpha_1},\widetilde{\alpha_2}) \in E_{\alpha_1}^{\alpha_2}(n)$, on a :

\begin{eqnarray*}
\big((a,b)\bullet \widetilde{ \alpha_1})\circ ((b,c)\bullet\widetilde{\alpha_2}\big)&=&\big(a\widetilde{\sigma_1}b^{-1},b(\widetilde{d_1}),a(\widetilde{d'_1})\big)\circ \big(b\widetilde{\sigma_2}c^{-1},c(\widetilde{d_2}),b(\widetilde{d'_2})\big)\\
&=&(a\widetilde{\sigma_1}\widetilde{\sigma_2}c^{-1},c(\widetilde{d_2}),a(\widetilde{d'_1}))\\
&=&(a,c)\bullet (\widetilde{\alpha_1}\circ \widetilde{\alpha_2}).
\end{eqnarray*}
Donc, on peut écrire :
\begin{eqnarray*}
(a,c)\bullet(\alpha_1\pr \alpha_2)&=&\frac{1}{| E_{\alpha_1}^{\alpha_2}(n)|}\sum_{(\widetilde{\alpha_1},\widetilde{\alpha_2}) \in E_{\alpha_1}^{\alpha_2}(n)} (a,c)\bullet (\widetilde{\alpha_1}\circ \widetilde{\alpha_2})\\
&=&\frac{1}{| E_{\alpha_1}^{\alpha_2}(n)|}\sum_{(\widetilde{\alpha_1},\widetilde{\alpha_2}) \in E_{\alpha_1}^{\alpha_2}(n)}  \big((a,b)\bullet \widetilde{ \alpha_1}\big)\circ \big((b,c)\bullet\widetilde{\alpha_2}\big)\\
&=&\frac{1}{| E_{(a,b)\bullet \alpha_1}^{(b,c)\bullet\alpha_2}(n)|}\sum_{(\widetilde{(a,b)\bullet \alpha_1},\widetilde{(b,c)\bullet \alpha_2}) \in E_{(a,b)\bullet \alpha_1}^{(b,c)\bullet \alpha_2}(n)}  \widetilde{(a,b)\bullet  \alpha_1}\circ \widetilde{(b,c)\bullet\alpha_2}\\
&=&((a,b)\bullet \alpha_1)\pr ((b,c)\bullet \alpha_2)).
\end{eqnarray*}
\end{proof}
On considère l'algèbre $\mathcal{A}'_n$ des éléments invariants par l'action de $\mathcal{B}_n\times \mathcal{B}_n$ sur $\mathcal{D}_n$ :
$$\mathcal{A}'_n=\mathcal{D}_n^{\mathcal{B}_n\times \mathcal{B}_n}=\lbrace x\in \mathcal{D}_n ~|~ (a,b)\bullet x=x~~\text{pour tout } (a,b)\in \mathcal{B}_n\times \mathcal{B}_n\rbrace.$$
Pour toute partition $\lambda$ telle que $|\lambda|\leq n$, on définit l'ensemble $A_{\lambda, n}$ comme étant l'ensemble des bijections partielles $\alpha$ de $n$ tel que $ct(\alpha)=\lambda.$ La somme des éléments de $A_{\lambda, n}$ est notée par ${\bf{A}}_{\lambda, n},$ suivant les notations du premier chapitre.
\begin{prop}\label{basis of An}
L'ensemble $\mathcal{A}'_n$ est une algèbre et les éléments de la famille $({\bf{A}}_{\lambda, n})_{|\lambda|\leq n}$ forment une base pour $\mathcal{A}'_n.$
\end{prop}
\begin{proof}
Pour tout $ (a,b)\in \mathcal{B}_n\times \mathcal{B}_n$ et pour tout $x,y \in \mathcal{A}'_n$ on a par linéarité :
$$(a,b)\bullet(x\pr y)= ((a,id)\bullet x)\pr ((id,b)\bullet y)=x\pr y.$$
Donc $\mathcal{A}'_n$ est une algèbre.\\
Tout élément $x\in \mathcal{D}_n$ s'écrit $\displaystyle{x= \sum_{k=1}^{n}\sum_{d,d^{'}\in \mathbf{P}_n \atop{|d|=|d^{'}|=2k}}\sum_{\sigma:d\rightarrow d^{'}\atop{\text{bijection}}}c_{(\sigma,d,d^{'})}(\sigma,d,d^{'})}$. Si de plus $x$ est dans $\mathcal{A}'_n$ alors pour tout $(a,b)\in \mathcal{B}_n\times \mathcal{B}_n$ on a : $$\sum_{k=1}^{n}\sum_{d,d^{'}\in \mathbf{P}_n\atop{|d|=|d^{'}|=2k}}\sum_{\sigma:d\rightarrow d^{'}\atop{\text{bijection}}}c_{(\sigma,d,d^{'})}(a\sigma b^{-1},b(d),a(d^{'}))=\sum_{k=1}^{n}\sum_{d,d^{'}\in \mathbf{P}_n\atop{|d|=|d^{'}|=2k}}\sum_{\sigma:d\rightarrow d^{'}\atop{\text{bijection}}}c_{(\sigma,d,d^{'})}(\sigma,d,d^{'}).$$
Donc, pour tout $(a,b)\in \mathcal{B}_n\times \mathcal{B}_n$ on a $c_{(a\sigma b^{-1},b(d),a(d^{'}))}=c_{(\sigma,d,d^{'})}$. Cela veut dire que si $x\in \mathcal{A}'_n$, toutes les bijections partielles appartenant à la même orbite possèdent le même coefficient. On conclut que les éléments $({\bf{A}}_{\lambda, n})_{|\lambda|\leq n}$ forment une base pour $\mathcal{A}'_n.$
\end{proof}
\begin{cor}\label{corollary 3.1}
Si $\lambda$ et $\delta$ sont deux partitions telles que $|\lambda|,|\delta|\leq n$ alors il existe une unique famille de nombres complexes $c_{\lambda\delta}^{\rho}(n)\in \mathbb{C}$ telle que :
$${\bf{A}}_{\lambda,n}\pr {\bf{A}}_{\delta,n}=\sum_{\rho \text{ partition} \atop {\max{(|\lambda|,|\delta|)}\leq|\rho|\leq \min{(|\lambda|+|\delta|,n)}}}c_{\lambda\delta}^{\rho}(n){\bf{A}}_{\rho,n}.$$
\end{cor}
\begin{proof} Il nous suffit de démontrer l'inégalité pour la taille de $\rho$. Soient $\alpha_1$ et $\alpha_2$ deux bijections partielles de $n$ de coset-type $\lambda$ et $\delta$ respectivement. Par définition (voir Figure \ref{fig:composition}), chaque bijection partielle de $[2n]$ qui apparaît dans la somme du produit $\alpha_1\pr \alpha_2$ possède un coset-type $\rho$ avec $|\rho|=\frac{|d_1\cup d_2'|}{2}.$ Mais
$$ \max\big( \frac{|d_1|}{2},\frac{|d_2'|}{2}\big)=\max (|\lambda|,|\delta|) \leq |\rho|=\frac{|d_1\cup d_2'|}{2}\leq \frac{|d_1|+|d_2'|}{2}=|\lambda|+|\delta|. \qedhere$$
\end{proof}

\setcounter{subsection}{1}
\subsection{Relations avec les algèbres de doubles-classes de $\mathcal{B}_n$ dans $\mathcal{S}_{2n}$}

\begin{prop}\label{proposition 3.2}
La fonction suivante 
$$\begin{array}{ccccc}
\psi_n & : & \mathbb{C}[Q_n] & \longrightarrow & \mathbb{C}[\mathcal{S}_{2n}] \\
& & \alpha & \mapsto & \frac{1}{2^{n-\frac{|d|}{2}}(n-\frac{|d|}{2})!}\displaystyle{\sum_{\hat{\alpha}\in \mathcal{S}_{2n}\cap P_\alpha(n) }}\hat{\sigma}\\
\end{array},$$
définit un morphisme d'algèbre entre $\mathbb{C}[Q_n]$ et $\mathbb{C}[\mathcal{S}_{2n}].$
\end{prop}
\begin{proof}
Soient $\alpha_1$ et $\alpha_2$ deux éléments de base de $\mathbb{C}[Q_n]$. On reprend la Figure \ref{fig:composition} et on note $$\text{$2b=|d_2|=|d'_2|, 2c=|d_1|=|d'_1|$ et $2e=|d'_2\cap d_1|$}.$$ On démontre d'abord que :
\begin{multline}\label{T}
\sum_{\widehat{\alpha_1}\in \mathcal{S}_{2n}\cap P_{\alpha_1}(n) } \quad
\sum_{\widehat{\alpha_2}\in \mathcal{S}_{2n}\cap P_{\alpha_2}(n) }
\widehat{\sigma_1}\circ \widehat{\sigma_2} \\
=2^{n-(b+c-e)}(n-(b+c-e))!\sum_{(\widetilde{\alpha_1},\widetilde{\alpha_2})\in E^{\alpha_1}_{\alpha_2}(n)} \quad \sum_{\widehat{\widetilde{\alpha_1}\circ \widetilde{\alpha_2}}\in \mathcal{S}_{2n}\cap P_{\widetilde{\alpha_1}\circ \widetilde{\alpha_2}}(n)}\widehat{\widetilde{\sigma_1}\circ \widetilde{\sigma_2}}.
\end{multline}
On fixe $(\widetilde{\alpha_1},\widetilde{\alpha_2})\in E_{\alpha_1}^{\alpha_2}(n)$ et $\omega\in \mathcal{S}_{2n}\cap P_{\widetilde{\alpha_1}\circ \widetilde{\alpha_2}}(n)$, \textit{i.e.}: $$\omega_{|_{\widetilde{d_2}}}=\widetilde{\sigma_1}\circ \widetilde{\sigma_2} \text{ et } ct(\omega)=ct(\widetilde{\sigma_1}\circ\widetilde{\sigma_2})\cup (1^{(n-(b+c-e))}).$$
On va compter les permutations $\widehat{\sigma_1}$ et $\widehat{\sigma_2}$ dans $\mathcal{S}_{2n}\cap P_{\alpha_1}(n)$ et $\mathcal{S}_{2n}\cap P_{\alpha_2}(n)$ telles que $\widehat{\sigma_1}\circ \widehat{\sigma_2}=\omega$. Dans l'équation précédente, $\widehat{\sigma_2}$ détermine $\widehat{\sigma_1}.$ Mais la condition $\omega_{|_{\widetilde{d_2}}}=\widetilde{\sigma_1}\circ \widetilde{\sigma_2}$ donne la valeur de $\widehat{\sigma_2}$ sur $\widetilde{d_2}$ (car $\widehat{\sigma_2}(x)=\sigma_2(x)$ si $x\in d_2$ et $\widehat{\sigma_2}(x)= \sigma_1^{-1}(\omega(x))$ si $x\in \widetilde{d_2}\setminus d_2$). Donc, le nombre de façons de choisir $\widehat{\sigma_2}$ est le nombre de façons d'étendre trivialement $\widetilde{\sigma_2}$ en une permutation de $2n$, ce qui est  $2^{n-(b+c-e)}(n-(b+c-e))!$ d'après le Lemme \ref{Lemma 3.1}. Cela prouve l'équation (\ref{T}).

Maintenant on a :
\begin{multline}\label{N}
\psi_n(\alpha_1)\psi_n(\alpha_2)=\frac{1}{2^{2n-b-c}(n-c)!(n-b)!}\displaystyle{\sum_{\widehat{\alpha_1}\in \mathcal{S}_{2n}\cap P_{\alpha_1}(n)}}~\displaystyle{\sum_{\widehat{\alpha_2}\in \mathcal{S}_{2n}\cap P_{\alpha_2}(n)}}\widehat{\sigma_1}\circ \widehat{\sigma_2}\\
=\frac{(n-b-c+e)!}{2^{n-e}(n-c)!(n-b)!}\sum_{(\widetilde{\alpha_1},\widetilde{\alpha_2})\in E^{\alpha_1}_{\alpha_2}(n)}~\sum_{\widehat{\widetilde{\alpha_1}\circ \widetilde{\alpha_2}}\in \mathcal{S}_{2n}\cap P_{\widetilde{\alpha_1}\circ \widetilde{\alpha_2}}(n)}\widehat{\widetilde{\sigma_1}\circ \widetilde{\sigma_2}}.~~~~~~~~~~~
\end{multline}
D'autre part on a :
$$\psi_n(\alpha_1\pr \alpha_2)=\frac{1}{2^{b+c-2e}(n-c)_{(b-e)}(n-b)_{(c-e)}}\sum_{(\widetilde{\alpha_1},\widetilde{\alpha_2})\in E^{\alpha_1}_{\alpha_2}(n)}\psi_n\big((\widetilde{\sigma_1}\circ \widetilde{\sigma_2},\widetilde{d_2},\widetilde{d_1'})\big).$$
Mais
$$\psi_n\big((\widetilde{\sigma_1}\circ \widetilde{\sigma_2},\widetilde{d_2},\widetilde{d_1'})\big)=\frac{1}{2^{n-(b+c-e)}(n-(b+c-e))!}\sum_{\widehat{\widetilde{\alpha_1}\circ \widetilde{\alpha_2}}\in \mathcal{S}_{2n}\cap P_{\widetilde{\alpha_1}\circ \widetilde{\alpha_2}}(n)}\widehat{\widetilde{\sigma_1}\circ \widetilde{\sigma_2}}.$$
Donc
\begin{equation}\label{M}
\psi_n(\alpha_1\pr \alpha_2)=\frac{(n-b-c+e))!}{2^{n-e}(n-c)!(n-b)!}\sum_{(\widetilde{\alpha_1},\widetilde{\alpha_2})\in E^{\alpha_1}_{\alpha_2}(n)}~~\sum_{\widehat{\widetilde{\alpha_1}\circ \widetilde{\alpha_2}}\in \mathcal{S}_{2n}\cap P_{\widetilde{\alpha_1}\circ \widetilde{\alpha_2}}(n)}\widehat{\widetilde{\sigma_1}\circ \widetilde{\sigma_2}}.
\end{equation}
En comparant les équations (\ref{N}) et (\ref{M}), on peut voir que pour deux bijections partielles $\alpha_1$ et $\alpha_2$ de $n$, on a $\psi_n(\alpha_1 \pr \alpha_2)=\psi_n(\alpha_1)\psi_n(\alpha_2)$. Autrement dit, $\psi_n$ est un morphisme d'algèbres.
\end{proof} 

\begin{lem}\label{image du base}
Soit $\lambda$ une partition tel que $|\lambda|=r\leq n$, on a :
$$\psi_n({\bf{A}}_{\lambda,n})=\frac{1}{2^{n-|\lambda|}(n-|\lambda|)!}\begin{pmatrix}
n-|\bar{\lambda}|\\
m_1(\lambda)
\end{pmatrix}{\bf K}_{\underline{\lambda}_n}.$$
\end{lem}
\begin{proof}On démontre d'abord l'équation :
\begin{equation}\label{projection sur S2n}
\sum_{\alpha\in A_{\lambda,n}}\sum_{\hat{\alpha}\in \mathcal{S}_{2n}\cap P_\alpha(n) }\hat{\sigma}=\begin{pmatrix}
n-|\bar{\lambda}|\\
m_1(\lambda)
\end{pmatrix}{\bf K}_{\underline{\lambda}_n}.
\end{equation} 
Fixons une permutation $\omega\in {K}_{\underline{\lambda}_n}$, ce qui veut dire que $\omega\in \mathcal{S}_{2n}$ et $ct(\omega)=\bar{\lambda}\cup 1^{n-|\bar{\lambda|}}$. On va compter les bijections partielles $\alpha\in A_{\lambda,n}$ telles que $\omega$ est une extension triviale de $\alpha.$ Il y a un ensemble unique $S$ tel que $ct(\omega_{|_{S}})=\bar{\lambda}.$ On appelle cet ensemble le support de $\omega$ et on le note $\supp(\omega).$ La condition suivante est nécessaire et suffisante pour que $\omega$ soit une extension triviale de $\alpha$ : $\supp(\omega)\subseteq d$ et $\sigma$ est égale à $\omega_{|_{d}}.$ Donc les bijections partielles $\alpha$ qu'on cherche sont les restrictions de $\omega$ sur des ensembles de la forme $\supp(\omega)\sqcup x$, avec $|\supp(\omega)\sqcup x|=2|\lambda|.$ Puisque $|\supp(\omega)|=2|\bar{\lambda}|$, on doit avoir nécessairement $|x|=2(|\lambda|-|\bar{\lambda}|).$ Donc le nombre de $\alpha$ cherchées est $\begin{pmatrix}
n-|\bar{\lambda}|\\
|\lambda|-|\bar{\lambda}|
\end{pmatrix}=\begin{pmatrix}
n-|\bar{\lambda}|\\
m_1(\lambda)
\end{pmatrix}$. On a terminé la démonstration de \eqref{projection sur S2n}.\\
En appliquant $\psi_n$ à ${\bf{A}}_{\lambda,n}$, on aura :
\begin{align*}
\psi_n({\bf{A}}_{\lambda,n})&=\psi_n\big(\sum_{\alpha\in A_{\lambda,n}}\alpha\big)\\
&=\frac{1}{2^{n-|\lambda|}(n-|\lambda|)!}\sum_{\alpha\in A_{\lambda,n}}\sum_{\hat{\alpha}\in \mathcal{S}_{2n}\cap P_\alpha(n) }\hat{\sigma}\\
&=\frac{1}{2^{n-|\lambda|}(n-|\lambda|)!}\begin{pmatrix}
n-|\bar{\lambda}|\\
m_1(\lambda)
\end{pmatrix}{\bf K}_{\underline{\lambda}_n}. && \qedhere
\end{align*}
\end{proof}

Ce lemme implique que $\psi_n(\mathcal{A}'_n)\subseteq \mathbb{C}[\mathcal{B}_n/\mathcal{S}_{2n}\setminus\mathcal{B}_n].$ Le morphisme $\mathcal{A}'_n\rightarrow \mathbb{C}[\mathcal{B}_n/ \mathcal{S}_{2n}\setminus \mathcal{B}_n]$ mentionné dans la Section \ref{sec:res_idee} est la restriction $\psi_{n_{|_{\mathcal{A}'_n}}}.$

\subsection{Morphisme entre les algèbres $\mathcal{A'}_{n+1}$ et $\mathcal{A'}_n$}\label{sec:morph_A'}

Dans ce paragraphe on va donner un morphisme d'algèbres entre $\mathcal{A}'_{n+1}$ et $\mathcal{A}'_n.$
\begin{prop}\label{prop:morph_A'_n+1_et_A'_n}
L'application $\varphi_{n}$ définie ainsi :
$$\begin{array}{ccccc}
\varphi_{n} & : & \mathcal{A}'_{n+1} & \to & \mathcal{A}'_n \\
& & {\bf{A}}_{\lambda,n+1} & \mapsto  & \left\{
\begin{array}{ll}
  \frac{n+1}{(n+1-|\lambda|)}{\bf{A}}_{\lambda,n} & \qquad \mathrm{si}\quad |\lambda| < n+1 ,\\
  0 & \qquad \mathrm{si}\quad |\lambda|=n+1 ,\\
 \end{array}
 \right. \\
\end{array}$$
est un morphisme d'algèbres.
\end{prop}
Soient ${\bf{A}}_{\lambda,n+1}$ et ${\bf{A}}_{\delta,n+1}$ où $|\lambda|\leq n+1$ et $|\delta|\leq n+1$ deux éléments de base de $\mathcal{A'}_{n+1}$. Si $\lambda$ (resp. $\delta$) est une partition de $n+1$, alors $\varphi_{n}({\bf{A}}_{\lambda,n+1})$ (resp. $\varphi_{n}({\bf{A}}_{\delta,n+1})$) est égale à zéro, et par le Corollaire \ref{corollary 3.1} on a :
$${\bf{A}}_{\lambda,n+1}\pr {\bf{A}}_{\delta,n+1}=\sum_{\rho \text{ partition} \atop {|\rho|=n+1 }}c_{\lambda\delta}^{\rho}(n+1){\bf{A}}_{\rho,n+1}.$$
Il faut noter que la taille des partitions $\rho$ qui indexe la somme dans cette équation est $n+1$. En appliquant $\varphi_n$, on aura : $$\varphi_{n}({\bf{A}}_{\lambda,n+1}\pr {\bf{A}}_{\delta,n+1})=\sum_{\rho \text{ partition} \atop {|\rho|=n+1 }}c_{\lambda\delta}^{\rho}(n+1)\varphi_{n}({\bf{A}}_{\rho,n+1})=0.$$
Donc dans ce cas on a : $\varphi_{n}({\bf{A}}_{\lambda,n+1}\pr {\bf{A}}_{\delta,n+1})=\varphi_{n}({\bf{A}}_{\lambda,n+1})\pr \varphi_{n}({\bf{A}}_{\delta,n+1}).$ 

Dans l'autre cas ($|\lambda|\leq n$ et $|\delta|\leq n$) on a, par le Corollaire \ref{corollary 3.1}:
$${\bf{A}}_{\lambda,n+1}\pr {\bf{A}}_{\delta,n+1}=\sum_{r\leq n+1\atop{\rho\vdash r}}c_{\lambda\delta}^{\rho}(n+1)
{\bf{A}}_{\rho,n+1}.$$
Cela nous donne l'équation suivante en appliquant $\varphi_n$:
\begin{eqnarray*}
\varphi_n({\bf{A}}_{\lambda,n+1}\pr {\bf{A}}_{\delta,n+1})&=&\varphi_n\left(\sum_{r\leq n+1\atop{\rho\vdash r}}c_{\lambda\delta}^{\rho}(n+1)
{\bf{A}}_{\rho,n+1}\right)\\
&=&\sum_{r\leq n\atop{\rho\vdash r}}c_{\lambda\delta}^{\rho}(n+1)
\frac{n+1}{(n+1-|\rho|)}{\bf{A}}_{\rho,n}.
\end{eqnarray*}
De l'autre côté, on a :
\begin{eqnarray*}
\varphi_n({\bf{A}}_{\lambda,n+1})\pr \varphi_n({\bf{A}}_{\delta,n+1})&=&\frac{n+1}{(n+1-|\lambda|)}{\bf{A}}_{\lambda,n}\pr \frac{n+1}{(n+1-|\delta|)}{\bf{A}}_{\delta,n}\\
&=&\frac{n+1}{(n+1-|\lambda|)}\frac{n+1}{(n+1-|\delta|)}\sum_{r\leq n\atop{\rho\vdash r}}c_{\lambda\delta}^{\rho}(n){\bf{A}}_{\rho,n}.
\end{eqnarray*}
Donc, $\varphi_n$ est un morphisme si on a l'égalité suivante pour toute partition $\rho$ de taille au plus $n$ :
$$\frac{c_{\lambda\delta}^{\rho}(n+1)}{c_{\lambda\delta}^{\rho}(n)}=\frac{\frac{n+1}{(n+1-|\lambda|)}\frac{n+1}{(n+1-|\delta|)}}{\frac{n+1}{(n+1-|\rho|)}}.$$
Soit $\rho$ une partition de taille au plus $n$ et $\alpha$ un élément de $A_{\rho,n}$. On définit $H_{\lambda\delta}^{\rho}(n)$ ainsi :
$$H_{\lambda\delta}^{\rho}(n):=\lbrace \big(\alpha_1,\alpha_2\big)\in A_{\lambda,n}\times A_{\delta,n}
\text{ tel qu'il existe } (\widetilde{\alpha_1},\widetilde{\alpha_2})\in E_{\alpha_1}^{\alpha_2}(n)\text{ avec }\alpha=\widetilde{\alpha_1}\circ \widetilde{\alpha_2}\rbrace.
$$
Cet ensemble dépend de $\alpha$ par définition, bien que, $\alpha$ n'apparaisse pas dans notre notation. Cela ne devrait pas causer de problèmes, puisque $\alpha$ est fixé dans la démonstration.

Le coefficient $c_{\lambda\delta}^{\rho}(n)$ s'écrit :
$$c_{\lambda\delta}^{\rho}(n)=\displaystyle{\sum_{(\alpha_1,\alpha_2)\in H_{\lambda\delta}^{\rho}(n)}}\frac{1}{|E_{\alpha_1}^{\alpha_2}(n)|}.$$
De même, on a :
$$c_{\lambda\delta}^{\rho}(n+1)=\displaystyle{\sum_{(\alpha_1,\alpha_2)\in H_{\lambda\delta}^{\rho}(n+1)}}\frac{1}{|E_{\alpha_1}^{\alpha_2}(n+1)|}.$$
D'après l'équation \eqref{cardinal de E}, si $(\alpha_1,\alpha_2)\in H_{\lambda\delta}^{\rho}(n)$, on a :
$$|E_{\alpha_1}^{\alpha_2}(n)|=2^{2|\rho|-|\lambda|-|\delta|}(n-|\lambda|)_{(|\rho|-|\lambda|)}(n-|\delta|)_{(|\rho|-|\delta|)}.$$
De même, si $(\alpha_1,\alpha_2)\in H_{\lambda\delta}^{\rho}(n+1)$, on a :
$$|E_{\alpha_1}^{\alpha_2}(n+1)|=2^{2|\rho|-|\lambda|-|\delta|}(n+1-|\lambda|)_{(|\rho|-|\lambda|)}(n+1-|\delta|)_{(|\rho|-|\delta|)}.$$
Donc, on obtient :
\begin{equation}\label{form_reliant_c_et_h}
c_{\lambda\delta}^{\rho}(n)=\frac{|H_{\lambda\delta}^{\rho}(n)|}{2^{2|\rho|-|\lambda|-|\delta|}(n-|\lambda|)_{(|\rho|-|\lambda|)}(n-|\delta|)_{(|\rho|-|\delta|)}},
\end{equation}
et
$$c_{\lambda\delta}^{\rho}(n+1)=\frac{|H_{\lambda\delta}^{\rho}(n+1)|}{2^{2|\rho|-|\lambda|-|\delta|}(n+1-|\lambda|)_{(|\rho|-|\lambda|)}(n+1-|\delta|)_{(|\rho|-|\delta|)}}.$$
Après simplification, cela nous donne :
$$\frac{c_{\lambda\delta}^{\rho}(n+1)}{c_{\lambda\delta}^{\rho}(n)}=\frac{|H_{\lambda\delta}^{\rho}(n+1)|}{|H_{\lambda\delta}^{\rho}(n)|}\cdot \frac{n+1-|\rho|}{n+1-|\lambda|}\cdot \frac{n+1-|\rho|}{n+1-|\delta|}$$
Maintenant, on va évaluer le quotient $\frac{|H_{\lambda\delta}^{\rho}(n+1)|}{|H_{\lambda\delta}^{\rho}(n)|}$. Soit $\bm{u}=(u_1, u'_1,u_2,u'_2)$ un élément de $\mathbf{P}_{n}^{4}$ tel que :
\begin{eqnarray}
&&u_2\subseteq d,\label{condition 1}\\
&&u'_{1}\subseteq d',\label{condition 2}\\
&&|u_1|= |u'_1|=2|\lambda|,\label{condition 3}\\
&&|u_2|= |u'_2|=2|\delta|,\label{condition 4}\\
&&|u'_2\cup u_1|=2|\rho|.\label{condition 5}
\end{eqnarray} 
On introduit \begin{eqnarray*}
N_{\bm{u}}&=&\lbrace \big(h_1=(f_1,u_1,u'_1),h_2=(f_2,u_2,u'_2)\big)\in A_{\lambda,n}\times A_{\delta,n}\\
&&\text{tel qu'il existe }(\widetilde{h_1},\widetilde{h_2})\in E_{h_1}^{h_2}(n) \text{ avec } \alpha=\widetilde{h_1}\circ \widetilde{h_2}\rbrace.
\end{eqnarray*}
Ces éléments sont représentés dans la Figure \ref{fig:elements of Nu}. L'ensemble $H_{\lambda\delta}^{\rho}(n)$ est l'union disjointe des $N_{\bm{u}}$ où $\bm{u}$ satisfait les conditions \eqref{condition 1} à \eqref{condition 5}.
\begin{figure}[htbp]
\begin{center}
\begin{tikzpicture},fill opacity=0.5]
\draw[] (0,0) ellipse (1cm and 2cm)++(8,0) ellipse (1cm and 2cm);
\draw (0,0) node {$d$};
\draw (8,0) node {$d'$};
\draw[->] (1,0) to node [sloped,below] {$\sigma$} (7,0);
\draw[] (0,1) ellipse (.4cm and .8cm)++(8,0) ellipse (.4cm and .8cm);
\draw (0,1) node {$u_2$};
\draw (8,1) node {$u'_1$};
\draw[] (4,3) ellipse (.4cm and .8cm);
\draw[] (4,2) ellipse (.4cm and .8cm);
\draw (4,3) node {$u'_2$};
\draw (4,2) node {$u_1$};
\draw[->,dashed] (.4,1) to node [sloped,above] {$f_2$} (3.6,3);
\draw[->,dashed] (4.4,2) to node [sloped,above] {$f_1$} (7.6,1);
\end{tikzpicture}
\caption{Représentation des éléments de $N_{\bm{u}}$ }
\label{fig:elements of Nu}
\end{center}
\end{figure}
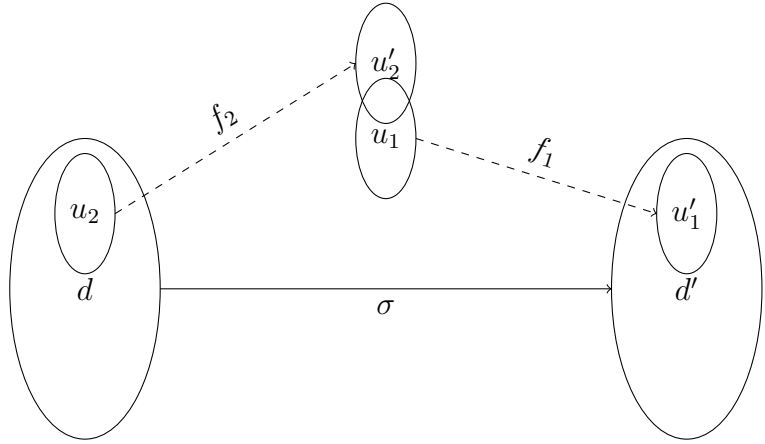
\begin{lem}\label{bijection Nu and Nv}
Soit $\bm{v}=(v_1, v'_1,v_2,v'_2)$ un élément de $\mathbf{P}_{n}^{4}$ satisfaisant les conditions \eqref{condition 1} à \eqref{condition 5}. Si $v'_1=u'_1$ et $v_2=u_2$, alors il existe une bijection entre $N_{\bm{u}}$ et $N_{\bm{v}}.$
\end{lem}
\begin{proof}
On prend une permutation $b\in \mathcal{B}_n$, tel que $b(u_1)=v_1$ et $b(u'_2)=v'_2.$ Une telle permutation existe car $|u_1|=|v_1|$, $|u'_2|=|v'_2|$ et $|u'_2\cup u_1|=|v'_2\cup v_1|.$ On associe à une paire $(h_1,h_2)$ de $N_{\bm{u}}$ la paire $((id,b)\bullet h_1,(b,id)\bullet h_2).$ On vérifie que l'image est bien dans $N_{\bm{v}}$ :
\begin{eqnarray*}
\big((id,b)\bullet h_1,(b,id)\bullet h_2\big)&=&\big((f_1b^{-1},b(u_1),u'_1),(bf_2,u_2,b(u'_2))\big)\\
&=&\big((f_1b^{-1},v_1,u^{'}_1),(bf_2,u_2,v'_2)\big)\in A_{\lambda,n}\times A_{\delta,n}.
\end{eqnarray*}
On peut vérifier que $\big((id,b)\bullet \widetilde{h_1},(b,id)\bullet \widetilde{h_2}\big)\in E_{(id,b)\bullet h_1}^{(b,id)\bullet h_2}(n)$, et on a :
\begin{eqnarray*}
(id,b)\bullet \widetilde{h_1}\circ (b,id)\bullet \widetilde{h_2}&=&(\widetilde{f_1}b^{-1},b(\widetilde{u_1}),\widetilde{u'_1})\circ (b\widetilde{f_2},\widetilde{u_2},b(\widetilde{u'_2}))\\
&=&(\widetilde{f_1}b^{-1}b\widetilde{f_2},\widetilde{u_2},\widetilde{u'_1})\\
&=&\widetilde{h_1}\circ \widetilde{h_2}\\
&=&\alpha .
\end{eqnarray*}
Donc cette correspondance définit bien une bijection entre $N_{\bm{u}}$ et $N_{\bm{v}}$. Les détails sont les mêmes que pour le Lemme \ref{compatibility of E}.
\end{proof}
Donc le cardinal de $N_{\bm{u}}$ dépend uniquement de $u'_1$ et $u_2$. On le note par $f(u'_1,u_2)$. Si on note par $\bm{U}$ l'ensemble des vecteurs $\bm{u}\in \mathbf{P}_n^4$ satisfaisants les conditions  \eqref{condition 1} à \eqref{condition 5}, l'ensemble $H_{\lambda\delta}^{\rho}(n)$ peut s'écrire : $$\displaystyle{H_{\lambda\delta}^{\rho}(n)=\bigsqcup_{\bm{u}\in \bm{U}} N_{\bm{u}}}.$$
En utilisant le Lemme \ref{bijection Nu and Nv}, on obtient :
$$|H_{\lambda\delta}^{\rho}(n)|=\sum_{\bm{u}\in \bm{U}}|N_{\bm{u}}|=\sum_{u'_1,u_2}\sum_{u_1,u'_2}f(u'_1,u_2).$$
La première (resp. deuxième) somme est indexée par les vecteurs $u'_1$ et $u_2$ (resp. $u_1$ et $u'_2$) satisfaisant les conditions \eqref{condition 1} jusqu'à \eqref{condition 4} (resp. \eqref{condition 3} jusqu'à \eqref{condition 5}).
Comme $N_{\bm{u}}$ ne dépend que du $u'_1$ et $u_2$, on obtient :
$$|H_{\lambda\delta}^{\rho}(n)|=\sum_{u'_1,u_2}f(u'_1,u_2)k_n,$$ 
où $k_n$ est le nombre des choix possibles des vecteurs $u_1$ et $u'_2$ satisfaisant les conditions \eqref{condition 3} à \eqref{condition 5}. Il y a $\begin{pmatrix}
n\\
|\lambda|
\end{pmatrix}$ ensembles $u_1\in \mathbf{P}_n$ qui vérifient l'équation \eqref{condition 3}. Une fois $u_1$ choisi, il reste $\begin{pmatrix}
|\lambda|\\
|\lambda|+|\delta|-|\rho|
\end{pmatrix}\cdot \begin{pmatrix}
n-|\lambda|\\
|\rho|-|\lambda|
\end{pmatrix}$ choix pour choisir $u'_2$ avec les conditions \eqref{condition 4} et \eqref{condition 5}. Le premier binomial est le nombre de choix possibles de $u_1\cap u'_2$ et le second est le nombre de choix possibles de $u'_2\setminus u_1.$ Donc, on a :
$$k_n=\begin{pmatrix}
n\\
|\lambda|
\end{pmatrix}\cdot \begin{pmatrix}
n-|\lambda|\\
|\rho|-|\lambda|
\end{pmatrix}\cdot \begin{pmatrix}
|\lambda|\\
|\lambda|+|\delta|-|\rho|
\end{pmatrix}.$$
Alors, le cardinal de $H_{\lambda\delta}^{\rho}(n)$ est :
\begin{equation}\label{card_de_H_lambdadeltarho}
|H_{\lambda\delta}^{\rho}(n)|=\begin{pmatrix}
n\\
|\lambda|
\end{pmatrix}\cdot \begin{pmatrix}
n-|\lambda|\\
|\rho|-|\lambda|
\end{pmatrix}\cdot \begin{pmatrix}
|\lambda|\\
|\lambda|+|\delta|-|\rho|
\end{pmatrix}\sum_{u'_1,u_2}f(u'_1,u_2).
\end{equation}
Les vecteurs qui indexent la somme ne dépendent pas de $n$ car $u'_1$ et $u_2$ doivent vérifier les conditions \eqref{condition 1} et \eqref{condition 2}.
De même, on obtient :
$$|H_{\lambda\delta}^{\rho}(n+1)|=\begin{pmatrix}
n+1\\
|\lambda|
\end{pmatrix}\cdot \begin{pmatrix}
n+1-|\lambda|\\
|\rho|-|\lambda|
\end{pmatrix}\cdot \begin{pmatrix}
|\lambda|\\
|\lambda|+|\delta|-|\rho|
\end{pmatrix}\sum_{u'_1,u_2}f(u'_1,u_2),$$
ce qui nous donne :
$$\frac{|H_{\lambda\delta}^{\rho}(n+1)|}{|H_{\lambda\delta}^{\rho}(n)|}=\frac{\begin{pmatrix}
n+1\\
|\delta|
\end{pmatrix}\begin{pmatrix}
n+1-|\delta|\\
|\rho|-|\delta|
\end{pmatrix}}{\begin{pmatrix}
n\\
|\delta|
\end{pmatrix}\begin{pmatrix}
n-|\delta|\\
|\rho|-|\delta|
\end{pmatrix}}=\frac{n+1}{n+1-|\rho|}.$$
Donc, on a :
\begin{eqnarray}\label{formula of c}
\frac{c_{\lambda\delta}^{\rho}(n+1)}{c_{\lambda\delta}^{\rho}(n)}&=&\frac{n+1}{n+1-|\rho|}\cdot \frac{n+1-|\rho|}{n+1-|\lambda|}\cdot \frac{n+1-|\rho|}{n+1-|\delta|}\\
&=&\frac{\frac{n+1}{(n+1-|\lambda|)}\frac{n+1}{(n+1-|\delta|)}}{\frac{n+1}{(n+1-|\rho|)}}, \nonumber
\end{eqnarray}
ce qui termine la preuve de la Proposition \ref{prop:morph_A'_n+1_et_A'_n}.
\section{Polynomialité des coefficients de structures}\label{sec:polynomialite}
\subsection{L'algèbre $\mathcal{A'}_\infty$ des invariants à l'infini}
On considère la limite projective $\mathcal{A}'_\infty$\label{nomen:alg_A_prime_infini} \nomenclature[34]{$\mathcal{A}'_\infty$}{L'algèbre qui se projette sur l'algèbre de Hecke \Hecke pour tout $n$ \quad \pageref{nomen:alg_A_prime_infini}} de la suite $(\mathcal{A}'_n)$. On démontre dans la Proposition \ref{proposition 3.4} que chaque élément de $\mathcal{A}'_\infty$ s'écrit d'une manière unique comme une combinaison linéaire infinie d'éléments indexés par des partitions.\\

À partir de l'équation \eqref{formula of c}, on peut obtenir le corollaire suivant :
\begin{cor}\label{c}
Soient $\lambda$, $\delta$ et $\rho$ trois partitions telles que $$\max{(|\lambda|,|\delta|)}\leq |\rho|\leq |\lambda|+|\delta|.$$ Pour tout $n\geq |\rho|$, on a :
$$c_{\lambda\delta}^{\rho}(n)=\frac{c_{\lambda\delta}^{\rho}(|\rho|)}{\begin{pmatrix}
|\rho|\\
|\lambda|
\end{pmatrix}\begin{pmatrix}
|\rho|\\
|\delta|
\end{pmatrix}}\cdot \frac{\begin{pmatrix}
n\\
|\lambda|
\end{pmatrix}\begin{pmatrix}
n\\
|\delta|
\end{pmatrix}}{\begin{pmatrix}
n\\
|\rho|
\end{pmatrix}}.$$
\end{cor}
\begin{proof} Récurrence immédiate en utilisant l'équation \eqref{formula of c}.
\end{proof}
Soit $\mathcal{A}'_{\infty}$ la limite projective des $(\mathcal{A}'_n,\varphi_n),$ c'est-à-dire :
$$\mathcal{A}'_{\infty}=\lbrace (a_n)_{n\geq 1} \mid \text{pour tout } n\geq 1, a_n\in \mathcal{A}'_n\text{ et } \varphi_n(a_{n+1})=a_n\rbrace.$$
\begin{lem}\label{Projective limit}Un élément $a=(a_n)_{n\geq 1}$ est dans $\mathcal{A}'_{\infty}$ si et seulement s'il existe une famille $(x^{a}_{\lambda})_{\lambda~partition}$ d'éléments de $\mathbb{C}$ telle que pour tout $n\geq 1$, $$a_n=\displaystyle{\sum_{\lambda \text{ partition}\atop{|\lambda|\leq n}}\frac{x^{a}_{\lambda}}{\begin{pmatrix}
n\\
|\lambda|
\end{pmatrix}}{{\bf{A}}}_{\lambda,n}}.$$
\end{lem}
\begin{proof}\
Soit $a=(a_n)_{n\geq 1}$ une suite dans $\mathcal{A}'_{\infty}$, $a_n\in \mathcal{A}'_n$ pour tout $n\geq 1$. D'après la Proposition \ref{basis of An}, les éléments $({\bf{A}}_{\lambda,n})_{\lambda\vdash r\leq n}$ forment une base de $\mathcal{A}'_n$, donc pour tout $n\geq 1$ et pour toute partition $\lambda$ tel que $|\lambda|\leq n$, il existe $a_\lambda(n)\in \mathbb{C}$ tel que \begin{displaymath}
a_n=\sum_{\lambda \text{ partition}\atop {|\lambda|\leq n}}a_{\lambda}(n){\bf{A}}_{\lambda,n}.
\end{displaymath}
La condition $\varphi_n(a_{n+1})=a_n$ peut être écrite ainsi :
$$\varphi_n\left(\sum_{\lambda \text{ partition}\atop{|\lambda|\leq n+1}}a_{\lambda}(n+1){\bf{A}}_{\lambda,n+1}\right)=\sum_{\lambda \text{ partition}\atop{|\lambda|\leq n}}a_{\lambda}(n){\bf{A}}_{\lambda,n}.$$
Par définition de $\varphi_n$, on peut simplifier cette égalité pour obtenir :
$$\sum_{\lambda \text{ partition}\atop{|\lambda|\leq n}}a_{\lambda}(n+1)\frac{n+1}{n+1-|\lambda|}{\bf{A}}_{\lambda,n}=\sum_{\lambda \text{ partition}\atop{|\lambda|\leq n}}a_{\lambda}(n){\bf{A}}_{\lambda,n}.$$
En identifiant les coefficients de ${\bf{A}}_{\lambda,n}$ on voit que pour toute partition $\lambda$ vérifiant $|\lambda|\leq n$, on a :
$$\frac{a_{\lambda}(n+1)}{a_{\lambda}(n)}=\frac{n+1-|\lambda|}{n+1}.$$
Par récurrence, on obtient :
$$a_{\lambda}(n)=\frac{a_{\lambda}(|\lambda|)}{\begin{pmatrix}
n\\
|\lambda|
\end{pmatrix}}.$$ 
Prendre $x^a_\lambda=a_\lambda(|\lambda|)$ prouve le sens directe. L'autre sens est évident.
\end{proof}
Pour tout $\lambda$, on définit la suite $T_{\lambda}$ de terme général :
$$(T_{\lambda})_{n}=\left\{
\begin{array}{ll}
  0 & \qquad \mathrm{si}\quad n< |\lambda|,\\\\
  \frac{1}{\begin{pmatrix}
n\\
|\lambda|
\end{pmatrix}}{{\bf{A}}}_{\lambda,n} & \qquad \mathrm{si}\quad  n\geq |\lambda| .\\
 \end{array}
 \right.$$
D'après le Lemme \ref{Projective limit}, on obtient directement la proposition suivante :
\begin{prop}\label{proposition 3.4}
 Tout élément $a \in \mathcal{A}'_\infty$ s'écrit d'une manière unique comme combinaison linéaire infinie des éléments $T_{\lambda}$.
 \end{prop}
Cette proposition montre que l'algèbre $\mathcal{A}'_\infty$ vérifie la deuxième propriété demandée dans la Section \ref{sec:res_idee}. En particulier, $T_{\lambda}\pr T_{\delta}$ s'écrit comme combinaison linéaire des éléments $T_\rho$. Plus précisément on a le corollaire suivant.
\begin{cor}\label{corollary 3.2}
Soient $\lambda$ et $\delta$ deux partitions, Il existe une unique famille $b_{\lambda\delta}^{\rho}$ telle que :
$$T_{\lambda}\pr T_{\delta}=\sum_{\rho \text{ partition } \atop {\max{(|\lambda|,|\delta|)}\leq |\rho|\leq |\lambda|+|\delta|}}b_{\lambda\delta}^{\rho}T_{\rho}.$$
De plus, $b_{\lambda\delta}^{\rho}=\frac{c_{\lambda\delta}^{\rho}(|\rho|)}{\begin{pmatrix}
|\rho|\\
|\lambda|
\end{pmatrix}\begin{pmatrix}
|\rho|\\
|\delta|
\end{pmatrix}}$. En particulier, c'est un nombre rationnel positif.
\end{cor}
\begin{proof} D'après la Proposition \ref{proposition 3.4}, $T_{\lambda}\pr T_{\delta}$ s'écrit comme combinaison linéaire des $T_\rho$. 
$$T_{\lambda}\pr T_{\delta}=\sum_{\rho \text{ partition }}b_{\lambda\delta}^{\rho}T_{\rho}.$$
Il reste à démontrer comment obtenir la condition sur la taille des partitions $\rho$ indexant la somme et la formule pour $b_{\lambda\delta}^{\rho}$. \\Si $n< \max{(|\lambda|,|\delta|)}$, on a :
$$(T_{\lambda}\pr T_{\delta})_n=0.$$
Soit $n\geq \max{(|\lambda|,|\delta|)}$, on utilise le Corollaire \ref{corollary 3.1} et le Corollaire \ref{c} pour avoir : 
\begin{eqnarray*}
(T_{\lambda}\pr T_{\delta})_n&=&\frac{1}{\begin{pmatrix}
n\\
|\lambda|
\end{pmatrix}}{{\bf{A}}}_{\lambda,n}\pr \frac{1}{\begin{pmatrix}
n\\
|\delta|
\end{pmatrix}}{{\bf{A}}}_{\delta,n}\\
&=&\sum_{\rho \text{ partition}\atop{\max {(|\lambda|,|\delta|)}\leq |\rho| \leq \min {(|\lambda|+|\delta|,n )}}}\frac{c_{\lambda\delta}^{\rho}(|\rho|)}{\begin{pmatrix}
|\rho|\\
|\lambda|
\end{pmatrix}\begin{pmatrix}
|\rho|\\
|\delta|
\end{pmatrix}\begin{pmatrix}
n\\
|\rho|
\end{pmatrix}}{\bf{A}}_{\rho,n}\\
&=&\left(\sum_{\rho \text{ partition}\atop{\max {(|\lambda|,|\delta|)}\leq |\rho| \leq |\lambda|+|\delta|}}\frac{c_{\lambda\delta}^{\rho}(|\rho|)}{\begin{pmatrix}
|\rho|\\
|\lambda|
\end{pmatrix}\begin{pmatrix}
|\rho|\\
|\delta|
\end{pmatrix}}T_{\rho}\right)_n
\end{eqnarray*}
Comparer les deux expressions de $T_{\lambda}\pr T_{\delta}$ termine la preuve.
\end{proof}
\begin{ex}\label{example T2*T2}
On calcule dans cet exemple le produit $T_{(2)}\pr T_{(2)}.$ D'après le Corollaire \ref{corollary 3.2}, on peut écrire 
$\displaystyle{T_{(2)}\pr T_{(2)}=\sum_{\rho \text{ partition } \atop {2\leq |\rho|\leq 4}}b_{(2)(2)}^{\rho}T_{\rho}}$, ce qui nous donne :
\begin{eqnarray*}
T_{(2)}\pr T_{(2)}&=& ~b_{(2)(2)}^{(1^2)}T_{(1^2)}+b_{(2)(2)}^{(1^3)}T_{(1^3)}+b_{(2)(2)}^{(1^4)}T_{(1^4)}\\
&& +b_{(2)(2)}^{(2)}T_{(2)}+b_{(2)(2)}^{(2,1)}T_{(2,1)}+b_{(2)(2)}^{(2,1^2)}T_{(2,1^2)}+b_{(2)(2)}^{(2^2)}T_{(2^2)}\\
&& +b_{(2)(2)}^{(3)}T_{(3)}+b_{(2)(2)}^{(3,1)}T_{(3,1)}\\
&& +b_{(2)(2)}^{(4)}T_{(4)}.
\end{eqnarray*}
La formule pour $b_{\lambda\delta}^{\rho}$ donnée dans le Corollaire \ref{corollary 3.2} montre qu'on peut utiliser le produit ${\bf{A}}_{(2),|\lambda|+|\delta|}\pr {\bf{A}}_{(2),|\lambda|+|\delta|}$ dans $\mathcal{A}'_{|\lambda|+|\delta|}$, pour obtenir les coefficients.\\
On a implémenté l'algèbre $\mathcal{A}'_n$ dans Sage \cite{Sage} et on a obtenu l'équation suivante pour le produit ${\bf{A}}_{(2),4}*{\bf{A}}_{(2),4}$ dans $\mathcal{A}'_4$:
$${\bf{A}}_{(2),4}*{\bf{A}}_{(2),4}=96 {\bf{A}}_{(1^2),4}+48{\bf{A}}_{(2),4}+36{\bf{A}}_{(3),4}+12{\bf{A}}_{(2^2),4}.$$
En utilisant les formules pour $c_{\lambda\delta}^\rho(|\rho|)$ et $b_{\lambda\delta}^{\rho}$ données par les Corollaires \ref{c} et \ref{corollary 3.2}, on obtient :
$$T_{(2)}\pr T_{(2)}=16T_{(1^2)}+8T_{(2)}+4T_{(3)}+\frac{1}{3}T_{(2^2)}.$$
\end{ex}

\begin{rem}
En général, d'après le Corollaire \ref{corollary 3.2}, on peut calculer le produit $T_{\lambda}\pr T_{\delta}$ en faisant le calcul du produit ${\bf{A}}_{\lambda, |\lambda|+|\delta|}*{\bf{A}}_{\delta,|\lambda|+|\delta|}$ dans l'algèbre $\mathcal{A}'_{|\lambda|+|\delta|}.$ 
\end{rem}

\begin{cor}
L'ensemble des combinaisons linéaire finies des ($T_\lambda$), noté par $\widetilde{\mathcal{A}'_\infty}$, est une sous-algèbre de $\mathcal{A}'_\infty$. La famille $(T_\lambda)_{\lambda \text{ partition}}$ forme une base pour $\widetilde{\mathcal{A}'_\infty}$.
\end{cor}
\begin{proof}
Cela vient du fait que les partitions $\rho$ indexant la somme du produit $T_\lambda \pr T_\delta$ vérifient :
\begin{equation*}
|\rho|\leq |\lambda|+|\delta|. \qedhere
\end{equation*}
\end{proof}
L'algèbre $\widetilde{\mathcal{A}'_\infty}$ va nous être utile dans la Section \ref{sec:isom_fct_decal_2}.
\subsection{Preuve de la polynomialité}
Dans la section précédente, on a donné toutes les algèbres et morphismes nécessaires pour prouver le Théorème \ref{Theorem 2.1}.

Soient $\lambda$ et $\delta$ deux partitions propres, d'après le Corollaire \ref{corollary 3.2}, on a :
$$T_{\lambda}\pr T_{\delta}=\sum_{\rho \text{ partition}\atop {\max{(|\lambda|,|\delta|)}\leq |\rho|\leq |\lambda|+|\delta|}}b_{\lambda\delta}^{\rho}T_{\rho}.$$
On rappelle que c'est une égalité de suites. En prenant le $n$-ième terme, on aura :
\begin{eqnarray*}
\frac{1}{\begin{pmatrix}
n\\
|\lambda|
\end{pmatrix}}{{\bf{A}}}_{\lambda,n}\pr \frac{1}{\begin{pmatrix}
n\\
|\delta|
\end{pmatrix}}{{\bf{A}}}_{\delta,n}&=&\sum_{\rho~\text{partition}\atop {\max{(|\lambda|,|\delta|)}\leq|\rho|\leq \min{(|\lambda|+|\delta|,n)}}}b_{\lambda\delta}^{\rho}\frac{1}{\begin{pmatrix}
n\\
|\rho|
\end{pmatrix}}{{\bf{A}}}_{\rho,n} .
\end{eqnarray*}
En appliquant $\psi_n$ on obtient (voir Lemme \ref{image du base}):
\begin{multline*}
\frac{1}{2^{n-|\lambda|}(n-|\lambda|)!}{\bf K}_{\underline{\lambda}_n}\cdot \frac{1}{2^{n-|\delta|}(n-|\delta|)!}{\bf K}_{\underline{\delta}_n}=\\
\sum_{\rho~\text{partition}\atop {\max{(|\lambda|,|\delta|)}\leq|\rho|\leq \min{(|\lambda|+|\delta|,n)}}}b_{\lambda\delta}^{\rho}\frac{\begin{pmatrix}
n\\
|\lambda|
\end{pmatrix}\begin{pmatrix}
n\\
|\delta|
\end{pmatrix}}{\begin{pmatrix}
n\\
|\rho|
\end{pmatrix}2^{n-|\rho|}(n-|\rho|)!}\begin{pmatrix}
n-|\bar{\rho}|\\
m_1(\rho)
\end{pmatrix}{\bf K}_{\underline{\rho}_n}.
\end{multline*}
Après simplification, on aura :
\begin{equation*}
{\bf K}_{\underline{\lambda}_n}\cdot {\bf K}_{\underline{\delta}_n}=
\sum_{\rho\text{ partition} \atop \max{(|\lambda|,|\delta|)}\leq|\rho|\leq \min{(|\lambda|+|\delta|,n)}}b_{\lambda\delta}^{\rho}\frac{(|\rho|)_{|\bar{\rho}|}}{|\lambda|!|\delta|!}2^{n+|\rho|-|\lambda|-|\delta|}n!(n-|\bar{\rho}|)_{m_1(\rho)}{\bf K}_{\underline{\rho}_n}.
\end{equation*}
\begin{obs} Une partition $\rho$ tel que $|\rho|\leq \min{(|\lambda|+|\delta|,n)}$ peut s'écrire d'une manière unique
$\rho=\tau \cup (1^j),$
où $\tau$ est une partition propre et $j\leq \min{(|\lambda|+|\delta|,n)}-|\tau|$.
\end{obs}
En utilisant ce fait, le produit peut s'écrire ainsi : 
$${\bf K}_{\underline{\lambda}_n}\cdot {\bf K}_{\underline{\delta}_n}=\sum_{\tau\text{ partition propre}\atop {|\tau|\leq \min{(|\lambda|+|\delta|,n)}}}\alpha_{\lambda\delta}^{\tau}(n){\bf K}_{\underline{\tau}_n},$$
où
\begin{eqnarray}\label{alpha}
\alpha_{\lambda\delta}^{\tau}(n)&=&\frac{1}{|\lambda|!|\delta|!}\sum_{j=0}^{\min{(|\lambda|+|\delta|,n)}-|\tau|}b_{\lambda\delta}^{\tau\cup (1^{j})}n!(n-|\tau|)_{j}(|\tau|+j)_{|\tau|}2^{n+|\tau|+j-|\lambda|-|\delta|}\\
&=&\frac{2^nn!}{|\lambda|!|\delta|!}\sum_{j=0}^{|\lambda|+|\delta|-|\tau|}b_{\lambda\delta}^{\tau\cup (1^{j})}(n-|\tau|)_{j}(|\tau|+j)_{|\tau|}2^{|\tau|+j-|\lambda|-|\delta|}. \nonumber
\end{eqnarray}
Le changement d'indice dans la dernière égalité vient du fait que si $n< |\lambda|+|\delta|$, on a :
$$(n-|\tau|)_{j}=0 \text{ pour tout $j$ avec } n-|\tau|< j\leq |\lambda|+|\delta|-|\tau|,$$
ce qui termine la preuve du Théorème \ref{Theorem 2.1}. \bigskip

On obtient aussi que la partie polynomiale de certains coefficients de structure est constante, en particulier on a le corollaire suivant.
\begin{cor}
Si $\lambda$, $\delta$ et $\rho$ sont trois partitions propres tel que $|\rho|=|\lambda|+|\delta|$, alors :
$$\alpha_{\lambda\delta}^{\rho}(n)=b_{\lambda\delta}^{\rho}\frac{|\rho|!}{|\lambda|!|\delta|!}2^nn!=c_{\lambda\delta}^\rho(|\rho|)\frac{|\lambda|!|\delta|!}{(|\lambda|+|\delta|)!}2^nn!.$$
\end{cor}
\begin{ex}\label{ex:produit_hecke_algebra}
On rappelle que le produit $T_{(2)}\pr T_{(2)}$ est donné dans l'Exemple \ref{example T2*T2} ce qui nous donne le produit ${\bf K}_{\underline{(2)}_n}\cdot {\bf K}_{\underline{(2)}_n}$ pour tout $n\geq 4.$ En utilisant la formule \eqref{alpha}, on obtient :
$${\bf K}_{\underline{(2)}_n}\cdot {\bf K}_{\underline{(2)}_n}=2^nn!n(n-1){\bf K}_{\underline{\emptyset}_n}+2^{n}n!{\bf K}_{\underline{(2)}_n}+2^{n}n!3{\bf K}_{\underline{(3)}_n}+2^{n}n!2{\bf K}_{\underline{(2^2)}_n}.$$
\end{ex}

\setcounter{subsection}{2}
\subsection{Objets aléatoires} La propriété de polynomialité des coefficients de structure de l'algèbre de Hecke de la paire $(\mathcal{S}_{2n},\mathcal{B}_n)$ est intéressante pour l'étude du comportement asymptotique des diagrammes de Young comme dans le cas du centre de l'algèbre du groupe symétrique. On n'entre pas dans les détails ici. Le lecteur est invité à regarder le papier \cite{2014arXiv1402.4615D} de Do{\l}{\c e}ga et F{\'e}ray dans lequel les auteurs traitent un cas général dont le centre de l'algèbre du groupe symétrique et l'algèbre de Hecke de la paire $(\mathcal{S}_{2n},\mathcal{B}_n)$ sont des cas particuliers. Ils présentent, pour le cas de $\mathbb{C}[\mathcal{B}_n\setminus \mathcal{S}_{2n}/\mathcal{B}_n]$ une étude de comportement asymptotique des diagrammes de Young, avec une généralisation de la mesure de Plancherel, similaire à celle détaillée dans la Section \ref{sec:appl_diag_de_young} dans le cas du centre de l'algèbre du groupe symétrique.

Au début de la Section \ref{sec:appl_diag_de_young}, on a donné une mesure de probabilité sur $\hat{G}$ et on a étudié une variable aléatoire définie avec les caractères irréductibles normalisés grâce aux résultats de la Section \ref{sec:Th_de_rep}. Dans le cas général des paires de Gelfand et grâce aux résultats qu'on a donné dans la Section \ref{sec:paire_des_Gelfand}, il est naturel de poser la question suivante:
 
\begin{que}
Si $(G,K)$ est une paire de Gelfand, peut-on donner des analogues aux trois premières propositions de la Section \ref{sec:appl_diag_de_young} ?
\end{que}

\section{Filtrations sur $\mathcal{A'}_\infty$ et degré des coefficients de structure}\label{sec:filtr}

Dans le Théorème \ref{Theorem 2.1}, on a donné une propriété de polynomialité pour les coefficients de structure de l'algèbre de Hecke de la paire $(\mathcal{S}_{2n},\mathcal{B}_n).$ Pour obtenir plus d'informations sur le degré des polynômes qui apparaissent, on va étudier plusieurs filtrations sur l'algèbre $\mathcal{A}'_\infty$ dans cette section. Les filtrations qu'on donne ici sont des analogues des filtrations sur l'algèbre $\mathcal{A}_\infty$ données dans la Section \ref{sec:filt}. Elles nous donnent des majorations du degré des polynômes des coefficients de structure de l'algèbre de Hecke de la paire \hecke.  

D'après la formule du produit des éléments de base de $\mathcal{A}'_\infty$, donnée dans le Corollaire \ref{corollary 3.2}, on peut voir que la fonction
\begin{equation*}
\deg'_1(T_\lambda)=|\lambda|
\end{equation*} 
définit une filtration sur $\mathcal{A}'_\infty$.

Pour obtenir d'autres filtrations sur $\mathcal{A}'_\infty$, on donne une décomposition d'une bijection partielle quelconque en bijections partielles de coset-type $(1)$ ou $(2).$ On appelle \textit{cycle} de taille $r+1$ une bijection partielle de coset-type $(r+1)$, où $r$ est un entier positif. Un cycle $\mathcal{C}$ de taille $r+1$ est écrit ainsi (voir \cite[page 2480]{Aker20122465}):
\begin{equation}\label{genericform}
\mathcal{C}=(c_1,c_2:c_3,c_4:\cdots:c_{4r+1},c_{4r+2}:c_{4r+3},c_{4r+4}).
\end{equation}
Ça signifie que :
\begin{align*}
&\lbrace c_{2i+1},c_{2i+2}\rbrace=p(k_i)\text{ pour une $k_i$} &\text{$i=0,\cdots,2r+1$}\\
&\mathcal{C}(c_1)=c_{4r+4}&\\
&\mathcal{C}(c_{4l+1})=c_{4l}&\text{$l=1,\cdots,r$}\\
&\mathcal{C}(c_{4l+2})=c_{4l+3}&\text{$l=0,\cdots,r$}
\end{align*}
\begin{ex}
Avec cette notation, le cycle le plus long de la Figure \ref{fig:coset-type}, que l'on redessine dans la Figure \ref{plus_long_cycle}, s'écrit ainsi :
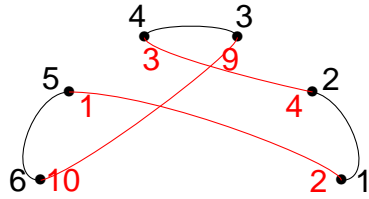
\begin{figure}[htbp]
\begin{center}
\begin{tikzpicture}
\foreach \angle / \label in
{ 0/2, 36/4, 72/9, 108/3, 144/1, 180/10}
{
\node at (\angle:2cm) {\footnotesize $\bullet$};
\draw[red] (\angle:1.7cm) node{\textsf{\label}};
}
\foreach \angle / \label in
{ 0/1, 36/2, 72/3, 108/4, 144/5, 180/6}
{\draw (\angle:2.3cm) node{\textsf{\label}};}
\draw (0:2cm) .. controls + (5mm,0) and +(.5,-.1) .. (36:2cm);
\draw[red] (36:2cm) .. controls + (-.5,.1) and + (0,-.2) .. (108:2);
\draw (108:2cm) .. controls + (0,.2) and + (0,.2) .. (72:2);
\draw[red]	(72:2)		  .. controls + (0,-.2) and + (+.5,.1) .. (180:2);
\draw	(180:2)		  .. controls + (-5mm,0) and + (-0.5,0) .. (144:2);
\draw[red]	(144:2)		  .. controls + (1.2,0) and + (0,.2) .. (0:2);
\end{tikzpicture}
\caption{Le plus long cycle de la Figure \ref{fig:coset-type}.}
\label{plus_long_cycle}
\end{center}
\end{figure}
\begin{equation*}
(1,2:4,3:4,3:9,10:6,5:1,2)
\end{equation*}
\end{ex}
\begin{notation}Pour une bijection partielle $\alpha$ et $x\in \mathcal{D}_n$, on écrit $\alpha \in x$ pour dire que le coefficient de $\alpha$ dans $x$ est non-nul.
\end{notation}
\begin{obs}\label{observation}
Soient $\alpha$, $\beta$ et $\gamma$ trois bijections partielles tel que $\gamma\in \alpha \pr \beta$. Si $\alpha\in x$ et $\beta\in y$ où $x$ et $y$ sont deux éléments de $\mathcal{D}_n$ avec des coefficients positifs, alors $\gamma\in x\pr y$.
\end{obs}
\begin{lem}\label{decompb}
Pour un cycle $\mathcal{C}$ de taille $r+1$, il existe $r$ bijections partielles $\tau_1,\cdots,\tau_r$ de coset-type $(2)$, tel que $\mathcal{C}\in \tau_1\pr \cdots\pr \tau_r$.
\end{lem}
\begin{proof}
Soit $\mathcal{C}$ un cycle de taille $r+1$ écrit comme dans (\ref{genericform}). On peut vérifier que $\mathcal{C}$ s'écrit :
\begin{equation*}
\mathcal{C}=\mathcal{K}\circ \mathcal{J},
\end{equation*} 
où
\begin{equation*}
\mathcal{K}=(c_3,c_4:c_4,c_3)\cdots(c_{4r-5},c_{4r-4}:c_{4r-4},c_{4r-5})(c_{4r-1},c_{4r}:c_{4r+4},c_{4r+3}:c_{4r+3},c_{4r+4}:c_{4r},c_{4r-1})
\end{equation*}
et
\begin{equation*}
\mathcal{J}=(c_1,c_2:c_3,c_4:\cdots:c_{4r-3},c_{4r-2}:c_{4r-1},c_{4r})(c_{4r+1},c_{4r+2}:c_{4r+3},c_{4r+4}).
\end{equation*}
Ainsi, si on note par $\tau_1$ le cycle de taille $2$ de $\mathcal{K}$ et par $\mathcal{C}_1$ le cycle de taille $r$ de $\mathcal{J}$, on aura :
\begin{align*}
\mathcal{C}\in \tau_1\pr \mathcal{C}_1&\text{, $ct(\tau_1)=(2)$ et $ct(\mathcal{C}_1)=(r)$}.
\end{align*}
De la même manière, on peut écrire :
\begin{align*}
\mathcal{C}_1\in \tau_2\pr \mathcal{C}_2&\text{ avec $ct(\tau_2)=(2)$ et $ct(\mathcal{C}_2)=(r-1)$}.
\end{align*} 
En utilisant l'Observation \ref{observation}, on obtient :
\begin{align*}
\mathcal{C}\in \tau_1\pr\tau_2\pr \mathcal{C}_2.
\end{align*} 
Donc, par itération on a bien :
\begin{align*}
\mathcal{C}\in \tau_1\pr\tau_2\pr \cdots\pr\tau_r&\text{ avec $ct(\tau_i)=(2)$ pour tout $i=1,\cdots, r$}.
\end{align*} 
Ce qui termine la preuve du lemme.
\end{proof}
\begin{lem}\label{decompb2}
Pour toute bijection partielle $\tau$ de coset-type $\rho=(\rho_1,\cdots,\rho_l)$, il existe $r$ bijections partielles $\tau_1,\cdots,\tau_r$ de coset-type $(2)$ et $m_1(\rho)$ bijections partielles $\beta_1,\cdots,\beta_{m_1(\rho)}$ de coset-type $(1)$ tel que :
\begin{equation*}
\tau\in \tau_1\pr \cdots \pr \tau_r \pr \beta_1\pr \cdots\pr \beta_{m_1(\rho)},
\end{equation*}
où $r=|\rho|-m_1(\rho)$.
\end{lem}
\begin{proof}
Chaque cycle dans le graphe $\Gamma(\tau)$ associé à $\tau$ peut être vu comme une bijection partielle. Soient $\beta_1,\cdots,\beta_{m_1(\rho)}$ les $m_1(\rho)$ bijections partielles correspondants aux cycles de taille $1$ de $\tau$. Les autres $m(\rho)=l(\rho)-m_1(\rho)$ bijections partielles correspondant aux cycles de $\tau$ de taille plus grande que $1$ sont nommées par $\alpha_1,\cdots,\alpha_{m(\rho)}$. La taille de $\alpha_i$ est $\rho_i$. D'après le Lemme \ref{decompb}, pour tout $\alpha_i$ on peut écrire :
\begin{align*}
\alpha_i\in \tau^i_1\pr \cdots \pr \tau^i_{\rho_i-1}&\text{ avec $ct(\tau^i_j)=(2),$ $j=1,\cdots,\rho_i-1$.}
\end{align*}
Puisque \begin{align*}
\tau\in \alpha_1\pr \cdots \pr \alpha_{m(\rho)}\pr \beta_1\pr\cdots\pr\beta_{m_1(\rho)},
\end{align*} on a :
\begin{align*}
\tau\in \tau^1_1\pr \cdots \pr \tau^1_{\rho_1-1}\pr \cdots \pr \tau^{m(\rho)}_1\pr \cdots \pr \tau^{m(\rho)}_{\rho_{m(\rho)}-1}\pr \beta_1\pr\cdots\pr\beta_{m_1(\rho)}.
\end{align*}
Le nombre de $\tau$ qui apparaissent dans cette décomposition est $(\rho_1-1)+\cdots+(\rho_{m(\rho)}-1)=|\rho|-m_1(\rho)-l(\rho)+m_1(\rho)=|\rho|-l(\rho)$, ce qui termine la preuve.
\end{proof}
Pour une bijection partielle $\tau$ de coset-type $\rho$, on définit les applications suivantes :
\begin{align*}
&\deg'_2(\tau)=|\rho|-l(\rho)\\
&\deg'_3(\tau)=|\rho|-l(\rho)+m_1(\rho).
\end{align*}
En considérant la décomposition de $\tau$ donnée dans le Lemme \ref{decompb2}. on obtient :
\begin{align*}
&\deg'_2(\tau_i)=1,\\
&\deg'_2(\beta_i)=0,\\
&\deg'_3(\tau_i)=1,\\
&\deg'_3(\beta_i)=1.
\end{align*}
Donc
\begin{align*} 
&\deg'_2(\tau)=|\rho|-l(\rho)=\deg'_2(\tau_1)+\cdots +\deg'_2(\tau_r)+\deg'_2(\beta_1)+\cdots +\deg'_2(\beta_{m_1(\rho)}),
\end{align*}
et 
\begin{align*} 
&\deg'_3(\tau)=|\rho|-l(\rho)+m_1(\rho)=\deg'_3(\tau_1)+\cdots +\deg'_3(\tau_r)+\deg'_3(\beta_1)+\cdots +\deg'_3(\beta_{m_1(\rho)}).
\end{align*}
\begin{prop}\label{filtrations}
Les fonctions
\begin{align*}
&\deg'_2(T_\rho)=|\rho|-l(\rho),\\
&\deg'_3(T_\rho)=|\rho|-l(\rho)+m_1(\rho),
\end{align*}
définissent des filtrations sur $\mathcal{A}'_\infty$.
\end{prop}
\begin{proof}
Soient $\sigma$ et $\tau$ deux bijections partielles de coset-type $\lambda$ et $\rho$. On doit montrer que :
\begin{align*}
&\deg'_i(\sigma\pr \tau)\leq \deg'_i(\sigma)+\deg'_i(\tau),&\text{pour $i=2,3$,}
\end{align*}
où $\displaystyle{\deg'_i(\sigma\pr \tau)=\max_{\alpha\in \sigma\pr \tau}\deg'_i(\alpha)}$.\\
D'après le Lemme \ref{decompb2}, $\tau$ peut se décomposer en cycles de taille $2$ et $1$. Autrement dit, il existe $r=|\rho|-m_1(\rho)$ bijections partielles $\tau_1,\cdots,\tau_r$ de coset-type $(2)$ et $m_1(\rho)$ bijections partielles $\beta_1,\cdots,\beta_{m_1(\rho)}$ de coset-type $(1)$ tel que :
\begin{equation*}
\tau\in \tau_1\pr \cdots \pr \tau_r \pr \beta_1\pr \cdots\pr \beta_{m_1(\rho)}.
\end{equation*}

Pour une bijection partielle $\alpha$ tel que $\alpha\in \sigma\pr \tau$, on a $\alpha\in \sigma\pr\tau_1\pr \cdots \pr \tau_r \pr \beta_1\pr \cdots\pr \beta_{m_1(\rho)}$ d'après l'Observation \ref{observation} puisque $\tau\in \tau_1\pr \cdots \pr \tau_r \pr \beta_1\pr \cdots\pr \beta_{m_1(\rho)}$. Donc,
\begin{align*}
\deg'_i(\sigma\pr \tau)&\leq \deg'_i(\sigma\pr \tau_1\pr \cdots \pr \tau_r \pr \beta_1\pr \cdots\pr \beta_{m_1(\rho)})&\text{pour $i=2,3.$}
\end{align*}
Si pour tout cycle $\mathcal{C}$ de taille $2$ ou $1$ et pour toute bijection partielle $\theta$ on a,
\begin{align}\label{filt}
&\deg'_i(\theta\pr \mathcal{C})\leq \deg'_i(\theta)+\deg'_i(\mathcal{C}), & \text{pour $i=2,3$,}
\end{align} alors on peut écrire :
\begin{align*}
\deg'_i(\sigma\pr \tau)&\leq \deg'_i(\sigma\pr \tau_1\pr \cdots \pr \tau_r \pr \beta_1\pr \cdots\pr \beta_{m_1(\rho)})\\
&\leq \deg'_i(\sigma\pr \tau_1\pr \cdots \pr \tau_r \pr \beta_1\pr \cdots\pr \beta_{m_1(\rho)-1})+\deg'_i(\beta_{m_1(\rho)})\\
&\vdots \\
&\leq \deg'_i(\sigma)+\deg'_i(\tau_1)+ \cdots + \deg'_i(\tau_r) + \deg'_i(\beta_1)+ \cdots + \deg'_i(\beta_{m_1(\rho)})\\
&\leq \deg'_i(\sigma)+\deg'_i(\tau).
\end{align*}
Par suite, il suffit de démontrer la formule (\ref{filt}). Soit $\theta$ une bijection partielle de coset-type $\delta$. Si $$\mathcal{C}=(c_1,c_2:c_3,c_4)$$ est un cycle de taille $1$, on a deux cas :
\begin{enumerate}
\item Si $\lbrace c_3, c_4\rbrace$ est dans l'ensemble de départ de $\theta$: dans ce cas la bijection partielle qui apparaît dans le développement du produit $\theta\pr \mathcal{C}$ possède le même coset-type que $\theta$ et on a 
\begin{align*}
&\deg'_i(\theta\pr \mathcal{C})\leq \deg'_i(\theta)+\deg'_i(\mathcal{C}), & \text{ pour $i=2,3$}.
\end{align*}
\item Sinon, toutes les bijections partielles qui apparaissent dans le développement du produit $\theta\pr \mathcal{C}$ possèdent le coset-type $\delta\cup (1)$. Donc, on peut vérifier que $\deg'_i(\theta\pr \mathcal{C})\leq \deg'_i(\theta)+\deg'_i(\mathcal{C})$, pour $i=2,3$.
\end{enumerate} 
Maintenant, supposons que $\mathcal{C}=(c_1,c_2:c_3,c_4:c_5,c_6:c_7,c_8)$ est un cycle de taille $2$, la figure de $\mathcal{C}$ est représentée par la Figure \ref{cycle C}.
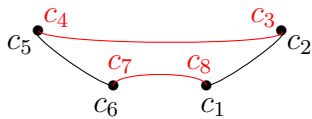
\begin{figure}[htbp]
\begin{center}
\begin{tikzpicture}
\foreach \angle / \label in
{216/$c_4$,
252/$c_7$, 288/$c_8$, 324/$c_3$}
{
\node at (\angle:2cm) {\footnotesize $\bullet$};
\draw[red] (\angle:1.7cm) node{\textsf{\label}};
}
\foreach \angle / \label in
{216/$c_5$,
252/$c_6$, 288/$c_1$, 324/$c_2$}
{\draw (\angle:2.3cm) node{\textsf{\label}};}
\draw[red] (216:2cm) .. controls + (0.2,-0.2) and +(0.2,-0.2) .. (324:2cm);
\draw (216:2cm) .. controls + (-0.2,0) and +(-0.2,0) .. (252:2cm);
\draw[red]	(252:2cm)		    .. controls + (0,0.2) and + (0,0.2) .. (288:2);
\draw	(288:2)		    .. controls + (0.2,0) and + (0.2,0) .. (324:2);
\end{tikzpicture}
\caption{Le cycle $\mathcal{C}$.}
\label{cycle C}
\end{center}
\end{figure}\\
On a quatre cas à voir. On donne pour chacun d'entre eux le résultat général sans les détails de la preuve. Les preuves consistent à compter les compositions des permutations.   
\begin{enumerate}
\item $\lbrace c_3,c_4\rbrace$ et $\lbrace c_7,c_8\rbrace$ n'apparaissent pas dans l'ensemble de départ de $\theta$:

Dans ce cas, les bijections partielles qui apparaissent dans le développement de $\theta\pr \mathcal{C}$ possèdent $\delta\cup (2)$ comme coset-type.
\item Un des ensembles $\lbrace c_3,c_4\rbrace$ et $\lbrace c_7,c_8\rbrace$ (par exemple $\lbrace c_3,c_4\rbrace$) apparaît dans l'ensemble de départ d'un cycle $\omega$ de $\theta$ et l'autre n'apparaît pas.

Supposons que $\omega$ est représentée comme dans la Figure \ref{figure in second case}.
\begin{figure}[htbp]
\begin{center}
\begin{tikzpicture}
\foreach \angle / \label in
{ 0/$\omega_{20}$, 36/$\omega_{3}$, 72/$\omega_{4}$, 108/$\omega_{7}$, 144/$\omega_{8}$, 180/$\omega_{11}$, 216/$\omega_{12}$,
252/$\omega_{15}$, 288/$\omega_{16}$, 324/$\omega_{19}$}
{
\node at (\angle:2cm) {\footnotesize $\bullet$};
\draw[red] (\angle:1.7cm) node{\textsf{\label}};
}
\foreach \angle / \label in
{ 0/$\omega_1$, 36/$\omega_2$, 72/$\omega_5=c_3$, 108/$\omega_6=c_4$, 144/$\omega_9$, 180/$\omega_{10}$, 216/$\omega_{13}$,
252/$\omega_{14}$, 288/$\omega_{17}$, 324/$\omega_{18}$}
{\draw (\angle:2.3cm) node{\textsf{\label}};}
\draw (0:2cm) .. controls + (.2,.1) and +(.2,.1) .. (36:2cm);
\draw[red] (36:2cm) .. controls + (-.5,.1) and + (0,-.2) .. (72:2);
\draw (72:2cm) .. controls + (0,.2) and + (0,.2) .. (108:2);
\draw[red]	(108:2)		  .. controls + (0,-.2) and + (.5,.1) .. (144:2);
\draw	(144:2)		  .. controls + (-.2,0) and + (-0.2,0) .. (180:2);
\draw[red]	(180:2)		  .. controls + (.2,0) and + (0.2,0) .. (216:2);
\draw (216:2cm) .. controls + (-0.2,0) and +(-0.2,0) .. (252:2cm);
\draw[red]	(252:2cm)		    .. controls + (0,0.2) and + (0,0.2) .. (288:2);
\draw	(288:2)		    .. controls + (0.2,0) and + (0.2,0) .. (324:2);
\draw[red]	(324:2)		    .. controls + (-0.2,0) and + (-0.2,0) .. (0:2);
\end{tikzpicture}
\caption{Le cycle $\omega$ de $\theta$.}
\label{figure in second case}
\end{center}
\end{figure}
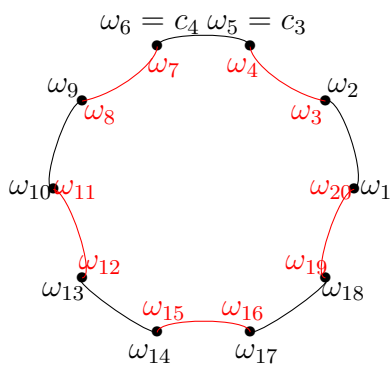\\
Alors, le cycle dont la forme est comme dans la Figure \ref{form in second case} apparaît dans le développement du produit $\theta\pr \mathcal{C}$.
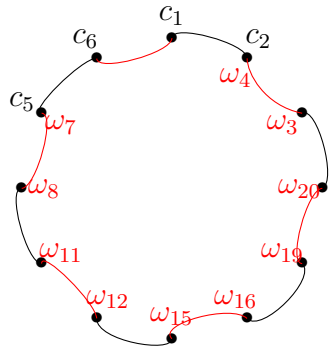
\begin{figure}[htbp]
\begin{center}
\begin{tikzpicture}
\foreach \angle / \label in
{ 0/$\omega_{20}$, 30/$\omega_{3}$, 60/$\omega_{4}$, 90/, 120/, 150/$\omega_{7}$, 180/$\omega_{8}$,
210/$\omega_{11}$, 240/$\omega_{12}$, 270/$\omega_{15}$, 300/$\omega_{16}$, 330/$\omega_{19}$}
{
\node at (\angle:2cm) {\footnotesize $\bullet$};
\draw[red] (\angle:1.7cm) node{\textsf{\label}};
}
\foreach \angle / \label in
{ 0/, 30/, 60/$c_2$, 90/$c_1$, 120/$c_6$, 150/$c_5$, 180/,
210/, 240/, 270/, 300/, 330/}
{\draw (\angle:2.3cm) node{\textsf{\label}};}
\draw (0:2cm) .. controls + (.2,.1) and +(.2,.1) .. (30:2cm);
\draw[red] (30:2cm) .. controls + (-.5,.1) and + (0,-.2) .. (60:2);
\draw (60:2cm) .. controls + (0,.2) and + (0,.2) .. (90:2);
\draw[red]	(90:2)		  .. controls + (0,-.2) and + (0,-.2) .. (120:2);
\draw	(120:2)		  .. controls + (-.2,0) and + (-0.2,0) .. (150:2);
\draw[red]	(150:2)		  .. controls + (.2,0) and + (0.2,0) .. (180:2);
\draw (180:2cm) .. controls + (-0.2,0) and +(-0.2,0) .. (210:2cm);
\draw[red]	(210:2cm)		    .. controls + (0,0.2) and + (0,0.2) .. (240:2);
\draw	(240:2)		    .. controls + (0,-0.3) and + (0,-0.2) .. (270:2);
\draw[red]	(270:2)		    .. controls + (-0.2,0.2) and + (-0.2,0.2) .. (300:2);
\draw	(300:2)		    .. controls + (0.2,-0.2) and + (0.2,-0.2) .. (330:2);
\draw[red]	(330:2)		    .. controls + (-0.2,0) and + (-0.2,0) .. (0:2);
\end{tikzpicture}
\caption{La forme du cycle.}
\label{form in second case}
\end{center}
\end{figure}\\
Il faut noter que certaines étiquettes extérieures n'apparaissent pas sur la figure. Pour expliquer ce fait, on rappelle que le produit $\alpha_1\pr \alpha_2$ des deux bijections partielles $\alpha_1$ et $\alpha_2$ est définie comme la moyenne de certaines bijections partielles $\alpha$. Lorsque l'extrémité  d'une arête ne possède pas une étiquette extérieure, ça veut dire que les éléments de n'importe quelle paire $p(k)$ différent du $\lbrace c_1,c_2\rbrace$ et $\lbrace c_5,c_6\rbrace$ peuvent être utiliser comme étiquettes et qu'on doit prendre la moyenne de toutes les possibilités. On va utiliser la même convention dans les deux derniers cas.\\
Donc, dans ce cas, le coset-type de chaque bijection partielle qui apparaît  dans le développement du produit $\theta\pr \mathcal{C}$ possède le même nombre des parts que $\delta$ et sa taille est égale à $|\delta|+1$.
\item Les deux ensembles $\lbrace c_3,c_4\rbrace$ et $\lbrace c_7,c_8\rbrace$ apparaissent dans l'ensemble de départ d'un même cycle $\omega$ de $\theta$.\\
Considérons les étiquettes extérieures de $\omega$. Parmi elles se trouvent $c_3,c_4,c_7$ et $c_8$ et on sait que $c_3$ et $c_4$ (resp. $c_7$ et $c_8$) apparaissent consécutivement. Donc il y a deux cas à considérer. Les étiquettes apparaissent soit dans l'ordre cyclique $c_3,c_4,\cdots,c_7,c_8,$ soit $c_3,c_4,\cdots,c_8,c_7$. Ces deux cas sont représentés par les Figures \ref{figure in third case} et \ref{second figure in third case}.
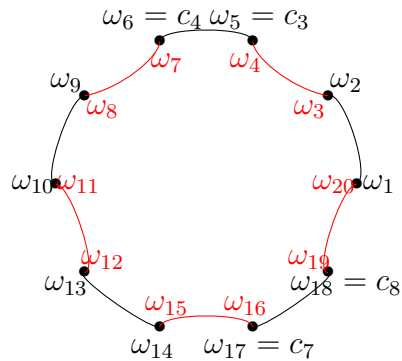
\begin{figure}[htbp]
\begin{center}
\begin{tikzpicture}
\foreach \angle / \label in
{ 0/$\omega_{20}$, 36/$\omega_{3}$, 72/$\omega_{4}$, 108/$\omega_{7}$, 144/$\omega_{8}$, 180/$\omega_{11}$, 216/$\omega_{12}$,
252/$\omega_{15}$, 288/$\omega_{16}$, 324/$\omega_{19}$}
{
\node at (\angle:2cm) {\footnotesize $\bullet$};
\draw[red] (\angle:1.7cm) node{\textsf{\label}};
}
\foreach \angle / \label in
{ 0/$\omega_1$, 36/$\omega_2$, 72/$\omega_5=c_3$, 108/$\omega_6=c_4$, 144/$\omega_9$, 180/$\omega_{10}$, 216/$\omega_{13}$,
252/$\omega_{14}$, 288/$\omega_{17}=c_7$, 324/$\omega_{18}=c_8$}
{\draw (\angle:2.3cm) node{\textsf{\label}};}
\draw (0:2cm) .. controls + (.2,.1) and +(.2,.1) .. (36:2cm);
\draw[red] (36:2cm) .. controls + (-.5,.1) and + (0,-.2) .. (72:2);
\draw (72:2cm) .. controls + (0,.2) and + (0,.2) .. (108:2);
\draw[red]	(108:2)		  .. controls + (0,-.2) and + (.5,.1) .. (144:2);
\draw	(144:2)		  .. controls + (-.2,0) and + (-0.2,0) .. (180:2);
\draw[red]	(180:2)		  .. controls + (.2,0) and + (0.2,0) .. (216:2);
\draw (216:2cm) .. controls + (-0.2,0) and +(-0.2,0) .. (252:2cm);
\draw[red]	(252:2cm)		    .. controls + (0,0.2) and + (0,0.2) .. (288:2);
\draw	(288:2)		    .. controls + (0.2,0) and + (0.2,0) .. (324:2);
\draw[red]	(324:2)		    .. controls + (-0.2,0) and + (-0.2,0) .. (0:2);
\end{tikzpicture}
\caption{La première forme possible de $\omega$.}
\label{figure in third case}
\end{center}
\end{figure}
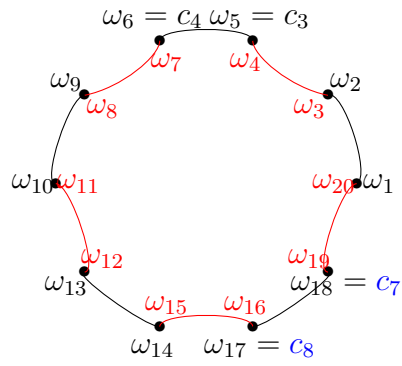
\begin{figure}[htbp]
\begin{center}
\begin{tikzpicture}
\foreach \angle / \label in
{ 0/$\omega_{20}$, 36/$\omega_{3}$, 72/$\omega_{4}$, 108/$\omega_{7}$, 144/$\omega_{8}$, 180/$\omega_{11}$, 216/$\omega_{12}$,
252/$\omega_{15}$, 288/$\omega_{16}$, 324/$\omega_{19}$}
{
\node at (\angle:2cm) {\footnotesize $\bullet$};
\draw[red] (\angle:1.7cm) node{\textsf{\label}};
}
\foreach \angle / \label in
{ 0/$\omega_1$, 36/$\omega_2$, 72/$\omega_5=c_3$, 108/$\omega_6=c_4$, 144/$\omega_9$, 180/$\omega_{10}$, 216/$\omega_{13}$,
252/$\omega_{14}$, 288/$\omega_{17}=\color{blue}{c_8}$, 324/$\omega_{18}=\color{blue}{c_7}$}
{\draw (\angle:2.3cm) node{\textsf{\label}};}
\draw (0:2cm) .. controls + (.2,.1) and +(.2,.1) .. (36:2cm);
\draw[red] (36:2cm) .. controls + (-.5,.1) and + (0,-.2) .. (72:2);
\draw (72:2cm) .. controls + (0,.2) and + (0,.2) .. (108:2);
\draw[red]	(108:2)		  .. controls + (0,-.2) and + (.5,.1) .. (144:2);
\draw	(144:2)		  .. controls + (-.2,0) and + (-0.2,0) .. (180:2);
\draw[red]	(180:2)		  .. controls + (.2,0) and + (0.2,0) .. (216:2);
\draw (216:2cm) .. controls + (-0.2,0) and +(-0.2,0) .. (252:2cm);
\draw[red]	(252:2cm)		    .. controls + (0,0.2) and + (0,0.2) .. (288:2);
\draw	(288:2)		    .. controls + (0.2,0) and + (0.2,0) .. (324:2);
\draw[red]	(324:2)		    .. controls + (-0.2,0) and + (-0.2,0) .. (0:2);
\end{tikzpicture}
\caption{La deuxième forme possible de $\omega$.}
\label{second figure in third case}
\end{center}
\end{figure}\\
Il faut noter que dans la Figure \ref{second figure in third case} , les étiquettes $c_7$ et $c_8$ sont échangées.\\
La forme des cycles qui apparaissent dans le développement du produit $\theta\pr \mathcal{C}$ dans les deux cas est donnée par les Figures \ref{form in the third case} et \ref{second form in the third case}.
\begin{figure}[htbp]
\begin{center}
\begin{tikzpicture}
\foreach \angle / \label in
{ 0/$\omega_{20}$, 36/$\omega_{3}$, 72/$\omega_{4}$, 108/$\omega_{19}$, 144/$\omega_{8}$, 180/$\omega_{11}$, 216/$\omega_{12}$,
252/$\omega_{15}$, 288/$\omega_{16}$, 324/$\omega_{7}$}
{
\node at (\angle:2cm) {\footnotesize $\bullet$};
\draw[red] (\angle:1.7cm) node{\textsf{\label}};
}
\foreach \angle / \label in
{ 0/, 36/, 72/$c_2$, 108/$c_1$, 144/, 180/, 216/,
252/, 288/$c_6$, 324/$c_5$}
{\draw (\angle:2.3cm) node{\textsf{\label}};}
\draw (0:2cm) .. controls + (.2,.1) and +(.2,.1) .. (36:2cm);
\draw[red] (36:2cm) .. controls + (-.5,.1) and + (0,-.2) .. (72:2);
\draw (72:2cm) .. controls + (0,.2) and + (0,.2) .. (108:2);
\draw[red]	(108:2)		  .. controls + (0,-.2) and + (.5,.1) .. (0:2);
\draw	(144:2)		  .. controls + (-.2,0) and + (-0.2,0) .. (180:2);
\draw[red]	(180:2)		  .. controls + (.2,0) and + (0.2,0) .. (216:2);
\draw (216:2cm) .. controls + (-0.2,0) and +(-0.2,0) .. (252:2cm);
\draw[red]	(252:2cm)		    .. controls + (0,0.2) and + (0,0.2) .. (288:2);
\draw	(288:2)		    .. controls + (0.2,0) and + (0.2,0) .. (324:2);
\draw[red]	(324:2)		    .. controls + (-0.2,0) and + (-0.2,0) .. (144:2);
\end{tikzpicture}
\caption{Un cycle d'une bijection partielle de $\theta\pr \mathcal{C}$ correspondant à la 1ère forme de $\omega$.}
\label{form in the third case}
\end{center}
\end{figure}
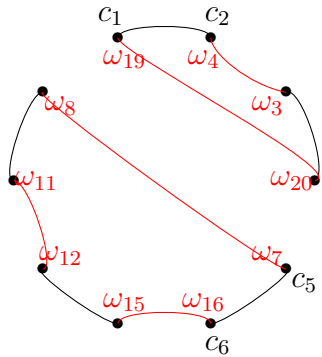
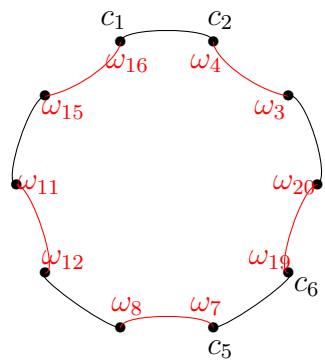
\begin{figure}[htbp]
\begin{center}
\begin{tikzpicture}
\foreach \angle / \label in
{ 0/$\omega_{20}$, 36/$\omega_{3}$, 72/$\omega_{4}$, 108/$\omega_{16}$, 144/$\omega_{15}$, 180/$\omega_{11}$, 216/$\omega_{12}$,
252/$\omega_{8}$, 288/$\omega_{7}$, 324/$\omega_{19}$}
{
\node at (\angle:2cm) {\footnotesize $\bullet$};
\draw[red] (\angle:1.7cm) node{\textsf{\label}};
}
\foreach \angle / \label in
{ 0/, 36/, 72/$c_2$, 108/$c_1$, 144/, 180/, 216/,
252/, 288/$c_5$, 324/$c_6$}
{\draw (\angle:2.3cm) node{\textsf{\label}};}
\draw (0:2cm) .. controls + (.2,.1) and +(.2,.1) .. (36:2cm);
\draw[red] (36:2cm) .. controls + (-.5,.1) and + (0,-.2) .. (72:2);
\draw (72:2cm) .. controls + (0,.2) and + (0,.2) .. (108:2);
\draw[red]	(108:2)		  .. controls + (0,-.2) and + (.5,.1) .. (144:2);
\draw	(144:2)		  .. controls + (-.2,0) and + (-0.2,0) .. (180:2);
\draw[red]	(180:2)		  .. controls + (.2,0) and + (0.2,0) .. (216:2);
\draw (216:2cm) .. controls + (-0.2,0) and +(-0.2,0) .. (252:2cm);
\draw[red]	(252:2cm)		    .. controls + (0,0.2) and + (0,0.2) .. (288:2);
\draw	(288:2)		    .. controls + (0.2,0) and + (0.2,0) .. (324:2);
\draw[red]	(324:2)		    .. controls + (-0.2,0) and + (-0.2,0) .. (0:2);
\end{tikzpicture}
\caption{Un cycle d'une bijection partielle de $\theta\pr \mathcal{C}$ correspondant à la 2ème forme de $\omega$.}
\label{second form in the third case}
\end{center}
\end{figure}\\
Donc, dans le premier cas le cycle est coupé en deux, donc le coset-type de chaque bijection partielle qui apparaît dans le développement du produit $\theta\pr \mathcal{C}$ possède la même taille que $\delta$ et $l(\delta)+1$ parts. Au contraire, dans le second, la taille des cycles ne change pas et le coset-type de chaque bijection partielle qui apparaît dans le développement du produit $\theta\pr \mathcal{C}$ possède la même taille que $\delta$ et le même nombre des parts. 
\item Les deux ensembles $\lbrace c_3,c_4\rbrace$ et $\lbrace c_7,c_8\rbrace$ apparaissent dans l'ensemble de départ des deux cycles différents de $\theta$.\\
Par exemple on prend les deux cycles représentés dans la Figure \ref{figure of the fourth case}.\\
\begin{figure}[htbp]
\begin{center}
\begin{tikzpicture}
\foreach \angle / \label in
{ 0/$\omega_{12}$, 36/$\omega_3$, 72/$\omega_4$, 108/$\omega_7$, 144/$\omega_{8}$, 180/$\omega_{11}$, 216/$t_8$,
252/$t_3$, 288/$t_4$, 324/$t_7$}
{
\node at (\angle:2cm) {\footnotesize $\bullet$};
\draw[red] (\angle:1.7cm) node{\textsf{\label}};
}
\foreach \angle / \label in
{ 0/$\omega_1$, 36/$\omega_2$, 72/$\omega_5=c_3$, 108/$\omega_6=c_4$, 144/$\omega_{9}$, 180/$\omega_{10}$, 216/$t_1=c_7$,
252/$t_2=c_8$, 288/$t_5$, 324/$t_6$}
{\draw (\angle:2.3cm) node{\textsf{\label}};}
\draw (0:2cm) .. controls + (.2,.1) and +(.2,.1) .. (36:2cm);
\draw[red] (36:2cm) .. controls + (-.5,.1) and + (0,-.2) .. (72:2);
\draw (72:2cm) .. controls + (0,.2) and + (0,.2) .. (108:2);
\draw[red]	(108:2)		  .. controls + (0,-.2) and + (.5,.1) .. (144:2);
\draw	(144:2)		  .. controls + (-.2,0) and + (-0.2,0) .. (180:2);
\draw[red]	(180:2)		  .. controls + (.2,0.2) and + (0.2,0.2) .. (0:2);
\draw[red] (216:2cm) .. controls + (0.2,-0.2) and +(0.2,-0.2) .. (324:2cm);
\draw (216:2cm) .. controls + (-0.2,0) and +(-0.2,0) .. (252:2cm);
\draw[red]	(252:2cm)		    .. controls + (0,0.2) and + (0,0.2) .. (288:2);
\draw	(288:2)		    .. controls + (0.2,0) and + (0.2,0) .. (324:2);
\end{tikzpicture}
\caption{Les deux cycles.}
\label{figure of the fourth case}
\end{center}
\end{figure}
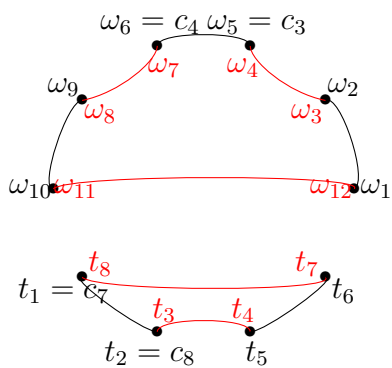
Un cycle dont la forme est représenté par la Figure \ref{form of the fourth case} apparaît dans le développement du produit $\theta\pr \mathcal{C}$.\\
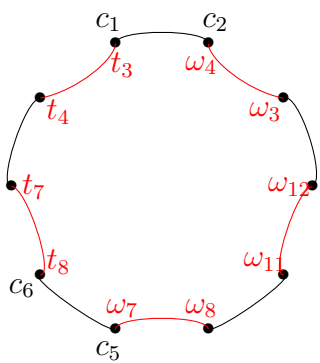
\begin{figure}[htbp]
\begin{center}
\begin{tikzpicture}
\foreach \angle / \label in
{ 0/$\omega_{12}$, 36/$\omega_{3}$, 72/$\omega_{4}$, 108/$t_{3}$, 144/$t_4$, 180/$t_7$, 216/$t_8$,
252/$\omega_{7}$, 288/$\omega_{8}$, 324/$\omega_{11}$}
{
\node at (\angle:2cm) {\footnotesize $\bullet$};
\draw[red] (\angle:1.7cm) node{\textsf{\label}};
}
\foreach \angle / \label in
{ 0/, 36/, 72/$c_2$, 108/$c_1$, 144/, 180/, 216/$c_6$,
252/$c_5$, 288/, 324/}
{\draw (\angle:2.3cm) node{\textsf{\label}};}
\draw (0:2cm) .. controls + (.2,.1) and +(.2,.1) .. (36:2cm);
\draw[red] (36:2cm) .. controls + (-.5,.1) and + (0,-.2) .. (72:2);
\draw (72:2cm) .. controls + (0,.2) and + (0,.2) .. (108:2);
\draw[red]	(108:2)		  .. controls + (0,-.2) and + (.5,.1) .. (144:2);
\draw	(144:2)		  .. controls + (-.2,0) and + (-0.2,0) .. (180:2);
\draw[red]	(180:2)		  .. controls + (.2,0) and + (0.2,0) .. (216:2);
\draw (216:2cm) .. controls + (-0.2,0) and +(-0.2,0) .. (252:2cm);
\draw[red]	(252:2cm)		    .. controls + (0,0.2) and + (0,0.2) .. (288:2);
\draw	(288:2)		    .. controls + (0.2,0) and + (0.2,0) .. (324:2);
\draw[red]	(324:2)		    .. controls + (-0.2,0) and + (-0.2,0) .. (0:2);
\end{tikzpicture}
\caption{Les deux cycles sont reliés dans toute bijection partielle de $\theta\pr \mathcal{C}.$}
\label{form of the fourth case}
\end{center}
\end{figure}
Dans ce cas, les deux cycles sont reliés pour former un cycle, donc le coset-type de chaque bijection partielle qui apparaît dans le développement du produit $\theta\pr \mathcal{C}$ possède la même taille que $\delta$ et $l(\delta)-1$ parts.
\end{enumerate}
Dans ces quatre cas, on peut vérifier qu'on a $deg_i(\theta\pr \mathcal{C})\leq deg_i(\theta)+deg_i(\mathcal{C})$, pour $i=2,3$ ce qui achève la preuve de la Proposition \ref{filtrations}.
\end{proof}

Ces filtrations nous permettent de majorer le degré des polynômes $f_{\lambda\delta}^{\rho}(n)$ qui apparaissent dans le Théorème \ref{Theorem 2.1}.
\begin{prop}
Soient $\lambda$, $\delta$ et $\rho$ trois partitions propres, le degré de $f_{\lambda\delta}^{\rho}(n)$ vérifie :
\begin{align*}
&\deg(f_{\lambda\delta}^{\rho}(n))\leq |\lambda|+|\delta|-|\rho|,\\
&\deg(f_{\lambda\delta}^{\rho}(n))\leq |\lambda|+|\delta|-|\rho|-l(\lambda)-l(\delta)+l(\rho).
\end{align*}
\end{prop}
\begin{proof}
D'après la preuve du Théorème \ref{Theorem 2.1}, le degré de $f_{\lambda\delta}^{\rho}(n)$ vérifie : 
\begin{align}
\deg(f_{\lambda\delta}^{\rho}(n))=\max_{b_{\lambda\delta}^{\rho\cup(1^j)}\neq 0} j.
\end{align}
De l'autre côté, puisque $\deg'_i$ est une filtration pour $i=1,2,3$, on obtient l'inégalité suivante :
\begin{align}
\deg'_i(T_\lambda\pr T_\delta)=\max_{\tau \text{ proper},j\geq 0 \atop{b_{\lambda\delta}^{\tau \cup (1^j)}\neq 0}} \deg'_i(T_{\tau\cup (1^j)})\leq \deg'_i(T_\lambda)+\deg'_i(T_\delta).
\end{align}
Donc on obtient les trois inégalités qui correspondent aux trois filtrations $\deg'_1$, $\deg'_2$ et $\deg'_3$:
\begin{align}
&\max_{\tau \text{ propre},j\geq 0 \atop{b_{\lambda\delta}^{\tau \cup (1^j)}\neq 0}} |\tau|+j\leq |\lambda|+|\delta|,&\\
&\max_{\tau \text{ propre},j\geq 0 \atop{b_{\lambda\delta}^{\tau \cup (1^j)}\neq 0}} |\tau|+j-l(\tau)-j\leq |\lambda|-l(\lambda)+|\delta|-l(\delta),&\\
&\max_{\tau \text{ propre},j\geq 0 \atop{b_{\lambda\delta}^{\tau \cup (1^j)}\neq 0}} |\tau|+j-l(\tau)-j+j\leq |\lambda|-l(\lambda)+|\delta|-l(\delta).
\end{align}
La deuxième inégalité ne donne aucune information sur le degré de $f_{\lambda\delta}^{\rho}(n)$. En utilisant le fait suivant avec la première et la troisième égalité,
\begin{align*}
\deg(f_{\lambda\delta}^{\rho}(n))=\max_{b_{\lambda\delta}^{\rho\cup(1^j)}\neq 0} j\leq \max_{\tau \text{ propre},j\geq 0 \atop{b_{\lambda\delta}^{\tau \cup (1^j)}\neq 0}}j,
\end{align*}
on obtient notre résultat.
\end{proof}
\section{Relation entre $\mathcal{A'}_\infty$ et l'algèbre des fonctions symétriques décalées d'ordre $2$}\label{sec:isom_fct_decal_2}

La paire $(\mathcal{S}_{2n},\mathcal{B}_n)$ est une paire de Gelfand (voir \cite[Chapitre VII, 2]{McDo}) et ses fonctions sphériques zonales sont indexées par les partitions de $n$. Elles sont notées $\omega^{\rho}$ et définies par : $$\omega^{\rho}(x)=\frac{1}{|\mathcal{B}_n|}\sum_{k\in \mathcal{B}_n}\chi^{2\rho}(xk),$$
pour $x\in \mathcal{S}_{2n}$, où $\chi^{2\rho}$ est le caractère du $\mathcal{S}_{2n}$-module irréductible correspondant à $2\rho:=(2\rho_1,2\rho_2,\cdots)$. Deux permutations $x$ et $y$ dans la même $\mathcal{B}_n$-double-classe $K_\lambda$ possèdent la même image par $\omega^\rho$ notée par $\omega^\rho_\lambda$.

Quand $\alpha=2$, les polynômes de Jack sont liés aux fonctions sphériques zonales de la paire de Gelfand $(\mathcal{S}_{2n},\mathcal{B}_n)$ par l'équation suivante (cf. \cite[page 408]{McDo}):
\begin{equation}\label{jackpolynomial2}J_{\rho}(2)=|\mathcal{B}_{n}|\sum_{|\lambda|=n}z_{2\lambda}^{-1}\omega_{\lambda}^{\rho}p_\lambda,
\end{equation}
pour toute partition $\rho$ de $n$. La définition des polynômes de Jack est donnée dans la Section \ref{sec:alg_fct_sym}. La formule de l'équation \eqref{jackpolynomial2} peut être vue comme un analogue dans le cas où $\alpha=2$ de la formule de Frobenius. Rappelons que les nombres $\theta^{\rho}_{\lambda}(\alpha)$ sont par définition, les coefficients qui apparaissent dans le développement du polynômes de Jack $J_{\rho}(\alpha)$ en fonction des fonctions puissances $p_\lambda.$ Les équations (\ref{jackpolynomial1}) et (\ref{jackpolynomial2}) nous donnent l'égalité suivante dans le cas où $\alpha=2$:
\begin{equation*}\label{}
\theta^{\rho}_{\lambda}(2)=|\mathcal{B}_{|\rho|}|z_{2\lambda}^{-1}\omega_{\lambda}^{\rho}.
\end{equation*}

La définition de l'algèbre des fonctions symétriques décalées $\Lambda^*(\alpha)$ ainsi que celle de la fonction $sh_\alpha$ ont été données dans la Section \ref{sec:fct_sym_dec}. Dans le cas où $\alpha=2,$ on a le théorème suivant.
\begin{theoreme}\label{isom:A'_infini_et_fct_decal}
L'application linéaire $F:\widetilde{\mathcal{A}'_{\infty}}\longrightarrow \Lambda^*(2)$ définie par :
\begin{equation}\label{isomorphism}F(T_{\lambda})=2^{2|\lambda|-l(\lambda)}|\lambda|!\frac{sh_2(p_{\lambda})}{z_{\lambda}},
\end{equation}
est un isomorphisme d'algèbres.
\end{theoreme} 
\begin{proof}
Soit $\lambda$ une partition et $T_\lambda$ l'élément correspondant dans $\widetilde{\mathcal{A}'_\infty}$. Soit $n$ un entier vérifiant $n\geq |\lambda|$. Par définition, $T_\lambda$ est une suite et son $n$-ème terme $(T_\lambda)_n=\frac{1}{\begin{pmatrix}
n\\
|\lambda|
\end{pmatrix}}{\bf{A}}_{\lambda,n}$ est dans $\mathcal{A}'_n$. On projette dans $\mathbb{C}[\mathcal{B}_n\setminus\mathcal{S}_{2n}/ \mathcal{B}_n]$ en utilisant $\psi_n$. D'après le Lemme \ref{image du base}, on obtient :
\begin{eqnarray*}
\psi_n((T_\lambda)_n)&=&\frac{1}{\begin{pmatrix}
n\\
|\lambda|
\end{pmatrix}}\frac{1}{2^{n-|\lambda|}(n-|\lambda|)!}\begin{pmatrix}
n-|\bar{\lambda}|\\
m_1(\lambda)
\end{pmatrix}{\bf K}_{\underline{\lambda}_n}\\
&=&\frac{2^{|\lambda|}|\lambda|!}{2^nn!}\begin{pmatrix}
n-|\bar{\lambda}|\\
m_1(\lambda)
\end{pmatrix}{\bf K}_{\underline{\lambda}_n}.
\end{eqnarray*}
Pour une partition $\rho$ de taille $n$, en appliquant $\omega^\rho$ à $\psi_n((T_\lambda)_n)$, on a :
\begin{eqnarray*}
\omega^\rho(\psi_n((T_\lambda)_n))&=&\frac{2^{|\lambda|}|\lambda|!}{2^nn!}\begin{pmatrix}
n-|\bar{\lambda}|\\
m_1(\lambda)
\end{pmatrix}|{K}_{\underline{\lambda}_n}|\omega^\rho_{\bar{\lambda}\cup (1^{n-|\bar{\lambda}|})}\\
&=&\frac{2^{|\lambda|}|\lambda|!}{2^nn!}\begin{pmatrix}
n-|\bar{\lambda}|\\
m_1(\lambda)
\end{pmatrix}|{K}_{\underline{\lambda}_n}|\theta^\rho_{\bar{\lambda}\cup (1^{n-|\bar{\lambda}|})}(2)\frac{1}{|\mathcal{B}_{n}|z_{2(\bar{\lambda}\cup (1^{n-|\bar{\lambda}|})}^{-1}}.
\end{eqnarray*}
La taille de ${K}_{\underline{\lambda}_n}$ est égale à $\frac{|\mathcal{B}_n|^2}{z_{2(\bar{\lambda}\cup (1^{n-|\bar{\lambda}|})}}$ (voir \cite[page 402]{McDo}). Donc, après simplification, on obtient :
\begin{equation*}
\omega^\rho(\psi_n((T_\lambda)_n))=2^{|\lambda|}|\lambda|!\begin{pmatrix}
n-|\bar{\lambda}|\\
m_1(\lambda)
\end{pmatrix}\theta^\rho_{\bar{\lambda}\cup (1^{n-|\bar{\lambda}|})}(2),
\end{equation*}
qui avec (\ref{lassalleequation}) nous donne l'équation suivante : 
\begin{equation*}
\omega^\rho(\psi_n((T_\lambda)_n))=2^{2|\lambda|-l(\lambda)}|\lambda|!\frac{sh_2(p_\lambda)(\rho)}{z_{\lambda}}.
\end{equation*}
Cette formule est vraie pour n'importe quelle partition $\rho$ vérifiant $|\rho|\geq |\lambda|$. On peut encore montrer que ça reste vrai pour toute partition $\rho$ de taille plus petite que $|\lambda|$, car dans ce cas $(T_\lambda)_{|\rho|}$ et $sh_2(p_\lambda)(\rho)$ sont égaux à zéro.\\
Finalement, pour toute partition $\rho$, son image par $F(T_\lambda)$ définie dans le Théorème \ref{isom:A'_infini_et_fct_decal}, peut encore s'écrire ainsi :
\begin{equation*}
F(T_\lambda)(\rho)=\omega^\rho(\psi_{|\rho|}((T_\lambda)_{|\rho|})).
\end{equation*}
Soit $\delta$ une partition, pour toute partition $\rho$, on a :
\begin{align}
F(T_\lambda\pr T_\delta)(\rho)&=\omega^\rho(\psi_{|\rho|}((T_\lambda\pr T_\delta)_{|\rho|}))\\
&=\omega^\rho(\psi_{|\rho|}((T_\lambda)_{|\rho|}\pr (T_\delta)_{|\rho|}))&\\
&=\omega^\rho(\psi_{|\rho|}(T_\lambda)_{|\rho|}\pr \psi_{|\rho|}(T_\delta)_{|\rho|})&\text{(Proposition \ref{proposition 3.2})}\\
&=\omega^\rho(\psi_{|\rho|}(T_\lambda)_{|\rho|})\cdot \omega^\rho(\psi_{|\rho|}(T_\delta)_{|\rho|}).&
\end{align}
La dernière égalité vient du fait que $\omega^\rho$ définit un morphisme de $\mathbb{C}[\mathcal{B}_n\setminus\mathcal{S}_{2n}/ \mathcal{B}_n]$ vers $\mathbb{C^*}$ (Proposition \ref{homomorphism}). Donc,
$$F(T_\lambda\pr T_\delta)(\rho)=(F(T_\lambda)\cdot F(T_\delta))(\rho),$$
pour toutes partitions $\lambda,$ $\delta$ et $\rho.$ Cela veut dire que $F$ est un morphisme d'algèbres de $\widetilde{\mathcal{A}'_\infty}$ dans $\Lambda^*(2)$. Comme $(T_\lambda)_{\lambda\text{ partition}}$ et $(sh_2(p_\lambda))_{\lambda\text{ partition}}$ sont des bases de $\widetilde{\mathcal{A}'_\infty}$ et $\Lambda^*(2)$, $F$ est bien un isomorphisme d'algèbres. 
\end{proof}
\begin{rem}
On aurait pu définir $F(T_{\lambda})$ par $F(T_\lambda)(\rho)=\omega^\rho(\psi_{|\rho|}((T_\lambda)_{|\rho|})),$ ce qui montre directement que $F(T_{\lambda})$ est un morphisme comme étant la composition des morphismes. Cependant, on préfère utiliser la définition donnée par l'équation \eqref{isomorphism} car elle nous donne d'une façon explicite l'action de $F$ sur les éléments de base de $\widetilde{\mathcal{A}'_\infty}$.
\end{rem}

Considérons à nouveau l'isomorphisme $sh_\alpha,$ du premier chapitre, entre $\Lambda(\alpha)$ et $\Lambda^{*}(\alpha)$. La famille $(sh_\alpha(p_\lambda))_{\lambda\text{ partition}}$ forme une base linéaire de $\Lambda^{*}(\alpha)$. Ceci nous permet d'écrire l'équation suivante :
\begin{equation*}
sh_\alpha(p_\lambda)sh_\alpha(p_\delta)=\sum_\rho g_{\lambda,\delta}^{\rho}(\alpha)sh_\alpha(p_\rho).
\end{equation*}
Dans \cite{dolega2012kerov}, les auteurs montrent que les coefficients $g_{\lambda,\delta}^{\rho}(\alpha)$ sont des polynômes en $\frac{\alpha-1}{\sqrt{\alpha}}$. Ces coefficients de structure sont reliés à la conjecture de Goulden et Jackson, voir \cite[Section 4.5]{dolega2012kerov}.

On s'intéresse au cas $\alpha=2$. On a montré dans le Corollaire \ref{corollary 3.2} que les coefficients $b_{\lambda\delta}^{\rho}$ qui apparaissent dans le produit $T_{\lambda}\pr T_{\delta}$ sont des nombres rationnels positifs. En appliquant l'isomorphisme $F$ donné dans le Théorème \ref{isom:A'_infini_et_fct_decal} au produit $T_{\lambda}\pr T_{\delta}$, on obtient directement la proposition suivante.
\begin{prop}
Les coefficients $g_{\lambda,\delta}^{\rho}(2)$ sont des nombres rationnels positifs.
\end{prop}

\chapter{Généralisation: Un cadre pour la polynomialité des coefficients de structure des algèbres de doubles-classes}
\label{chapitre4}

Dans les chapitres \ref{chapitre2} et \ref{chapitre3}, on a présenté deux propriétés de polynomialité pour les coefficients de structure de $Z(\mathbb{C}[\mathcal{S}_n])$ et \Hecke. La démarche qu'on a suivie pour obtenir le résultat du troisième chapitre était similaire à celle d'Ivanov et Kerov dans \cite{Ivanov1999}, présentée dans le deuxième chapitre. Dans \cite{meliot2013partial}, Méliot a utilisé la même idée pour donner une propriété de polynomialité des coefficients de structure de $Z(\mathbb{C}[GL_n(\mathbb{F}_q)]),$  le centre de l'algèbre du groupe des matrices inversibles de taille $n\times n$ à coefficients dans un corps fini d'ordre $q.$ 

La définition d'une bijection partielle dans le troisième chapitre était semblable, dans le cas de \Hecke , à celle des permutations partielles donnée par Ivanov et Kerov dans \cite{Ivanov1999}. De même, dans \cite{meliot2013partial}, l'auteur définit ce qu'il appelle "isomorphisme partiel" pour arriver à son résultat. La ressemblance des démarches suivies par Ivanov et Kerov, moi-même et Méliot dans \cite{Ivanov1999}, \cite{toutarxiv} (détails dans le troisième chapitre) et \cite{meliot2013partial} m'a poussé à chercher un cadre général pour la polynomialité des coefficients de structure des algèbres de doubles-classes.

Le choix de faire une généralisation pour les coefficients de structure des algèbres de doubles-classes est motivé par le fait que les centres des algèbres de groupe fini sont des cas particuliers des algèbres de doubles-classes, voir Section \ref{sec:ctr_double_classe}. De plus, bien que Méliot considère le centre de l'algèbre $GL_n(\mathbb{F}_q),$ où $\mathbb{F}_q$ est un corps fini de $q$ éléments, dans \cite{meliot2013partial}, les isomorphismes partiels qu'il utilise sont définis naturellement comme des éléments partiels de l'algèbre de Hecke de la paire $(GL_n(\mathbb{F}_q)\times GL_n(\mathbb{F}_q)^{opp},\diag(GL_n(\mathbb{F}_q)))$ et non pas $Z(\mathbb{C}[GL_n(\mathbb{F}_q)])$ --qui est isomorphe à $(GL_n(\mathbb{F}_q)\times GL_n(\mathbb{F}_q)^{opp},\diag(GL_n(\mathbb{F}_q)))$. En effet, ils sont définis comme des couples d'isomorphismes entre sous-espaces vectoriels de $GL_n(\mathbb{F}_q).$ C'est pour cela que c'est naturel de considérer directement les algèbres de doubles-classes au lieu de commencer par les centres d'algèbres de groupe.

On s'intéresse à la propriété de polynomialité des coefficients de structure dans ce chapitre. Pour une suite de paires $(G_n,K_n)_n,$ où $G_n$ est un groupe et $K_n$ un sous-groupe de $G_n$ pour tout $n,$ on propose une liste de conditions qui nous permet de comprendre la dépendance en $n$ des coefficients de structure des algèbres de doubles-classes $\mathbb{C}[K_n\setminus G_n/K_n]$ quand $n$ est suffisamment grand. Une de ces conditions que l'on note H.0 (voir le début de la Section \ref{sec:hypo_defi}) porte sur les $K_n$-doubles-classes dans $G_n$ et elle est tout à fait naturelle. Les autres ne dépendent que de la suite des sous-groupes $(K_n)_n$. Nous montrons que les conditions requises sur les sous-groupes $K_n$ sont vérifiées par les groupes symétriques et par les groupes hyperoctaédraux.

Le cadre général présenté dans ce chapitre contient les résultats de polynomialité des coefficients de structure de $Z(\mathbb{C}[\mathcal{S}_n])$ et \Hecke. En plus de retrouver ces résultats, on montre grâce à notre cadre général que les coefficients de structure de l'algèbre de doubles-classes de $(\mathcal{S}_n\times \mathcal{S}^{opp}_{n-1},diag(\mathcal{S}_{n-1})),$ et du centre de l'algèbre de groupe hyperoctaédral possèdent une propriété de polynomialité (voir Sections \ref{alg_double_classe_diag(S_n-1)} et \ref{cen_grp_hyp}). Selon les connaissances de l'auteur, ces résultats ne sont pas présentés ailleurs. Le lecteur peut se référer aux articles \cite{strahov2007generalized}, \cite{Jackson20121856} et \cite{jackson2012character} pour l'étude de l'algèbre de doubles-classes de $(\mathcal{S}_n\times \mathcal{S}^{opp}_{n-1},diag(\mathcal{S}_{n-1})),$ et à \cite{geissinger1978representations} et \cite{stembridge1992projective} pour $Z(\mathbb{C}[\mathcal{B}_n]).$ 

Le chapitre s'achève par une série de questions ouvertes. À la fin de ce chapitre, on considère le centre de l'algèbre du groupe $GL_n(\mathbb{F}_q),$ les super-classes dans un cadre général et particulièrement celles du groupe unitriangulaire, ainsi que l'algèbre de Iwahori-Hecke et une généralisation de l'algèbre de Hecke de la paire \hecke. Les coefficients de structure du centre de l'algèbre du groupe $GL_n(\mathbb{F}_q)$ et ceux du centre de l'algèbre de Iwahori-Hecke possèdent une propriété de polynomialité, similaire à celle de Farahat et Higman, donnée par Méliot dans \cite{meliot2013partial} et \cite{meliot2010products} respectivement tandis que pour les autres algèbres (les algèbres engendrées par les super-classes et la généralisation de l'algèbre de Hecke de la paire \hecke), il n'y a pas --selon les connaissances de l'auteur-- de résultat de ce type. Ces algèbres --hormis celle présentée comme une généralisation de l'algèbre de Hecke de \hecke-- n'entrent pas dans notre cadre général dans son état actuel. Il sera intéressant de voir s'il peut être modifié afin de les contenir.




\section{Cadre général: Définitions et Théorème principal}

Dans cette section, on présente notre cadre général pour la propriété de polynomialité des coefficients de structure et on donne toutes les définitions nécessaires pour la présentation de notre théorème principal. Notre cadre concerne les algèbres de doubles-classes.
\subsection{Hypothèses et définitions}\label{sec:hypo_defi}
Soit $(G_n,K_n)$ une suite de paires où $G_n$ est un groupe et $K_n$ est un sous-groupe de $G_n$ pour tout $n.$ On suppose que $G_n\subseteq G_{n+1}$ et que $K_n\subseteq K_{n+1}$ pour tout $n.$ On considère pour tout $n$ l'ensemble $K_{n}\setminus G_{n}/K_{n}$ des doubles-classes de $K_{n}$ dans $G_{n}.$ On demande que si $x \in G_n,$ alors l'intersection de la $K_{n+1}$-double-classe de $x$ avec $G_n$ est la $K_n$-double-classe de $x$: formellement,
\begin{enumerate}
\item[H.0]\label{hyp0} $K_{n+1}xK_{n+1}\cap G_n=K_{n}xK_{n}$ pour tout $x\in G_n.$
\end{enumerate}

Pour un élément $x$ de $G_n,$ on note $\overline{x}^{n}$ la $K_n$-double-classe $K_nxK_n.$ Rappelons que $y\in \overline{x}^{n}$ si et seulement s'il existe deux éléments $k,k'\in K_{n}$ tel que $x=kyk'.$

Pour tout triplet $(x_1,x_2,x_3)$ de $G_n,$ on définit $c_{1,2}^3(n)$ comme étant le coefficient de $\overline{\bf x_3}^{n}$ dans le produit $\overline{\bf x_1}^{n}\cdot \overline{\bf x_2}^{n}.$


Dans notre cadre, on suppose que pour tout $k\leq n,$ il existe un sous-groupe $K_n^k$ de $K_n$ vérifiant les hypothèses suivantes :

\begin{enumerate}
\item[H.1]\label{hyp1} $K_n^k$ est isomorphe en tant que groupe à $K_{n-k}$.
\item[H.2]\label{hyp2} Si $x\in K_k$ et $y\in K_n^k$, alors on a : $x\cdot y=y\cdot x.$
\item[H.3]\label{hyp3} $K_{n+1}^k \cap K_n=K_n^k$ si $k\leq n.$
\end{enumerate}

Dans ces hypothèses, il n'y a pas de conditions sur les doubles-classes. Pour donner les hypothèses sur les doubles-classes, on définit la fonction $\mathrm{k}$ ainsi :
\begin{equation*}
\mathrm{k}(X):=\min_{k\atop{X\cap K_k\neq \emptyset}}k,
\end{equation*}
pour n'importe quel sous-ensemble $X$ de $\cup_{n\geq 1}K_n.$ Cette définition va nous être cruciale pour la démonstration ainsi que la présentation de notre résultat principal. 

\begin{definition}
Soit $y\in K_n$, on dit que $y$ est $(k_1,k_2)$-minimal si $y\in K_m$ où $m=\mathrm{k}(K_n^{k_1}yK_n^{k_2})$.
\end{definition}

C'est aussi une définition très importante pour la suite. On présente ici les conditions nécessaires sur les doubles-classes dans notre cadre général. La liste des conditions est ainsi :

\begin{enumerate}
\item[H.4]\label{hyp4} Pour tout élément $z\in K_n$, on a $K_{n+1}^{k_1}zK_{n+1}^{k_2}\cap K_n=K_{n}^{k_1}zK_{n}^{k_2}.$
\item[H.5]\label{hyp5} $m_{k_1,k_2}(x):=\mathrm{k}(K_n^{k_1}xK_n^{k_2})\leq k_1+k_2$ pour tout $x\in K_n.$
\item[H.6]\label{hyp6} $yK_{n}^{k_1}y^{-1}\cap K_{n}^{k_2}=K_{n}^{m_{k_1,k_2}(y)},$ si $y$ est $(k_1,k_2)$-minimal.
\end{enumerate}

\begin{rem}
Les hypothèses présentées dans ce cadre général sont inspirées des cas particuliers déjà traités pour la polynomialité des coefficients de structure. Il est à remarquer que la suite des groupes $G_n$ n'intervient pas dans les hypothèses H.1 à H.6 et que ces hypothèses concernent la suite des sous-groupes $K_n.$ Cela va être pratique pour les applications qu'on donne dans les sections qui suivent dans ce chapitre. 
L'hypothèse 5 est peut-être la plus importante parmi ces hypothèses (le lecteur peut voir les applications qu'on donne à la fin de ce chapitre pour mieux comprendre cette hypothèse). La seule hypothèse qui dépend de la suite $(G_n,K_n)_n$ et pas seulement de $(K_n)_n$ est H.0. Cette hypothèse nous garantit l'indépendance en $n$ de l'intersection de $G_{n_0}$ avec les $K_n$-doubles-classes de ses éléments pour $n_0$ fixé et $n$ suffisamment grand. Elle est aisément vérifiée dans les cas particuliers de suites $(G_n,K_n)_n$ qu'on considère dans les sections d'applications \ref{sec:appl_double_classe} et \ref{sec:appl_centr}.
\end{rem}
\begin{rem}
Dans les hypothèses H.0 à H.6, on donne toutes les conditions qui, à notre avis, sont nécessaires pour avoir la propriété de polynomialité. On les a trouvé en s'inspirant des résultats de polynomialité déjà connus, notamment ceux de $Z(\mathbb{C}[\mathcal{S}_n])$ et de $\mathbb{C}[\mathcal{B}_n\setminus \mathcal{S}_{2n}/\mathcal{B}_n].$ Il est important de savoir si la liste des conditions qu'on donne est "minimale" (dans le sens qu'aucune hypothèse n'est pas impliqué par un ensemble d'autres hypothèses) ou pas. La liste H.0 à H.6 n'est pas minimale car H.3 implique H.4 et vice-versa (c'est-à-dire les hypothèses H.3 et H.4 sont équivalentes). On donne la preuve de cette équivalence dans l'Observation \ref{obs:Equiv_H3_H4} suivante. On pense que ce sont les seules implications directes entre les hypothèses H.0 à H.6. Dans la suite, on s'intéresse à H.4 plus qu'à H.3 parce qu'elle est plus appropriée à notre étude.
\end{rem}
\begin{obs}\label{obs:Equiv_H3_H4}
L'hypothèse H.3 s'obtient à partir de H.4 dans le cas particulier où $k_1=k_2=k$ et $z$ est l'élément neutre. Dans l'autre sens, si H.3 est vérifiée, pour tout $z\in K_n$ on a: $K_{n+1}^{k_1}z \cap K_n=K_n^{k_1}z$ pour tout $k_1\leq n.$ Si on multiplie cette égalité à gauche par le sous-groupe $K_{n+1}^{k_2}$ pour un certain $k_2\leq n$ on obtient $K_{n+1}^{k_1}zK_{n+1}^{k_2} \cap K_n=K_n^{k_1}zK_{n+1}^{k_2}$ qui est l'hypothèse H.4.
\end{obs}
\begin{obs}\label{obs:H_4_ind_n}
D'après H.4, la minimalité ne dépend pas de $n.$ En effet, fixons trois entiers $n_0,$ $k_1$ et $k_2,$ et un élément $z\in K_{n_0},$ et supposons que $\mathrm{k}(K_{n_0}^{k_1}zK_{n_0}^{k_2})$ est un entier $k_{n_0,k_1,k_2}.$ D'après H.5,  $k_{n_0,k_1,k_2}\leq k_1+k_2.$ Pour tout $n\geq n_0+1,$ on a $$K_{n}^{k_1}zK_n^{k_2}\cap K_{k_{n_0,k_1,k_2}}=K_{n_0}^{k_1}zK_{n_0}^{k_2}\cap K_{k_{n_0,k_1,k_2}}$$ d'après H.4. Donc, pour tout $n\geq n_0,$ on a $\mathrm{k}(K_{n}^{k_1}zK_{n}^{k_2})=\mathrm{k}(K_{n_0}^{k_1}zK_{n_0}^{k_2}).$ Cela montre que $m_{k_1,k_2}(x)$ ne dépend pas de $n.$ On va utiliser cette observation dans la preuve du Théorème \ref{mini_th}.
\end{obs}

\setcounter{subsection}{1}
\subsection{Résultat principal}

Notre résultat principal dans ce chapitre est le Théorème \ref{main_th} présenté ci-dessous. La propriété de polynomialité dans certaines algèbres peut être obtenue directement de ce théorème. De plus, ce théorème nous permet dans des cas particuliers d'obtenir non seulement la propriété de polynomialité des coefficients de structure mais les valeurs exactes de ces coefficients.

\begin{theoreme}\label{main_th}
Soit $(G_n,K_n)_n$, une suite de paires où $G_n$ est un groupe et $K_n$ est un sous-groupe de $G_n$ pour tout $n$ vérifiant les hypothèses H.0 à H.6. Pour un entier $n_0$ fixé et trois éléments $x_1$, $x_2$ et $x_3$ de $G_{n_0},$ on note par $k_1$ (resp. $k_2$, $k_3$) l'entier $\mathrm{k}(\bar{x_1}^{n_0})$ (resp. $\mathrm{k}(\bar{x_2}^{n_0})$, $\mathrm{k}(\bar{x_3}^{n_0})$). Le coefficient de structure $c_{1,2}^3(n_0)$ de $\bar{{\bf x_3}}^{n_0}$ dans le produit $\bar{{\bf x_1}}^{n_0}\bar{{\bf x_2}}^{n_0}$ est donné par la formule suivante :
\begin{eqnarray}\label{eq:forme_coef_stru}
c_{1,2}^3(n_0)=&&\frac{|\bar{x_1}^{n_0}||\bar{x_2}^{n_0}||K_{n_0-k_1}||K_{n_0-k_2}|}{|K_{n_0}||\bar{x_3}^{n_0}|} \nonumber \\
&&\sum_{ \max(k_1,k_2,k_3)\leq k\leq \min(k_1+k_2,n_0), x\in G_k, \atop{\text{ $x_1^{-1}xx_2^{-1}\in K_k$ et est $(k_1,k_2)$-minimal} \atop{\bar{x}^{n_0}=\bar{x_3}^{n_0}}}}\frac{1}{|K_{n_0-k}||K_{n_0}^{k_1}x_1^{-1}xx_2^{-1}K_{n_0}^{k_2}\cap K_{m_{k_1,k_2}(x_1^{-1}xx_2^{-1})}|}.
\end{eqnarray}
\end{theoreme}
\begin{proof}
Voir la Section \ref{sec:preuve_th_princ} dédiée à la preuve de ce théorème.
\end{proof}

Des exemples d'application du Théorème \ref{main_th} sont donnés dans les Sections \ref{sec:alg_de_hecke} et \ref{alg_double_classe_diag(S_n-1)}. Ce théorème donne une formule exacte pour les coefficients de structure mais la taille de $K_n^{k_1}x_1^{-1}xx_2^{-1}K_n^{k_2}\cap K_{m_{k_1,k_2}(x_1^{-1}xx_2^{-1})}$ est en générale difficile à calculer (contrairement à celle de $K_{n-k}$). De plus, l'ensemble de sommation dans l'équation \eqref{eq:forme_coef_stru} est très compliqué. On peut rendre ce théorème moins compliqué et plus facile à utiliser dans le cas où on a besoin juste de montrer une propriété de polynomialité pour les coefficients de structure. C'est l'objet de la présentation ci-dessous.





\begin{theoreme}\label{mini_th} Soit $n_1$ un entier fixé et soient $x_1,x_2$ et $x_3$ trois éléments de $G_{n_1}.$ On note par $k_1$ (resp. $k_2$ et $k_3$) l'entier $\mathrm{k}(\bar{x_1}^n)$ (resp. $\mathrm{k}(\bar{x_2}^n)$ et $\mathrm{k}(\bar{x_3}^n)$). Pour tout $n$ suffisamment grand, il existe des nombres rationnels positifs $a_{1,2}^{3}(k)$ pour $k_3\leq k\leq k_1+k_2$ indépendants de $n$ tels que le coefficient de structure $c_{1,2}^3(n)$ s'écrit :

\begin{equation}
c_{1,2}^3(n)=\frac{|\bar{x_1}^n||\bar{x_2}^n||K_{n-k_1}||K_{n-k_2}|}{|K_n||\bar{x_3}^n|}\sum_{ \max(k_1,k_2,k_3)\leq k\leq k_1+k_2}\frac{a_{1,2}^{3}(k)}{|K_{n-k}|}.
\end{equation}

\end{theoreme}

\begin{proof} On va commencer cette preuve par appliquer le Théorème \ref{main_th} à $x_1,x_2,x_3$ pour tout entier $n_0\geq n_1.$ On peut faire cela car l'inclusion $G_n\subset G_{n+1}$ permet de voir $x_1,x_2,x_3$ comme des éléments de $G_{n_0}.$ D'après H. 0, les entiers $k_1$, $k_2$ et $k_3$ ne dépendent pas de $n_0$ si $n_0$ est suffisamment grand, pas plus que la taille de 
$K_{n_0}^{k_1}x_1^{-1}xx_2^{-1}K_{n_0}^{k_2}\cap K_{m_{k_1,k_2}(x_1^{-1}xx_2^{-1})}$ (d'après H.4 et l'Observation \ref{obs:H_4_ind_n} implique que $m_{k_1,k_2}(x_1^{-1}xx_2^{-1})$ ne dépend pas de $n_0$). Il reste à remarquer que l'hypothèse H.0 sur les doubles-classes de $K_n$ dans $G_n$ et l'Observation \ref{obs:H_4_ind_n} nous garantit que l'ensemble de sommation du Théorème \ref{main_th} ne dépend pas de $n$ quand $n$ est suffisamment grand. Le Théorème \ref{mini_th} est ainsi une conséquence du Théorème \ref{main_th}.
\end{proof}

\begin{rem}
Le lecteur a pu remarquer en lisant la preuve du Théorème \ref{mini_th} que celui-ci peut être vu comme un corollaire du Théorème \ref{main_th}. Néanmoins, nous préférons le nommer théorème car il est l'objectif initial et principal de ce chapitre.
\end{rem}

Le Théorème \ref{main_th} permet de trouver une formule exacte pour la fonction $c_{1,2}^3(n)$ pour des exemples de $(x_1,x_2,x_3)$ tandis que le Théorème \ref{mini_th} est adapté pour montrer une propriété de polynomialité de ces coefficients pour $(x_1,x_2,x_3)$ général.

\section{L'algèbre des éléments partiels}

On commence cette section par la définition des éléments partiels.

\begin{definition}
Un \textit{élément partiel} de $G_n$ est un triplet $(C,(x;k),C')$, où $k$ est un entier entre $1$ et $n,$ $C$ (resp. $C'$) une classe à gauche (resp. droite) de $K_n^k$ dans $K_n$ et $x\in G_k.$
\end{definition}
Par exemple, $(\mathcal{B}_3^2, ((1\,\, 4\,\,3)(2); 2),\mathcal{B}_3^2)$ est un élément partiel de $\mathcal{S}_{6}$ (associé à la suite des paires $(\mathcal{S}_{2n},\mathcal{B}_n)$) où $\mathcal{B}_3^2$ (voir Section \ref{sec:cond_group_hyp} pour une définition explicite) est un groupe isomorphe à $\mathcal{B}_1.$

Les éléments partiels sont les objets essentiels pour montrer notre résultat. Il est à noter que bien que la notation laisse entendre une généralisation des définitions des permutations partielles de \cite{Ivanov1999}, bijections partielles de \cite{toutarxiv} et isomorphismes partiels de \cite{meliot2013partial}, ce n'est pas le cas. En effet, il suffit de remarquer qu'en général le nombre des éléments partiels de $G_n$ est :
$$\sum_{k=1}^n\left(\frac{|K_n|}{|K_{n-k}|}\right)^2|G_k|,$$
ce qui ne coïncide pas avec le nombre des bijections partielles de $[2n]$ dans le cas de la suite des paires $(\mathcal{S}_{2n},\mathcal{B}_n).$


\begin{definition}\label{def:_prod_ele_part}
Soient 
$pe_1=(C_1,(x_1;k_1),C'_1)$ et $pe_2=(C_2,(x_2;k_2),C'_2)$ deux éléments partiels de $G_n.$ On définit le produit $pe_1\cdot pe_2$ ainsi :
$$pe_1\cdot pe_2:=\frac{1}{n^{k_1}_m n^{k_2}_m |C'_1C_2\cap K_{\mathrm{k}(C'_1C_2)}|}\sum_{h\in C'_1C_2 \atop{h\text{ $(k_1,k_2)$-minimal}}}\sum_{i=1}^{n^{k_1}_m}\sum_{j=1}^{n^{k_2}_m}(C_1^i,(x_1hx_2;m){C'}_2^j),$$
où $m=\max(k_1,k_2,\mathrm{k}(C'_1C_2))$, $n^{k_2}_m=\frac{|K_n^{k_2}|}{|K_n^{m}|}$, $n^{k_1}_m=\frac{|K_n^{k_1}|}{|K_n^{m}|}$ et les classes $C_1^{i}$ (resp. ${C'}_{2}^j$) sont définies par les équations suivantes : 
\begin{equation}\label{eq:dec_classe}
C_1=\bigsqcup_{i=1}^{n^{k_1}_m}C_1^j~~~~(\text{resp. } C'_2=\bigsqcup_{j=1}^{n^{k_2}_m}{C'}_2^j).\end{equation}
\end{definition}

\begin{rem}
Comme le nombre $m$ défini dans la Définition \ref{def:_prod_ele_part} est plus grand que $k_1$ et $k_2$ alors $K_n^m$ est un sous-groupe 
de $K_n^{k_1}$ et de $K_n^{k_2}.$ L'équation \eqref{eq:dec_classe} est l'écriture formelle du fait que les classes à gauche (resp. droite) de $K_n^{k_1}$ (resp. $K_n^{k_2}$) sont des unions disjointes des classes à gauche (resp. droite) de $K_n^m.$
\end{rem}

Après avoir défini le produit des éléments partiels, il est naturel de vérifier si ce produit est associatif ou pas. Le produit des bijections partielles est associatif et cela nous a permis de construire une algèbre universelle qui se projette sur l'algèbre \Hecke \, pour tout $n.$ De même, les produits des permutations partielles et des isomorphismes partiels définis par Ivanov/Kerov et Méliot dans \cite{Ivanov1999} et \cite{meliot2013partial} sont associatifs. Des algèbres universelles sont aussi présentées dans les deux papiers.

Les preuves de l'associativité du chapitre \ref{chapitre3} et de \cite{meliot2013partial} sont difficiles. On a choisi de ne pas prouver l'associativité car on n'a en fait pas besoin de cette propriété pour montrer notre résultat de polynomialité. De plus, la preuve paraissait difficile.

On note par $PE_n$ l'ensemble des éléments partiels de $G_n$. L'ensemble $K_n\times K_n$ agit sur $PE_n$ par l'action suivante :
$$(a,b)\cdot (C,(x;k),C')=(aC,(x;k),C'b^{-1}).$$
Ceci définit bien une action de groupe car :
\begin{eqnarray*}
(a_1,b_1)\cdot (a_2,b_2)\cdot (C,(x;k),C')&=&(a_1,b_1)\cdot (a_2C,(x;k),C'b_2^{-1})\\
&=&(a_1a_2C,(x;k),C'b_2^{-1}b_1^{-1})\\
&=&(a_1a_2,b_1b_2)\cdot (C,(x;k),C').
\end{eqnarray*}
\begin{prop}\label{P}
L'action de $K_n\times K_n$ sur $PE_n$ est compatible avec le produit dans $PE_n,$ dans le sens où :
$$(a,b)\cdot (pe_1\cdot pe_2)=((a,c)\cdot pe_1) \cdot ((c,b)\cdot pe_2),$$
pour $a,b$ et $c$ dans $K_n.$
\end{prop}
\begin{proof}
Le produit $(a,b)\cdot (pe_1\cdot pe_2)$ égale :
\begin{equation*}
\frac{1}{n^{k_1}_m n^{k_2}_m |C'_1C_2\cap K_{\mathrm{k}(C'_1C_2)}|}\sum_{h\in C'_1C_2 \atop{h\text{ $(k_1,k_2)$-minimal}}}\sum_{i=1}^{n^{k_1}_m}\sum_{j=1}^{n^{k_2}_m}(aC_1^i,(x_1hx_2;m),{C'}_2^jb^{-1}),
\end{equation*}
qui est aussi la valeur de $((a,c)\cdot pe_1) \cdot ((c,b)\cdot pe_2).$
\end{proof}
On considère maintenant l'ensemble $\mathbb{C}[PE_n]$ des combinaisons linéaires des éléments partiels de $G_n$ à coefficients dans $\mathbb{C}.$ On étend l'action de $K_n\times K_n$ sur $PE_n$ par linéarité sur $\mathbb{C}[PE_n]$ et on note par $\mathcal{A}_n$ l'ensemble des invariants par cette action.
\begin{lem}
$\mathcal{A}_n$ est stable par multiplication. En d'autres termes, si $\alpha_1$ et $\alpha_2$ sont deux éléments de $\mathcal{A}_n$, alors $\alpha_1\cdot \alpha_2$ est aussi dans $\mathcal{A}_n.$
\end{lem}
\begin{proof}
Ce résultat est une conséquence de la Proposition \ref{P}.
\end{proof}
\begin{notation}
Pour un élément $x\in G_n$, on va parfois noter la double-classe $K_n^{k_1}xK_n^{k_2}$ par $Cl_{k_1,k_2}(x).$ 
De même, on définit $Cl_{k_1,*}(x)$ (resp. $Cl_{*,k_2}(x)$) comme étant la classe à gauche (resp. droite) $K_n^{k_1}x$ (resp. $xK_n^{k_2}$). Finalement, considérons les trois ensembles suivants :
\begin{enumerate}
\item $CL_{k_1,k_2}(G_n):=\lbrace Cl_{k_1,k_2}(x), x\in G_n\rbrace$.
\item $CL_{k_1,*}(G_n):=\lbrace Cl_{k_1,*}(x), x\in G_n\rbrace$.
\item $CL_{*,k_2}(G_n):=\lbrace Cl_{*,k_2}(x), x\in G_n\rbrace$.
\end{enumerate}
\end{notation}
\begin{prop}
$\mathcal{A}_n$ est engendré par la famille $(\mathbf{a}_{(x;k)}(n))_{1\leq k\leq n\atop{x\in G_k}}$ où :
$$\mathbf{a}_{(x;k)}(n)=\sum_{C\in CL_{*,k}(K_n)}\sum_{C'\in CL_{k,*}(K_n)}(C,(x;k),C').$$
\end{prop}
\begin{proof}
Soit $\alpha\in \mathcal{A}_n$ i.e. pour toute paire $(a,b)$ de $K_n\times K_n$, on a $(a,b)\cdot \alpha=\alpha$. Comme $\alpha\in \mathbb{C}[PE_n]$, on peut écrire :
\begin{equation*}
\alpha=\sum_{1\leq k\leq n\atop{x\in G_k}}\sum_{C\in CL_{*,k}(K_n)}\sum_{C'\in CL_{k,*}(K_n)}c_{k,x,C,C'}(C,(x;k),C'),
\end{equation*}
où les coefficients $c_{k,x,C,C'}$ sont dans $\mathbb{C}.$  La condition $(a,b)\cdot \alpha=\alpha$ pour tout $(a,b)$ dans $K_n\times K_n$ nous donne la relation suivante pour les coefficients $c_{k,x,C,C'}$: 
$$c_{k,x,aC,C'b^{-1}}=c_{k,x,C,C'}\text{ pour tout $(a,b)\in K_n\times K_n$.}$$
Cela veut dire que tous les éléments de la forme $(*,(x;k),*)$ possèdent le même coefficient dans le développement de $\alpha$ ce qui termine la preuve de la proposition. 
\end{proof}
\begin{prop}\label{F} Soient $x_1$ et $x_2$ deux éléments $G_{k_1}$ et $G_{k_2}$ où $k_1$ et $k_2$ sont deux entiers inférieurs ou égaux à $n$, alors on a :
\begin{equation}\label{E}\mathfrak{a}_{(x_1;k_1)}(n)\cdot \mathfrak{a}_{(x_2;k_2)}(n)=\sum_{\max(k_1,k_2)\leq k\leq \min(k_1+k_2,n) \atop{x\in G_k}}c_{(x_1;k_1),(x_2;k_2)}^{(x;k)}(n)\mathfrak{a}_{(x;k)}(n),
\end{equation}
où, $c_{(x_1;k_1),(x_2;k_2)}^{(x;k)}(n)$ est égal à :
$$\left\{
\begin{array}{ll}
  \frac{|K_n||K_{n-k}|}{|K_{n-k_1}||K_{n-k_2}||K_n^{k_1}XK_n^{k_2}\cap K_{m_{k_1,k_2}(X)}|} & \qquad \mathrm{si}\quad X=x_1^{-1}xx_2^{-1}\in K_n \text{ et est $(k_1,k_2)$-minimal,}\\
  0 & \qquad \mathrm{sinon},\\
 \end{array}
 \right. \\$$
\end{prop}
\begin{proof} 
On fixe un élément $x\in G_k$ pour un $k\leq k_1+k_2$ et deux classes $C$ et $C'.$ Posons $X=x_1^{-1}xx_2^{-1}.$ Soit $\mathcal{C}_{(x_1;k_1),(x_2;k_2)}^{(x;k)}(n)$ l'ensemble des paires $(pe_1,pe_2)$ tel que $pe_i$ apparaît dans le développement de $\mathfrak{a}_{(x_i;k_i)}(n)$ pour $i=1,2$ et $(C,(x;k),C')$ apparaît dans le développement du produit $pe_1\cdot pe_2.$ Alors, on peut écrire :
$$c_{(x_1;k_1),(x_2;k_2)}^{(x;k)}(n)=\sum_{(pe_1,pe_2)\in \mathcal{C}_{(x_1;k_1),(x_2;k_2)}^{(x;k)}(n)}\frac{1}{n^{k_1}_k n^{k_2}_k|C'_1C_2\cap K_{C'_1C_2}|}.$$ 

Comme $C$ et $C'$ sont fixés, ils déterminent $C_1$ et $C'_2$ pour n'importe quelle paire $(pe_1,pe_2)$ dans $\mathcal{C}_{(x_1;k_1),(x_2;k_2)}^{(x;k)}(n).$ On note par $A_{(x_1;k_1),(x_2;k_2)}^{(x;k)}(n)$ l'ensemble suivant :
\begin{eqnarray*}
{A}_{(x_1;k_1),(x_2;k_2)}^{(x;k)}(n)=\lbrace (C'_1,C_2)\in CL_{k_1,*}(K_n) \times CL_{*,k_2}(K_n) \text{ tel que }X\in C'_1C_2\cap K_{k(C'_1C_2)}\rbrace.
\end{eqnarray*}
Cet ensemble est en bijection avec $\mathcal{C}_{(x_1;k_1),(x_2;k_2)}^{(x;k)}(n).$
Notons aussi qu'il est vide si $X$ n'est pas dans $K_n$ ou n'est pas $(k_1,k_2)$-minimal. Regardons donc le cas où $X$ est dans $K_n$ et $(k_1,k_2)$-minimal. Dans ce cas là, $c_{(x_1;k_1),(x_2;k_2)}^{(x;k)}(n)$ peut s'écrire comme une sommation sur les éléments de ${A}_{(x_1;k_1),(x_2;k_2)}^{(x;k)}(n)$ ainsi :
\begin{eqnarray*}c_{(x_1;k_1),(x_2;k_2)}^{(x;k)}(n)&=&\sum_{(C'_1,C_2)\in A_{(x_1;k_1),(x_2;k_2)}^{(x;k)}(n)}\frac{1}{n^{k_1}_k n^{k_2}_k|K_n^{k_1}XK_n^{k_2}\cap K_{m_{k_1,k_2}(X)}|}\\
&=&\frac{|A_{(x_1;k_1),(x_2;k_2)}^{(x;k)}(n)|}{n^{k_1}_k n^{k_2}_k|K_n^{k_1}XK_n^{k_2}\cap K_{m_{k_1,k_2}(X)}|}.
\end{eqnarray*}

Nous allons montrer que l'action de $K_n$ sur $A_{(x_1;k_1),(x_2;k_2)}^{(x;k)}(n)$ définie ainsi est transitive :
$$h\cdot (C'_1,C_2)=(C'_1h,h^{-1}C_2).$$
Soient $(A_1,A_2)$ et $(B_1,B_2)$ deux éléments de $A_{(x_1;k_1),(x_2;k_2)}^{(x;k)}(n).$ Étant des classes, soient $a_1,a_2,b_1$ et $b_2$ quatre éléments de $K_n$ représentant $A_1,A_2,B_1,$ et $B_2.$ Pour montrer que l'action est transitive, on doit trouver un élément $h\in K_n$ tel que :
\begin{eqnarray*}
a_1h&=&xb_1\\
h^{-1}a_2&=&b_2y,
\end{eqnarray*}
où $x\in K_n^{k_1}$ et $y\in K_n^{k_2}.$ Cela veut dire que l'ensemble $a_1^{-1}K_n^{k_1}b_1\cap a_2K_n^{k_2}b_2^{-1}$ doit être non vide ce qui est équivalent au fait que $K_n^{k_1}\cap a_1a_2K_n^{k_2}(b_1b_2)^{-1}$ est non vide. Ceci est vrai puisque comme $(A_1,A_2)$ et $(B_1,B_2)$ appartiennent à $A_{(x_1;k_1),(x_2;k_2)}^{(x;k)}(n)$ --donc $X\in A_1A_2$ et $X\in B_1B_2$-- alors $a_1a_2$ peut s'écrire $h_1Xh_2$ et $b_1b_2$ peut s'écrire $h'_1Xh'_2$ où $h_1,h'_1\in K_n^{k_1}$ et $h_2,h'_2\in K_n^{k_2}$ et $K_n^{k_1}\cap XK_n^{k_2}X^{-1}$ est non vide par H. 6.\\

Donc il n'y a qu'une seule orbite et on a,
$$|A_{(x_1;k_1),(x_2;k_2)}^{(x;k)}(n)|=|K_n\cdot (K_n^{k_1}X,K_n^{k_2})|=\frac{|K_n|}{|X^{-1}K_n^{k_1}X\cap K_n^{k_2}|},$$
car le stabilisateur de $(K_n^{k_1}X,K_n^{k_2})$ est l'ensemble des éléments $h\in K_n$ tel que $h\in X^{-1}K_n^{k_1}X$ et $h\in K_n^{k_2}.$
Par H. 6, le dénominateur est égal à $|K_{n-k}|$, donc :
$$c_{(x_1;k_1),(x_2;k_2)}^{(x;k)}(n)=\left\{
\begin{array}{ll}
  \frac{|K_n||K_{n-k}|}{|K_{n-k_1}||K_{n-k_2}||K_n^{k_1}XK_n^{k_2}\cap K_{m_{k_1,k_2}(X)}|} & \qquad \mathrm{si}\quad X \text{ est $(k_1,k_2)$-minimal},\\
  0 & \qquad \mathrm{sinon} .\\
 \end{array}
 \right. \\$$
\end{proof}

Soit $\psi_n:PE_n\rightarrow \mathbb{C}[K_n\setminus G_n/K_n]$ la fonction définie ainsi :
$$\psi_n((C,(x;k),C'))=\frac{1}{|C||C'|}\sum_{c\in C, c'\in C'}cxc'.$$
\begin{prop}\label{prop:compa_psi_n}
$\psi_n$ est compatible avec le produit défini sur $PE_n,$ c'est-à-dire $\psi_n(pe_1\cdot pe_2)=\psi_n(pe_1)\cdot \psi_n(pe_2).$
\end{prop}
\begin{proof} Soient $pe_1$ et $pe_2$ deux éléments de $PE_n$, d'après la définition du produit (voir Définition \ref{def:_prod_ele_part}) on a :
 $$\psi_n(pe_1\cdot pe_2)=\frac{1}{n^{k_1}_m n^{k_2}_m |C'_1C_2\cap K_{k(C'_1C_2)}|}\sum_{h}\sum_{i}\sum_{j}\frac{1}{|C_{1}^{i}||C'^{j}_{2}|}\sum_{c_{1}^{i}\in C_{1}^{i},c'^{j}_{2}\in C'^{j}_{2}}c_{1}^{i} x_{1} h x_{2} c'^{j}_{2}.$$
On a enlevé les ensembles de sommation de l'équation ci-dessus pour la rendre plus facile à lire. L'équation complète est celle obtenue de la définition \ref{def:_prod_ele_part}.
 Après simplification, on obtient :
 $$\psi_n(pe_1\cdot pe_2)=\frac{1}{|K_{n-k_1}||K_{n-k_2}||C'_1C_2\cap K_{k(C'_1C_2)}|}\sum_h\sum_{c_1\in C_1, c'_2\in C'_2}c_1x_1hx_2c'_2.$$
 D'autre part, on a :
 $$\psi_n(pe_1)\cdot \psi_n(pe_2)=\frac{1}{|K_{n-k_1}|^2|K_{n-k_2}|^2}\sum_{c_1,c'_1,c_2,c'_2}c_1x_1c'_1c_2x_2c'_2.$$
 donc $\psi_n$ est compatible avec le produit défini sur $PE_n$ si on a l'égalité suivante :
 \begin{equation}\label{coef}
\sum_{c_1,c'_1,c_2,c'_2}c_1x_1c'_1c_2x_2c'_2=\frac{|K_{n-k_1}||K_{n-k_2}|}{|C'_1C_2\cap K_{k(C'_1C_2)}|}\sum_h\sum_{c_1\in C_1, c'_2\in C'_2}c_1x_1hx_2c'_2.
\end{equation}
Soit $h$ un élément fixe de l'ensemble $C'_1C_2\cap K_{k(C'_1C_2)}$(l'ensemble de sommation de $h$ dans l'équation \eqref{coef}). Fixons $c'_1\in C'_1$ et $c_2\in C_2$, il existe deux éléments $h_1\in K_n^{k_1}$ et $h_2\in K_n^{k_2}$ tel que $c'_1c_2=h_1hh_2.$ Par H.2, comme $x_1\in G_{k_1}$ (resp. $x_2\in G_{k_2})$, on a donc $x_1c'_1c_2x_2=x_1h_1hh_2x_2=h_1x_1hx_2h_2.$ Donc, on a :
$$\sum_{c_1,c'_2}c_1x_1c'_1c_2x_2c'_2=\sum_{c_1\in C_1, c'_2\in C'_2}c_1h_1x_1hx_2h_2c'_2=\sum_{c_1\in C_1, c'_2\in C'_2}c_1x_1hx_2c'_2.$$
La dernière égalité vient du fait que $C_1h_1=C_1$ et $h_2{C'}_2={C'}_2.$ Comme le membre à droite ne dépend pas de $c'_1$ et $c_2,$ on obtient :
$$\sum_{c_1,c'_1,c_2,c'_2}c_1x_1c'_1c_2x_2c'_2=|K_{n-k_1}||K_{n-k_2}|\sum_{c_1\in C_1, c'_2\in C'_2}c_1x_1hx_2c'_2.$$
Rappelons que $h$ est ici un élément fixé de $C'_1C_2\cap K_{k(C'_1C_2)}.$ Si on prend la somme des deux côtés sur les éléments $h$ de $C'_1C_2\cap K_{k(C'_1C_2)}$ et comme le membre à gauche de l'équation ci-dessus ne dépend pas de $h$, on obtient l'équation \eqref{coef}.
\end{proof} 

\section{Preuve du théorème principal}\label{sec:preuve_th_princ} On fixe un entier $n_0$ et soient $x_1$ et $x_2$ deux éléments de $G_{n_0}.$ Soient $k_1=\mathrm{k}(\bar{x_1}^{n_0})$ et $k_2=\mathrm{k}(\bar{x_2}^{n_0}).$ Le produit $\mathfrak{a}_{(x_1;k_1)}(n_0)\cdot \mathfrak{a}_{(x_2;k_2)}(n_0)$ est donné par l'équation \eqref{E}. On applique $\psi_{n_0}$ au produit et on obtient, grâce à la Proposition \ref{prop:compa_psi_n} :
$$\psi_{n_0}(\mathfrak{a}_{(x_1;k_1)}(n_0))\cdot \psi_{n_0}(\mathfrak{a}_{(x_2;k_2)}(n_0))=\sum_{\max(k_1,k_2)\leq k\leq \min(k_1+k_2,n_0) \atop{x\in G_k}}c_{(x_1;k_1),(x_2;k_2)}^{(x;k)}(n_0)\psi_{n_0}(\mathfrak{a}_{(x;k)}(n_0)).$$
On a,
$$\psi_{n_0}(\mathfrak{a}_{(x;k)}(n_0))=\sum_{C,C'}\frac{1}{|C||C'|}\sum_{c,c'}cxc'=\frac{1}{|K_{n_0-k}|^2}\sum_{h\in K_{n_0},h'\in K_{n_0}}hxh'=\frac{|K_{n_0}|^2}{|K_{n_0}xK_{n_0}||K_{n_0-k}|^2}\overline{\bf x}^{n_0}.$$
Donc, on a :
\begin{eqnarray*}
\frac{|K_{n_0}|^2}{|K_{n_0}x_1K_{n_0}||K_{n_0-k_1}|^2}\overline{\bf x_1}^{n_0}&\cdot &\frac{|K_{n_0}|^2}{|K_{n_0}x_2K_{n_0}||K_{n_0-k_2}|^2}\overline{\bf x_2}^{n_0}=\\
&&\sum_{\max(k_1,k_2)\leq k\leq \min(k_1+k_2,n_0) \atop{x\in G_k}}c_{(x_1;k_1),(x_2;k_2)}^{(x;k)}(n_0)\frac{|K_{n_0}|^2}{|K_{n_0}xK_{n_0}||K_{n_0-k}|^2}\overline{\bf x}^{n_0},
\end{eqnarray*}
ce qui nous donne :
$$\overline{\bf x_1}^{n_0}\cdot \overline{\bf x_2}^{n_0}=\sum_{\max(k_1,k_2)\leq k\leq \min(k_1+k_2,n_0) \atop{x\in G_k}}c_{(x_1;k_1),(x_2;k_2)}^{(x;k)}(n_0)\frac{|K_{n_0}x_1K_{n_0}||K_{n_0-k_1}|^2|K_{n_0}x_2K_{n_0}||K_{n_0-k_2}|^2}{|K_{n_0}|^2|K_{n_0}xK_{n_0}||K_{n_0-k}|^2}\overline{\bf x}^{n_0}.$$
En utilisant la formule des coefficients de structure $c_{(x_1;k_1),(x_2;k_2)}^{(x;k)}(n_0)$ donnée par la Proposition \ref{F}, on obtient :
\begin{eqnarray*}
&&\overline{\bf x_1}^{n_0}\cdot \overline{\bf x_2}^{n_0}=\\
&&\sum_{\max(k_1,k_2)\leq k\leq \min(k_1+k_2,n_0) \atop{x\in G_k,\text{ $x_1^{-1}xx_2^{-1}$ est dans $K_k$ et $(k_1,k_2)$-minimal}}}\frac{|K_{n_0}x_1K_{n_0}||K_{n_0-k_1}||K_{n_0}x_2K_{n_0}||K_{n_0-k_2}|}{|K_{n_0}||K_{n_0}xK_{n_0}||K_{n_0-k}||Cl_{k_1,k_2}(x_1^{-1}xx_2^{-1})\cap K_{m_{k_1,k_2}(x_1^{-1}xx_2^{-1})}|}\overline{\bf x}^{n_0}.
\end{eqnarray*}

Pour obtenir l'expression des coefficients de structure donnée par le Théorème \ref{main_th}, il suffit de remarquer qu'en fixant $\bar{x_3}^{n_0},$ pour obtenir son coefficient on doit sommer sur tous les $x$ tel que $\bar{x}^{n_0}=\bar{x_3}^{n_0}$ ce qui implique que $\mathrm{k}(\bar{x}^{n_0})=\mathrm{k}(\bar{x_3}^{n_0})$. Donc les $x$ qui apparaissent dans la somme doivent être dans $G_k$ où $k\geq \mathrm{k}(\bar{x_3}^{n_0})$, sinon $\bar{x}^{n_0}$ est différent de $\bar{x_3}^{n_0}.$

\begin{rem}
Pour prouver le Théorème \ref{Theorem 2.1} de polynomialité des coefficients de structure de l'algèbre de Hecke de la paire \hecke, on a construit une algèbre universelle $\mathcal{A'}_\infty$ qui se projette sur l'algèbre de Hecke de la paire $(\mathcal{S}_{2n},\mathcal{B}_n)$ pour tout $n.$ Cette algèbre, comme on a vu dans la Section \ref{sec:isom_fct_decal_2}, était isomorphe à l'algèbre des fonctions symétriques décalées d'ordre $2.$ On a vu dans la Section \ref{sec:approche_Ivanov_Kerov} que Ivanov et Kerov ont construit dans \cite{Ivanov1999} une algèbre universelle similaire pour prouver la propriété de polynomialité des coefficients de structure du centre de l'algèbre du groupe symétrique. Une algèbre universelle aussi a été construite par Méliot, voir \cite{meliot2013partial}, pour montrer une propriété de polynomialité pour les coefficients de structure du centre de l'algèbre du groupe des matrices inversibles.


Il est possible dans notre cadre général aussi, grâce à la Proposition \ref{F}, de construire une "algèbre non associative" universelle. Ce qui est remarquable dans notre preuve est qu'on n'a pas eu besoin de construire une telle algèbre pour donner la forme des coefficients de structure. En fait, on aurait pu faire la même chose (obtenir la propriété de polynomialité en évitant la construction d'une algèbre universelle) dans le troisième chapitre. En effet, à partir de l'équation \eqref{card_de_H_lambdadeltarho} et de la formule \eqref{form_reliant_c_et_h} reliant $c_{\lambda,\delta}^\rho(n)$ au cardinal de $H_{\lambda\delta}^{\rho}(n)$, on obtient directement un résultat sur la dépendance en $n$ des $c_{\lambda,\delta}^\rho(n).$ En appliquant le morphisme de la Proposition \ref{prop:morph_A'_n+1_et_A'_n}, on obtiendrait notre résultat de polynomialité pour les coefficients de structure de l'algèbre de Hecke de la paire $(\mathcal{S}_{2n},\mathcal{B}_n)$ sans avoir besoin de construire l'algèbre à l'infini. On utilise une idée similaire dans ce chapitre pour donner la propriété de polynomialité des coefficients de structure sans la construction d'une algèbre universelle.


\end{rem}

\section{Applications et résultats de polynomialité}\label{sec:appl_double_classe}

\subsection{Les hypothèses du cadre général dans le cas du groupe symétrique}\label{sec:hyp_grp_sym}

On montre dans cette sous-section que le groupe symétrique $\mathcal{S}_n$ vérifie toutes les hypothèses demandées sur le sous-groupe $K_n$ (c'est-à-dire H.1 à H.6; on ne regarde pas H.0 qui dépend aussi du groupe $G_n$). 
Soit $1\leq k\leq n$, on définit $\mathcal{S}_n^k$ comme étant le groupe symétrique qui agit sur les $n-k$ derniers éléments de l'ensemble $[n]$. Explicitement,
$$\mathcal{S}_n^k:=\lbrace x\in \mathcal{S}_n \text{ tel que }x(1)=1,\, x(2)=2, \cdots , x(k)=k\rbrace.$$
Il est clair que $\mathcal{S}_n^k$ est isomorphe à $\mathcal{S}_{n-k}$ pour tout $1\leq k\leq n$ et donc H.1 est vérifiée. Pour tout $x\in \mathcal{S}_k$ et pour tout $y\in \mathcal{S}_{n-k}$, la composition de $x$ et $y$ est commutative car $x$ et $y$ agissent sur des ensembles disjoints. Cela veut dire que H.2 est aussi vérifiée. Si $1\leq k\leq n,$ alors $\mathcal{S}_{n+1}^{k}\cap \mathcal{S}_n$ est l'ensemble des permutations de $n+1$ qui fixe $1,2,\cdots,k$ et $n+1$, donc $\mathcal{S}_{n+1}^{k}\cap \mathcal{S}_n=\mathcal{S}_{n}^{k}$ et H.3 est vérifiée. Les trois autres hypothèses sont prouvées dans les lemmes ci-dessous.


\begin{lem}\label{lem:inter_dc_S} (H.4 pour $\mathcal{S}_n$)
Soit $z\in \mathcal{S}_n$, alors on a :
$$ \mathcal{S}_{n+1}^{k_1}z\mathcal{S}_{n+1}^{k_2}\cap \mathcal{S}_n=\mathcal{S}_n^{k_1}z\mathcal{S}_n^{k_2}.$$
\end{lem}
\begin{proof}
Le sens $\supseteq$ est évident. Pour l'inclusion inverse, prenons $h=azb\in \mathcal{S}_n$ où $a$ et $b$ sont dans $\mathcal{S}_{n+1}^{k_1}$ et $\mathcal{S}_{n+1}^{k_2}$ respectivement. Notons $a(n+1)=i$ et $a^{-1}(n+1)=j$, on obtient :
$$b(z(i))=h(a^{-1}(i))=h(n+1)=n+1\text{ et } h(j)=b(z(a(j)))=b(z(n+1))=b(n+1).$$
Alors les décompositions en cycle de $a$ et $b$ sont :
$$a=(1)\cdots (k_1)(\cdots j~~n+1~~i\cdots)\cdots$$
et
$$b=(1)\cdots (k_2)(\cdots z(i)~~n+1~~h(j)\cdots)\cdots.$$
Donc $h$ peut s'écrire comme $a'zb'$ où :
$$a'=(1)\cdots (k_1)(\cdots j~~i\cdots)\cdots (n+1)=(n+1~~ j)a\in \mathcal{S}_n^{k_1}$$
et
$$b'=(1)\cdots (k_2)(\cdots z(i)~~h(j)\cdots)\cdots (n+1)=b(h(j)~~n+1)\in \mathcal{S}_n^{k_2}.$$
Cela termine la preuve de ce lemme.
\end{proof}

\begin{lem}\label{lem:k<k_1+k_2_group_sym} (H.5 pour $\mathcal{S}_n$)
Soit $z\in \mathcal{S}_n$, alors on a :
$$\mathrm{k}(\mathcal{S}_n^{k_1} z\mathcal{S}_n^{k_2})\leq |\lbrace 1,\cdots,k_2,z(1),\cdots,z(k_1)\rbrace|\leq k_1+k_2.$$
\end{lem}
\begin{proof}
Il nous convient d'utiliser l'écriture en deux lignes pour les permutations dans cette preuve. La première ligne est constituée des nombres ordonnés de $1$ à $n$ et la deuxième donne l'image de ces entiers. L'ensemble $\mathcal{S}_n^{k_1}z$ contient des éléments de la forme suivante :
$$\begin{matrix}
1 & 2 & \cdots & k_1 & k_1+1 & \cdots & k_1+k_2  & \cdots & n \\
z(1) & z(2) & \cdots & z(k_1) & * & \cdots & * & \cdots & *
\end{matrix}.$$
Les étoiles sont utilisées pour dire que les images ne sont pas fixées. Explicitement, on a : 
$$\mathcal{S}_n^{k_1}z=\lbrace x\in S_n \text{ tel que } x(i)=z(i) \text{ pour } i=1,\cdots, k_1\rbrace$$
et
$$z\mathcal{S}_n^{k_2}=\lbrace x\in S_n \text{ tel que } x(z^{-1}(i))=z(i) \text{ pour } i=1,\cdots, k_2\rbrace.$$ 
Alors on peut explicitement écrire :
$$\mathcal{S}_n^{k_1}z\mathcal{S}_n^{k_2}=\bigcup_{x\in \mathcal{S}_n^{k_1}z}\lbrace y\in S_n \text{ tel que } y(x^{-1}(i))=x(i) \text{ pour } i=1,\cdots, k_2\rbrace.$$
Notons $r$ la taille de l'ensemble $\lbrace 1,\cdots,k_2\rbrace \cap \lbrace z(1),\cdots,z(k_1)\rbrace$ et supposons que $\lbrace h_1,\cdots, h_{k_2-r}\rbrace =\lbrace1,\cdots,k_2\rbrace\setminus \lbrace z(1),\cdots,z(k_1)\rbrace.$ On peut trouver une permutation de la forme
$$\begin{matrix}
z(1) & z(2) & \cdots & z(k_1) & h_1 & \cdots & h_{k_2-r} & * & \cdots & *
\end{matrix},$$
dans $\mathcal{S}_n^{k_1}z.$ Comme la multiplication par un élément de $\mathcal{S}_n^{k_2}$ à droite permet de permuter les éléments plus grands que $k_2$ dans la deuxième ligne définissant cette permutation, l'ensemble $\mathcal{S}_n^{k_1}z\mathcal{S}_n^{k_2}$ contient donc une permutation de la forme
$$\begin{matrix}
1 & 2 & \cdots & k_1 & k_1+1 & \cdots & k_1+k_2-r & k_1+k_2-r+1 & \cdots & n \\
* & ** & \cdots & * & h_1 & \cdots & h_{k_2-r} & k_1+k_2-r+1 & \cdots & n
\end{matrix}$$
Cette permutation est aussi dans $\mathcal{S}_{k_1+k_2-r}.$ J'ai mis $**$ pour dire qu'il y a $r$ éléments fixes (les éléments parmi $z(1),z(2), \cdots, z(k_1)$ qui ne bougent pas après multiplication à droite par $\mathcal{S}_n^{k_2}$) dans les $k_1$ premières images mais qu'on ne s'intéresse pas à ces positions. Il suffit de remarquer maintenant que $k_1+k_2-r=|\lbrace 1,\cdots,k_2,z(1),\cdots,z(k_1)\rbrace|$ ce qui termine la preuve du lemme.
\end{proof}
\begin{lem}\label{lem:hyp_6_grp_sym}
Soit $z$ un élément de $\mathcal{S}_n$, alors on a :
$$z\mathcal{S}_n^{k_1}z^{-1}\cap \mathcal{S}_n^{k_2}\simeq \mathcal{S}_n^{r(z)},$$
où \begin{eqnarray*}
r(z)&=&|\lbrace z^{-1}(1),z^{-1}(2),\cdots , z^{-1}(k_1),1,\cdots ,k_2\rbrace| \\
&=&k_1+k_2-|\lbrace z^{-1}(1),z^{-1}(2),\cdots , z^{-1}(k_1)\rbrace\cap\lbrace 1,\cdots ,k_2\rbrace|.
\end{eqnarray*}
Si $z$ est $(k_1,k_2)$-minimale, alors $r(z)=\mathrm{k}(\mathcal{S}_n^{k_1}z\mathcal{S}_n^{k_2})$ ce qui donne H.6 pour $\mathcal{S}_n.$
\end{lem}
\begin{proof}
Soit $a=zbz^{-1}$ un élément de $\mathcal{S}_n$ qui fixe les $k_2$ premiers éléments tandis que $b$ fixe les $k_1$ premiers éléments. Alors $a$ fixe encore les éléments $z^{-1}(1),\cdots, z^{-1}(k_1)$ ce qui prouve l'inclusion $\subseteq.$ Dans le sens contraire si $x$ est une permutation de $n$ qui fixe les éléments de l'ensemble $\lbrace z^{-1}(1),z^{-1}(2),\cdots , z^{-1}(k_1),1,\cdots ,k_2\rbrace$ alors $x$ est dans $\mathcal{S}_n^{k_2}$ et de plus $z^{-1}xz$ est dans $\mathcal{S}_n^{k_1}$ ce qui implique que $x=zz^{-1}xzz^{-1}$ est dans $z\mathcal{S}_n^{k_1}z^{-1}$ ce qui prouve ce lemme.
\end{proof}

\subsection{Les hypothèses du cadre général dans le cas du groupe hyperoctaédral}\label{sec:cond_group_hyp}

Ici on montre que le groupe hyperoctaédral $\mathcal{B}_n$ vérifie les hypothèses demandées sur le sous-groupe $K_n$ (c'est-à-dire H.1 à H.6; on ne regarde pas H.0 qui dépend encore du groupe $G_n$). 
Soit $1\leq k\leq n$, l'ensemble $\mathcal{B}_n^k$ désigne le sous-groupe hyperoctaédral de $\mathcal{S}_{2n}$ qui agit sur les $2n-2k$ derniers éléments de l'ensemble $[2n].$ Explicitement,
$$\mathcal{B}_n^k:=\lbrace x\in \mathcal{S}_{2n} \text{ tel que }x(1)=1,\, x(2)=2, \cdots , x(2k)=2k\rbrace.$$
Il est clair que $\mathcal{B}_n^k$ est isomorphe à $\mathcal{B}_{n-k}$ pour tout $1\leq k\leq n$ et donc H.1 est vérifiée. Pour tout $x\in \mathcal{B}_k$ et pour tout $y\in \mathcal{B}_{n-k}$, la composition de $x$ et $y$ est commutative puisque les deux permutations $x$ et $y$ agissent sur des ensembles disjoints, cela veut dire que H.2 est vérifiée. Si $1\leq k\leq n,$ alors $\mathcal{B}_{n+1}^{k}\cap \mathcal{B}_n$ est l'ensemble des permutations de $2n+2$ qui fixent $1,2,\cdots,2k, 2n+1$ et $2n+2$, donc $\mathcal{B}_{n+1}^{k}\cap \mathcal{B}_n=\mathcal{B}_{n}^{k}$ et H.3 est vérifiée. Les trois autres hypothèses demandées sont prouvées dans les lemmes ci-dessous.


\begin{lem}\label{lem:hyp_4_hyper} (H.4 pour $\mathcal{B}_n$)
Soit $z\in \mathcal{B}_n$, on a :
$$ \mathcal{B}_{n+1}^{k_1}z\mathcal{B}_{n+1}^{k_2}\cap \mathcal{B}_n=\mathcal{B}_n^{k_1}z\mathcal{B}_n^{k_2}.$$
\end{lem}
\begin{proof}
Le sens $\supseteq$ est évident. Pour l'autre inclusion, si $h=azb\in \mathcal{B}_n$ où $a$ et $b$ sont dans $\mathcal{B}_{n+1}^{k_1}$ et $\mathcal{B}_{n+1}^{k_2}$, alors si on suppose que : $a(2n+2)=i$ et $a(j)=2n+2$, on obtient :
$$b(z(i))=h(a^{-1}(i))=h(2n+2)=2n+2\text{ et } h(j)=b(z(a(j)))=b(z(2n+2))=b(2n+2).$$
Si on désigne par $\rho(i)$ l'entier $i+1$ si $i$ est impair et $i-1$ si $i$ est pair alors la décomposition en produit de cycles de $a$ et $b$ est ainsi :
$$a=(1)\cdots (2k_1)(\cdots j~~2n+2~~i\cdots \rho(j)~~2n+1~~\rho(i)\cdots)\cdots,$$
ou 
$$a=(1)\cdots (2k_1)(\cdots j~~2n+2~~i\cdots)\cdots (\cdots \rho(j)~~2n+1~~\rho(i)\cdots)\cdots,$$
et
$$b=(1)\cdots (2k_2)(\cdots z(i)~~2n+2~~h(j)\cdots \rho(z(i))~~2n+1~~\rho(z(h(j)))\cdots)\cdots,$$
ou 
$$b=(1)\cdots (2k_2)(\cdots z(i)~~2n+2~~h(j)\cdots)\cdots (\cdots\rho(z(i))~~2n+1~~\rho(z(h(j)))\cdots)\cdots.$$
Donc, $h$ s'écrit comme $a'zb'$ où :
$$a'=(1)\cdots (2k_1)(\cdots j~~i\cdots\rho(j)~~\rho(i)\cdots)\cdots (2n+1)(2n+2)=(2n+1~~\rho(j))(2n+2~~ j)a\in \mathcal{B}_n^{k_1},$$
ou 
$$a'=(1)\cdots (2k_1)(\cdots j~~i\cdots)\cdots(\cdots\rho(j)~~\rho(i)\cdots)\cdots (2n+1)(2n+2)$$
$$=(2n+1~~\rho(j))(2n+2~~ j)a\in \mathcal{B}_n^{k_1},$$
et
$$b'=(1)\cdots (2k_2)(\cdots z(i)~~h(j)\cdots \rho(z(i))~~\rho(z(h(j)))\cdots)\cdots(2n+1)(2n+2)$$$$=b(h(j)~~2n+2)(\rho(z(h(j)))~~2n+1)\in \mathcal{B}_n^{k_2},$$
ou 
$$b'=(1)\cdots (2k_2)(\cdots z(i)~~h(j)\cdots)\cdots (\cdots\rho(z(i))~~\rho(z(h(j)))\cdots)\cdots(2n+1)(2n+2)$$$$=b(h(j)~~2n+2)(\rho(z(h(j)))~~2n+1)\in \mathcal{B}_n^{k_2}.$$
Cela termine la preuve de ce lemme.
\end{proof}

\begin{lem} (H.5 pour $\mathcal{B}_n$)
Soit $z\in \mathcal{B}_n$, alors on a :
$$\mathrm{k}(\mathcal{B}_n^{k_1} z\mathcal{B}_n^{k_2})\leq \frac{|\lbrace 1,\cdots,2k_2,z(1),\cdots,z(2k_1)\rbrace|}{2}\leq k_1+k_2.$$
\end{lem}
\begin{proof}
La preuve de ce lemme est identique à celle du Lemme \ref{lem:k<k_1+k_2_group_sym}.
\end{proof}
\begin{lem} (H.6 pour $\mathcal{B}_n$)
Soit $z$ un élément de $\mathcal{B}_n$, on a :
$$z\mathcal{B}_n^{k_1}z^{-1}\cap \mathcal{B}_n^{k_2}\simeq \mathcal{B}_n^{r(z)},$$
où \begin{eqnarray*}
r(z)&=&|\lbrace z^{-1}(1),z^{-1}(2),\cdots , z^{-1}(2k_1),1,\cdots ,2k_2\rbrace| \\
&=&2k_1+2k_2-|\lbrace z^{-1}(1),z^{-1}(2),\cdots , z^{-1}(2k_1)\rbrace\cap\lbrace 1,\cdots ,2k_2\rbrace|.
\end{eqnarray*}
Si $z$ est $(k_1,k_2)-minimal$, alors $r(z)=\mathrm{k}(\mathcal{B}_n^{k_1}z\mathcal{B}_n^{k_2}),$ d'où H.6 pour $\mathcal{B}_n.$
\end{lem}
\begin{proof}
La preuve de ce lemme est identique à celle du Lemme \ref{lem:hyp_6_grp_sym}.
\end{proof}

\subsection{Algèbre de Hecke de la paire $(\mathcal{S}_{2n},\mathcal{B}_n)$}\label{sec:alg_de_hecke}

Comme on vient de le voir, le groupe hypéroctaédral vérifie les hypothèses H.1 à H.6 requises à la Section \ref{sec:hypo_defi}. Pour pouvoir appliquer notre résultat, il nous reste à vérifier l'hypothèse H.0 dans le cas de la suite $(\mathcal{S}_{2n},\mathcal{B}_n)_n.$ En d'autre termes, on doit montrer que pour toute permutation $x$ de $\mathcal{S}_{2n}$ on a :
$$\mathcal{B}_{n+1}x\mathcal{B}_{n+1}\cap \mathcal{S}_{2n}=\mathcal{B}_{n}x\mathcal{B}_{n}.$$
Cela est vrai et une démonstration directe de ce fait semblable à la démonstration de l'hypothèse 4 (voir Lemme \ref{lem:hyp_4_hyper}) pour le groupe hyperoctaédral est facile à établir. Une autre manière de prouver l'hypothèse H.0 pour la suite $(\mathcal{S}_{2n},\mathcal{B}_n)_n$ est d'utiliser la description combinatoire des $\mathcal{B}_n$-doubles-classes. En effet si $x\in \mathcal{S}_{2n}$ alors la double-classe $\mathcal{B}_nx\mathcal{B}_n$ n'est autre que l'ensemble des permutations de $2n$ de coset-type égal à $\ct(x).$ De même, la double-classe $\mathcal{B}_{n+1}x\mathcal{B}_{n+1}$ ($x$ est maintenant vue comme une permutation de $2n+2$) est l'ensemble des permutations de $2n+2$ de coset-type égal à $\ct(x)\cup (1)$ ce qui correspond à $\mathcal{B}_nx\mathcal{B}_n$ quand on prend l'intersection avec $\mathcal{S}_{2n}.$

Soit $\lambda$ un élément de $\mathcal{PP}_{\leq n}$, la double-classe de $\mathcal{B}_n$ dans $\mathcal{S}_{2n}$ associée à $\lambda$ est, d'après le troisième chapitre, ainsi :
$${K}_{\underline{\lambda}_n}=\lbrace \omega\in \mathcal{S}_{2n}\text{ tel que } \ct(\omega)=\lambda\cup (1^{n-|\lambda|})\rbrace.$$
La taille de ${K}_{\underline{\lambda}_n}$ est donnée ainsi :
\begin{equation}
|{K}_{\underline{\lambda}_n}|=\frac{(2^nn!)^2}{z_{2\lambda}2^{n-|\lambda|}(n-|\lambda|)!}.
\end{equation} 

On fixe trois partitions propres $\lambda$, $\delta$ et $\rho.$ Soit $n$ un entier suffisamment grand. D'après le Théorème \ref{mini_th}, le coefficient ${\bf K}_{\underline{\rho}_n}$ dans le produit ${\bf K}_{\underline{\lambda}_n}{\bf K}_{\underline{\delta}_n}$ est de la forme :
\begin{equation}
\frac{\frac{(2^nn!)^2}{z_{2\lambda}2^{n-|\lambda|}(n-|\lambda|)!}\frac{(2^nn!)^2}{z_{2\delta}2^{n-|\delta|}(n-|\delta|)!}2^{n-|\lambda|}(n-|\lambda|)!2^{n-|\delta|}(n-|\delta|)!}{2^nn!\frac{(2^nn!)^2}{z_{2\rho}2^{n-|\rho|}(n-|\rho|)!}}\sum_{|\rho|\leq k\leq |\lambda|+|\delta|}\frac{a_{\lambda\delta}^{\rho}(k)}{2^{n-k}({n-k)!}}
\end{equation}
ce qui est égal à :
\begin{equation}
2^nn!\frac{z_{2\rho}}{z_{2\lambda}z_{2\delta}}\sum_{|\rho|\leq k\leq |\lambda|+|\delta|}a_{\lambda\delta}^{\rho}(k)2^{k-|\rho|}\frac{(n-|\rho|)!}{(n-k)!}.
\end{equation}
Cela nous donne le corollaire suivant.

\begin{cor}
Soient $\lambda$ et $\delta$ deux partitions propres. Soit $n$ un entier suffisamment grand et considérons l'équation suivante :
$${\bf K}_{\underline{\lambda}_n}{\bf K}_{\underline{\delta}_n}=\sum_{\rho \text{ partition propre}} {c'}_{\lambda\delta}^{\rho}(n) {\bf K}_{\underline{\rho}_n}.$$
Les coefficients $\frac{{c'}_{\lambda\delta}^{\rho}(n)}{2^nn!}$ sont des polynômes en $n$ avec coefficients rationnels.
\end{cor}

Ce corollaire est le résultat de polynomialité des coefficients de structure de l'algèbre\linebreak \Hecke\,  présenté dans le troisième chapitre. 


Par ailleurs, le Théorème \ref{main_th} peut être utilisé pour trouver les valeurs exactes des coefficients de structure de l'algèbre de Hecke de la paire $(\mathcal{S}_{2n},\mathcal{B}_n)$ mais c'est assez complexe, comme le montre l'exemple suivant.

\begin{ex}


Soient $k_1=k_2=2,$ $x_1=(1~~2~~4~~3)$ et $x_2=(1~~4~~2)(3).$ Soit $n$ un entier suffisamment grand, alors $$\sum_{y_1\sim x_1}y_1=\sum_{y_2\sim x_2}y_2={\bf K}_{\underline{(2)}_n}.$$
Comme l'indice $k$ dans la somme du Théorème \ref{main_th} doit être inférieur à $k_1+k_2$ et plus grand que $k_1$ et $k_2$, $k$ est donc $2$, $3$ ou bien $4.$\\
Pour $k=2$, toutes les permutations de $\mathcal{B}_2$ sont $(2,2)$-minimales.\\
Pour $k=3$, les permutations $(2,2)$-minimales sont celles qui appartiennent à $\mathcal{B}_3$ et qui envoient l'ensemble $\lbrace 5,6\rbrace$ sur $\lbrace 1,2\rbrace$ ou bien $\lbrace 3,4\rbrace.$\\
Pour $k=4$, les permutations $(2,2)$-minimales sont celles qui appartiennent à $\mathcal{B}_4$ et qui envoient l'ensemble $\lbrace 1,2,3,4\rbrace$ sur $\lbrace 5,6,7,8\rbrace.$\\
Donc, pour $k=2$, les permutations $x\in \mathcal{S}_4$ tel que $x_1^{-1}xx_2^{-1}$ est $(2,2)$-minimale sont les permutations de l'ensemble $x_1\mathcal{B}_2x_2,$ qui sont, ${\color{green}(1~~2)(3~~4)}$, ${\color{blue}(1~~2~~4)(3)}$, ${\color{blue}(1)(2~~3~~4)}$, ${\color{green}(1~~4~~2~~3)}$, ${\color{blue}(1~~3~~2)(4)}$, ${\color{green}(1~~3~~2~~4)}$, ${\color{green}(1)(2)(3~~4)}$, ${\color{blue}(1~~4~~3)(2)}$.\\
De la même façon, pour $k=3$, les permutations $x\in \mathcal{S}_4$ tel que $x_1^{-1}xx_2^{-1}$ est $(2,2)$-minimale sont :
$\begin{pmatrix}
2&6&3&5&4&1\\
3&6&2&5&4&1\\
2&5&3&6&4&1\\
3&5&2&6&4&1\\
5&2&6&3&4&1\\
6&2&5&3&4&1\\
5&3&6&2&4&1\\
6&3&5&2&4&1
\end{pmatrix}$~~$\begin{pmatrix}
2&6&3&5&1&4\\
3&6&2&5&1&4\\
2&5&3&6&1&4\\
3&5&2&6&1&4\\
5&2&6&3&1&4\\
6&2&5&3&1&4\\
5&3&6&2&1&4\\
6&3&5&2&1&4
\end{pmatrix}$~~$\begin{pmatrix}
1&6&4&5&2&3\\
1&5&4&6&2&3\\
4&6&1&5&2&3\\
4&5&1&6&2&3\\
5&1&6&4&2&3\\
5&4&6&1&2&3\\
6&1&5&4&2&3\\
6&4&5&1&2&3
\end{pmatrix}$~~$\begin{pmatrix}
1&6&4&5&3&2\\
1&5&4&6&3&2\\
4&6&1&5&3&2\\
4&5&1&6&3&2\\
5&1&6&4&3&2\\
5&4&6&1&3&2\\
6&1&5&4&3&2\\
6&4&5&1&3&2
\end{pmatrix}.$\\
Dans chaque matrice ci-dessus, chaque ligne définie une permutation.\\
Pour $k=4$, les permutations $x\in \mathcal{S}_4$ tel que $x_1^{-1}xx_2^{-1}$ est $(2,2)$-minimale sont celles de coset-type $(2,2)$ tel que l'image de $\lbrace 1,2,3,4\rbrace$ est $\lbrace 5,6,7,8\rbrace.$\\
Les permutations écrites en vert possèdent $(1^2)$ comme coset-type. Ces permutations donnent le coefficient de ${\bf K}_{\underline{\emptyset}_n}$. Pour chacune d'elles $|Cl_{k_1,k_2}(x_1^{-1}xx_2^{-1})\cap K_{m_{k_1,k_2}(x_1^{-1}xx_2^{-1})}|=1.$ Donc, le coefficient de ${\bf K}_{\underline{\emptyset}_n}$ est, d'après le Théorème \ref{main_th} :
$$4\frac{(2^nn!n(n-1))^2(2^{n-2}(n-2)!)^2}{2^nn!2^nn!2^{n-2}(n-2)!}=2^nn!n(n-1).$$
Les permutations écrites en bleu possèdent $(2)$ comme coset-type. Ces permutations donnent le coefficient de ${\bf K}_{\underline{(2)}_n}$. Pour chacune d'elles  $|Cl_{k_1,k_2}(x_1^{-1}xx_2^{-1})\cap K_{m_{k_1,k_2}(x_1^{-1}xx_2^{-1})}|=1$. Donc, le coefficient de ${\bf K}_{\underline{(2)}_n}$ est, d'après le Théorème \ref{main_th} :
$$4\frac{(2^nn!n(n-1))^2(2^{n-2}(n-2)!)^2}{2^nn!2^nn!n(n-1)2^{n-2}(n-2)!}=2^nn!.$$
Toutes les permutations écrites dans les "matrices" ci-dessus possèdent $(3)$ comme coset-type. Les permutations donnent le coefficient de ${\bf K}_{\underline{(3)}_n}$. Pour chacune d'elles $|Cl_{k_1,k_2}(x_1^{-1}xx_2^{-1})\cap K_{m_{k_1,k_2}(x_1^{-1}xx_2^{-1})}|=4.$ Donc, le coefficient de ${\bf K}_{\underline{(3)}_n}$ est, d'après le Théorème \ref{main_th} :
$$8.4\frac{(2^nn!n(n-1))^2(2^{n-2}(n-2)!)^2}{2^nn!\frac{4}{3}2^nn!n(n-1)(n-2)2^{n-3}(n-3)!4}=3\cdot2^nn!.$$
Les permutations $x\in \mathcal{S}_4$ de coset-type $(2,2)$ tel que l'image de $\lbrace 1,2,3,4\rbrace$ est $\lbrace 5,6,7,8\rbrace$ nous donnent le coefficient ${\bf K}_{\underline{(2^2)}_n}$. Pour chacune d'elles, on a $|Cl_{k_1,k_2}(x_1^{-1}xx_2^{-1})\cap K_{m_{k_1,k_2}(x_1^{-1}xx_2^{-1})}|=64$. Donc, le coefficient de ${\bf K}_{\underline{(2,2)}_n}$ est, d'après le Théorème \ref{main_th} :
$$64\frac{(2^nn!n(n-1))^2(2^{n-2}(n-2)!)^2}{2^nn!2^{n-1}n!n(n-1)(n-2)(n-3)2^{n-4}(n-4)!64}=2\cdot2^nn!.$$

On en déduit la formule complète du produit ${\bf K}_{\underline{(2)}_n}\cdot {\bf K}_{\underline{(2)}_n}$ pour tout $n\geq 4$,
$${\bf K}_{\underline{(2)}_n}\cdot {\bf K}_{\underline{(2)}_n}=2^nn!n(n-1){\bf K}_{\underline{\emptyset}_n}+2^{n}n!{\bf K}_{\underline{(2)}_n}+2^{n}n!3{\bf K}_{\underline{(3)}_n}+2^{n}n!2{\bf K}_{\underline{(2^2)}_n}.$$

\end{ex}

On retrouve les valeurs des coefficients de structure du produit ${\bf K}_{\underline{(2)}_n}\cdot {\bf K}_{\underline{(2)}_n}$ données dans l'Exemple \ref{ex:produit_hecke_algebra}.

\subsection{Algèbre de doubles-classes de $\diag(\mathcal{S}_{n-1})$ dans $\mathcal{S}_n\times \mathcal{S}_{n-1}^{opp}$}\label{alg_double_classe_diag(S_n-1)}

Dans cette section, on considère $\mathcal{S}_{n-1}$ comme étant le sous-groupe $\mathcal{S}_n^1$ de $\mathcal{S}_n.$ Cela veut dire que $\mathcal{S}_{n-1}$ est l'ensemble des permutations de $n$ qui fixent $1.$ Il faut noter que lorsqu'on a montré les hypothèses H.1 à H.6 pour le groupe symétrique dans la Section \ref{sec:hyp_grp_sym}, on a vu le groupe $\mathcal{S}_{n-1}$ comme étant le sous-groupe de $\mathcal{S}_{n}$ des permutations qui fixent $n.$ En voyant $\mathcal{S}_{n-1}$ comme étant le sous-groupe $\mathcal{S}_n^1$ de $\mathcal{S}_n,$ les hypothèses H.1 à H.6 restent valables et leurs preuves sont les mêmes 'à des isomorphismes près' que dans la Section \ref{sec:hyp_grp_sym}.

L'hypothèse H.0 est encore assurée pour la suite $(\mathcal{S}_n\times \mathcal{S}_{n-1}^{opp},\diag(\mathcal{S}_{n-1})).$ Pour la montrer il suffit de vérifier que si $(x,y)\in \mathcal{S}_n\times \mathcal{S}_{n-1}^{opp}$ et si $a,b\in \mathcal{S}_n$ tel que $(axb,b^{-1}ya^{-1})\in \mathcal{S}_n\times \mathcal{S}_{n-1}^{opp}$ alors il existe $a',b'\in \mathcal{S}_{n-1}$ tel que $(axb,b^{-1}ya^{-1})=(a'xb',{b'}^{-1}y{a'}^{-1}).$ Cela se prouve d'une manière similaire à la preuve du Lemme \ref{lem:inter_dc_S}.

L'algèbre de doubles-classes de $\diag(\mathcal{S}_{n-1})$ dans $\mathcal{S}_n\times \mathcal{S}_{n-1}^{opp}$ a été étudiée par Brender en 1976, voir \cite{brender1976spherical}. En 2007, Strahov a montré, voir \cite[Proposition 2.2.1]{strahov2007generalized}, que la paire $(\mathcal{S}_n\times \mathcal{S}_{n-1}^{opp},\diag(\mathcal{S}_{n-1}))$ est une paire de Gelfand -- l'algèbre de doubles-classes de $\diag(\mathcal{S}_{n-1})$ dans $\mathcal{S}_n\times \mathcal{S}_{n-1}^{opp}$ est commutative -- et il a étudié les fonctions sphériques zonales associées. Ici on s'intéresse aux coefficients de structure de cette algèbre et on établit à la fin de cette sous-section une propriété de polynomialité.

Deux permutations $x$ et $y$ de $\mathcal{S}_n$ sont conjuguées par rapport à $\mathcal{S}_{n-1}$ si $x=zyz^{-1}$ pour un certain élément $z\in S_{n-1}.$ Soit $x$ une permutation de $\mathcal{S}_n$ et soit $k$ une permutation de $\mathcal{S}_{n-1},$ on sait que $x$ et $kxk^{-1}$ possèdent le même type-cyclique, mais de plus si $c=(1,a_2, \cdots ,a_{l(c)})$ est le cycle de $x$ contenant $1$ alors le cycle $(1,z^{-1}(a_2) \cdots ,z^{-1}(a_{l(c)}))$ de $zxz^{-1}$ contient $1$ et il possède la même longueur que $c.$ D'un autre côté si deux permutations de $n$ possèdent le même type-cyclique et si les cycles contenant $1$ dans leur décomposition possèdent la même longueur, alors il est facile de voir que les deux permutations sont conjugués par rapport à $\mathcal{S}_{n-1}.$

Les classes de conjugaison par rapport à $\mathcal{S}_{n-1}$ sont donc indexées par les paires $(i,\lambda)$ où $i$ est un entier entre $1$ et $n$ et $\lambda$ est une partition de $n-i.$ La classe de conjugaison par rapport à $\mathcal{S}_{n-1}$ associée à la paire $(i,\lambda)$ est donnée par :
$$C_{(i,\lambda)}=\lbrace x\in \mathcal{S}_n \text{ tel que $1$ est dans un cycle $c$ de longueur $i$ et } type-cyclique(x\setminus c)=\lambda\rbrace.$$
La taille d'une telle classe de conjugaison est d'après \cite[page 118]{strahov2007generalized} :
$$|C_{(i,\lambda)}|=\frac{(n-1)!}{z_\lambda}.$$
L'étude des classes de conjugaison par rapport à $\mathcal{S}_{n-1}$ est détaillée dans l'article \cite{jackson2012character} de Jackson et Sloss où les auteurs utilisent les paires $(\lambda,i)$ où $\lambda$ est une partition de $n$ contenant forcément une part $i$ (ces partitions sont en bijection avec l'ensemble des partitions de $n-i$) pour indexer les classes de conjugaison.

Soit $(a,b)$ un élément de $\mathcal{S}_n\times \mathcal{S}^{opp}_{n-1}$ et soient $x$ et $y$ deux éléments de $\mathcal{S}_{n-1}$, alors on a :
$$(x,x^{-1})\cdot (a,b)\cdot (y,y^{-1})=(xay,y^{-1}bx^{-1}).$$
Donc deux éléments $(a,b)$ et $(c,d)$ de $\mathcal{S}_n\times \mathcal{S}^{opp}_{n-1}$ sont dans la même $diag(\mathcal{S}_{n-1})$-double-classe si et seulement si $ab$ et $cd$ sont conjugués par rapport à $\mathcal{S}_{n-1}.$

L'ensemble de $\diag(\mathcal{S}_{n-1})$-double-classes est donc aussi indexé par les paires $(i,\lambda)$ où $i$ est un entier entre $1$ et $n$ et $\lambda$ est une partition de $n-i.$  La double-classe associée à la  paire $(i,\lambda)$ est donnée par :
$$DC_{(i,\lambda)}=\lbrace (a,b)\in \mathcal{S}_n\times \mathcal{S}^{opp}_{n-1} \text{ tel que } ab\in C_{(i,\lambda)} \rbrace.$$ 

Pour tout $a\in C_{(i,\lambda)}$ et pour tout $x\in S_{n-1},$ les éléments $(ax,x^{-1})$ sont tous différents dans $DC_{(i,\lambda)}.$ Alors, on a :
$$|DC_{(i,\lambda)}|=|S_{n-1}||C_{(i,\lambda)}|=\frac{(n-1)!^2}{z_\lambda}.$$

Soient $i$ et $j$ deux entiers entre $1$ et $n$ et soient $\lambda$ et $\delta$ deux partitions de $n-i$ et $n-j$. Les coefficients de structure $c_{(i,\lambda)(j,\delta)}^{(r,\rho)}$ de l'algèbre des doubles-classes $\mathbb{C}[\diag(\mathcal{S}_{n-1}))\setminus \mathcal{S}_n\times \mathcal{S}^{opp}_{n-1}/ \diag(\mathcal{S}_{n-1}))]$ sont définis par l'équation suivante :

\begin{equation*}DC_{(i,\lambda)}DC_{(j,\delta)}=\sum_{1\leq k\leq n \atop{\rho\vdash n-r}}c^{(r,\rho)}_{(i,\lambda)(j,\delta)}DC_{(r,\rho)}.
\end{equation*}

On dit qu'une paire $(i,\lambda)$ où $i$ est un entier et $\lambda$ est une partition est {\em propre} si la partition $\lambda$ est propre. Si $(i,\lambda)$ est propre, pour tout entier $n\geq i+|\lambda|$ la paire $\underline{(i,\lambda)}_n$ est définie ainsi :
$$\underline{(i,\lambda)}_n=(i,\underline{\lambda}_{(n-i)}),$$
et on a 
$$|DC_{\underline{(i,\lambda)}_n}|=\frac{(n-1)!^2}{z_\lambda(n-i)!}.$$
Si $(i,\lambda)$ et $(j,\delta)$ sont deux paires propres, alors d'après l'équation précédente, pour tout entier $n\geq i+|\lambda|, j+|\delta|$ on peut écrire :
\begin{equation*}DC_{\underline{(i,\lambda)}_n}DC_{\underline{(j,\delta)}_n}=\sum_{1\leq r\leq n \atop{\rho\in \mathcal{PP}_{\leq n-r}}}c^{(r,\rho)}_{(i,\lambda)(j,\delta)}(n)DC_{\underline{(r,\rho)}_n}.
\end{equation*}

D'après notre résultat pour les coefficients de structure des algèbres de doubles-classes donné par le Théorème \ref{mini_th}, il existe des entiers rationnels $a^{(r,\rho)}_{(i,\lambda)(j,\delta)}(k)$ tous indépendants de $n$ tel que :

\begin{eqnarray*}
c^{(r,\rho)}_{(i,\lambda)(j,\delta)}(n)&=&\frac{\frac{(n-1)!^2}{z_\lambda(n-i)!}\frac{(n-1)!^2}{z_\delta(n-j)!}(n-1-i-|\lambda|)!(n-1-j-|\delta|)!}{(n-1)!\frac{(n-1)!^2}{z_\rho(n-r)!}}\\
&&~~~~~~~~~~~~~~~~~~~~~~~~~~\sum_{ r+|\rho|\leq k\leq \min(i+|\lambda|+j+|\delta|,n)}\frac{a^{(r,\rho)}_{(i,\lambda)(j,\delta)}(k)}{(n-1-k)!}\\
&=&\frac{z_\rho(n-1)!}{z_\lambda z_\delta}\frac{1}{(n-i-|\lambda|)\cdots (n-i)\cdot (n-j-|\delta|)\cdots (n-j)} \\
&&~~~~~~~~~~~~~~~~~~~~~~~~~~\sum_{ r+|\rho|\leq k\leq \min(i+|\lambda|+j+|\delta|,n)}a^{(r,\rho)}_{(i,\lambda)(j,\delta)}(k) (n-k)\cdots (n-r).
\end{eqnarray*} 

\begin{cor}
Soient $\lambda, \delta$ et $\rho$ trois partitions propres, $i,j$ et $r$ trois entiers et soit $n$ un entier plus grand que $i+|\lambda|, j+|\delta|$ et $r+|\rho|,$ alors le quotient

\begin{equation*}
(n-i-|\lambda|)\cdots (n-i)\cdot (n-j-|\delta|)\cdots (n-j)\cdot \frac{c^{(r,\rho)}_{(i,\lambda)(j,\delta)}(n)}{(n-1)!}
\end{equation*}
est un polynôme en $n$ de degré inférieur ou égale à $i+|\lambda|+j+|\delta|-r+1.$
\end{cor}

Les coefficients de structure $c^{(r,\rho)}_{(i,\lambda)(j,\delta)}(n)$ possèdent une interprétation combinatoire via des graphes spéciaux appelés \textit{dipôles}, voir l'article \cite{Jackson20121856}
de Jackson et Sloss pour plus de détails. Il faut noter que ces deux auteurs donnent, voir \cite{jackson2012character}, un théorème analogue à celui de Frobenius exprimant ces coefficients de structure en fonction des caractères (généralisés) du groupe symétrique.

\section{Un cadre pour les centres des algèbres de groupe}\label{sec:appl_centr}

Dans la sous-section \ref{sec:ctr_double_classe}, on a montré que le centre d'une algèbre d'un groupe fini $G$ peut être vu comme étant l'algèbre de doubles-classes de $\diag(G)$ dans $G\times G^{opp}.$ Cela nous permet de donner la version des centres d'algèbres de groupe des Théorèmes \ref{main_th} et \ref{mini_th}. 

On considère une suite $(G_n)_n$ où $G_n$ est un groupe pour tout $n.$ D'après l'étude faite dans la Section \ref{sec:ctr_double_classe}, il faut que la suite $(G_n\times G_n^{opp},\diag(G_n))_n$ vérifie les hypothèses H.0 à H.6 de la Section \ref{sec:hypo_defi} pour pouvoir appliquer les Théorèmes \ref{main_th} et \ref{mini_th}. Comme on a déjà mentionné, les hypothèses H.1 à H.6 ne dépendent que de la suite $(\diag(G_n))_n$ --c'est-à-dire de la suite $(G_n)_n$--. Il reste à voir ce que veut dire que H.0 est vérifiée pour $(G_n\times G_n^{opp},\diag(G_n))_n$ pour la suite $(G_n)_n.$ On montre le lemme suivant :
\begin{lem}
l'hypothèse H.0 est vérifiée pour $(G_n\times G_n^{opp},\diag(G_n))_n$ si et seulement si la suite $(G_n)_n$ vérifie l'hypothèse $\text{H}^{'}.0$ suivante :
\begin{enumerate}
\item[$\text{H}^{'}.0$]\label{hyp'0} $C_g(n+1) \cap G_n=C_g(n)$ pour tout $g\in G_n,$ où $C_g(n)$\footnote{On a utilisé jusqu'ici $C_g$ pour désigner la classe de conjugaison de $g.$ On préfère la notation $C_g(n)$ dans ce chapitre pour éviter la confusion puisqu'on travaille avec une suite de groupe.} est la classe de conjugaison de $g$ dans $G_n.$
\end{enumerate}
\end{lem} 
\begin{proof}
En effet, si $(G_n\times G_n^{opp},\diag(G_n))_n$ vérifie H.0 et si $y=xgx^{-1}$ est un élément de $C_g(n+1)\cap G_n$ avec $g\in G_n$ et $x\in G_{n+1}$ alors $$(1,y)=(x^{-1},x)(1,g)(x,x^{-1})\in \diag(G_{n+1})(1,g)\diag(G_{n+1})\cap G_n\times G_n^{opp}.$$ Mais $\diag(G_{n+1})(1,g)\diag(G_{n+1})\cap G_n\times G_n^{opp}$ n'est autre que $\diag(G_{n})(1,g)\diag(G_{n})$ d'après H.0. Cela veut dire qu'il existe $x'\in G_n$ tel que $y=x'g{x'}^{-1}$ et donc $y\in C_g(n).$ Donc si H.0 est vérifiée pour $(G_n\times G_n^{opp},\diag(G_n))_n$ alors $\text{H}^{'}.0$ est vérifiée pour $(G_n)_n.$ Réciproquement, si $(G_n)_n$ vérifie $\text{H}^{'}.0$ et si $(x,y)\in \diag(G_{n+1})(g,f)\diag(G_{n+1})$ avec $(x,y),(g,f)$ dans $G_n\times G_n^{opp}$ alors il existe $t\in G_{n+1}$ et $r\in G_{n+1}$ tel que :
$$(x,y)=(t,t^{-1})(g,f)(r,r^{-1})=(tgr,r^{-1}ft^{-1}).$$
Donc $xy=tgft^{-1}\in C_{gf}(n+1)\cap G_n$ ce qui est $C_{gf}(n)$ d'après $\text{H}^{'}.0$ (car $gf\in G_n$) donc $(x,y)\in \diag(G_{n})(g,f)\diag(G_{n})$ ce qui termine la preuve de ce lemme.
\end{proof}


\begin{theoreme}\label{center_main_th}
Soit $(G_n)_n$ une suite de groupes finis vérifiant l'hypothèse $\text{H}^{'}.0$ et les hypothèses H.1 à H.6 de la Section \ref{sec:hypo_defi}. Soient $f,$ $h$ et $g$ trois éléments de $G_{n_0}$ pour un entier $n_0$ fixé et soient $k_1=\mathrm{k}(C_{f}(n_0)),$ $k_2=\mathrm{k}(C_{h}(n_0))$ et $k_3=\mathrm{k}(C_{g}(n_0)).$ Le coefficient de structure ${c}_{f,h}^{g}(n_0)$ de ${\bf C}_{g}(n_0)$ dans le produit ${\bf C}_{f}(n_0){\bf C}_{h}(n_0)$ est donné par la formule suivante :
\begin{eqnarray*}
{c}_{f,h}^{g}(n_0)&=&\frac{|C_{f}(n_0)||C_{h}(n_0)||G_{n_0-k_1}||G_{n_0-k_2}|}{|G_{n_0}||C_{g}(n_0)|} \nonumber \\
&&\sum_{ \max(k_1,k_2,k_3)\leq k\leq \min(k_1+k_2,n_0), x\in G_k, \atop{\text{ $f^{-1}xh^{-1}\in G_k$ et est $(k_1,k_2)$-minimal} \atop{xhx^{-1}f\in C_{g}(n_0)}}}\frac{1}{|G_{n_0-k}||G_{n_0}^{k_1}f^{-1}xh^{-1}G_{n_0}^{k_2}\cap G_{m_{k_1,k_2}(f^{-1}xh^{-1})}|}.
\end{eqnarray*}
\end{theoreme}
\begin{proof}
La formule de ${c}_{f,h}^{g}(n_0)$ s'obtient à partir du résultat du Théorème \ref{main_th} dans le cas particulier de la suite $(G_n\times G_n^{opp},\diag(G_n))_n$ en utilisant les Propositions \ref{prop:relation_taille_dc_class} et \ref{prop:lien_taille_classe_et_double_classe}. Comme on suppose que les hypothèses $\text{H}^{'}.0$ et H.1 à H.6 sont vérifiées pour la suite $(G_n)_n$ alors la suite $(G_n\times G_n^{opp},\diag(G_n))_n$ vérifie les hypothèses H.0 à H.6. Pour retrouver le résultat de ce théorème, on applique le Théorème \ref{main_th} à cette suite pour les éléments $(f,1),$ $(h,1)$ et $(g,1).$ On obtient :
\begin{tiny}
$$c_{(f,1),(h,1)}^{(g,1)}(n_0)=\frac{|DC_{(f,1)}(n_0)||DC_{(h,1)}(n_0)||\diag(G)_{n_0-k_1}||\diag(G)_{n_0-k_2}|}{|\diag(G)_{n_0}||DC_{(g,1)}(n_0)|} $$
$${\small \sum_{ \max(k_1,k_2,k_3)\leq k\leq \min(k_1+k_2,n_0), (x,y)\in G_k\times G_k, \atop{\text{ $(f^{-1}xh^{-1},y)\in \diag(G_k)$ et est $(k_1,k_2)-minimal$} \atop{DC_{(x,y)}(n_0)=DC_{(g,1)}(n_0)}}}\frac{1}{|\diag(G)_{n_0-k}||\diag(G)_{n_0}^{k_1}(f^{-1}xh^{-1},y)\diag(G)_{n_0}^{k_2}\cap \diag(G)_{m_{k_1,k_2}(f^{-1}xh^{-1},y)}|}.}$$
\end{tiny}
La condition $(f^{-1}xh^{-1},y)\in \diag(G_k)$ équivaut à $y=hx^{-1}f$ et nous permet donc de ramener cette somme sur les éléments de $G_k.$ De plus, la condition $DC_{(x,y)}(n_0)=DC_{(g,1)}(n_0)$ sera équivaut à $xhx^{-1}f\in C_g(n_0).$ D'après la Proposition \ref{prop:lien_taille_classe_et_double_classe}, on a :
$${c}_{f,h}^{g}(n_0)=\frac{{c}_{(f,1),(h,1)}^{(g,1)}(n_0)}{|G_{n_0}|}.$$ En utilisant la Proposition \ref{prop:relation_taille_dc_class} et après simplification, on obtient :
\begin{eqnarray*}
{c}_{f,h}^{g}(n_0)&=&\frac{|C_{f}(n_0)||C_{h}(n_0)||G_{n_0-k_1}||G_{n_0-k_2}|}{|G_{n_0}||C_{g}(n_0)|} \nonumber \\
&&\sum_{ \max(k_1,k_2,k_3)\leq k\leq \min(k_1+k_2,n_0), x\in G_k, \atop{\text{ $f^{-1}xh^{-1}\in G_k$ et est $(k_1,k_2)$-minimal} \atop{xhx^{-1}f\in C_{g}(n_0)}}}\frac{1}{|G_{n_0-k}||G_{n_0}^{k_1}f^{-1}xh^{-1}G_{n_0}^{k_2}\cap G_{m_{k_1,k_2}(f^{-1}xh^{-1})}|}.
\end{eqnarray*}
Cela termine la preuve de ce théorème.
\end{proof}


\begin{theoreme}\label{center_mini_th}
Soit $(G_n)_n$ une suite de groupes finis vérifiant l'hypothèse $\text{H}^{'}.0$ et les hypothèses H.1 à H.6 de la Section \ref{sec:hypo_defi}. Soient $f,$ $h$ et $g$ trois éléments de $G_{n_0}$ pour un entier $n_0$ fixé et soient $k_1=\mathrm{k}(C_f(n_0)),$ $k_2=\mathrm{k}(C_h(n_0))$ et $k_3=\mathrm{k}(C_g(n_0)).$ 
Pour tout $n\geq n_0,$ le coefficient de structure $c_{f,h}^g(n)$ de $\mathbf{C}_g(n)$ dans le produit $\mathbf{C}_f(n)\mathbf{C}_h(n)$ du centre de l'algèbre de groupe $G_n$ s'écrit sous la forme suivante :

\begin{equation}
c_{f,h}^g(n)=\frac{|C_f(n)||C_h(n)||G_{n-k_1}||G_{n-k_2}|}{|G_n||C_g(n)|}\sum_{ k_3\leq k\leq \max(k_1+k_2,n)}\frac{a_{f,h}^{g}(k)}{|G_{n-k}|},
\end{equation}
où les nombres $a_{f,h}^{g}(k)$ sont des rationnels positifs indépendants de $n.$
\end{theoreme}
\begin{proof}
C'est une conséquence du Théorème \ref{center_main_th}.
\end{proof}

\subsection{Le centre de l'algèbre du groupe symétrique}\label{sec:cent_grp_symetrique}

On commence par rappeler que les hypothèses H.1 à H.6 sont vérifiées pour les groupes symétriques. Pour appliquer notre résultat il faut vérifier aussi l'hypothèse $\text{H}^{'}.0$ pour la suite $(Z(\mathbb{C}[\mathcal{S}_n]))_n.$ Soit $\omega$ une permutation de $n,$ la classe de conjugaison $C_\omega(n)$ de $\omega$ dans $\mathcal{S}_n$ correspond à l'ensemble des permutations de $n$ ayant le même type-cyclique que $\omega.$ De même, en regardant $\omega$ comme permutation de $n+1,$ la classe la classe de conjugaison $C_\omega(n+1)$ de $\omega$ est l'ensemble des permutations de $n+1$ ayant le type-cyclique $\type-cyclique(\omega)\cup (1)$ ce qui correspond à $C_\omega(n)$ quand on prend l'intersection avec $\mathcal{S}_n.$ Donc la suite $(Z(\mathbb{C}[\mathcal{S}_n]))_n$ vérifie bien $\text{H}^{'}.0$ et le Théorème \ref{center_mini_th} peut être appliqué dans ce cas.

On rappelle que la famille $({\bf C}_{\underline{\lambda}_n})_{\lambda\in \mathcal{PP}_{\leq n}},$ où
$${C}_{\underline{\lambda}_n}=\lbrace \omega\in \mathcal{S}_n \text{ tel que } type-cyclique(\omega)=\lambda\cup (1^{n-|\lambda|})\rbrace,$$ 
est une base du centre de l'algèbre du groupe symétrique.

D'après le Corollaire \ref{cor:taille_classe_de_conj_S_n}, la taille de ${C}_{\underline{\lambda}_n}$ est ainsi :
$$|{C}_{\underline{\lambda}_n}|=\frac{n!}{z_\lambda \cdot (n-|\lambda|)!}.$$ 

Soient $\lambda$ et $\delta$ deux partitions propres. Soit $n$ un entier suffisamment grand, on va appliquer le Théorème \ref{center_mini_th} dans le cas du groupe symétrique. Pour une partition propre $\rho$ fixée, le coefficient de ${\bf C}_{\underline{\rho}_n}$ dans le produit ${\bf C}_{\underline{\lambda}_n}{\bf C}_{\underline{\delta}_n}$ est, d'après le Théorème \ref{center_mini_th}, donné ainsi :

\begin{equation}
\frac{\frac{n!}{z_\lambda(n-|\lambda|)!}\frac{n!}{z_\delta(n-|\delta|)!}(n-|\lambda|)!(n-|\delta|)!}{n!\frac{n!}{z_\rho(n-|\rho|)!}}\sum_{|\rho|\leq k\leq |\lambda|+|\delta|}a_{\lambda\delta}^{\rho}(k)\frac{1}{(n-k)!},
\end{equation}
ce qui est égal à :
\begin{equation}
\frac{z_\rho}{z_\lambda z_\delta}\sum_{|\rho|\leq k\leq |\lambda|+|\delta|}a_{\lambda\delta}^{\rho}(k)\frac{(n-|\rho|)!}{(n-k)!}.
\end{equation}

Pour tout $|\rho|\leq k$, le quotient $\frac{(n-|\rho|)!}{(n-k)!}$ est un polynôme en $n$ avec un degré égal à $k-|\rho|.$
\begin{cor}
Soient $\lambda$ et $\delta$ deux partitions propres. Soit $n$ un entier suffisamment grand et considérons l'équation :
$${\bf C}_{\underline{\lambda}_n}{\bf C}_{\underline{\delta}_n}=\sum_{\rho \text{ partition propre}} c_{\lambda\delta}^{\rho}(n) {\bf C}_{\underline{\rho}_n}.$$
Les coefficients de structure $c_{\lambda\delta}^{\rho}(n)$ sont des polynômes en $n$ avec des coefficients rationnels positifs.
\end{cor}

Le résultat de polynomialité de Farahat et Higman, présenté dans le Théorème \ref{pol_coef_cen_grp_sym}, est donc une conséquence directe du Théorème \ref{center_mini_th}.

On montre dans l'exemple suivant qu'il est possible, pour certaines partitions, d'obtenir les valeurs exactes des coefficients de structure du centre de l'algèbre du groupe symétrique en appliquant le Théorème \ref{center_main_th}.

\begin{ex}\label{ex:valeur_exacte_coef_S_n}
Supposons que $f$ et $h$ sont la permutation $(1\,\,2)$ de $2$ alors dans ce cas $k_1=k_2=2.$ Pour $n$ suffisamment grand, la classe de conjugaison dans $\mathcal{S}_n$ associée à $f$ et $h$ est $C_{(2,1^{n-2})}.$ Supposons qu'on cherche le coefficient de $C_{(2^2,1^{n-4})}$ dans le produit $C_{(2,1^{n-2})}\cdot C_{(2,1^{n-2})}.$ Dans le Théorème \ref{center_main_th}, la valeur de l'indice de sommation $k$ peut être $2,$ $3$ ou bien $4.$ Pour trouver le coefficient de $C_{(2^2,1^{n-4})}$ dans le produit $C_{(2,1^{n-2})}\cdot C_{(2,1^{n-2})},$ il faut chercher d'abord les permutations de $\mathcal{S}_4$ qui sont $(2,2)$-minimal. Elles sont celles qui envoient $\lbrace 3,4\rbrace$ à $\lbrace 1,2\rbrace.$ Il y a $4$ telles permutations : $(1\,\,3)(2\,\,4),$ $(1\,\,4\,\,2\,\,3),$ $(1\,\,3\,\,2\,\,4)$ et $(1\,\,4)(2\,\,3).$ L'ensemble de sommation du Théorème \ref{center_main_th} dans ce cas est formé des permutations suivantes : $(1\,\,4)(2\,\,3),$ $(1\,\,4\,\,2\,\,3),$ $(1\,\,3\,\,2\,\,4)$ et $(1\,\,3)(2\,\,4).$ Pour chaque permutation $x$ d'entre elles, $xhx^{-1}f$ est la permutation $(1\,\,2)(3\,\,4)$ qui est de type-cyclique $(2^2)$ et $|\mathcal{S}_4^2 f^{-1}xh^{-1}\mathcal{S}_4^2\cap \mathcal{S}_{4}|=4.$ D'après le Théorème \ref{center_main_th}, le coefficient qu'on cherche est égal à :
$$\frac{\frac{n!}{2(n-2)!}\frac{n!}{2(n-2)!}(n-2)!(n-2)!}{n!\frac{n!}{2^2\cdot 2\cdot (n-4)!}}\cdot 4\cdot \frac{1}{(n-4)!4}=2.$$
\end{ex}

En suivant le même raisonnement de l'Exemple \ref{ex:valeur_exacte_coef_S_n}, on peut retrouver l'équation suivante de l'Exemple \ref{ex:calcul_coef_str}.

$${\bf C}_{(1^{n-2},2)}^2=\frac{n(n-1)}{2}{\bf C}_{(1^n)}+3{\bf C}_{(1^{n-3},3)}+2{\bf C}_{(1^{n-4},2^2)}.$$

Cette équation peut être obtenue directement de l'équation du produit $A_{(2)}\cdot A_{(2)}$ donnée dans \cite{Ivanov1999} en appliquant le morphisme $\psi$ du Théorème 7.1 du même papier.

\subsection{Le centre de l'algèbre du groupe hyperoctaédral}\label{cen_grp_hyp}


Les classes de conjugaisons du groupe hyperoctaédral $\mathcal{B}_n$ sont indexées par des paires de partitions $(\lambda,\delta)$ tels que $|\lambda|+|\delta|=n,$ voir \cite{geissinger1978representations} ou bien \cite{stembridge1992projective}. On commence cette sous-section par détailler ce fait et décrire les classes de conjugaison du groupe hyperoctaédral afin de définir les coefficients de structure de l'algèbre de ce groupe. 

Il nous est utile dans cette sous-section d'utiliser la notation suivante: 
$$x(p(i)):=\lbrace x(2i-1),x(2i)\rbrace,$$
pour tout $x\in \mathcal{S}_{2n}$ et tout $1\leq i\leq n.$ En utilisant cette notation, on a :
$$\mathcal{B}_n=\lbrace x\in \mathcal{S}_{2n} \text{ tel que pour tout $1\leq i\leq n,$ il existe $1\leq j\leq n$: }x(p(i))=p(j)\rbrace.$$
Si $a\in p(i)$, on désigne par $\overline{a}$ l'élément de l'ensemble $p(i)\setminus \lbrace a\rbrace$. Donc on a, $\overline{\overline{a}}=a$ pour tout $a=1,\cdots 2n.$

La décomposition d'une permutation de $\mathcal{B}_n$ en cycles disjoints possède une forme remarquable. Elle contient deux sortes de cycle. Supposons que $\omega$ est une permutation de $\mathcal{B}_n$ et prenons un cycle $\mathcal{C}$ de cette décomposition, $\mathcal{C}$ peut s'écrire ainsi :
$$\mathcal{C}=(a_1,\cdots ,a_{l(\mathcal{C})}),$$
où $l(\mathcal{C})$ est la longueur du cycle $\mathcal{C}.$ On peut distinguer deux cas :
\begin{enumerate}
\item premier cas: $\overline{a_1}$ apparaît dans le cycle $\mathcal{C},$ par exemple $a_j=\overline{a_1}.$ Comme $\omega\in \mathcal{B}_n$ et $\omega(a_1)=a_2,$ on a  $\omega(\overline{a_1})=\overline{a_2}=\omega(a_j).$ De même comme $\omega(a_{j-1})=\overline{a_1}$, on a $\omega(\overline{a_{j-1}})=a_1$ ce qui veut dire que $a_{l(\mathcal{C})}=\overline{a_{j-1}}.$ Donc,
$$\mathcal{C}=(a_1,\cdots a_{j-1},\overline{a_1},\cdots,\overline{a_{j-1}})$$
et $l(\mathcal{C})=2(j-1)$ est paire. On va noter un tel cycle par $(\mathcal{O},\overline{\mathcal{O}})$.
\item deuxième cas: $\overline{a_1}$ n'apparaît pas dans le cycle $\mathcal{C}.$ Alors prenons le cycle $\mathcal{C'}$ qui contient $\overline{a_1}.$ Comme $\omega(a_1)=a_2$ et $\omega\in \mathcal{B}_n$, on a $\omega(\overline{a_1})=\overline{a_2}$ et ainsi de suite. Cela veut dire que le cycle $\mathcal{C'}$ est de la forme suivante,
$$\mathcal{C'}=(\overline{a_1},\overline{a_2},\cdots ,\overline{a_{l(\mathcal{C})}})$$
et que $\mathcal{C}$ et $\mathcal{C'}$ apparaissent dans la décomposition de $\omega.$ On va maintenant noter $\overline{\mathcal{C}}$ au lieu de $\mathcal{C'}.$ 
\end{enumerate} 

Supposons maintenant que la décomposition d'une permutation $\omega$ de $\mathcal{B}_n$ est ainsi,
$$\omega=\mathcal{C}_1\overline{\mathcal{C}_1}\mathcal{C}_2\overline{\mathcal{C}_2}\cdots \mathcal{C}_k\overline{\mathcal{C}_k}(\mathcal{O}^1,\overline{\mathcal{O}^1})(\mathcal{O}^2,\overline{\mathcal{O}^2})\cdots (\mathcal{O}^l,\overline{\mathcal{O}^l}).$$
Soit $\lambda$ la partition dont les parts sont les longueurs des cycles $\mathcal{C}_i,$ $i=1,\cdots, k$ et $\delta$ la partition dont les parts sont les longueurs des $\mathcal{O}^j,$ $j=1,\cdots, l.$ On a, $|\lambda|+|\delta|=n.$ Le $\textit{type}$ de $\omega$ est défini comme étant la paire des partitions $(\lambda,\delta).$


\begin{prop}
Deux permutations de $\mathcal{B}_n$ sont dans la même classe de conjugaison si et seulement si elles possèdent le même type.
\end{prop}
\begin{proof}
Voir la section 2 dans \cite{stembridge1992projective}.
\end{proof}
\begin{rem}
Deux permutations de $\mathcal{B}_n$ peuvent avoir le même type-cyclique -- c'est à dire être dans la même classe de conjugaison dans $\mathcal{S}_{2n}$ -- sans être dans la même classe de conjugaison dans $\mathcal{B}_n.$ Par exemple, les permutations $\omega=(12)(34)(56)$ et $\psi=(13)(24)(56)$ de $\mathcal{B}_3$ possèdent le même type-cyclique $(2^3)$ mais elles ne sont pas dans la même classe de conjugaison dans $\mathcal{B}_3$ car $\omega$ est de type $(\emptyset, (1^3))$ tandis que le type de $\psi$ est $((2),1)$.
\end{rem}

\begin{cor}
Soit $\omega\in \mathcal{B}_n$ et supposons que $type(\omega)=(\lambda,\delta),$ $|\lambda|+|\delta|=n.$ Alors,
$$C_\omega=\lbrace \theta\in \mathcal{B}_n \text{ tel que } type(\theta)=(\lambda,\delta)\rbrace.$$
\end{cor} 
 
Cela montre que les classes de conjugaison du groupe hyperoctaédral sont indexées par des paires de partition $(\lambda,\delta)$ tel que $|\lambda|+|\delta|=n$ et que pour une telle paire, sa classe de conjugaison associée est :
$$\mathcal{H}_{(\lambda,\delta)}=\lbrace \theta\in \mathcal{B}_n \text{ tel que } type(\theta)=(\lambda,\delta)\rbrace.$$

Soit $\omega$ une permutation de $\mathcal{B}_n$ de type $(\lambda,\delta),$ la taille de  $\mathcal{H}_{(\lambda,\delta)}$ est :
$$|\mathcal{H}_{(\lambda,\delta)}|=\frac{|\mathcal{B}_n|}{|S_\omega|},$$
où $S_\omega=\lbrace \theta \in \mathcal{B}_n\text{ tel que }\theta\omega\theta^{-1}=\omega\rbrace.$ la taille de $S_\omega$ est :
$$|S_\omega|=\prod_{i\geq 1}(2i)^{m_i(\lambda)}m_i(\lambda)!\prod_{j\geq 1}(2i)^{m_j(\delta)}m_j(\delta)!=2^{l(\lambda)}z_\lambda 2^{l(\delta)}z_\delta.$$
\begin{prop}
Soit $(\lambda,\delta)$ une paire de partition tel que $|\lambda|+|\delta|=n,$ alors :
$$|\mathcal{H}_{(\lambda,\delta)}|=\frac{2^nn!}{2^{l(\lambda)+l(\delta)}z_\lambda z_\delta}.$$
\end{prop}



On dit qu'une paire de partitions $(\lambda,\delta)$ est \textit{propre} si et seulement si la partition $\lambda$ est propre. Pour une paire propre de partitions $(\lambda,\delta)$ et pour tout entier $n\geq |\lambda|+|\delta|,$ on définit $(\lambda,\delta)^{\uparrow^n}$ comme la paire de partitions suivante dont la somme des tailles est $n$ :
$$(\lambda,\delta)^{\uparrow^n}:=(\lambda\cup (1^{n-|\lambda|-|\delta|}),\delta).$$
Cela définit une bijection entre l'ensemble des paires propres de partition de taille inférieure ou égale à $n$ et l'ensemble des paires de partition de taille $n.$

Il n'est pas difficile de vérifier que :

$$|\mathcal{H}_{(\lambda,\delta)^{\uparrow^n}}|=\frac{2^nn!}{2^{l(\lambda)+n-|\lambda|-|\delta|+l(\delta)}z_\lambda(n-|\lambda|-|\delta|)! z_\delta}=|\mathcal{H}_{(\lambda,\delta)}|\frac{n!}{(n-|\lambda|-|\delta|)!(|\lambda|+|\delta|)!}.$$

Soient $(\lambda,\delta)$ et $(\beta,\gamma)$ deux paires propres de partition. Pour tout entier $n\geq |\lambda|+|\delta|, |\beta|+|\gamma|,$ il existe des constantes $c_{(\lambda,\delta)(\beta,\gamma)}^{(\rho,\nu)}(n)$ telles que :
$$\mathcal{H}_{(\lambda,\delta)^{\uparrow^n}}\mathcal{H}_{(\beta,\gamma)^{\uparrow^n}}=\sum_{(\rho,\nu) proper\atop{|\rho|+|\nu|\leq n}} c_{(\lambda,\delta)(\beta,\gamma)}^{(\rho,\nu)}(n)\mathcal{H}_{(\rho,\nu)^{\uparrow^n}}.$$

On a déjà montré que les groupes hyperoctaédraux vérifient les hypothèses H.1 à H.6. Pour appliquer le Théorème \ref{center_mini_th} dans le cas du centre de l'algèbre du groupe hyperoctaédral, il nous reste à vérifier l'hypothèse $\text{H}^{'}.0$ pour la suite $(Z(\mathbb{C}[\mathcal{B}_n]))_n.$ Cela se fait d'une façon similaire à la preuve de cette hypothèse dans le cas de la suite $(Z(\mathbb{C}[\mathcal{S}_n]))_n.$ Considérons une permutation $\omega$ de $\mathcal{B}_n.$ La classe de conjugaison de $\omega$ dans $\mathcal{B}_n$ est l'ensemble des permutations de $\mathcal{B}_n$ ayant le même type, disons $(\lambda(\omega),\delta(\omega)),$ que $\omega.$ De même, la classe de conjugaison de $\omega,$ vue comme une permutation de $\mathcal{B}_{n+1},$ dans $\mathcal{B}_{n+1}$ est l'ensemble des permutations de $\mathcal{B}_{n+1}$ ayant $(\lambda(\omega)\cup (1),\delta(\omega))$ comme type, ce qui est la classe de conjugaison de $\omega$ dans $\mathcal{B}_n$ quand on l'intersecte avec $\mathcal{B}_{n+1}.$ Donc, d'après le Théorème \ref{center_mini_th}, il existe des nombres rationnels $a_{(\lambda,\delta)(\beta,\gamma)}^{(\rho,\nu)}(k)$ indépendants de $n$ tel que :
\begin{equation*}
c_{(\lambda,\delta)(\beta,\gamma)}^{(\rho,\nu)}(n)=\frac{\frac{n!|\mathcal{H}_{(\lambda,\delta)}|2^{n-|\lambda|-|\delta|}}{(n-|\lambda|-|\delta|)!(|\lambda|+|\delta|)!}\frac{n!|\mathcal{H}_{(\beta,\gamma)}|2^{n-|\beta|-|\gamma|}}{(n-|\beta|-|\gamma|)!(|\beta|+|\gamma|)!}(n-|\lambda|-|\delta|)!(n-|\beta|-|\gamma|)!}{2^nn!|\mathcal{H}_{(\rho,\nu)}|\frac{n!}{(n-|\rho|-|\nu|)!(|\rho|+|\nu|)!}}
\end{equation*}
$$\sum_{|\rho|+|\nu|\leq k\leq \min(|\lambda|+|\delta|+|\beta|+|\gamma|,n)}a_{(\lambda,\delta)(\beta,\gamma)}^{(\rho,\nu)}(k)\frac{1}{2^{n-k}(n-k)!},$$
pour des paires propres $(\lambda,\delta),(\beta,\gamma),(\rho,\nu)$ et pour tout entier $n\geq |\lambda|+|\delta|,|\beta|+|\gamma|,|\rho|+|\nu|.$ Après simplification, cela peut s'écrire ainsi,
\begin{eqnarray*}
c_{(\lambda,\delta)(\beta,\gamma)}^{(\rho,\nu)}(n)&=&\frac{|\mathcal{H}_{(\lambda,\delta)}||\mathcal{H}_{(\beta,\gamma)}|(|\rho|+|\nu|)!}{|\mathcal{H}_{(\rho,\nu)}|(|\lambda|+|\delta|)!(|\beta|+|\gamma|)!}\\
&&\sum_{|\rho|+|\nu|\leq k\leq \min(|\lambda|+|\delta|+|\beta|+|\gamma|,n)}a_{(\lambda,\delta)(\beta,\gamma)}^{(\rho,\nu)}(k)\frac{(n-k+1)\cdots (n-|\rho|-|\nu|)}{2^{|\lambda|+|\delta|+|\beta|+|\gamma|-k}}.
\end{eqnarray*}

\begin{cor}
Soient $(\lambda,\delta), (\beta,\gamma)$ et $(\rho,\nu)$ trois paires propres de partitions, alors pour tout $n\geq |\lambda|+|\delta|,|\beta|+|\gamma|,|\rho|+|\nu|,$ le coefficient de structure $c_{(\lambda,\delta)(\beta,\gamma)}^{(\rho,\nu)}(n)$ du centre de l'algèbre du groupe hyperoctaédral est un polynôme en $n$ avec des coefficients positifs et on a :
$$\deg(c_{(\lambda,\delta)(\beta,\gamma)}^{(\rho,\nu)}(n))\leq |\lambda|+|\delta|+|\beta|+|\gamma|-|\rho|-|\nu|.$$
\end{cor}

\section{Directions de recherche future}

On présente dans cette section plusieurs algèbres intéressantes. Pour certaines d'entre elles un résultat de polynomialité pour les coefficients de structure existe, pour d'autres non. Le problème principal est que notre cadre général, dans son état actuel, ne peut pas être appliqué pour la plupart d'entre elles. Une direction de recherche intéressante consiste à essayer d'alléger les hypothèses demandées dans notre cadre afin d'inclure le plus grand nombre d'elles, tout en gardant possible d'avoir un théorème pour la propriété de polynomialité. Le lecteur peut voir aussi l'article \cite{strahov2007generalized} de Strahov où l'auteur donne aussi une liste d'algèbre de doubles-classes intéressante déjà étudiée où une propriété de polynomialité pour les coefficients de structure serait un résultat intéressant.

\subsection{Le centre de l'algèbre du groupe des matrices inversibles à éléments dans un corps fini}\label{sec:centre_Gl_n}

On présente dans cette sous-section le travail de Méliot dans \cite{meliot2013partial} autour des coefficients de structure du centre de l'algèbre du groupe $GL_n(\mathbb{F}_q),$ où $GL_n(\mathbb{F}_q)$ est le groupe des matrices inversibles de taille $n\times n$ à coefficients dans le corps fini $\mathbb{F}_q$ à $q$ éléments. On utilise les mêmes notations que Méliot.  

Le centre de $GL_n(\mathbb{F}_q)$ est engendré par l'ensemble des classes $C_{\hat{\mu}}$ indexées par les polypartitions de taille $n$ sur le corps fini $\mathbb{F}_q.$ Une polypartition $\hat{\mu}=\lbrace \mu(P_1),\cdots,\mu(P_r) \rbrace$ de taille $n$ sur le corps fini $\mathbb{F}_q$ est une famille de partitions indexées par les polynômes unitaires (à une variable) irréductibles sur $\mathbb{F}_q$, sauf $X$, telle que :

$$|\hat{\mu}|=\sum_{i=1}^r \deg(P_i)\mid\mu(P_i)\mid.$$  


La taille de $GL_n(\mathbb{F}_q)$ est $(q^n-1)(q^n-q)\cdots (q^n-q^{n-1}).$ Comme présenté dans \cite[section 1.2]{meliot2013partial}, pour une polypartition fixée $\hat{\mu}$ de taille $n$ sur le corps fini $\mathbb{F}_q,$ la taille de sa classe correspondante $C_{\hat{\mu}}$ est :

\begin{equation}
|C_{\hat{\mu}}|=\frac{(q^n-1)(q^n-q)\cdots (q^n-q^{n-1})}{q^{|\hat{\mu}|+2b(\hat{\mu})}\prod_{i=1}^{r}\prod_{k\geq 1}(q^{-\deg P_i})_{m_k(\mu(P_i))}},
\end{equation}

où $$b(\hat{\mu})=\sum_{i=1}^r (\deg P_i)b(\mu(P_i))=\sum_{i=1}^r\sum_{j=1}^{l(\mu(P_i))}(\deg P_i)(j-1)(\mu(P_i))_j,$$
$(x)_m=(x;x)_m$ est le symbole de Pochhammer $(1-x)(1-x^2)\cdots (1- x^m)$ et $m_k(\mu)$ est le nombre de parts $k$ d'une partition $\mu.$ 

\begin{definition}
Une polypartition $\hat{\mu}$ est dite propre si la partition $\mu(X-1)$ est propre. 
\end{definition}

L'ensemble des polypartitions de $n$ est en bijection avec l'ensemble des polypartitions propres de taille inférieure ou égale à $n.$ On indexe la base de $GL_n(\mathbb{F}_q)$ par les polypartitions propres pour mieux présenter la propriété de polynomialité des coefficients de structure de $Z(\mathbb{C}[GL_n(\mathbb{F}_q)]).$

Pour une polypartition propre $\hat{\mu}$ de taille $k$ inférieure à $n$, on associe à $\hat{\mu}$ une polypartition de $n$ notée $\hat{\mu}^{\uparrow n}.$ La partition $\mu^{\uparrow n}(X-1)=\mu(X-1)\cup (1^{n-k})$ tandis que les autres partitions de $\hat{\mu}^{\uparrow n}$ associées aux autres polynômes irréductibles sont les mêmes que dans $\hat{\mu}.$ Il n'est pas difficile de vérifier que :
$$|\hat{\mu}^{\uparrow n}|=\sum_{i=1}^r \deg(P_i)|\mu^{\uparrow n}(P_i)|=\sum_{i=1}^r \deg(P_i)|\mu(P_i)|+(n-k)=n,$$
et
\begin{eqnarray*}
2b(\hat{\mu}^{\uparrow n})&=&2\sum_{i=1}^r\sum_{j=1}^{l(\mu^{\uparrow n}(P_i))}(\deg P_i)(j-1)(\mu^{\uparrow n}(P_i))_j\\
&=&2b(\hat{\mu})+2\sum_{j=l(\mu(X-1))+1}^{l(\mu(X-1))+n-|\hat{\mu}|}j-1\\
&=&2b(\hat{\mu})+\big(n-|\hat{\mu}|\big)\big(n-|\hat{\mu}|+2l(\mu(X-1))-1\big).
\end{eqnarray*}

Donc, si $\hat{\mu}$ est une polypartition de taille inférieure à $n$, on obtient :

\begin{eqnarray*}
|C_{\hat{\mu}^{\uparrow n}}|&=&\frac{(q^n-1)(q^n-q)\cdots (q^n-q^{n-1})}{q^{n+2b(\hat{\mu})+\big(n-|\hat{\mu}|\big)\big(n-|\hat{\mu}|+2l(\mu(X-1))-1\big)}\prod_{i=1}^{r}\prod_{k\geq 1}(q^{-\deg P_i})_{m_k(\mu(P_i))}(q^{-1})_{n-|\hat{\mu}|}}\\
&=&|C_{\hat{\mu}}|\frac{|GL_n(\mathbb{F}_q)|}{|GL_{|\hat{\mu}|}(\mathbb{F}_q)|q^{\big(n-|\hat{\mu}|\big)\big(n-|\hat{\mu}|+2l(\mu(X-1))\big)}(q^{-1})_{n-|\hat{\mu}|}}\\
&=&|C_{\hat{\mu}}|\frac{|GL_n(\mathbb{F}_q)|}{|GL_{|\hat{\mu}|}(\mathbb{F}_q)|q^{\big(n-|\hat{\mu}|\big)\big(2l(\mu(X-1))\big)}|GL_{n-|\hat{\mu}|}(\mathbb{F}_q)|}.
\end{eqnarray*}

La dernière égalité vient du fait que :

$$(q^{-1})_n=\frac{|GL_{n}(\mathbb{F}_q)|}{q^{n^2}}.$$


Le théorème suivant est le Théorème 3.7 dans \cite{meliot2013partial} autour de la polynomialité des coefficients de structure de $Z(GL_n(\mathbb{F}_q)).$ Notre présentation est un peu différente de celle de Méliot puisqu'on utilise les polypartitions propres pour être cohérent avec les autres résultats de polynomialité des coefficients de structure déjà présentés. 

\begin{theoreme}[Méliot]\label{Th:Meliot}
Soit $q$ fixé et soient $\hat{\lambda}$, $\hat{\delta}$ deux polypartitions propres. Soit $n$ un entier suffisamment grand et considérons l'équation suivante :
\begin{equation}
\mathcal{C}_{\hat{\lambda}^{\uparrow n}}\mathcal{C}_{\hat{\delta}^{\uparrow n}}=\sum_{\hat{\rho} \text{ polypartition propre}} c_{\hat{\lambda}\hat{\delta}}^{\hat{\rho}}(n) \mathcal{C}_{\hat{\rho}^{\uparrow n}}.\end{equation}
Alors les coefficients $c_{\hat{\lambda}\hat{\delta}}^{\hat{\rho}}(n)$ sont des polynômes en $q^n$ avec des coefficients rationnels.
\end{theoreme}


\begin{ex}\label{ex:coeff_meliot}
Soit $a$ un élément de $\mathbb{F}^*_q$ et considérons son inverse $a^{-1}.$ Soit $\hat{\lambda}_a$ (resp. $\hat{\delta}_{a^{-1}}$) la polypartition dont la partition associée au polynôme $X-a$ (resp. $X-a^{-1}$) est $(1)$ et la partition vide est associée à tous les autres polynômes irréductibles. Il est clair que  $\hat{\lambda}_a$ et $\hat{\delta}_{a^{-1}}$ sont des polypartitions propres si $a$ est différent de  $1$. Soit $n$ un entier suffisamment grand et notons par $\hat{\emptyset}$ la polypartition propre qui associe à tous les polynômes irréductibles la partition vide. On s'intéresse au coefficient de $\mathcal{C}_{\hat{\emptyset}^{\uparrow n}}$ dans le produit $\mathcal{C}_{\hat{\lambda}_a^{\uparrow n}}\mathcal{C}_{\hat{\delta}_{a^{-1}}^{\uparrow n}}.$ La matrice identité de taille $n$, $I_n$, est le seul élément de $\mathcal{C}_{\hat{\emptyset}^{\uparrow n}}$, tandis que les éléments de $\mathcal{C}_{\hat{\lambda}_a^{\uparrow n}}$ (resp. $\mathcal{C}_{\hat{\delta}_{a^{-1}}^{\uparrow n}}$) sont les matrices conjuguées à $I_{a,n-1}$ (resp. $I_{a^{-1},n-1}$) où :
$$I_{a,n-1}=\begin{pmatrix}
a&0&0&\cdots &0\\
0&1&0&\cdots &0\\
\vdots &0 &\ddots &\cdots & \vdots\\
0&0&\cdots &\ddots &0\\
0&0&\cdots &\cdots &1
\end{pmatrix} \text{(resp.} I_{a^{-1},n-1}=\begin{pmatrix}
a^{-1}&0&0&\cdots &0\\
0&1&0&\cdots &0\\
\vdots &0 &\ddots &\cdots & \vdots\\
0&0&\cdots &\ddots &0\\
0&0&\cdots &\cdots &1
\end{pmatrix}).$$
Soit $A$ une matrice conjuguée à $I_{a,n-1}$, disons $A=MI_{a,n-1}M^{-1}$ pour une matrice $M\in GL_n(\mathbb{F}_q)$, alors $B=MI_{a^{-1},n-1}M^{-1}$ est conjuguée à $I_{a^{-1},n-1}$ et $AB=I_n.$ Cela montre que le coefficient qu'on cherche est égale à la taille de la classe de conjugaison de $I_{a,n-1}$ qui est :
$$q^{n-1}\frac{q^{n}-1}{q-1}=\frac{1}{q(q-1)}\big((q^n)^2-q^n\big).$$
\end{ex}

Malheureusement, notre cadre général n'est pas applicable dans le cas du centre de l'algèbre de $GL_n(\mathbb{F}_q).$ En fait, $GL_n(\mathbb{F}_q)$ ne vérifie pas l'hypothèse 5 de la Section \ref{sec:hypo_defi} pour appliquer nos théorèmes. Voilà, un contre-exemple explicite :

\begin{c-ex}\label{contre_exemple_GL}[à H.5 dans le cas de $GL_n(\mathbb{F}_q)$]
Pour $k_1=2$ et $k_2 = 1,$ la matrice $5\times 5$ suivante :
$$\begin{pmatrix}
1&1&1&0&0\\
1&0&0&0&0\\
0&1&0&0&1\\
0&0&1&0&0\\
0&0&0&1&0
\end{pmatrix}$$
ne peut pas être transformée en une matrice de la forme
$$\begin{pmatrix}
*&*&*&0&0\\
*&*&*&0&0\\
*&*&*&0&0\\
0&0&0&1&0\\
0&0&0&0&1
\end{pmatrix}$$
en faisant des opérations élémentaires sur les 3 dernières colonnes et les lignes de 2 jusqu'à 5. Ces opérations élémentaires correspondent à regarder la classe à gauche (droite) pour $GL_5(\mathbb{F}_q)^2$ ($GL_5(\mathbb{F}_q)^1$) où $GL_n(\mathbb{F}_q)^k$ est le sous-groupe de $GL_n(\mathbb{F}_q)$ des matrices de la forme :

$$ \left(\begin{array}{c|c}
  \begin{array}{cccccc}
    I_k
  \end{array}

 & \mbox{\Huge $0$}\\
 
 \hline
 
 \begin{matrix}
 \mbox{\Huge $0$}

\end{matrix}
 &
 
 \begin{matrix}
 * & * &\cdots &*\\
 * &\ddots &\cdots &* \\
 \vdots &\ddots &\cdots &\vdots \\
 * &\cdots &\cdots &*

\end{matrix}
 
 \end{array}\right)$$

\end{c-ex}

\begin{rem}
Puisqu'on présente ici le premier cas d'algèbre intéressante pour laquelle notre cadre général ne peut pas s'appliquer, il serait bien de clarifier le point suivant. Le lecteur aurait du remarquer dans le Contre-exemple \ref{contre_exemple_GL} à l'hypothèse H.5 dans le cas de $GL_n(\mathbb{F}_q)$ qu'on teste "un seul" sous-groupe qu'on note $GL_n(\mathbb{F}_q)^k$ pour lequel l'hypothèse H.5 n'est pas vérifiée. Par ailleurs, ce qu'on demande pour appliquer notre cadre général est "l'existence", pour tout $k,$ d'un sous-groupe $K_n^k$ qui vérifie les hypothèses H.0 à H.6 demandées. Comme on n'a pas testé tous les sous-groupes possibles de $GL_n(\mathbb{F}_q),$ on ne peut pas dire "directement" que notre cadre général ne s'applique pas dans ce cas. Ce qui nous rend presque-sûr que notre cadre ne s'applique pas dans le cas de $GL_n(\mathbb{F}_q)$ est le fait que les sous-groupes $GL_n(\mathbb{F}_q)^k$ qu'on teste sont naturels dans ce cas et ils ressemblent aux sous-groupes $\mathcal{S}_n^k$ et $\mathcal{B}_n^k$ (qui étaient des bons choix) dans le cas des groupes symétriques et hyperoctaédraux.

On va utiliser la même logique dans les prochaines cas d'algèbres où notre cadre général ne s'applique pas. Quand on donnera un contre-exemple, cela veut dire que les sous-groupes qu'on a choisi à tester dans ce cas sont les plus naturels à regarder. 
\end{rem}

Il faut noter que l'hypothèse H.5 de la Section \ref{sec:hypo_defi} est d'une importance cruciale dans notre raisonnement. Elle nous permet d'avoir un représentant de la classe $K_n^{k_1}xK_n^{k_2}$ dans $K_{k_1+k_2}$ 
et elle permet que l'indice $k$ dans la somme des théorèmes qu'on donne soit borné par un entier qui ne dépend pas de $n,$ ce qui est crucial pour obtenir des polynômes.

Puisqu'on ne peut pas appliquer les Théorèmes \ref{main_th} et \ref{mini_th} dans le cas du centre de l'algèbre $GL_n(\mathbb{F}_q),$ il est donc naturel de poser la question suivante :
\begin{que}\label{ques:GL}
Est-il possible de modifier les hypothèses de notre cadre général, tout en gardant la possibilité de donner des théorèmes similaires aux Théorème \ref{main_th} et Théorème \ref{mini_th}, afin d'inclure le cas du centre de l'algèbre de $GL_n(\mathbb{F}_q)$ dans nos applications ?
\end{que}

\subsection{Super-classes et doubles-classes}

La théorie des super-caractères et des super-classes des algèbres de groupe (voir la définition ci-dessous) est détaillée dans l'article \cite{diaconis2008supercharacters} de Diaconis et Isaacs. Par définition, les super-classes sont des unions de classes de conjugaisons. Dans cette section, on montre que, les super-classes d'un groupe fini $G$ (de la forme $1+J,$ où $J$ est une algèbre nilpotente) sont en bijection avec les doubles-classes d'une paire de groupes particuliers. On discutera dans la section suivante le cas où $G$ est le groupe des matrices uni-triangulaires.

Soit $J$ une algèbre nilpotente\footnote{Une algèbre $\mathcal{A}$ est nilpotente s'il existe un entier $n$ tel que $x^n=0$ pour tout $x\in \mathcal{A}$} associative et de dimension finie sur un corps fini. On considère l'ensemble $G=1+J$ (somme formelle) des éléments qui s'écrivent de la forme $1+x$ avec $x\in J.$ L'ensemble $G$ est un groupe pour le produit défini ainsi:
$$(1+x)(1+y)=1+x+y+xy,$$
pour tout $x,y\in J.$ Le groupe produit $H=G\times G$ agit sur $J$ par :
$$(u,v)\cdot x=uxv^{-1},$$
où $x\in J$ et $(u,v)\in G\times G.$ Une $\textit{super-classe}$ de $G$ est une orbite pour cette action, c'est-à-dire un ensemble de la forme $1+(G,G)\cdot x,$ pour un certain $x\in J$, où :
$$1+(G,G)\cdot x=\lbrace 1+uxv^{-1} \, ; \, u,v\in G\rbrace.$$

On considère le produit semi-direct de groupes $(G\times G)$ et $J$ noté par $(G\times G)\ltimes J.$ Comme ensemble, $(G\times G)\ltimes J$ est l'ensemble des éléments de la forme $((u,v),x),$ où $u,v\in G$ et $x\in J.$ Le produit dans $(G\times G)\ltimes J$ est défini ainsi :
$$((u,v),x)\cdot ((u',v'),x')=((uu',vv'), x+ux'v^{-1}),$$
pour tous éléments $((u,v),x), ((u',v'),x')\in (G\times G)\ltimes J.$ L'ensemble des éléments de la forme $((u,v),0)$ est un sous-groupe de $(G\times G)\ltimes J$ isomorphe à $G\times G.$ Pour deux éléments $(g_1,g_2)$ et $(g'_1,g'_2)$ de $G\times G,$ et pour n'importe quel élément $((u,v),x)$ de $(G\times G)\ltimes J,$ on a :
$$(g_1,g_2)\cdot ((u,v),x) \cdot (g'_1,g'_2):=((g_1,g_2),0)\cdot ((u,v),x) \cdot ((g'_1,g'_2),0)=((g_1ug'_1,g_2vg'_2),g_1xg_2^{-1}).$$
Il y a une bijection entre les super-classes de $G$ et les double-classes de $G\times G$ dans $(G\times G)\ltimes J.$ Pour un élément $x\in J,$ tous les éléments de la forme $((u,v),x)\in (G\times G)\ltimes J$ sont dans la même $G\times G$ double-classe dans $(G\times G)\ltimes J$ qui est l'image de la super-classe de $x$ par cette bijection. 
Explicitement, la fonction $\psi$ définie ainsi :
$$\begin{array}{ccccc}
\psi & : & \mathcal{SC}(G) & \longrightarrow & G\times G\setminus(G\times G)\ltimes J/ G\times G \\
& & 1+(G,G)\cdot x & \mapsto & (G\times G)\times GxG\\
\end{array},$$
où $\mathcal{SC}(G)$ est l'ensemble des super-classes de $G$ est une bijection. Dans la suite, pour un élément $x\in J$, on note par $\mathcal{SC}(x)$ (resp. $\mathcal{DC}(x)$) la super-classe (resp. double-classe) de $G$ (resp. $G\times G$ dans $(G\times G)\ltimes J$) associée à $x.$ Par linéarité, $\psi$ peut être étendue sur l'algèbre des super-classes de $G.$ Pour avoir un morphisme d'algèbres, il faut considérer $\frac{\psi}{|G|^2}.$ Autrement dit, on a la proposition suivante.
\begin{prop}
La fonction $\frac{1}{|G|^2}\psi:\mathbb{C}[\mathcal{SC}(G)]\rightarrow \mathbb{C}[G\times G\setminus(G\times G)\ltimes J/ G\times G]$ définie sur les éléments de base par: 
$$(\frac{1}{|G|^2}\psi)(\mathcal{SC}(x)):=\frac{1}{|G|^2}\psi(\mathcal{SC}(x))=\frac{1}{|G|^2}\mathcal{DC}(x),$$
est un morphisme d'algèbres.
\end{prop}
\begin{proof}
On va montrer que $\frac{1}{|G|^2}\psi$ est compatible avec les produits. On va montrer que 
\begin{equation}\label{ex:psi_morphisme}
\frac{1}{|G|^2}\psi(\mathcal{SC}(x)\mathcal{SC}(y))=\frac{1}{|G|^2}\psi(\mathcal{SC}(x))\frac{1}{|G|^2}\psi(\mathcal{SC}(y)), \text{ pour tout $x,y\in J$.}
\end{equation}
Soient $x,y$ et $z$ trois éléments fixés de $J,$ on note par $s_{xy}^{z}$ (resp. $d_{xy}^{z}$) le coefficient de $\mathcal{SC}(z)$ (resp. $\mathcal{DC}(z)$) dans l'expansion du produit $\mathcal{SC}(x)\mathcal{SC}(y)$ (resp. $\mathcal{DC}(x)\mathcal{DC}(y)$). Montrer l'équation \eqref{ex:psi_morphisme} revient à montrer que $d_{xy}^{z}=|G|^2s_{xy}^{z}.$ En utilisant la Proposition \ref{desc_coef}, on peut écrire :
$$s_{xy}^{z}=|\lbrace (a,b,c,d)\in G^4 \text{ tel que } axb+cyd+axbcyd=z\rbrace|=:|A_{xy}^{z}|,$$
et,
$$d_{xy}^{z}=|\lbrace (a,b,c,d,e,f)\in G^6 \text{ tel que } axb+cdyef=z\rbrace|$$$$=|G|^2|\lbrace (a,b,c',d')\in G^4 \text{ tel que } axb+c'yd'=z\rbrace|.$$
Si on définit $B_{xy}^{z}$ comme étant l'ensemble suivant :
$$B_{xy}^{z}:=\lbrace (a,b,c',d')\in G^4 \text{ tel que } axb+c'yd'=z\rbrace,$$
alors l'application définie sur $A_{xy}^{z}$ à valeur dans $B_{xy}^{z}$ qui pour un élément $(a,b,c,d)$ renvoie $(a,b,c(1+axb),d)$ est une bijection dont l'inverse est l'application $$(a,b,c,d)\mapsto (a,b,c(1+axb)^{-1},d).$$ Cela veut dire que $d_{xy}^{z}=|G|^2s_{xy}^{z}$ et que $\frac{1}{|G|^2}\psi$ est bien un morphisme.
\end{proof}

Dans l'annexe B du papier \cite{diaconis2008supercharacters}, Diaconis et Isaacs présentent le lien entre les super-caractères de $G$ et les fonctions sphériques zonales de la paire $(G\times G)\ltimes J,G\times G).$ Dans cette section, on a présenté explicitement le lien entre les super-classes de $G$ et les doubles-classes de la paire $(G\times G)\ltimes J,G\times G).$ Il est probablement possible de passer de l'un à l'autre de ces résultats en utilisant la Proposition \ref{egalité_doubleclasse_fctzonale} de cette thèse et un résultat équivalent pour les super-caractères (si un tel résultat existe). Mais j'ai préféré présenté ici une preuve directe et auto-contenue de l'équivalence entre algèbre des super-classes de $G$ et algèbre des double-classes de la paire $(G\times G)\ltimes J,G\times G).$


\subsection{Les super-classes des matrices uni-triangulaires}\label{super-classe-matr-uni}

La théorie des super-caractères et des super-classes de l'algèbre du groupe des matrices uni-triangulaires est en relation avec la théorie des fonctions symétriques sur des variables non-commutatives, voir \cite{andre2013supercharacters}. On commence cette section par la définition des super-classes du groupe uni-triangulaire et on expliquera cette relation à la fin de cette section après avoir donné tout ce qu'il nous faut pour la présenter.

Soit $\mathbb{K}$ un corps fini d'ordre $q.$ Pour tout $n\in \mathbb{N}$, on note par $U_n$ le groupe des matrices uni-triangulaires supérieures à coefficients dans $\mathbb{K}.$ Si on note par $\mathbf{u}_n$ la $\mathbb{K}$-algèbre des matrices triangulaires strictement supérieures à coefficients dans $\mathbb{K}$, alors on a, $U_n=I_n+\mathbf{u}_n$, où $I_n$ est la matrice identité de taille $n.$

Soit $n$ un entier positif, on définit $[[n]]:=\lbrace (i,j):1\leq i<j\leq n\rbrace.$
Une partition d'ensemble $\pi$ de $n$, écrite $\pi=B_1/B_2/\cdots /B_l$, est une famille d'ensembles non-vides $B_i$ tel que $B_1\sqcup B_2\sqcup\cdots \sqcup B_l=[n].$ Les $B_i$ sont appelés blocs de $\pi$ et $l(\pi)$ est le nombre $l$ de ces blocs. Il faut noter qu'on ne s'intéresse pas à l'ordre des blocs définissant une partition d'ensemble. Par exemple :
$$\sigma=1~~2~~5/3/4~~8~~9~~10/6~~7,$$
est une partition d'ensemble de $10$ et $l(\sigma)=4.$ On note par $\mathcal{SP}_n$ l'ensemble des partitions d'ensemble de $n$.

Par convention, on met toujours les éléments d'un bloc $B$ en ordre croissant, et pour un bloc $B=b_1b_2\cdots b_k$ on associe un ensemble d'arcs, noté $D(B)$, 
$$D(B):=\lbrace (b_1,b_2),(b_2,b_3),\cdots, (b_{k-1},b_k)\rbrace.$$
L'ensemble des arcs d'une partition d'ensemble $\pi$, noté $D(\pi)$ est l'union disjointe des ensembles des arcs des blocs de $\pi.$ Par exemple : $$D(\sigma)=\lbrace (1~~2),(2~~5),(4~~8),(8~~9),(9~~10),(6~~7)\rbrace.$$
Il est clair que pour une partition d'ensemble $\pi$ de $n$, $D(\pi)\subset [[n]].$ L'inverse n'est pas vrai. Cela veut dire qu'il existe des sous-ensembles de $[[n]]$ qui ne correspondent à aucune partition d'ensemble de $n.$ Par exemple, $D=\lbrace (1,2),(1,3)\rbrace\subset [[3]]$ ne peut pas être l'ensemble des arcs d'aucune partition d'ensemble de $3.$

Pour une partition d'ensemble $\pi$ de $n,$ on peut associer une $n\times n$-matrice, noté $M(\pi),$ triangulaire supérieure dont les entrées sont les entiers $0$ et $1.$ La matrice $M(\pi)$ est codée par les éléments de $D(\pi).$ L'entrée $m_{ij}$ est $1$ si l'arc $(i,j)$ est un élément de $D(\pi)$ et $0$ sinon. 

Soit $\mathbb{K}^*=\mathbb{K}\setminus\lbrace 0\rbrace$, une partition d'ensemble de $n$ $\mathbb{K}^*$-colorée est une paire $(\pi,\phi)$, où $\pi$ est une partition de $n$ et $\phi:D(\pi)\rightarrow \mathbb{K}^*$ est une application. On va écrire $(\pi,\phi)=((a_1,\alpha_1),(a_2,\alpha_2),\cdots,(a_r,\alpha_r))$, où $D(\pi)=\lbrace a_1,a_2,\cdots ,a_r\rbrace$ et $\alpha_i=\phi(a_i)$, $1\leq i\leq r.$ Pour une partition d'ensemble de $n$ $\mathbb{K}^*$-colorée on peut associer une $n\times n$-matrice triangulaire supérieure, noté $M(\pi,\phi),$ avec des entrées dans $\mathbb{K}^*.$ La matrice $M(\pi,\phi)$ possède la même forme que $M(\pi)$ avec l'entrée $\alpha(i,j)$ (au lieu de $1$) si l'arc $(i,j)$ est dans l'ensemble $D(\pi).$ 


Soit $\pi$ une partition d'ensemble de $k$ et soit $n$ un entier plus grand que $k.$ On peut obtenir d'une façon naturelle une partition d'ensemble de $n$ à partir de $\pi$ en ajoutant les $n-k$ blocs $k+1/k+2/\cdots/ n$ à $\pi.$ On note cette partition par $\pi^{\uparrow n} :$
$$\pi^{\uparrow n}:=\pi/k+1/k+2/\cdots/ n.$$
En terme de matrices, $M(\pi^{\uparrow n})$ est la $n\times n$-matrice triangulaire supérieure obtenue de $M(\pi)$ en ajoutant $n-k$ $0$-colonnes et $0$-lignes à $M(\pi).$

On dit qu'une partition d'ensemble $\pi$ de $n$ est \textit{propre} si $n$ n'est pas seul dans son bloc de $\pi.$ Par exemple, la partition d'ensemble $156/237/4$ de $7$ est propre mais pas $156/7/234.$ On note par $\mathcal{PSP}_n$ l'ensemble des partitions d'ensembles propres de $n.$ Il y a une bijection naturelle entre $\mathcal{SP}_n$ et l'ensemble $\mathcal{PSP}_{\leq n},$ 
$$\mathcal{PSP}_{\leq n}:=\bigsqcup_{1\leq k\leq n} \mathcal{PSP}_k .$$ 

Les super-classes du groupe uni-triangulaire sont indexées par les partitions d'ensemble, $\mathbb{K}^*$-colorées ,voir \cite{andre2013supercharacters} pour plus de détails sur la théorie des super-caractères et super-classes du groupe uni-triangulaire. Pour tout élément $(\pi,\phi)$ de $\mathcal{SP}_n(\mathbb{K})$, l'ensemble des partitions d'ensemble $\mathbb{K}^*$-colorées de $n$, on note par $\mathcal{O}_{\pi,\phi}$ la $U_n$-double classe $U_nM(\pi,\phi)U_n$ et par $\mathcal{K}_{\pi,\phi}=I_n+\mathcal{O}_{\pi,\phi}$ la super-classe de $U_n$ associée à $(\pi,\phi).$

Soient $(\pi,\phi)$ et $(\sigma,\psi)$ deux partitions d'ensembles de $k_1$ et $k_2$ respectivement, $\mathbb{K}$-colorées et propres et soit $n$ un entier plus grand que $k_1$ et $k_2.$ Alors on a :
\begin{equation}
\mathcal{K}_{\pi,\phi}(n)\mathcal{K}_{\sigma,\psi}(n)=\sum_{(\rho,\theta)\in \mathcal{PSP}_{\leq n}}d_{(\pi,\phi),(\sigma,\psi)}^{(\rho,\theta)}(n)\mathcal{K}_{\rho,\theta}(n).
\end{equation}

\begin{que}
Est-il possible de donner une propriété de polynomialité en $q^n$ pour les coefficients $d_{(\pi,\phi),(\sigma,\psi)}^{(\rho,\theta)}(n)$ ?
\end{que}

Comme dans le cas de $GL_n(\mathbb{F}_q),$ H.5 n'est pas vérifiée pour $U_n$ et donc probablement non plus\footnote{On suppose, pour arriver à cette conclusion, que les sous-groupes $(U_n\times U_n)^{k}$ qu'on cherche s'écrivent $(U_n)^{k} \times (U_n)^{k},$ ce qui nous paraît raisonnable dans ce cas. Mais rien ne prouve qu'on ne peut pas trouver de sous-groupes de $U_n \times U_n$ qui ne soient pas des produits cartésiens de sous-groupe de $U_n$ avec eux-même et pour laquelle l'hypothèse H.5 marche} pour $U_n\times U_n$ (car $\mathrm{k}( (U_n\times U_n)^{k_1} (x,y) (U_n\times U_n)^{k_2}) = \max 
(\mathrm{k}(U_n^{k_1}xU_n^{k_2}), \mathrm{k}(U_n^{k_1}yU_n^{k_2})$). Il serait donc intéressant de reposer la Question \ref{ques:GL} afin d'inclure ce nouveau cas dans notre cadre général.

L'étude des super-classes du groupe des matrices uni-triangulaires est en relation avec celle des fonctions symétriques sur des variables non-commutatives. En effet, l'espace vectoriel $$\SC:=\bigoplus_n \SC_n,$$ où $\SC_n$ est l'espace vectoriel engendré par les super-caractères du groupe uni-triangulaire $U_n,$ est isomorphe en tant qu'algèbre de Hopf à l'algèbre des fonctions symétriques en des variables non-commutatives notée $\NCSym,$ voir \cite[Section 4.4]{andre2013supercharacters}. Les fonctions symétriques sur des variables non-commutatives ont été étudiées par Wolf dans \cite{wolf1936symmetric}. 

L'algèbre $\NCSym$ introduite par Rosas et Sagan dans \cite{rosas2006symmetric} peut être vue comme une extension de l'algèbre des fonctions symétriques $\Lambda$ définie dans le deuxième chapitre. Elle possède plusieurs familles de base indexées par les partitions d'ensemble similaires à celles des fonctions puissances, monomiales, élémentaires, etc. Pour illustrer, on prend par exemple les fonctions monomiales. Si $\pi\in \mathcal{SP}_n,$ un monôme de forme $\pi$ en des variables non-commutatives est un produit $x_{i_1}x_{i_2} \cdots x_{i_n} $ où
$i_r = i_s$ si et seulement si $r$ et $s$ sont dans un même bloc de $\pi.$ Par exemple $x_1x_2x_1x_2$ est un monôme de forme $1\,\,3/2\,\,4$ en des variables non-commutatives. Si $\pi$ est une partition  d'ensemble, la fonction symétrique monomiale sur des variables non-commutatives $m_\pi$ est définie comme étant la somme de tous les  monômes de forme $\pi$ en des variables non-commutatives. Par exemple,
$$m_{1\,\,3/2\,\,4}=x_1 x_2 x_1 x_2 + x_2 x_1 x_2 x_1 + x_1 x_3 x_1 x_3 + x_3 x_1 x_3 x_1 + x_2 x_3 x_2 x_3 + \cdots.$$
La famille $(m_\pi)$ indexée par les partitions d'ensemble forme une base pour $\NCSym.$ Pour plus de détails sur cette algèbre, le lecteur peut voir \cite{wolf1936symmetric}, \cite{rosas2006symmetric}, \cite{gebhard2001chromatic} et \cite{andre2013supercharacters}.

\subsection{Généralisation de l'algèbre de Hecke de la paire $(\mathcal{S}_{2n},\mathcal{B}_n)$}

Soient $n$ et $k$ deux entiers naturels positifs. On considère le groupe symétrique $\mathcal{S}_{kn}.$ On note par $\mathcal{B}_{kn}^k$ l'ensemble suivant :
$$\mathcal{B}_{kn}^k:=\lbrace \sigma\in \mathcal{S}_{kn} \text{ tel que pour tout $1 \leq r \leq n$ il existe $1 \leq r' \leq n$: } \sigma(p_k (r)) =p_k( r' )\rbrace,$$
où $p_k(i)= \lbrace (i-1)k+1, .....,  ik \rbrace.$ L'ensemble $\mathcal{B}_{kn}^k$ est un sous-groupe de $\mathcal{S}_{kn}$ pour tout $n$ et $k.$ Avec ces notations, la paire $(\mathcal{S}_{2n},\mathcal{B}_n)$ n'est autre que la paire $(\mathcal{S}_{2n},\mathcal{B}_{2n}^2).$

\begin{que}\label{ques1_B}
Comment définir le 'type' d'une permutation $\sigma$ de $\mathcal{S}_{kn}$ de sorte que deux permutations soient dans la même double-classe $\mathcal{B}_{kn}^k\sigma \mathcal{B}_{kn}^k$ si et seulement si elles possèdent le même 'type' que $\sigma$ ?
\end{que}

\begin{que}\label{ques2_B}
L'algèbre de doubles-classes de $\mathcal{B}_{kn}^k$ dans $\mathcal{S}_{kn}$ est-elle commutative ?
\end{que}

\begin{que}\label{ques3_B}
L'algèbre de doubles-classes de $\mathcal{B}_{kn}^k$ dans $\mathcal{S}_{kn}$ entre-t-elle dans notre cadre général ? Autrement-dit, est-il possible de donner une propriété de polynomialité pour les coefficients de structure de cette algèbre en utilisant le Théorème \ref{mini_th} ?
\end{que}

Si la réponse à la Question \ref{ques3_B} est positive, une question subsidiaire est la suivante :

\begin{que}\label{ques4_B}
L'algèbre des éléments partiels est-elle associative dans ce cas ? Si elle ne l'est pas, est-il possible de construire une algèbre similaire qui joue le même rôle et qui soit associative ? En plus, existe-il une relation entre cette algèbre et les algèbres des fonctions symétriques ?
\end{que}

Dans le cas contraire, c'est-à-dire si la réponse à la la Question \ref{ques3_B} est négative, cela soulève la question suivante :

\begin{que}\label{ques5_B}
Peut-on adapter la démarche du troisième chapitre pour établir un résultat de polynomialité ?
\end{que}



Un sous-groupe intéressant de $\mathcal{B}_{kn}^k$ est celui des permutations de $nk$ colorées avec $k$ couleurs. Ce groupe possède plusieurs définitions équivalentes, voir \cite{bagno2006excedance}. Ici on utilise une définition plus cohérente avec notre travail et nos notations et on le note par $\mathcal{C}_{kn}^k.$

Pour tout $1\leq i \leq n,$ on note par $c^k_i$ le cycle de longueur $k$ suivant :
$$c_i^k:=( (i - 1)k + 1, \cdots , ik ).$$ 
On dit qu'une permutation $\omega$ de $nk$ est $k$-colorée si $\omega$ envoie chaque cycle $c_i^k$ sur un autre cycle de la même forme. Par exemple :
$$p=\begin{pmatrix}
1&2&3&4&5&6&7&8&9&10&11&12\\
6&4&5&11&12&10&8&9&7&1&2&3
\end{pmatrix}\in \mathcal{C}_{12}^3,$$
mais 
$$p'=\begin{pmatrix}
1&2&3&4&5&6&7&8&9&10&11&12\\
6& \color{red}{5}&\color{red}{4}&11&12&10&8&9&7&1&2&3
\end{pmatrix},$$
est dans $\mathcal{B}_{12}^3$ mais pas dans $\mathcal{C}_{12}^3$ car les images de $1,$ $2$ et $3$ ne sont pas bien (cycliquement) ordonnées. Le groupe $\mathcal{C}_{kn}^k$ est l'ensemble des permutations $k$-colorées de $nk,$ 
$$\mathcal{C}_{kn}^k:=\lbrace \omega\in \mathcal{S}_{nk} \text{ tel que pour tout $1 \leq i \leq n$ il existe $1 \leq i' \leq n$: } \omega(c_i^k)=c_{i'}^k, \,\, 1\leq i,i'\leq n\rbrace.$$

Il sera aussi intéressant de répondre aux questions \ref{ques1_B}, \ref{ques2_B}, \ref{ques3_B}, \ref{ques4_B} et \ref{ques5_B} en remplaçant $\mathcal{B}_{kn}^k$ par $\mathcal{C}_{kn}^k.$

\subsection{L'algèbre de Iwahori-Hecke et son centre}\label{sec:alg_Iwahori_Hecke}

L'algèbre de Iwahori-Hecke, notée $\mathcal{H}_{n,q}$ est une algèbre sur $\mathbb{C}(q)$ généralisant l'algèbre du groupe symétrique. Quand $q=1,$ cette algèbre est l'algèbre du groupe symétrique $\mathbb{C}[\mathcal{S}_n].$ D'après un résultat de Iwahori, voir \cite{iwahori1964structure}, quand $q$ est une puissance d'un nombre premier cette algèbre est isomorphe à l'algèbre des doubles-classes du groupe des matrices triangulaires supérieures dans $GL_n(\mathbb{F}_q).$ Ce résultat se trouve aussi dans \cite[Section 8.4]{geck2000characters}.

Les éléments de Geck-Rouquier définis dans \cite{geck1997centers} forment une base pour le centre $Z(\mathcal{H}_{n,q})$ de l'algèbre d'Iwahori-Hecke. Ils sont indexés par les partitions de $n$ et notés souvent par $\Gamma_{\lambda,n}.$ Dans \cite{francis1999minimal}, Francis donne une caractérisation pour ces éléments. Les éléments $\Gamma_{\lambda,n}$ de Geck-Rouquier deviennet des classes de conjugaison lorsque $q=1.$

Une propriété de polynomialité pour les coefficients de structure $a_{\lambda\delta}^\rho(n)$ du centre de l'algèbre d'Iwahori-Hecke définis par l'équation suivante :
\begin{equation}
\Gamma_{\lambda,n}\Gamma_{\delta,n}=\sum_{|\rho|\leq |\lambda|+|\delta|}a_{\lambda\delta}^\rho(n)\Gamma_{\rho,n}
\end{equation}
a été conjecturée par Francis et Wang dans \cite{francis1992centers}. Dans \cite{meliot2010products}, Méliot montre ce résultat.

L'auteur n'a pas pu trouver un résultat dans la littérature qui décrit le centre de l'algèbre de Iwahori-Hecke $\mathcal{H}_{n,q}$ (et non l'algèbre $\mathcal{H}_{n,q}$ elle-même) comme étant une algèbre de doubles-classes pour essayer de voir si le résultat de polynomialité donné par Méliot pour ces coefficients de structure peut être inclus dans le cadre général présenté dans ce chapitre.

\begin{que}
Le centre de l'algèbre de Iwahori-Hecke est-elle une algèbre de doubles-classes ? Si oui, est-ce qu'elle entre dans notre cadre général ? (c'est-à-dire: peut-on retrouver le résultat de polynomialité pour ses coefficients de structure, donné par Méliot, en appliquant le Théorème \ref{main_th} ?)  
\end{que}

\bibliographystyle{alpha}
\bibliography{biblio}

\begin{thebibliography}{{Las}08}

\bibitem[AC12]{Aker20122465}
Kürşat Aker and Mahir~Bilen Can.
\newblock Generators of the {H}ecke algebra of $({S}_{2n},{H}_n)$.
\newblock {\em Advances in Mathematics}, 231(5):2465 -- 2483, 2012.

\bibitem[And08]{UnitriangulargroupAndre}
Carlos A.~M. André.
\newblock Supercharacters of unitriangular groups and set partition
  combinatorics.
\newblock {\em Trans. Amer. Math. Soc. 360, 2359-2392}, 2008.

\bibitem[And13]{andre2013supercharacters}
Carlos A.~M. André.
\newblock Supercharacters of unitriangular groups and set partition
  combinatorics.
\newblock {\em course given in CIMPA school: Modern Methods in Combinatorics
  ECOS 2013}, 2013.

\bibitem[BC11]{bernardi2011counting}
Olivier Bernardi and Guillaume Chapuy.
\newblock Counting unicellular maps on non-orientable surfaces.
\newblock {\em Advances in Applied Mathematics}, 47(2):259--275, 2011.

\bibitem[BG06]{bagno2006excedance}
Eli Bagno and David Garber.
\newblock On the excedance number of colored permutation groups.
\newblock {\em S{\'e}minaire Lotharingien de Combinatoire}, 53:B53f, 2006.

\bibitem[Boc80]{boccara1980nombre}
G~Boccara.
\newblock Nombre de representations d'une permutation comme produit de deux
  cycles de longueurs donnees.
\newblock {\em Discrete Mathematics}, 29(2):105--134, 1980.

\bibitem[Bre76]{brender1976spherical}
M~Brender.
\newblock Spherical functions on the symmetric groups.
\newblock {\em Journal of Algebra}, 42(2):302--314, 1976.

\bibitem[BW80]{bertram1980decomposing}
Edward~A Bertram and Victor~K Wei.
\newblock Decomposing a permutation into two large cycles: an enumeration.
\newblock {\em SIAM Journal on Algebraic Discrete Methods}, 1(4):450--461,
  1980.

\bibitem[Can13]{Can}
Mahir~Bilen Can.
\newblock Personal communication.
\newblock 2013.

\bibitem[C{\"O}14]{2014arXiv1407.3700B}
M.~Bilen {Can} and {\c S}.~{{\"O}zden}.
\newblock {Corrigendum to ''Generators of the Hecke algebra of
  $({S}_{2n},{B}_{n})$}.
\newblock {\em ArXiv e-prints}, July 2014.

\bibitem[Cor75]{CoriHypermaps}
R.~Cori.
\newblock Un code pour les graphes planaires et ses applications.
\newblock Number~27 in Ast{\'e}risque. Soci{\'e}t{\'e} Math{\'e}matique de
  France, 1975.
\newblock 169 pages.

\bibitem[DF12]{dolega2012kerov}
M.~Dołęga and V.~Féray.
\newblock On {K}erov polynomials for {J}ack characters.
\newblock preprint arXiv:1201.1806, 2012.

\bibitem[DF14]{2014arXiv1402.4615D}
M.~{Do{\l}{\c e}ga} and V.~{F{\'e}ray}.
\newblock {Gaussian fluctuations of {Y}oung diagrams and structure constants of
  Jack characters}.
\newblock {\em ArXiv e-prints}, February 2014.

\bibitem[DF{\'S}13]{2013arXiv1301.6531D}
M.~{Do{\l}ega}, V.~{F{\'e}ray}, and P.~{{\'S}niady}.
\newblock {Jack polynomials and orientability generating series of maps}.
\newblock {\em ArXiv e-prints}, January 2013.

\bibitem[DI08]{diaconis2008supercharacters}
Persi Diaconis and I~Isaacs.
\newblock Supercharacters and superclasses for algebra groups.
\newblock {\em Transactions of the American Mathematical Society},
  360(5):2359--2392, 2008.

\bibitem[Dud08]{dudko2008asymptotics}
A~Dudko.
\newblock Asymptotics of {P}lancherel measures for $ {GL} (n, q) $.
\newblock {\em arXiv preprint arXiv:0806.1345}, 2008.

\bibitem[Fé12]{feray2012complete}
Valentin Féray.
\newblock On complete functions in {J}ucys-{M}urphy elements.
\newblock {\em Annals of Combinatorics}, 16(4):677--707, 2012.

\bibitem[Fé14]{feray2014}
Valentin Féray.
\newblock
  \href{http://user.math.uzh.ch/feray/Slides/1402_CourseEdimburg.pdf}{Random
  Representations of the symmetric group}.
\newblock {\em mini-course given in workshop "Probability and representation
  theory in Edinburgh"}, 2014.

\bibitem[FH59]{FaharatHigman1959}
H.~Farahat and G.~Higman.
\newblock The centres of symmetric group rings.
\newblock {\em Proc. Roy. Soc. (A)}, 250:212--221, 1959.

\bibitem[Fra99]{francis1999minimal}
Andrew Francis.
\newblock The minimal basis for the centre of an {I}wahori-{H}ecke algebra.
\newblock {\em Journal of Algebra}, 221(1):1--28, 1999.

\bibitem[F{\'S}11]{Feray2011338}
Valentin F{\'e}ray and Piotr {\'S}niady.
\newblock Zonal polynomials via {S}tanley coordinates and free cumulants.
\newblock {\em Journal of Algebra}, 334(1):338--373, 2011.

\bibitem[Ful08]{fulman2008convergence}
Jason Fulman.
\newblock Convergence rates of random walk on irreducible representations of
  finite groups.
\newblock {\em Journal of Theoretical Probability}, 21(1):193--211, 2008.

\bibitem[FW09]{francis1992centers}
A.~Francis and W.~Weiqiang.
\newblock The centers of {I}wahori-{H}ecke algebras are filtered.
\newblock {\em Representation Theory, Comtemporary Mathematics}, 478:29–38,
  2009.

\bibitem[GJ96]{GouldenJacksonLocallyOrientedMaps}
I.~P. Goulden and D.~M. Jackson.
\newblock Maps in locally orientable surfaces, the double coset algebra, and
  zonal polynomials.
\newblock {\em Can. J. Math.}, 48(3):569--584, 1996.

\bibitem[GK78]{geissinger1978representations}
Ladnor Geissinger and D~Kinch.
\newblock Representations of the hyperoctahedral groups.
\newblock {\em Journal of algebra}, 53(1):1--20, 1978.

\bibitem[Gou14]{Goupil-Personal}
Alain Goupil.
\newblock Personal communication.
\newblock 2014.

\bibitem[GP00]{geck2000characters}
Meinolf Geck and G{\"o}tz Pfeiffer.
\newblock {\em Characters of finite {C}oxeter groups and {I}wahori-{H}ecke
  algebras}.
\newblock Number~21. Oxford University Press, 2000.

\bibitem[GR97]{geck1997centers}
Meinolf Geck and Rapha{\"e}l Rouquier.
\newblock Centers and simple modules for {I}wahori-{H}ecke algebras.
\newblock In {\em Finite Reductive Groups: Related Structures and
  Representations}, pages 251--272. Springer, 1997.

\bibitem[GS98]{GoupilSchaefferStructureCoef}
A.~Goupil and G.~Schaeffer.
\newblock Factoring n-cycles and counting maps of given genus.
\newblock {\em Eur. J. Comb.}, 19(7):819--834, 1998.

\bibitem[GS01]{gebhard2001chromatic}
David~D Gebhard and Bruce~E Sagan.
\newblock A chromatic symmetric function in noncommuting variables.
\newblock {\em Journal of Algebraic Combinatorics}, 13(3):227--255, 2001.

\bibitem[HAO07]{hora2007quantum}
Akihito Hora, L~Accardi, and Nobuaki Obata.
\newblock {\em Quantum probability and spectral analysis of graphs}.
\newblock Springer, 2007.

\bibitem[IK99]{Ivanov1999}
V.~Ivanov and S.~Kerov.
\newblock The algebra of conjugacy classes in symmetric groups, and partial
  permutations.
\newblock {\em Zap. Nauchn. Sem. S.-Peterburg. Otdel. Mat. Inst. Steklov.
  (POMI)}, 256(3):95--120, 1999.

\bibitem[IO02]{ivanov2002olshanski}
Vladimir Ivanov and Grigori Olshanski.
\newblock Kerov’s central limit theorem for the {P}lancherel measure on
  {Y}oung diagrams.
\newblock In {\em Symmetric functions 2001: surveys of developments and
  perspectives}, pages 93--151. Springer, 2002.

\bibitem[Iwa64]{iwahori1964structure}
Nagayoshi Iwahori.
\newblock On the structure of a {H}ecke ring of a {C}hevalley group over a
  finite field.
\newblock 1964.

\bibitem[Jac70]{jack1970class}
Henry Jack.
\newblock I.—{A} class of symmetric polynomials with a parameter.
\newblock {\em Proceedings of the Royal Society of Edinburgh. Section A.
  Mathematical and Physical Sciences}, 69(01):1--18, 1970.

\bibitem[Jac72]{jack1972xxv}
Henry Jack.
\newblock Xxv.—{A} surface integral and symmetric functions.
\newblock {\em Proceedings of the Royal Society of Edinburgh. Section A.
  Mathematical and Physical Sciences}, 69(04):347--364, 1972.

\bibitem[Jac87]{jackson1987counting}
David~Martin Jackson.
\newblock Counting cycles in permutations by group characters, with an
  application to a topological problem.
\newblock {\em Transactions of the American Mathematical Society},
  299(2):785--801, 1987.

\bibitem[Jam61]{James1961}
Alan~T. James.
\newblock Zonal polynomials of the real positive definite symmetric matrices.
\newblock {\em Annals of Mathematics}, 74(3):456--469, 1961.

\bibitem[JS12a]{jackson2012character}
David~M Jackson and Craig~A Sloss.
\newblock Character-theoretic techniques for near-central enumerative problems.
\newblock {\em Journal of Combinatorial Theory, Series A}, 119(8):1665--1679,
  2012.

\bibitem[JS12b]{Jackson20121856}
David~M. Jackson and Craig~A. Sloss.
\newblock Near-central permutation factorization and {S}trahov's generalized
  {M}urnaghan–{N}akayama rule.
\newblock {\em Journal of Combinatorial Theory, Series A}, 119(8):1856 -- 1874,
  2012.

\bibitem[JV90a]{jackson1990character}
David~Martin Jackson and TI~Visentin.
\newblock Character theory and rooted maps in an orientable surface of given
  genus: face-colored maps.
\newblock {\em Transactions of the American Mathematical Society},
  322(1):365--376, 1990.

\bibitem[JV90b]{JaVi90}
D.M. Jackson and T.I. Visentin.
\newblock A character theoretic approach to embeddings of rooted maps in an
  orientable surface of given genus.
\newblock {\em Trans. AMS}, 322:343--363, 1990.

\bibitem[Ker93]{kerov1993gaussian}
Serguei Kerov.
\newblock Gaussian limit for the {P}lancherel measure of the symmetric group.
\newblock In {\em CR Acad. Sci. Paris}. Citeseer, 1993.

\bibitem[KP86]{KatrielPaldus}
J.~Katriel and J.~Paldus.
\newblock Explicit expressions for the product of the class of transpositions
  with an arbitrary class of the symmetric group.
\newblock {\em group theoretical methods in Physics, Gilmore, Editor}, 1986.

\bibitem[{Las}08]{Lassalle}
M.~{Lassalle}.
\newblock {A positivity conjecture for Jack polynomials}.
\newblock {\em Math. Res. Lett.}, 15(4):661--681, 2008.

\bibitem[LS77]{logan1977variational}
Benjamin~F Logan and Larry~A Shepp.
\newblock A variational problem for random {Y}oung tableaux.
\newblock {\em Advances in mathematics}, 26(2):206--222, 1977.

\bibitem[LZ04]{lando2004graphs}
Sergei~K Lando and Alexander~K Zvonkin.
\newblock {\em Graphs on surfaces and their applications}, volume~2.
\newblock Springer, 2004.

\bibitem[Mac95]{McDo}
I.G. Macdonald.
\newblock {\em Symmetric functions and Hall polynomials}.
\newblock Oxford Univ. Press, second edition, 1995.

\bibitem[M{\'e}l10]{meliot2010products}
Pierre-Lo{\"\i}c M{\'e}liot.
\newblock Products of {G}eck-{R}ouquier conjugacy classes and the {H}ecke
  algebra of composed permutations.
\newblock {\em DMTCS Proceedings}, (01):921--932, 2010.

\bibitem[M{\'e}l13]{meliot2013partial}
Pierre-Lo{\"\i}c M{\'e}liot.
\newblock Partial isomorphisms over finite fields.
\newblock {\em Journal of Algebraic Combinatorics}, pages 1--54, 2013.

\bibitem[MV11]{morales2011bijective}
Alejandro~H Morales and Ekaterina~A Vassilieva.
\newblock Bijective evaluation of the connection coefficients of the double
  coset algebra.
\newblock {\em DMTCS Proceedings}, (01):681--692, 2011.

\bibitem[OO97a]{1996q.alg.....8020O}
Andrei Okounkov and Grigori Olshanski.
\newblock Shifted {J}ack polynomials, binomial formula, and applications.
\newblock {\em Mathematical Research Letters}, 4:69--78, 1997.

\bibitem[OO97b]{okounkov1997shifted}
Andrei~Yur{'}evich Okounkov and Grigorii~Iosifovich Olshanskii.
\newblock Shifted {S}chur functions.
\newblock {\em Algebra i Analiz}, 9(2):73--146, 1997.

\bibitem[RS06]{rosas2006symmetric}
Mercedes Rosas and Bruce Sagan.
\newblock Symmetric functions in noncommuting variables.
\newblock {\em Transactions of the American Mathematical Society},
  358(1):215--232, 2006.

\bibitem[Sag]{Sage}
Sage mathematical software, version 4.8, http://www.sagemath.org.

\bibitem[Sag01]{sagan2001symmetric}
Bruce~E Sagan.
\newblock {\em The symmetric group: representations, combinatorial algorithms,
  and symmetric functions}, volume 203.
\newblock Springer, 2001.

\bibitem[Sol90]{Solomon1990}
Louis Solomon.
\newblock The {B}ruhat decomposition, {T}its system and {I}wahori ring for the
  monoid of matrices over a finite field.
\newblock {\em Geometriae Dedicata}, 36(1):15--49, 1990.

\bibitem[Sta81]{Stanley1981255}
Richard~P. Stanley.
\newblock Factorization of permutations into n-cycles.
\newblock {\em Discrete Mathematics}, 37(2–3):255 -- 262, 1981.

\bibitem[Ste92]{stembridge1992projective}
John~R Stembridge.
\newblock The projective representations of the hyperoctahedral group.
\newblock {\em Journal of Algebra}, 145(2):396--453, 1992.

\bibitem[Str07]{strahov2007generalized}
Eugene Strahov.
\newblock Generalized characters of the symmetric group.
\newblock {\em Advances in Mathematics}, 212(1):109--142, 2007.

\bibitem[Tou12]{toutarxiv}
Omar Tout.
\newblock Structure coefficients of the {H}ecke algebra of $({S}_{2n}, {B}_n)
  $.
\newblock {\em arXiv preprint arXiv:1212.5375}, 2012.

\bibitem[Tou13]{tout2013structure}
Omar Tout.
\newblock Structure coefficients of the {H}ecke algebra of $({S}_{2n}, {B}_n)$.
\newblock {\em DMTCS Proceedings}, (01):551--562, 2013.

\bibitem[Vas12]{vassilieva2012explicit}
Ekaterina~A Vassilieva.
\newblock Explicit monomial expansions of the generating series for connection
  coefficients.
\newblock {\em DMTCS Proceedings}, 2012.

\bibitem[VK77]{vervsik1977asymptotic}
AM~Ver{\v{s}}ik and SV~Kerov.
\newblock Asymptotic behavior of the {P}lancherel measure of the symmetric
  group and the limit form of {Y}oung tableaux.
\newblock In {\em Dokl. Akad. Nauk SSSR}, volume 233, pages 1024--1027, 1977.

\bibitem[Wal79]{walkup1979many}
David~W Walkup.
\newblock How many ways can a permutation be factored into twoo n-cycles ?
\newblock {\em Discrete Mathematics}, 28(3):315--319, 1979.

\bibitem[Wol36]{wolf1936symmetric}
M.~C. Wolf.
\newblock Symmetric functions of non-commutative elements.
\newblock {\em Duke Mathematical Journal}, 2(4):626--637, 1936.

\bibitem[Yan01]{yan2001representation}
Ning Yan.
\newblock Representation theory of the finite unipotent linear groups.
\newblock 2001.

\end{thebibliography}

\appendix
\backmatter
\end{document}